\documentclass[12pt]{amsart}

\usepackage[foot]{amsaddr}
\usepackage{mathtools}
\mathtoolsset{showonlyrefs}

\usepackage{graphicx}
\usepackage{amsmath,amssymb,amsthm}
\usepackage[english]{babel}

\usepackage[Symbol]{upgreek}

\usepackage{amsfonts}
\usepackage{amsmath}
\usepackage{latexsym}
\usepackage{amscd}

\usepackage{fancyvrb}
\usepackage{xcolor}
\usepackage{shadow}
\usepackage{a4wide}

\usepackage{endnotes}
\usepackage{amsopn}
\usepackage{url}
\usepackage{dsfont}

\usepackage{bbm}

\usepackage{centernot}

\usepackage{enumitem}

\usepackage{tensor}

\usepackage{hyperref}
\hypersetup{colorlinks=true
}

\usepackage[margin=2cm,bottom=2.5cm,top=2.5cm]{geometry}

\numberwithin{figure}{section}
\theoremstyle{plain}
\newtheorem{thm}{Theorem}[section]
\newtheorem{theoreme}{Theorem}[section]
\newtheorem{prop}[thm]{Proposition}

\newtheorem{cor}[thm]{Corollary}
\newtheorem{lemma}[thm]{Lemma}
\numberwithin{equation}{section}

\theoremstyle{remark}

\newtheorem{rmq}[thm]{Remark}
\newtheorem{rem}[thm]{Remark}

\theoremstyle{definition}
\newtheorem{dfn}[thm]{Definition}
\newtheorem{definition}[thm]{Definition}

\newcommand{\eps}{\varepsilon}

\newcommand{\ceps}{\epsilon} \newcommand{\cveps}{\varepsilon}

\newcommand{\hatun}[1]{\overset{\lower.9em\hbox{${\scriptscriptstyle 1 \wedge}$}}{#1}}
\newcommand{\hatdeux}[1]{\overset{\lower.9em\hbox{${\scriptscriptstyle \!\!2 \wedge}$}}{#1}}

\newcommand\R{{\mathbb R}} \newcommand\N{{\mathbb N}}
\newcommand\Z{{\mathbb Z}}

\def\ra{\rightarrow}

 \def\cdotv{\raise 2pt\hbox{,}}

\newcommand\anabla{\mathord{\centernot \nabla}}

\makeatletter
\def\@tvsp{\mathchoice{{}\mkern-4.5mu}{{}\mkern-4.5mu}{{}\mkern-2.5mu}{}}
\def\ltrivert{\left|\@tvsp\left|\@tvsp\left|}
\def\rtrivert{\right|\@tvsp\right|\@tvsp\right|}
\makeatother

\def\ra{\rightarrow}

 \def\cdotv{\raise 2pt\hbox{,}}

\newcommand{\Grond}{\mathcal{G}}
\newcommand{\Prond}{\mathcal{P}}
\newcommand{\Srond}{\mathcal{S}}
\newcommand{\Trond}{\mathcal{T}}
\newcommand{\Arond}{\mathcal{A}}

\newcommand{\Nrond}{\mathcal{N}}

\newcommand{\Ab}{\mathbb{A}}
\newcommand{\Af}{\mathfrak{A}}
\newcommand{\Bf}{\mathfrak{B}}

\newcommand{\tf}{\mathfrak{S}}
\newcommand{\xx}{\mathbbmss{x}}
\newcommand{\yy}{\mathbbmss{y}}

\newcommand{\tprime}{\varrho}
\newcommand{\Tprime}{\Upsilon}
\newcommand{\Trondprime}{\Trond}
\newcommand{\Lp}{\ell}
\newcommand{\Lo}{\mathfrak{L}}
\newcommand{\tiTheta}{\upsilon}

\newcommand{\canonchi}{\boldsymbol{\chi}}
\newcommand{\cutoffchi}{\chi}
\newcommand{\WF}{W\!F}
\newcommand{\qtau}{\tau_{q}}

\usepackage{xcolor}

\makeatletter
\def\mathcolor#1#{\@mathcolor{#1}}
\def\@mathcolor#1#2#3{%
  \protect\leavevmode
  \begingroup
    \color#1{#2}#3%
  \endgroup
}
\makeatother

\begin{document}

\title[Dispersion for the wave equation inside strictly convex domains]{Dispersion for the wave equation inside strictly
  convex domains II: the general case}

\author{Oana Ivanovici${}^{1}$}
\address{${}^{1}$Sorbonne Université, CNRS, LJLL, F-75005 Paris, France} 

\author{Richard Lascar${}^{2}$}

\author{Gilles Lebeau${}^{3}$}
  \address{${}^{2,3}$Universit\'e Côte d'Azur, CNRS, LJAD, France}

  \author{Fabrice Planchon${}^{4,*}$}
  \address{${}^{4}$Sorbonne Université, CNRS, IMJ-PRG F-75005 Paris, France}
  \email{${}^{*}$corresponding author: fabrice.planchon@sorbonne-universite.fr}
  \email{oana.ivanovici@math.cnrs.fr, gilles.lebeau@univ-cotedazur.fr, richard.lascar@univ-cotedazur.fr}

      \thanks{
   O.Ivanovici and F. Planchon were supported by ERC grant ANADEL 757 996.}
\date{}

\begin{abstract}
We consider the wave equation on a manifold $(\Omega,g)$ of dimension $d\geq 2$ with smooth
strictly convex boundary $\partial\Omega\neq\emptyset$, with Dirichlet boundary conditions. We construct a sharp local in time parametrix and then proceed to obtain dispersion estimates: our fixed time decay rate for the Green function exhibits  a $t^{1/4}$
loss with respect to the boundary less case. We precisely describe where and when these losses occur and relate them to swallowtail type singularities in the wave front set, proving that our decay is optimal. Moreover, we derive better than expected Strichartz estimates, balancing lossy long time estimates at a given incidence with short time ones with no loss: for $d=3$, it heuristically means that, on average the decay loss is only $t^{1/6}$.
\end{abstract}

\maketitle

\section{Introduction}\label{intro}
Let us consider the wave
equation on a smooth $d-$dimensional manifold $(\Omega,g)$, with $d\geq 2$, a strictly convex boundary $\partial
\Omega$, and $\Delta_g$ its Laplace-Beltrami operator:
\begin{equation} \label{WE} 
\left\{ \begin{array}{l}
   (\partial^2_t- \Delta_g) u=0,  \;\; \text{in } \Omega
\\ u|_{t=0}=u_0, \; \partial_t u|_{t=0}=u_1,\;
 u|_{\partial\Omega}=0, \end{array} \right.
 \end{equation}
On any smooth
Riemannian manifold without boundary, one may construct an approximate solution, i.e. a parametrix, to any order by microlocal methods. In a suitable patch around $x_{0}\in \Omega$ (within the radius of injectivity at $x_{0}$), such an approximate solution is a Fourier integral operator whose phase is a solution to the eikonal equation. That phase is non degenerate in a suitable way and one recovers pointwise decay estimates for the kernel of such parametrix similar to that for the flat case: let us denote by $e^{\pm it\sqrt{-\Delta_{g}}}$ the
half-wave propagators on $\Omega$ with $\partial \Omega=\varnothing$, and $\varkappa\in C_{0}^\infty
(]0,\infty[)$.  Then, possibly only for (small) finite $|t|$, we have the so-called dispersion estimate,
\begin{equation}\label{disprd}
\|\varkappa(-h^{2}\Delta_{g})e^{\pm it\sqrt{-\Delta_{g}}}\|_{L^1\rightarrow L^{\infty}}\leq C(d)h^{-d}\min\{1,(h/|t|)^{\frac{d-1}{2}}\}.
\end{equation}
Such fixed time decay estimates have been the key tool to obtain other families of estimates, from Strichartz to spectral projector estimates, all of which are of space-time type in (mixed) Lebesgue spaces, for data in Sobolev spaces. These in turn are invaluable tools for studying a large range of problems, from nonlinear waves to localization of eigenfunctions.  

In the presence of a boundary, much less is known on the decay of the wave equation. In fact, before our recent work \cite{Annals} on the wave equation on a model strictly convex domain, there were no known results on fixed time dispersion, even with lesser bounds than \eqref{disprd}. Boundaries induce reflections, and the geometry of broken light rays can be quite complicated. These already cause difficulties in proving propagation of singularities results, and obtaining such results led to major developments \cite{AndMel77, Ivrii78, mesj78,mesj82}, along with constructions of suitable parametrices, see \cite{esk77,meta85p,meta85}. However, such parametrices, while efficient at proving that singularities travel
along the (generalized) bi-characteristic flow, do not seem
strong enough to obtain dispersion, at least in the presence of
gliding rays and the more flexible microlocal energy arguments from \cite{Ivrii78, mesj78,mesj82} do not provide any information on the amplitude of the wave. Nevertheless, outside strictly convex obstacles, parametrices from \cite{meta85p,meta85} were instrumental in matching results from $\R^{d}$:  Strichartz estimates for the wave equation were obtained in \cite{smso95}, and dispersion estimates were finally proved to hold for $d=3$ in \cite{ildispext}). For generic boundaries, some positive results for mixed space-time estimates (\cite{smso06,blsmso08} and references therein) have been obtained using the machinery developed for low regularity metrics (\cite{tat02}):  reflect the
metric across the boundary and consider a boundary less manifold
with a Lipschitz metric across an interface. These arguments require to work on very short time intervals, in order to consider only one reflection (and this, in turn, induces losses when summing time intervals). Counterexamples to the full set of Strichartz estimates inside a strictly convex domain were later constructed in \cite{doi,doi2}, by carefully propagating a cusp singularity along the boundary and across a large number of successive reflections, and these carefully crafted solutions provided hindsight for the parametrix construction on the model domain from \cite{Annals}.

Before stating our main result, let us define strict convexity: our boundary $\partial\Omega\neq \varnothing$ is said to be strictly (geodesically) convex if the induced second fundamental form on $\partial\Omega$ is positive definite. If $\Omega$ is actually a domain in $\R^{d}$ with the identity metric, this definition is equivalent to strict positivity of all principal curvatures at any point of the boundary, and $\Omega$ is a strictly convex domain (it admits a gauge function that is strictly convex.)
\begin{thm}\label{disper}
  
  Let $\varkappa \in C_{0}^{\infty}(]0,+\infty[)$. There exist
$C>0$, $T_{0}>0$ and $a_{0}>0$ such that, uniformly in $a\in ]0,a_{0}]$, $h\in (0,1)$ and $t\in [-T_{0},T_{0}]$, the solution $u_a$   to \eqref{WE} with $(u_0,u_1)=(\delta_a,0)$,  $\delta_a$ being any Dirac mass at distance  $a$ from $\partial\Omega$, is such that
\begin{equation}\label{dispco}
\|\varkappa(-h^2 \Delta_{g})u_{a}(t,\cdot)\|_{L^{\infty}}\leq\frac C {h^{d}} \min\left\{1,\left(\frac h{|t|}\right)^{\frac{d-2}{2}+\frac{1}{4}}\right\}\,.
\end{equation}
\end{thm}
\begin{rmq}
By finite speed of propagation for the wave equation, estimate \eqref{dispco} is local in time {\bf and} space. Hence, compactness for $\Omega$ may be dropped if appropriate uniform assumptions are made on the metric.
\end{rmq}
The dispersion estimate \eqref{dispco} may be compared to
\eqref{disprd}: we notice a $1/4$ loss in the $h/t$ exponent, which
we may informally relate to the presence of caustics in arbitrarily
small times if $a$ is small. Moreover, one of the key features in Theorem \ref{disper} is that $T_{0}$ depends only on the geometry of $\partial\Omega$ and the metric $g$: \eqref{dispco} holds  { uniformly} with respect to both the source point and its distance $a$ to the boundary and the frequency $1/h$. In fact, say for $a=h^{\nu}$, $\nu>0$, there are at most $1/\sqrt a =h^{-\nu/2}$ reflections, and caustics in between them, as we will see later; so in the large frequencies regime $h\ra 0$, we have to deal with an increasingly large number of caustics, even to travel a small distance over a small time $T_{0}$. These caustics occur because optical rays are no longer
diverging from each other in the normal direction, where less
dispersion occurs when compared to the $\mathbb{R}^d$ case. 
In fact, we
can track caustics and therefore Theorem \ref{disper} is optimal.
\begin{thm}\label{disperoptimal}
Let  $u_a$ be the solution to \eqref{WE} with data
$(u_0,u_1)=(\delta_a,0)$. Let $h\in (0,1)$ and $a\geq h^{1/3}$. There exist a constant $C>0$, such that for all $\vartheta\in \mathbb{S}^{d-2}$, there exist a finite sequence
$(t_{n},x_{n},y_{n})_n$, $1\leq n\leq a^{-1/2}$ with $d(x_{n},\partial\Omega)\sim a$, $y_{n}/|y_{n}|\sim \vartheta$, such that
\begin{equation}\label{dispcooptimal}
h^{-d}(h/t_{n})^{\frac{d-2}{2}}n^{-1/4} a^\frac 1 8 h^{1/4}\sim a^\frac
1 4 h^{-d}(h/t_n)^{\frac{d-2}{2}+\frac 1 4}\leq C
 |\varkappa(-h^2 \Delta_{g})u_{a}(t_n,x_{n},y_{n})|\,.
\end{equation}
\end{thm}
As a byproduct, we get that even for $t\in ]0,T_{0}]$ with $T_{0}$ small,
the $1/4$ loss is unavoidable for $a$ comparatively small to $T_0$ and
independent of $h$. Later this optimal loss will be related to swallowtail type singularities in the wave
front set of $u_a$. 
\begin{rmq}
There is nothing specific about the cosine part of the wave propagator in Theorem \ref{disper} and \ref{disperoptimal}. Both hold equally true if one replaces $\varkappa(-h^2 \Delta_{g})u_{a}(t,x)$ by the half-wave propagators $\varkappa(hD_{t})e^{\pm it\sqrt{-\Delta_{g}}}\delta_{a}$ with $\varkappa\in C_{0}^\infty(\R^*)$.
\end{rmq}
As a consequence of more elaborate estimates that lead to \eqref{dispco}, we obtain improved Strichartz estimates when compared to those that routinely follow from fixed time dispersion.
\begin{thm}\label{thStri}
 Let $d\geq 3$ and $u$ be a solution of \eqref{WE} on a manifold $\Omega$ with strictly convex boundary. Then there exist $T$ such that for all $\varepsilon>0$, there exists $C_{T,\varepsilon}$ such that
 \begin{equation}
   \label{eq:SEF}
\|u\|_{L^q(0,T) L^r(\Omega)}\leq
C_{T,\varepsilon} \bigl(\,||u_0||_{\dot{H}^{\beta}(\Omega)} +
||u_1||_{\dot{H}^{\beta-1}(\Omega)} \bigr)\,,
\end{equation}
where $\beta=d/2-1/q-d/r$ (scaling condition) and $(d,q,r)$ such that $q\geq 2$ ($q\neq 2$ for $d=3$),
\[
\frac{1}{q}\leq\Bigl(\frac{d-1}{2}-\upgamma(d)\Bigr)\Bigl(\frac{1}{2}-\frac{1}{r}\Bigr)\,, \text{ with } \upgamma(d)=\frac 1 4-\frac 1 {4d}+\varepsilon=\frac 1 6+\frac 1 4\Bigl(\frac 1 3-\frac 1 d\Bigr)+\varepsilon\,.
\]
\end{thm}
In dimension $d=2$ the known range of admissible indices for which
sharp Strichartz are already known to hold is in fact slightly larger, see
\cite{blsmso08} where $\upgamma(2)=1/6$ (which we may recover with our argument).  Especially noteworthy is $d=3$, for which we get $\upgamma(3)= 1/6+\varepsilon$: such a loss corresponds heuristically to a fixed time dispersion \eqref{dispco} where the $1/4$ loss would be replaced by a $1/6$ loss. In dimensions $d\geq 3$, Theorem
\ref{thStri} improves the known range of indices for which Strichartz estimates hold, and it does so in a uniform way with respect to dimension, in
contrast to \cite{blsmso08}, where $\upgamma(3)=2/3$ and $\upgamma(d)=(d-3)/2$ for $d\geq 4$. The results in \cite{blsmso08} however apply to any domain or manifold with non-empty boundary. 

In the negative direction, counterexamples from \cite{doi,doi2} prove that $\upgamma(d)\geq 1/12$, for $d=2,3,4$. In other recent works \cite{ILP3}, \cite{ILP4}, on the model domain, both positive and negative results for $d=2$ are pushed further. Estimates \eqref{eq:SEF} are proved to hold with $\upgamma(2)= 1/9$ ; improvements on counterexamples yield $\upgamma(2) \geq 1/10$. These results extend beyond the model case for $d=2$, and provide similar improvements in higher dimensions; these extensions, for the general case, will be addressed elsewhere, as they require significant new developments that are out of scope here.

In the present work, we mainly focus on constructing a sharp parametrix for the wave equation \eqref{WE}, providing optimal bounds on the amplitude of the wave, including at a discrete set of caustics of swallowtail type that increase to arbitrarily large numbers when the source gets closer and closer to the boundary. While a natural outcome of this parametrix is optimal dispersion bounds, we believe that such a sharp parametrix will prove useful for a broad range of applications beyond the study of dispersive effects and localization of eigenfunctions, including sharp quantitative versions of propagation of singularities results that are of importance in control theory. 

We conclude this introduction with a brief overview of the content in the next sections.
\begin{itemize}[leftmargin=5.5mm]
\item The second section is devoted to building our parametrix for the wave propagator, which is the key tool to prove Theorems \ref{disper} and \ref{disperoptimal}. While one may think of \cite{Annals} as inspirational, its inner knowledge is by no way a prerequisite and the present construction differs significantly for several reasons we briefly outline: unlike in the model case, we lack an explicit spectral representation. We therefore need to construct quasi-modes, and for this we rely on a parametrix for the Helmholtz equation (see \cite{meta} which relies crucially on \cite{mel76}). Using the Airy-Poisson formula that we introduced in \cite{ILP-XEDP}, we then obtain a parametrix, both as a ``spectral'' sum and its counterpart after Poisson summation. One obvious benefit from this approach is that the Dirichlet boundary condition holds easily, unlike in \cite{Annals}. Moreover, the Poisson sum turns out to coincide with the carefully constructed sum of reflected waves in \cite{Annals}, as each term has essentially the same phase (in the model case). The present construction is therefore a sophisticated version of the method of images, which was our inspiration for constructing suitably matching incoming and outgoing waves in between consecutive reflections in \cite{Annals} (in turn drawing upon \cite{doi}). An additional benefit is that our parametrix holds for any $a$ and $h$: we extend the reflected waves construction to the range $h^{2/3}<a<h^{4/7}$ (a crucial tool in further improvements alluded to earlier, see \cite{ILP3}). The range $0<a\leq h^{2/3}$ requires to properly define gallery modes from the quasi-modes and prove that their decay properties are uniform with respect to their discrete parameter, at least in a range useful for our purpose. To our knowledge, these gallery modes had never been defined in such a uniform way in the general case before now; then, one has to carefully construct the initial data by decomposing over the gallery modes, a delicate issue that was notably absent from the model case.
\item The third section deals with dispersion estimates for reflected waves. There the analysis of the oscillatory integrals follows \cite{Annals} in spirit but it departs from it on several counts. We can no longer reduce the higher dimensional case to $d=2$ by rotational invariance (i.e., the underlying model case is no longer isotropic). For $a<h^{4/7}$, we need to estimate both the size of each wave and their overlap, which is no longer bounded: we observe that after a very large number of reflections, waves start to exhibit dispersion along the tangential variable. We therefore obtain   bounds that are sharper and cover an extended region when compared to \cite{Annals}.
  \item In the fourth section, for $a\leq h^{2/3}$, we use a mix of dispersion estimates on each gallery mode, the spectral sum, and Poisson summation on the worst terms to obtain a sharper decay than in \cite{Annals}, thereby proving that the worst decay (with a $1/4$ loss) really only happens when $h^{1/3}< a < 1$, whereas a lesser $1/6$ loss is seen below $h^{1/3}$, essentially due to cusp propagating and accumulating.
  \item The fourth section deals with Strichartz estimates and how to derive Theorem \ref{thStri}, taking advantage of the previously introduced decomposition with respect to angles of incidence, following \cite{ILP3}, combined with short time Strichartz estimates (similar to those from \cite{blsmso08}).
  \item Finally, the appendix provides hindsight on how to obtain the key properties (and required uniformity, in a suitable sense) of the generating function associated to the equivalence of glancing hypersurfaces (\cite{mel76}) in our setting.
\end{itemize}
In the remaining of the paper, $A\lesssim B$ means that there exists a constant $C$ such that $A\leq CB$ and this constant may change from line to line but is independent of all parameters. It will be explicit when (very occasionally) needed. Similarly, $A\sim B$ means both $A\lesssim B$ and $B\lesssim A$.
\section{A parametrix construction}\label{parconstruction}
By finite speed of propagation, we may work locally near the boundary and chose boundary normal coordinates $(x,y)$ on $\Omega$, with $x>0$ on $\Omega$, $y\in\mathbb{R}^{d-1}$ such that $\partial\Omega=\{(0,y): y\in \mathbb{R}^{d-1}\}$ (these coordinates may be interpreted as Fermi coordinates relative to the hypersurface that is the boundary); local coordinates on $\Omega\times\mathbb{R}_t$ are then $(x,y,t)$. Local coordinates on the base induce local coordinates on the cotangent bundle, namely $(x,y,t,\xi,\eta,\mathcolor{red}{\mathbf{\tau}})$ on $T^*(\Omega\times\mathbb{R}_t)$. The corresponding local coordinates on the boundary are $(y,t,\eta,\mathcolor{red}{\mathbf{\tau}})$. In this coordinates (and up to conjugation by a non vanishing smooth
factor $e_{g}(x,y)$), the Laplacian $\Delta_g$ can be written as (\cite[III, Appendix C]{Hormander})
\begin{equation}
  \label{eq:laptilde}
 \Delta=e_{g}^{-1} \Delta_g e_{g}=\partial^2_x+R(x,y,\partial_y)\,.
\end{equation}
We assume that the boundary is everywhere strictly (geodesically) convex: for every point $(0,y_0)\in\partial\Omega$ and every $(0,y_{0},0,\eta_0)\in T^*\Omega$ with $\eta_{0}\neq 0$,
\begin{gather*}
  \{\xi^2+R(x,y,\eta),x\}(0,y_0,0,\eta_{0})=0\,,\\
 \{\{\xi^2+R(x,y,\eta),x\},\xi^2+R(x,y,\eta)\}(0,y_0,0,\eta_{0})=2\partial_xR(0,y_0,\eta_0)>0,
\end{gather*}
where $\{.,.\}$ denotes the Poisson bracket (see \cite[III, 24.3]{Hormander}). We assume (without loss of generality) that $y_0=0$, hence $\kappa_0=(0,0,0,\eta_0)$. On the boundary
and for $(0,y)$ near $(0,0)$, the metric reads $\xi^2+\sum_{j,k=1}^{d-1}R^{j,k}(0,y)\eta_j\eta_k$; using again \cite[III, Appendix C]{Hormander}, we assume moreover that $(R_{j,k}(0,0))_{j,k}$ is the identity matrix, and define
\begin{equation}
  \label{eq:R01}
  R_0(y,\partial_y)  :=R(0,y,\partial_y)=\sum_{j}\partial_{y_j}^2+O(|y|)\,,\quad 
  R_1(y,\partial_y)  :=\partial_x R(0,y,\partial_y)=\sum_{j,k}R_1^{j,k}(y)\partial_{y_j}\partial_{y_k}\,.
\end{equation}
Recall that strict convexity for $\partial\Omega$ is equivalent to $R_1$ being elliptic (the associated quadratic form is  positive definite). Define our model Laplacian $\Delta_{M}$ and (Fourier) multipliers $q, {\qtau}$
\begin{equation}
  \label{eq:LapM}
\Delta_M=\partial^2_x+\sum_{j}\partial_{y_j}^2+ x\sum_{j,k}R_1^{j,k}(0)\partial_{y_j}\partial_{y_k}\,,\, q(\eta)=\sum_{j,k}R_1^{j,k}(0)\eta_j \eta_k\,,\,  {\qtau}(\omega,\eta)=\sqrt{ |\eta|^2+\omega q(\eta)^\frac 2 3}\,.
\end{equation}
Later we will use various functions of variables $(x,y,\eta,\omega,\sigma)$ (where some variables may be omitted depending on context and both new variables $\omega, \sigma \in \R$)
that will be defined in a conic neighborhood of the set
\begin{equation}
  \label{eq:N0}
 N_0=\{x=0\,,\,\,y=0\,,\,\,\omega=0\,,\,\,\sigma=0\,,\,\,\eta\in \R^{d-1}\setminus\{0\}\}\,.
\end{equation}
Such a function $f$ is said to be homogeneous of degree $k$ if
\[
f(x,y, \lambda \eta, \lambda^{2/3}\omega, \lambda^{1/3}\sigma)=\lambda^{k}
f(x,y,\eta,\omega,\sigma)\,.
\]
\begin{definition}
    A symbol $a(x,y,\eta,\omega,\sigma)$ is of order $m$ and type $((1, 2/3,1/3),0)$ if 
\[
\forall \beta=(\beta_0,\beta_1,\beta_2,\beta_3) \quad \exists C_{\beta} \quad |\partial^{\beta_0}_{(x,y)}\partial^{\beta_{1}}_{\eta} \partial^{\beta_{2}}_{\omega}\partial^{\beta_{3}}_{\sigma}a(x,y,\eta,\omega,\sigma)|\leq C_{\beta}(1+|\eta|)^{m-|\beta_1|-\frac 23|\beta_2|-\frac 13|\beta_3|}.
\]

\end{definition}
We now recall the Airy function, defined for $z\in \R$ as the oscillatory integral
\begin{equation}
  \label{eq:15}
  Ai(-z)=\int e^{i(\frac{\sigma^{3}} 3-\sigma z)}\,d\sigma\,.
\end{equation}
The choice of $\sigma$ as an integration variable is consistent with our later use of oscillatory integrals with related phases and with symbols within the class we just defined; and $\omega$ may be chosen as a zero of the Airy function.

Constructing a parametrix near glancing or gliding rays has a long history, starting with Andersson-Melrose \cite{AndMel77} and Eskin \cite{esk77}. We also refer to Melrose and Taylor (\cite{meta} and references therein) and Zworski \cite{Zwo} for the exterior case. We now state an important theorem for our purposes. To our knowledge, this result is stated
(for glancing rays) in \cite{Zwo} and a proof is available in \cite{meta}.
\begin{theoreme}\label{thmMelrose}[Melrose-Taylor, Zworski]
Let ${\qtau}(\omega,\eta)$ be defined in \eqref{eq:LapM}.  There exist a neighborhood $U$ of $(x,y,\eta,\omega)=(0,0,1,0)$, phase functions $\psi(x,y,\eta,\omega)$ and
  $\zeta(x,y,\eta,\omega)$, symbols $p_0(x,y,\eta,\omega)$ and
  $p_1(x,y,\eta,\omega)$ and a function $e_0(x,y,\eta,\omega)$  such that
  \begin{itemize}[leftmargin=5.5mm]
  \item the function $\psi$ is homogeneous of degree $1$, $(\nabla_y(\partial_{\eta_j}\psi))_{j=1,\cdots,d-1}$ are linearly independent;
  \item the function $\zeta$ is homogeneous of degree $2/3$, and
    \begin{equation}
      \label{eq:16}
      \zeta=\omega-xq(\eta)^{1/3} e_0(x,y,\eta/|\eta|,\omega/q(\eta)^{1/3})\,,
    \end{equation}
i.e. $e_{0}$ is homogeneous of degree $0$;
\item the symbols    $p_0$, $p_1$ (which do not depend on $\sigma$) and $p_{0}+\sigma p_{1}$ are of order $0$ and type $((1,2/3,1/3),0)$;
    \item the phase functions $\psi$ and $\zeta$ are solutions to the following eikonal equations 
    \begin{equation} \label{systeikeq} 
<\nabla_{(x,y)}\psi,\nabla_{(x,y)}\psi>+\zeta<\nabla_{(x,y)}\zeta,\nabla_{(x,y)}\zeta>={\qtau}^2(\omega,\eta),\quad    <\nabla_{(x,y)}\psi,\nabla_{(x,y)}\zeta>=0.
 \end{equation}
 Here $<.,.>$ is the symmetric bilinear form obtained by polarization of the principal symbol $\xi^2+R(x,y,\eta)$ of the operator $\Delta$ (which is a second order homogeneous polynomial).
\item Define the function $G(x,y,\eta,\omega)$ to be
  \begin{equation}
    \label{eq:defG}
    G(x,y,\eta,\omega)=e^{i\psi(\cdot)}(p_0(\cdot) Ai(-\zeta(\cdot))+ i p_1(\cdot ) q^{-1/6}(\eta) Ai'(-\zeta(\cdot)))\,;
  \end{equation}
  Then the following equation holds in $U$,
  \begin{equation}
  \label{eq:eqG}
  -\Delta G={\qtau}^2 G+O_{C^\infty}({\qtau}^{-\infty})\,,
\end{equation}
with $p_0$, $e_0$ elliptic symbols, $e_0>0$ near any $(0,0,\eta,0)$ with $\eta\in \R^{d-1}\setminus\{0\}$ and $p_{1}=0$ on $\{x=0\}$. We call $G$ a quasimode in $U$.
  \end{itemize}
\end{theoreme}

\begin{rmq}\label{rmqzeta}
Constructing an asymptotic solution to equation \eqref{eq:eqG} with ansatz \eqref{eq:defG} is a classical result in geometrical optics. However, that such a solution can be constructed with $\zeta\vert_{x=0}=\omega$ independent of $(y,\eta)$ is delicate and is a key point of the result. Moreover, that the construction can be done such that the symbol $p_{1}$ in front of $Ai'$ in \eqref{eq:defG} vanishes on the boundary $\{x=0\}$ is not obvious and proved in \cite[Paragraph. 4.4, formula 4.4.6 and paragraph 7.1]{meta}.
\end{rmq}
\begin{rmq}
Near glancing rays, the same theorem holds true with $Ai(e^{\pm i\pi/3}\cdot)$ instead of
$Ai(\cdot)$. As $Ai(e^{\pm i\pi/3}s)$ does not vanish for real values of $s$, one may define outgoing and incoming parametrix for the wave operator with given Dirichlet data on
the boundary. Near gliding rays, which is our case, the Airy function may vanish and the same methodology no longer applies. In
\cite{esk77}, Eskin deals with this difficulty by a conjugation of the
wave operator by $e^{-\zeta t}$, replacing $\mathcolor{red}{\mathbf{\tau}}$ by $\mathcolor{red}{\mathbf{\tau}}-i\zeta$, therefore avoiding zeros of the Airy function. While one may then prove propagation of singularities, it is unknown (and unlikely) to be enough for dispersive estimates near gliding rays.
\end{rmq}
Let us now briefly review how to prove Theorem \ref{thmMelrose}.
First, observe that Melrose's classification Theorem for glancing
hypersurfaces (see \cite{mel76}) applies, in the non-homogeneous setting, locally near any point in the set $\Sigma_0$, defined as
\begin{equation}
  \label{eq:18}
\Sigma_{0}  =\{(X_{M},Y_{M},\Xi,\Theta)\,:\,\, X_{M}=0\,,\,\,Y_{M}=0\,,\,\,
\Xi=0\,,\,\,|\Theta|=1\}\,.
\end{equation}
Therefore,  there exists a canonical transform $\canonchi_M$ such that, near $\Sigma_0$
\begin{equation}
  \label{eq:melrose}
  \canonchi_M(\{X_{M}=0\})=\{x=0\},\;\;\;\canonchi_{M}(\{\Xi^2+|\Theta|^2+X_{M}q(\Theta)=1\})=\{\xi^2+R(x,y,\eta)=1\}\,.
\end{equation}
The crucial fact that such a canonical transformation $\canonchi_{M}$ may actually be defined in a neighborhood of $\Sigma_{0}$ then follows from the transversality of the Hamiltonian flow with respect to $\Sigma_{0}$. The following proposition will be essential for us.
\begin{prop}\label{lemgamma}
  The generating function for $\canonchi_M$ may be written as $\varphi_{\Gamma}(x,y,\Xi,\Theta)=x\Xi+y\Theta+\Gamma(x,y,\Xi,\Theta)$, where $\Gamma(0,y,\Xi,\Theta)$ is independent of $\Xi$ (as $\canonchi_M(\{X_{M}=0\})=\{x=0\}$) and
\begin{equation}
  \label{eq:GAB}
\Gamma(x,y,\Xi,\Theta)=B_{\Gamma}(y,\Theta)+xA_{\Gamma}(x,y,\Xi,\Theta).
\end{equation}
The transformation $\canonchi_M$ is such that $\canonchi_M(\partial_{\Xi}\varphi_{\Gamma},\nabla_{\Theta}\varphi_{\Gamma},\Xi,\Theta)=(x,y,\partial_x\varphi_{\Gamma},\nabla_y\varphi_{\Gamma})$, and therefore generated by the following relations:
\begin{equation} \label{genchi} 
\left\{ \begin{array}{l}
X_{M}=x+x\frac{\partial A_{\Gamma}}{\partial \Xi}(x,y,\Xi,\Theta)\,,\,\,Y_{M}=y+\frac{\partial B_{\Gamma}}{\partial \Theta}(y,\Theta)+x\frac{\partial A_{\Gamma}}{\partial \Theta}(x,y,\Xi,\Theta)\\ 
\xi=\Xi+A_{\Gamma}(x,y,\Xi,\Theta)+x\frac{\partial A_{\Gamma}}{\partial x}(x,y,\Xi,\Theta)\,,\,\,
\eta=\Theta +\frac{\partial B_{\Gamma}}{\partial y}(y,\Theta)+x\frac{\partial A_{\Gamma}}{\partial y}(x,y,\Xi,\Theta)\,.
\end{array} \right.
 \end{equation}
There exists an elliptic symbol $p(x,y,\eta,\omega,\sigma)$ of order $0$ and type $((1,2/3,1/3),0)$
with support near $N_0$ (from \eqref{eq:N0}) 
and
\begin{equation}\label{eq:Gosc}
G(x,y,\eta,\omega):=\frac{1}{2\pi} e^{-i{\qtau} B_{\Gamma}(0,\eta/{\qtau})}\int e^{i(y\cdot \eta+\frac {s^3} 3 +s(xq^{\frac 13}(\eta)-\omega)+{\qtau}\Gamma(x,y,s q^{\frac 13}(\eta)/{\qtau},\eta/{\qtau}))}p(x,y,\eta,\omega,s)ds
\end{equation}
such that Theorem \ref{thmMelrose} holds with this $G$.
\end{prop}
\begin{proof}
  We postpone to the appendix how to obtain the generating function $\Gamma$ and the $B_{\Gamma}(y,\Theta)$ and $A_{\Gamma}(x,y,\Xi,\Theta)$ terms.  The function $\Gamma(x,y,\Xi,\Theta)$ is not unique:  the group of canonical transformations under which the model $\{X_{M}=0,\Xi^2+|\Theta|^2+X_{M}q(\Theta)=1\}$ is invariant is non trivial and includes any symplectic transformation $(X_{M},Y_{M},\Xi,\Theta)\rightarrow (X_{M},Y_{M}+h'(\Theta),\Xi,\Theta)$, where $h$ is any function defined near the set $\{|\Theta|=1\}$. Thus we may replace $\Gamma(x,y,\Xi,\Theta)$ by $\Gamma(x,y,\Xi,\Theta)+h(\Theta)$. We therefore assume that $B_{\Gamma}(0,\Theta)=0$, which is equivalent to $\Gamma(0,0,0,\Theta)=0$. This explains the factor $e^{-i{\qtau} B_{\Gamma}(0,\eta/{\qtau})}$ in \eqref{eq:Gosc}. Let us now verify that there exists a symbol $p(x,y,\eta,\omega,\sigma)$ such that $G$ defined by \eqref{eq:Gosc} is such that $(-\Delta-{\qtau}^2)G=O({\qtau}^{-\infty})$ near $x=0,y=0$.

 We will work microlocally near the set $\Sigma_0$, defined in \eqref{eq:18}, in the semiclassical setting with $0<h<1$ as small parameter. Set $P=-h^{2}\Delta-1$, $p_2(x,y,\xi,\eta)=\xi^2+R(x,y,\eta)-1$,
 $P_M=h^{2}(D^2_{X_{M}}+D^2_{Y_{M}}+Xq(D_{Y_{M}}))-1$, and $ p_{M,2}(X_{M},Y_{M},\Xi,\Theta)=\Xi^2+|\Theta|^2+X_{M}q(\Theta)-1$, where $D_{X_{M}}=\frac 1i \partial_{x_{M}}$, $D_{Y_{M}}=\frac 1i\nabla_{Y_{M}}$. Let $W\subset \tilde W$ be small neighborhoods of $\Sigma_0$. Let $\chi(X_{M},Y_{M},\Xi,\Theta)\in C^{\infty}_0(\tilde W)$ such that $\chi=1$ in a neighborhood of $\overline{W}$ and such that $\chi|_{X_{M}=0}$ is independent of $\Xi$. Let $\mathcal{G}_{h}$ be the following semiclassical Fourier integral operator
\[
\mathcal{G}_{h}(F)(x,y):=\frac{1}{(2\pi h)^d}\int e^{\frac ih(\varphi_{\Gamma}(x,y,\Xi,\Theta)-X\Xi-Y\Theta)}\chi(X,Y,\Xi,\Theta)F(X,Y)dXdYd\Xi d\Theta,
\]
where $\varphi_{\Gamma}$ has been introduced in Proposition \ref{lemgamma}. Then for any semiclassical operator $Q_M$ such that $Q_{M}= \sum_{n\geq 0} (-ih)^{n}Q_{M,n}(X_{M},Y_{M},{h}D_{X_{M}},{h}D_{Y_{M}})$ defined on $\tilde W$, there exists a semiclassical operator $Q= \sum_{n\geq 0} (-ih)^{n}Q_{n}(x,y,{h}D_x,{h}D_y)$ defined on $\canonchi_M(\tilde W)$ and unique on $\canonchi_M(W)$ such that one has 
\[
\WF_{h}\Big((Q\mathcal{G}_{{h}}-\mathcal{G}_{{h}}Q_M)(F)\Big)\cap \canonchi_M(W)=\emptyset,
\]
where $\WF_{{h}}$ denotes the semiclassical wavefront set (see \cite{ZworBook}) and $F$ denotes any function such that $\WF_{{h}}(F)\subset W$. Moreover, $Q_0(\canonchi_M(X_{M},Y_{M},\Xi,\Theta))=Q_{M,0}(X_{M},Y_{M},\Xi,\Theta)$ for $(X_{M},Y_{M},\Xi,\Theta)\in W$. Taking $Q_M=P_M$ and using \eqref{eq:melrose}, $Q$ has simple characteristics on the set $p_2(x,y,\xi,\eta)=0$ near $\canonchi_M(\Sigma_0)$. Thus, if $W\subset \tilde W$ are small enough, there exists a function $l(x,y,\xi,\eta)\in C^{\infty}_0(\canonchi_M(\tilde W))$, which is elliptic on $\canonchi_M(W)$, such that $l(x,y,\xi,\eta)p_2(x,y,\xi,\eta)=Q_0(x,y,\xi,\eta)=p_{M,2}\circ \canonchi_M^{-1}(x,y,\xi,\eta)$ in a neighborhood of $\overline{\canonchi_M(W)}$. Set $L=l(x,y,{h}D_x,{h}D_y)$, then there exists $R_M$ with $R_{M}= \sum_{n\geq 0} (-i{h})^{n}R_{M,n}(X_{M},Y_{M},{h}D_{X_{M}},{h}D_{Y_{M}})$ such that
\begin{equation}\label{g2}
\WF_{{h}}\Big((LP\mathcal{G}_{{h}}-\mathcal{G}_{{h}}(P_M-i{h}R_M))(F)\Big)\cap \canonchi_M(W)=\emptyset.
\end{equation}
We now exhibit a suitable $F=F_{\theta,p_{q},{h}}(X_{M},Y_{M})$ as an oscillatory integral with symbol $p_{q}$ to be chosen later and $\theta\in\mathbb{R}^{d-1}$ with $|\theta|$ close to $1$. Define a function
\begin{equation}
  \label{eq:17}
  \alpha_{q}(\theta)=\frac{1-|\theta|^2}{q^{2/3}(\theta)}\,,
\end{equation}
let $p_{q}$ be 
such that $p_{q}(X_{M},Y_{M},s,\theta,{h^{-1}})= \sum_{n\geq 0}(-i{h})^{n}p_{q,n}(X_{M},Y_{M},s,\theta)$, compactly supported near $X=0$, $Y=0$, $s=0$, with $p_{q,n}$ homogeneous of order $0$ and $|\partial^{\beta} p_{q,n}|\leq C_{\beta}$, and set
\[
F_{\theta,p_{q},{h}}(X_{M},Y_{M}):=\frac{1}{2\pi h^{1/3}}\int e^{\frac i{h}(Y\cdot\theta+s^3/3+s(X_{M}q^{1/3}(\theta)-\alpha_{q}(\theta)))}p_{q}(X_{M},Y_{M},s,\theta,{h^{-1}})\,ds\,,
\]
where $Y_{M}\cdot\theta$ denotes the scalar product in $\mathbb{R}^{d-1}$. 
Define the model Lagrangian submanifold $\Lambda_{M,\theta}$  
\[
\Lambda_{M,\theta}:=\Big\{(X_{M},Y_{M},\Xi,\Theta=\theta): \exists s\in\mathbb{R}\text{ such that } X_{M}=q(\theta)^{-1/3}(\alpha_{q}(\theta)-s^2), \Xi=sq(\theta)^{1/3}\Big\},
\]
then one has $\WF_{{h}}(F_{\theta,p_{q},{h}})\subset \Lambda_{M,\theta}$ and $\Lambda_{M,\theta}$ is contained in the characteristic set of $P_M$, defined by $p_{M,2}(X_{M},Y_{M},\Xi,\Theta)=0$, which is the same as the characteristic set of $P_M-i{h}R_M$ as $i{h}R_M$ is a lower order term. Let $V\subset W$ be a small neighborhood of $p_{q,0}$. By solving transport equations, we can select the symbol $p_{q}$ to be elliptic on $X_{M}=0,Y_{M}=0,s=0,|\theta|=1$, such that for all $\theta\in\mathbb{R}^{d-1}$ with $|\theta|$ close to $1$ one has
\[
\WF_{{h}}\Big((P_M-{i{h}}R_M)(F_{\theta,p_{q},{h^{-1}}})\Big)\cap V=\emptyset, \quad \WF_{{h}}(F_{\theta,p_{q},{h}})\subset W.
\]
We now set $G_{\theta,p_{q},{h}}(x,y):=\mathcal{G}_{{h}}(F_{\theta,p_{q},{h}})(x,y)$. Using \eqref{g2} and the ellipticity of $L$ on $\chi_M(V)$ yields
\begin{equation}\label{gl-1}
\WF_{{h}}\Big(P(G_{\theta,p_{q},{h}})\Big)\cap \canonchi_M(V)=\emptyset, \quad \WF_{{h}}(G_{\theta,p_{q},{h}})\subset \canonchi_M(\WF_{{h}}(F_{\theta,p_{q},{h}}))\subset \canonchi_M(W).
\end{equation}
Moreover, we may write $G_{\theta,p_{q},{h}}(x,y)$ as follows
\[
\frac 1 {(2\pi h)^d} \frac 1 {2\pi {h}^{1/3}}\int e^{\frac i{h}\Phi_{\theta}(x,y,X_{M},Y_{M},\Xi,\Theta,s)}\chi(X_{M},Y_{M},\Xi,\Theta)p_{q}(X_{M},Y_{M},s,\theta,{h^{-1}})\,dX_{M}dY_{M}d\Xi d\Theta ds\,,
\]
where the phase is defined as
\begin{equation}
  \label{eq:19}
\Phi_{\theta}(x,y,X_{M},Y_{M},\Xi,\Theta,s)=\varphi_{\Gamma}(x,y,\Xi,\Theta)-X_{M}\Xi-Y_{M}\cdot\Theta+Y_{M}\cdot\theta+s^3/3+s(X_{M}q^{1/3}(\theta)-\alpha_{q}(\theta))\,,
\end{equation}
and we recall that $\varphi_{\Gamma}(x,y,\Xi,\Theta)=x\Xi+y\cdot\Theta+\Gamma(x,y,\Xi,\Theta)$ and we required $\Gamma(0,0,0,\Theta)=0$. Since at $(x,y)=(0,0)$ we have $\varphi_{\Gamma}(0,0,0,\Theta)=0$ ($\Gamma|_{x=0}$ is independent of $\Xi$), $\Phi_{\theta}(0,0,X_{M},Y_{M},\Xi,\Theta,s)$ is explicit and we easily check that it has an unique non degenerate stationary point in the variables $X_{M},Y_{M},\Xi,\Theta$ at $X_{M,c}=0,Y_{M,c}=0,\Xi_c=sq^{1/3}(\theta),\Theta_c=\theta$. Therefore, for $(x,y)$ close to $(0,0)$, the phase function $\Phi_{\theta}$ also has a unique non degenerate critical point in these variables, such that $\Xi_c=sq^{1/3}(\theta)$ and $\Theta_c=\theta$; the critical value of the phase $\Phi_{\theta}$, that we denote $\phi_{\theta}(x,y,s)$ is given by
\[
\phi_{\theta}(x,y,s)=y\cdot\theta+\Gamma(x,y,sq^{1/3}(\theta),\theta)+s^3/3+s(xq^{1/3}(\theta)-\alpha_{q}(\theta)).
\]
Using \eqref{eq:GAB}, we have for the stationary point $X_{M,c}=x(1+\partial_{\Xi}A_{\Gamma})$; stationary phase provides a symbol $\tilde p_{q} $ such that $\tilde p_{q}= \sum_{n\geq 0}(-i{h})^{n}\tilde p_{q,n}(x,y,s,\theta)$, elliptic on $x=0,y=0,s=0, |\theta|=1$  and 
\begin{equation}
  \label{g3}
  G_{\theta,p_{q},{h}}(x,y)=\frac{1}{2\pi h^{1/3}} \int
  e^{\frac i{h}\phi_{\theta}(x,y,s)} \tilde p_{q}(x,y,s,\theta,{h^{-1}}) ds\,.
\end{equation}
Notice that $\WF_{{h}}(G_{\theta,p_{q},{h}})\subset \canonchi_M(\Lambda_{M,\theta})$, $\chi_M^{-1}(\{x=0\})=\{X=0\}$ and $A_{\Gamma}(0,0,0,\theta)=0$ (and also $\nabla_yB_{\Gamma}(0,\theta)=0$ for $|\theta|=1$). From \eqref{genchi} we thus get $\{ |\theta|=1 \text{ and }(0,0,\xi,\eta)\in \WF_{{h}}(G_{\theta,p_{q},{h}})\}$ if and only if $\{\xi=0\text{ and }\eta=\theta\}$. Together with \eqref{gl-1} which gives $\WF_{{h}}(P(G_{\theta,p_{q},{h}}))\cap \canonchi_M(V)=\emptyset$, we proved that there exists a small neighborhood $U$ of $(x,y)=(0,0)$ such that for all $|\theta|\sim 1$, one has $\WF_{{h}}(P(G_{\theta,p_{q},{h}}))\cap \{(x,y)\in U\}=\emptyset$.

Taking $\theta=h\eta$ and $ h^{2/3}\omega=\alpha_{q}(\theta)$, we obtain by direct computation that $1/h={\qtau}(\eta,\omega)$,  and \eqref{eq:eqG} holds. Rescaling the variable $s\rightarrow {\qtau}^{-1/3}s$, $G(x,y,\eta,\omega)=G_{\theta,p_{q},{\qtau}}(x,y)$ is given by the formula \eqref{eq:Gosc} where $p(x,y,\eta,\omega,s)=\tilde p_{q}(x,y,s/{\qtau}^{1/3},\eta/{\qtau},{\qtau})$ and is a symbol or order $0$ and type $((1,2/3,1/3),0)$.  Finally, integration with respect to the variable $s$ in \eqref{eq:Gosc} yields $G$ of the form \eqref{eq:defG} near $x=0$ as $\Gamma|_{x=0}$ is independent of $s$. 
\end{proof}
Let $a_0>0$ be small and $a\in (0,a_0]$; denote by $\mathcal{G}(t,x,y,a)$ the
Green function for the wave equation with Dirichlet boundary condition, and $\delta_{(a,0)}:=\delta_{x=a,y=0}$ the source point,
\begin{equation}
\label{waveeq}
(  \partial_t^2 -\Delta) \Grond=0 \,,\text{ for }\, x>0\,,
 \, \Grond_{{\textstyle |}x=0}=0, \Grond|_{t=0}= \delta_{(a,0)} \text{ and } \partial_t \Grond |_{t=0}=0.
\end{equation}
We will frequently need smooth cut-off functions $\varkappa\geq 0$ in $ C_{0}^{\infty}(\R^{m})$ with $m=1$ or with $m=d-1$. For $m=1$, $\varkappa$ will be such that $\varkappa=1$ near $1$, $\varkappa=0$ outside a small neighborhood of $1$, and for $m=d-1$, $\varkappa$ will be radial and such that ${\varkappa}=1$ near $\mathbb{S}^{m-1}$, ${\varkappa}=0$ outside a small neighborhood of $\mathbb{S}^{m-1}$. We will abuse notations and retain $\varkappa$ as a generic notation, irrespective of the value of $m$ (which will be clear from context) as well as the size of the (small) support of $\varkappa$, which we assume from now on to be smaller than $0<\ceps_{0}<1/100$.
\begin{dfn}\label{dfnparametrix}
  Let $h\in (0,1)$. A function $\mathcal{P}_{h,a}(t,x,y)$ is a parametrix for \eqref{waveeq} if and only if
 there exists $a_0>0$, $r>0$ and a neighborhood $V$ of $(t,y)=(0,0)$ such that for all $\alpha$ one has  
 \[
\sup_{0<a\leq a_0}\sup_{0<x\leq r}\sup_{(t,y)\in V} \Big|\partial^{\alpha}_{t,x,y}({\varkappa}(hD_{t}){\varkappa}(hD_{y})(\mathcal{P}_{h,a}-\Grond(\cdot,a))\Big| \in O(h^{\infty}).
 \]
\end{dfn}
\begin{rmq}
We have $\Grond=\cos(t\sqrt{|\Delta|})$, but we will work with the half wave propagator $e^{\pm it\sqrt{|\Delta|}}$, from which we may obtain $\Grond$ and $\partial_{t} \Grond$.  The operator $\varkappa(hD_{t})$ is really a spectral localization with respect to $\Delta$, if applied to a solution to the wave equation. The operator $\varkappa(hD_{y})$ further restricts this localization to spatial frequencies whose dominant part is tangential: the general heuristic is that waves propagating along the boundary are the most dangerous ones, whereas other waves are transverse and can be handled by simpler arguments (with a finite number of reflections). While $\varkappa(hD_{y})$ does not commute with  $\Delta$ (unlike in the model case), the support of $\eta$ in phase space will not significantly move over a finite time interval as a consequence of the Melrose-Sj\"ostrand propagation of singularities theorem. Therefore, up to $O_{C^{\infty}}(h^{\infty})$ terms, we may insert $\varkappa(hD_{y})$ operators before and after the propagator.
\end{rmq}
Rescale
$\omega=\frac{\alpha}{h^{2/3}}$, $\eta=\frac{\theta}{h}$, $s=\frac{\sigma}{h^{1/3}}$ in \eqref{eq:Gosc} (defining $G$), hence $\tau_q(\alpha,\theta)= h\tau_q(\omega,\eta)$. 
Let also $\cutoffchi^{\flat}\in C^{\infty}(\mathbb{R})$ such that $\cutoffchi^{\flat}=1$ on $(-\infty,1]$ and $\cutoffchi^{\flat}=0$ on $[2,\infty)$, and $\cutoffchi^{\sharp}=1-\cutoffchi^{\flat}$. We let
  \begin{equation}
    \label{eq:Phi}
    \Phi(x,y,\theta,\alpha,\sigma)=y\cdot \theta+\sigma^3/3+\sigma(xq^{1/3}(\theta)-\alpha) 
    +\tau_q(\alpha,\theta)
    \Gamma(x,y,\sigma q^{1/3}(\theta)/\tau_q(\alpha,\theta),\theta/\tau_q(\alpha,\theta))\,,
  \end{equation}
and as our change of variables $s=\frac{\sigma}{h^{1/3}}$ provides a factor $h^{-1/3}$, we set
\begin{equation}
  \label{eq:qh}
  p_h(x,y,\theta,\alpha,\sigma):=h^{-1/3}p(x,y,\theta/h,\alpha/h^{2/3},\sigma/h^{1/3}) \varkappa(\theta)\varkappa(\tau_q(\alpha,\theta))\cutoffchi^{\sharp}(\alpha/h^{2/3})\,,
\end{equation}
where the relevance of all cut-off functions will reveal itself later on. We get 
\begin{equation}
  \label{eq:12}
\cutoffchi^{\sharp}(\alpha/h^{2/3})  \varkappa(\theta)\varkappa(\tau_q(\alpha,\theta))G(x,y,\theta/h,\alpha/h^{2/3})=\frac{1}{2\pi}\int e^{\frac ih  \Phi(x,y,\theta,\alpha,\sigma)}  p_h(x,y,\theta,\alpha,\sigma)d\sigma.
\end{equation}
We now define an operator acting on smooth $f(y',\rho)$, with $\hat f$ its Fourier transform in all variables, 
\begin{equation}
  \label{eq:J}
    J(f)(x,y)=\int G(x,y,\eta,\omega)\cutoffchi^{\sharp}(\omega)q(\eta)^{1/6}\varkappa(h\eta)\varkappa(h\tau_q(\omega,\eta))\hat{f}(\eta,\omega/h^{1/3})d\eta d\omega\,.
  \end{equation}
  After rescaling and subtitution of \eqref{eq:12} in \eqref{eq:J},
\begin{equation}
  J(f)(x,y) 
   = \frac 1 {2\pi h^d} \int e^{\frac ih (\Phi(x,y,\theta,\alpha,\sigma)-y'\cdot\theta-{\tprime}\alpha)}p_h(x,y,\theta,\alpha,\sigma)q(\theta)^{1/6}
  f(y',{\tprime}) \,dy'd{\tprime}d\theta d\alpha d\sigma\,.
\end{equation}
\begin{lemma}
The operator $J$ is well defined from tempered distributions $\mathcal{S}'_{y',\tprime}$ into smooth functions of $(x,y)$ near $(0,0)$. In the semiclassical setting with $h$ as small parameter, $J$ is a semi-classical Fourier integral operator associated to a
canonical transform $\canonchi_J$, defined near the set $\{y'=0,\tprime=0, |\theta|=1,\alpha=0 \}$ and such that
 $ \canonchi_J(y'=0,{\tprime}=0,|\theta|=1,\alpha=0)=\{y=0,x=0,|\theta|=1,\xi=0\}$.
Moreover, $J$ is elliptic on this set and, microlocally near this set, an intertwining relation holds,
\begin{equation}
  \label{eq:LapJ}
  -h^2 \Delta J(f)= J(\tau_{q}^2(hD_{{\tprime}},hD_{y'}) f)+O(h^{\infty}).
\end{equation}
\end{lemma}
As the symbol $p_h$ is smooth and compactly supported in $(\theta,\alpha,\sigma)$, $J$ is easily extended to $\mathcal{S}'_{y',\tprime}$. The Lemma then follows from Theorem \ref{thmMelrose} ($p_0$ is elliptic and $x p_1$ vanishes on $\partial\Omega$.)  
\begin{rem}
  When $\Gamma=0$ (the model case), this canonical transform is given explicitly:
  \[
    \canonchi_J(y',{\tprime},\theta,\alpha)=(y,x,\theta,\xi)
  \]
  where
  \[
    y=y'+\tprime(\alpha-\tprime^2)\nabla q(\theta)/(3q(\theta))\,,\,\,x=(\alpha-\tprime^2)/q^{1/3}(\theta)\,,\,\,\xi=-\tprime q^{1/3}(\theta)\,.
  \]
\end{rem}
\subsection{Some useful results on Airy functions} 
We now digress and present a variation on the Poisson summation
formula, the "Airy-Poisson summation formula".
For $z\in\mathbb{C}$ we set 
\begin{equation}
  \label{eq:Apm}
  A_\pm(z)=e^{\mp i\pi/3} Ai(e^{\mp i\pi/3} z)\,,\,\,\text{ then } \,
Ai(-z)=A_+(z)+A_-(z)\, \text{ and } \overline{A_+(z)}=A_-(\overline{z}).
\end{equation}
The next two Lemmas are proved in \cite[Lemmas 1 and 3]{ILP4} :
\begin{lemma}\label{lemL}
Define, for $\omega \in \R$, the function $  L(\omega)=\pi+i\log \frac{A_-(\omega)}{A_+(\omega)}$: $L$ is an analytic, real valued, strictly increasing function with $L(0)=\pi/3$, $\lim_{\omega\rightarrow -\infty} L(\omega)=0$, and, for $\omega\geq 1$,
\begin{equation}
  \label{eq:propL}
    L(\omega)=\frac 4 3 \omega^{\frac 3 2} +\frac{\pi}{2} -B_L(\omega^{\frac 3
    2})\,,\quad
  B_L(u)= \sum_{k=1}^\infty b_k u^{-k}\,,\,\, (b_k)_{k}\in\R\,,\,\,
  b_1>0\,.
\end{equation}
Finally, let $\{-\omega_k\}_{k\geq 1}$ denote the zeros of the Airy function in decreasing order,
\begin{equation}
  \label{eq:propL2}
 L(\omega_k)=2\pi k \text{ and }
  L'(\omega_k)= 2\pi \int_0^\infty Ai^2(x-\omega_k) \,dx\,.
\end{equation}
\end{lemma}
\begin{lemma}
Let $\N^{*}=\N\setminus\{0\}$.  In $\mathcal{D}'(\R_\omega)$, one has
  \begin{equation}
    \label{eq:AiryPoisson}
        \sum_{N\in \Z} e^{-i NL(\omega)}= 2\pi \sum_{k\in \N^*} \frac 1
    {L'(\omega_k)} \delta(\omega-\omega_k)\,.
  \end{equation}
\end{lemma}
Let us define, for $\omega\in \R$, and without loss of generality, an arbitrary choice of $+$ sign for the time propagator $\exp(it{\qtau}(\omega,\eta))$,
\begin{equation}
  \label{eq:Kequiv}
  K_\omega(f)(t,x,y)=\int e^{it{\qtau}(\omega,\eta)} G(x,y,\eta,\omega)\cutoffchi^{\sharp}(\omega) q^{1/6}(\eta)
\varkappa(h\eta) \varkappa(h{\qtau}(\omega,\eta)) \hat f(\eta,\frac{\omega}{h^{1/3}}) \, d\eta\,.
\end{equation}
Due to both cut-off in $\omega$ as well as that in $\eta$, $K_{\omega}(f)$ is supported in $1\leq \omega\leq \ceps_0h^{-2/3}$ and so is $R(t,x,y,\omega,a,h):=((\partial^2_{t}-\Delta)K_{\omega}(f))(t,x,y)$. By design of $G$, using \eqref{eq:eqG}, we have moreover that, for small $r_0$ and $a_0$ and for all (large) $M\in \N$,
\begin{equation}\label{eq:Kom1}
\sup_{|a|<a_0}\sup_{|(t,x,y)|<r_0}\sup_{\omega}\Big|\nabla^{\alpha}_{t,x,y,\omega}R\Big|\leq C_{M,\alpha}h^M\,.
\end{equation}
Moreover, at $x=0$, we have $K_{\omega_{k}}(f)(t,0,y)=0$ as $G(0,y,\eta,\omega_k)=0$ (recall $\zeta(x,y,\eta,\omega)|_{x=0}=\omega$ and \eqref{eq:defG}). In other words, $K_{\omega}(f)(t,x,y)$ is a solution to the wave equation,  up to $O(h^{\infty})$; and when $\omega=\omega_{k}$, it satisfies the Dirichlet boundary condition.

To get a sense of perspective, let us remark that, in the model case, $\Gamma=0$ and then (up to normalization) $G_{M}(x,y,\eta,\omega)=\int_{\eta} \exp(i y\cdot \eta) Ai(x q^{1/3}(\eta)-\omega)\, d\eta$; for $\omega=\omega_{k}$, $G_{M}$ is a so-called gallery mode, and $K_{\omega}(f)$ is an exact solution to the half-wave equation, satisfying the Dirichlet boundary condition if $\omega=\omega_{k}$, but $f$ should not be considered as its data: if one picks $f$ such that, on the model, $J(f)$ is a Dirac at $(x=a,y=0)$, then $ f=\int_{\eta} \exp(-iy\cdot \eta) Ai(aq^{1/3}(\eta)-\omega)\,d\eta$ and then integrating over $\omega$ recovers $\delta_{x=a,y=0}$ by a standard identity on Airy functions. For this $f$, the integral over $\omega$ of $K_{\omega}(f)(t,x,y)$ is then just an half-wave solution with no boundary condition. In \cite{Annals} such a solution is then iterated by reflecting it on the boundary; here, the Airy-Poisson formula would, on the model, directly provide a sum of waves (the sum over $N$) that may later be identified as analogue of the reflected waves from \cite{Annals}, while the spectral sum (over $k$) provides the boundary condition and a direct way to decompose the Dirac data.

We now revert to the general case, where we follow the model case strategy we just sketched, but replace gallery modes by $G(x,y,\eta,\omega)$. Recall we defined $J(f)(x,y)$ in \eqref{eq:J} and we may rewrite $J(f)(x,y)=\int_\R K_\omega(f)(0,x,y)\,d\omega$.
Let $\eta=\frac \theta h$, $q^{\frac 16}(\eta)=h^{-\frac 1 3}q^{\frac 1 6}(\theta)$ and $\alpha=h^{\frac 2 3}\omega$, then (with elliptic symbol $p_{h}$ defined in \eqref{eq:qh})
\begin{equation}
  \label{eq:Kom}
 K_\omega(f)(t,x,y)=\frac{h^{\frac 2 3-d}}{2\pi } \int e^{\frac i
    h(t\tau_q(\alpha,\theta)+\Phi(x,y,\theta,\alpha,\sigma)-y'\cdot \theta-{\tprime} \alpha )}
  p_h (x,y,\theta,\alpha,\sigma) q^{\frac 1 6}(\theta)f(y',{\tprime}) \,dy'd{\tprime}d\theta d\sigma\,.
\end{equation}
As we will see later, both cut-off functions $\varkappa$ in $K_{\omega}(f)$ relate to localization operators from Definition \ref{dfnparametrix}. Moreover, $K_{\omega}(f)$ is a suitable test function in $\omega$ (smooth and compactly supported in $\omega$). Using \eqref{eq:AiryPoisson},
\begin{equation}\label{formulafAP}
 \langle \sum_{N\in \Z} e^{-iN L(\omega)}
    , K_\omega(f)(t,x,y)\rangle_{\omega}=2\pi  \sum_{k\in \N^*} \frac 1
      {L'(\omega_k)}  K_{\omega_k}(f)(t,x,y)\,.
\end{equation}
The $N=0$ term in the sum over $N$ is $J(f)$. Moreover, at $x=0$ the RHS vanishes, as the sum over $k$ is finite and each term vanishes as we just observed, and this finite sum (on the RHS) satisfies the wave equation, up to $O(h^{\infty})$ terms, due to \eqref{eq:Kom1}.

These remarks will later be of crucial importance to verify that Definition \ref{dfnparametrix} will hold for the parametrix we shall introduce in the next sections, up to finding a suitable function $f$ that will recover the data at $t=0$ in \eqref{formulafAP}. This remaining step is far from trivial, unlike in the model case (see \cite{ILP4}), for which we know explicitly the spectral resolution of the Laplacian and can therefore expand a Dirac mass over the eigenmodes, as alluded to earlier.

One may expect that it should be enough to consider initial data (at time $0$) $\chi_0(hD_{x})\varkappa(hD_{y})\delta_{(a,0)}$, for $\varkappa$ supported near $\mathbb{S}^{d-1}$ and $\chi_0\in C^{\infty}_0$ supported near $0$. Indeed,  classical geometric optics arguments provide a parametrix for data $(1-\chi_0(hD_{x}))\varkappa(hD_{y})\delta_{(a,0)}$:  due to the cut-off $(1-\chi_0(hD_{x}))$, singularities are transverse to the boundary at $x=0$ (there is at most one reflection). However, 
\begin{equation}\label{datachi0var}
  \chi_0(hD_{x})\varkappa(hD_{y})\delta_{(a,0)}=\int e^{i((x-a)\xi+y\cdot\eta)}\chi_0(h\xi)\varkappa(h\eta)d\xi d\eta
  =\frac{1}{h^d}\widehat{\chi}_0\Big(\frac{x-a}{h}\Big)\widehat{\varkappa}\Big(\frac{y}{h}\Big)\,.
\end{equation}
Therefore, at $x=0$, this data will be $O(h^{\infty})$ only if we assume that $a\geq h^{1-\ceps}$ for some $\ceps>0$. For smaller $a$, in the case of the Friedlander model operator $\Delta_M$, we can take advantage of the known, explicit, spectral resolution of $-\Delta_{M}$  in order to consider an initial data $\chi_0(-h^2\Delta_M)\varkappa(hD_y)\delta_{(a,0)}$ that can be further expressed as a sum of eigenfunctions that vanish on the boundary. By contrast, in the general case, we only have quasimodes and this is a source of significant difficulties for these very small $a$. Nevertheless, we will decompose the parametrix construction according to the respective values of $a$ and $h^{2/3}$, with an overlap between the two regimes where any construction holds. In subsection \ref{ss22}, for $a\gg h^{2/3}$, we will use \eqref{datachi0var} as a data, and mainly proceed with the sum over $N$ in our Airy-Poisson formula \eqref{formulafAP}. In subsection \ref{lowparam}, dealing with $a \lesssim h^{2/3}$, stationary phase methods in this sum over $N$ break down (although one could push them down to $a\gg h$, matching the heuristic above, but with no obvious benefit) in addition to the problem of defining a suitable initial data. We will solve the data issue in subsection \ref{sss231} by choosing the model initial data $\chi_0(-h^2\Delta_M)\varkappa(hD_y)\delta_{(a,0)}$. In some sense, in the very narrow strip where it is located, the spectral localizations with respect to $\Delta_{M}$ or $\Delta$ are close enough that gallery modes are good substitutes to the quasimodes in defining said data.  One then proceeds with a parametrix construction where such data is, again, split according to the values of $k$ in the spectral sum defining it: either $k$ is large enough and we recover a large parameter and can proceed as in the previous regime, or we have the relatively small value of $k$ for which we proceed with the spectral sum, proving in subsections \ref{sectpseudocalcul} and \ref{sectproofs} that terms appearing in that expansion are close enough to the model gallery modes and therefore retain enough of their properties to provide a parametrix. This part of the construction is quite delicate and obviously absent in the model case, while of independent interest as far as uniform estimates on quasimodes are concerned as these will be proved in the range $k\ll h^{-1/4}$, exceeding by far what we need in our construction.
\subsection{Parametrix construction for $a\geq h^{\frac 23-\ceps}$, $0<\ceps<2/3$}\label{ss22}
An initial data \eqref{datachi0var} is $O(h^{\infty})$ on the boundary for any $\chi_0$, compactly supported near $0$. 
Let $\chi_0\in C^{\infty}_0(-2\ceps_0,2\ceps_0)$ with $\chi_{0\big|[-\ceps_0,\ceps_0]}=1$.
\begin{lemma}\label{lemgab}
Let $a_0>0$, $r_0>0$ be small enough. For all $a\in [h^{\frac 23-\ceps},a_0]$, there exists a smooth function $g_{h,a}$ such that $\varkappa(hD_{y'})g_{h,a}=g_{h,a}$ and
  \begin{equation}
    \label{eq:Jdirac}
    J(g_{h,a})(x,y)=\chi_0(hD_{x})\varkappa(hD_{y})\delta_{(a,0)} + O_{C^{\infty}(|(x,y)|\leq r_0)}(h^{\infty})\,,
  \end{equation}
where the remainder is  $O_{C^{\infty}(|(x,y)|\leq r_0)}(h^{\infty})$ uniformly in $a$.
\end{lemma}
The lemma follows  from the aforementioned fact that $J$ is an
elliptic Fourier integral operator, however, we compute $g_{h,a}$ explicitely:
\begin{lemma}
\label{lemp}
There exists a smooth phase function $\psi_a({\tprime},\theta')$ and a symbol $r_h(\tprime,\theta')$ of order $1/3$, with support near $\tprime =0$, $|\theta'|=1$, of the form $r_h(\tprime,\theta')=h^{1/3}\sum_{k\geq 0}r_{h,k}h^k$ with $r_{h,0}(0,\theta')\neq 0$ for $|\theta'|=1$, such that
\begin{equation}
  \label{eq:phia}
  \psi_a({\tprime},\theta')  ={\tprime}^3/3+a({\tprime}q^{1/3}(\theta')+O({\tprime}^3))+O(a^2)
\end{equation}
and the function $g_{h,a}$, defined as
\begin{equation}
    \label{eq:gosc}
        g_{h,a}(y',{\tprime})  = h^{-d}\int e^{\frac i h (\psi_a({\tprime},\theta')+y'\cdot \theta')}    r_h(\tprime, \theta')q^{-1/6}(\theta') \,d\theta'\,,
\end{equation}
 solves \eqref{eq:Jdirac}. Moreover, $\psi_a$ is the critical value of the phase ${\tprime}\alpha-\Phi(a,0,\theta',\alpha,s)$ at critical points in $(\alpha,s)$.
\end{lemma}
 \begin{proof}
We may invert microlocally the operator $J$ from \eqref{eq:J} by setting 
\[
J^{-1}(F)(y',\tprime)=h^{-d-1}\int e^{\frac ih(-\Phi(x,y,\theta',\alpha',s)+y'\cdot\theta'+\tprime\alpha')}q_h(x,y,\theta',\alpha',s)q^{-1/6}(\theta')F(x,y)dxdyd\theta' d\alpha' ds,
\]
where $q_h(x,y,\theta',\alpha',s)$ is a symbol of order $1/3$, $q_h= h^{1/3}\sum_{k\geq 0} h^kq_{h,k}$, with support near $\{x=0,y=0,\alpha'=0,s=0\}$ and elliptic on this set: we aim at proving that $J^{-1}\circ J(f)=f\text{ modulo } O(h^{\infty})$:
\begin{multline}
J^{-1}\circ J(f)=h^{-d-1}\int e^{\frac ih(-\Phi(x,y,\theta',\alpha',s)+y'\cdot\theta'+\tprime\alpha')}q_h(x,y,\theta',\alpha',s)q^{-1/6}(\theta')\\
\times \frac {1}{2\pi h^d} e^{\frac ih \Phi(x,y,\theta,\alpha,\sigma)}p_h(x,y,\theta,\alpha,\sigma)q^{1/6}(\theta) \hat{f}(\theta/h,\alpha/h)d\theta d\alpha d\sigma dxdy d\theta' d\alpha' ds.
\end{multline}
We now apply stationary phase in variables $(\sigma, s, x,y,\theta',\alpha')$: one checks that critical points are non-degenerate, such that $\theta'_c=\theta$, $\alpha'_c=\alpha$, and stationary phase provides a factor $h^{d-1+1+1}=h^{d+1}$ (one factor $h^{d-1}$ from $dyd\theta'$, one factor $h$ from $dxds$ and one factor $h$ from $d\sigma d\alpha'$). The critical value of the phase is $y'\theta+\tprime \alpha$ and we obtain (modulo $O(h^{\infty})$)
\[
J^{-1}\circ J(f)(y',\tprime)=h^{-d-1}\times \frac{h^{d+1}}{2\pi h^d}\int e^{\frac ih(y'\cdot\theta+\tprime \alpha)}\tilde q_h(\theta,\alpha)\hat{f}(\theta/h,\alpha/h)d\theta d\alpha,
\]
where $\tilde q_h$ is obtained from the product $q_h(x,y,\theta',\alpha',s)q^{-1/6}(\theta') p_h(x,y,\theta,\alpha,\sigma)q^{1/6}(\theta)$ after stationary phase; asking $\tilde q_h=1$ for $\theta$ such that $|\theta|\sim 1$ and $\tprime$ near $0$ allows to chose $q_h$; since $p_h=h^{-1/3} p(x,y,\theta/h,\alpha/h^{2/3},\sigma/h^{1/3})\cutoffchi^{\sharp}(\alpha/h^{2/3})\varkappa(\theta)\varkappa(\tau_q(\alpha,\theta))$, we obtain $q_h$ as announced. Define 
\begin{equation}\label{gl-2}
\tilde g_{h,a}:=J^{-1}(\chi_0(hD_x)\varkappa(hD_y)\delta_{(a,0)})\,.
\end{equation}
Then, using the second line in \eqref{datachi0var}, 
\begin{multline}\label{ghaF}
\tilde g_{h,a}(y',\tprime)=\frac 1{h^{d}}\int e^{\frac ih y'\cdot\theta'}F_{h,a}(\tprime,\theta')d\theta'\,,\,
  F_{h,a}(\tprime,\theta')=\frac 1 {h^{d+1}}\int e^{\frac ih (-\Phi(x,y,\theta',\alpha,s)+\tprime \alpha+(x-a)\sigma+y\cdot \theta)}\\
{}  \times q_h(x,y,\theta',\alpha,s)q^{-\frac 1 6}(\theta')\chi_0(\sigma)\varkappa(\theta)d\sigma d\theta dx dy d\alpha ds\,.
\end{multline}
We apply stationary phase to $F_{h,a}$ with respect to variables $(x,\sigma,y,\theta)$: (non-degenerate) critical points are $x=a$, $y=0$, $\sigma_c=\partial_x\Phi(a,0,\theta',\alpha,s)$ and $\theta_c=\partial_y\Phi(a,0,\theta',\alpha,s)$. The resulting symbol $\tilde q_h(\theta',\alpha,s)$ is of order $1/3$, with support near $\{|\theta'|=1,\alpha=0,s=0\}$ and elliptic on this set, and
\begin{equation}\label{Fha}
F_{h,a}(\tprime,\theta')=h^{-1}\int e^{\frac ih(-\Phi(a,0,\theta',\alpha,s)+\tprime \alpha)} \tilde q_h(\theta',\alpha,s)q^{-1/6}(\theta') d\alpha ds\,.
\end{equation}
Here $\alpha$ is bounded, as $\tau_q(\alpha,\theta')\in \mathrm{supp}\, \varkappa$; indeed, $\alpha=h^{2/3}\omega$ and  we assumed $|\omega|\leq \ceps_0h^{-2/3}$; therefore $\alpha\leq \ceps_0$ on the support of the symbol $p_h$, as well as on the support of $q_h$ and also $\tilde q_h$. The phase of $F_{h,a}$ is stationary in $\alpha$ for $s+\tprime+O(a)=0$ and in $s$ for $s^2+O(a)\sim \alpha\leq \ceps_0$ (as we will see below using the explicit form of $\Phi$) and $a\leq a_0<1$ is small enough, therefore $s^2\lesssim \ceps_0+a_0$ (otherwise non stationary phase in $s$ provides an $O(h^{\infty})$ contribution.) Hence there exists a cut-off $\chi(\tprime)\in C^{\infty}_0((-2r,2r))$, equal to $1$ on $[-r,r]$ for $r\sim \sqrt{\ceps_0}+ a_0$, such that $(1-\chi(\tprime))F_{h,a}(\tprime,\theta')=O(h^{\infty})$ in $\mathcal{S}_{\tprime}$,
 uniformly in $\theta'$ near $|\theta'|=1$.
 
We now apply stationary phase in \eqref{Fha} with respect to $\alpha$ and $s$: with $(\tprime,\theta',a)$ as parameters and for $\tprime\in (-2r,2r)$ and $|\theta'|$ close to $1$, let $(\alpha_c,s_c)$ denote the critical points of the phase, define $\psi_a({\tprime},\theta')={\tprime}\alpha_c-\Phi(a,0,\theta',\alpha_c,s_c)$, where $({\tprime},a)$ are small parameters in $(-2r,2r)\times (0,a_0)$. In order for  $\psi_a$ to be \eqref{eq:phia}, we need more information on the phase $\Phi$. Recall from \eqref{eq:Phi}
\[
\Phi(x,y,\theta,\alpha,\sigma)=y\theta+\sigma^3/3+\sigma(xq^{1/3}(\theta)-\alpha)+\tau_q(\alpha,\theta)\Gamma(x,y,\sigma q^{1/3}(\theta)/\tau_q(\alpha,\theta),\theta/\tau_q(\alpha,\theta)),
\]
and from the Appendix, Proposition \ref{propimpformgamma}, $\Gamma(x,y,\Xi,\Theta)=B_{\Gamma}(y,\Theta)+xA_{\Gamma}(x,y,\Xi,\Theta)$, where $B_{\Gamma}(0,\Theta)=0$ and 
 $A_{\Gamma}(x,y,\Xi,\Theta)=\Xi \Lp(y,\Theta/|\Theta|)+\mu(y,\Theta/|\Theta|)(\Xi^2+|\Theta|^2-1)+\mathcal{H}_{j\geq 3}$,
where $\Lp,\mu$ are smooth functions such that $\Lp(0,\omega)=0$, and $\mathcal{H}_{j\geq k}$ denotes any function which is an expansion of the form $\sum_{j\geq k}f_j$ with $f_j$ homogeneous of order $j$ with respect to weights on variable $(x,y,\Xi,\Theta)$: $x:2$, $(|\Theta|-1):2$, $\Xi:1$. Therefore, the phase of $F_{h,a}$ reads
\begin{multline}\label{phaFha}
-\Phi(a,0,\theta',\alpha,s)+\tprime \alpha=-s^3/3-s(aq^{1/3}(\theta')-\alpha)+\tprime\alpha\\
{}-a\tau_q(\alpha,\theta')A_{\Gamma}(a,0,sq^{1/3}(\theta')/\tau_q(\alpha,\theta'), \theta'/\tau_q(\alpha,\theta'))\,.
\end{multline}
Setting $\Xi:=s q^{1/3}(\theta')/\tau_q(\alpha,\theta')$ and $\Theta:=\theta'/\tau_q(\alpha,\theta')$ and using that $\Lp(0,\Theta/|\Theta|)=0$, 
\[
A_{\Gamma}(a,0,\Xi,\Theta)_{\bigl|\Xi=sq^{1/3}(\theta')/\tau_q(\alpha,\theta'),\Theta=\theta'/\tau_q(\alpha,\theta')}=\mu(0,\theta'/|\theta'|)\frac{(s^2-\alpha)q^{2/3}(\theta')}{{\qtau}^2(\alpha,\theta')}+\mathcal{H}_{j\geq 3},
\]
where $\mathcal{H}_{j\geq 3}$ contains terms with factors $as$, $s^3$ and $s\alpha$; as $A_{\Gamma}(a,\cdots)$ comes with a factor $a$, cancelling the derivative of the phase \eqref{phaFha} with respect to $s$ yields an equation for $s_{c}$,
\begin{gather}
  -s^2-(aq^{1/3}(\theta')-\alpha)-a\tau_q(\alpha,\theta')\Big(2s\mu(0,\theta'/|\theta'|)q^{2/3}(\theta')/{\qtau}^2(\alpha,\theta')+O(a,s^2,\alpha)\Big)=0\,,\\
s_c^2(1+O(a))+aq^{1/3}(\theta')(1+2s_{c}q^{1/3}(\theta')\mu/{\qtau})=\alpha_c(1+O(a))\,,
\end{gather}
while cancelling the derivative of \eqref{phaFha} with respect to $\alpha$ yields $\tprime+s_c(1+O(a))=0$.
Therefore, $s_c=-\tprime(1+O(a))$ and $\alpha_c=\tprime^2(1+O(a))+aq^{1/3}(\theta')(1+O(\tprime))$. We now compute the critical value of the phase \eqref{phaFha} at $s_c,\alpha_c$: let $\phi(a,\tprime,\theta'):=-\Phi(a,0,\theta',\alpha_{c},s_{c})+\tprime \alpha_{c}$
and write the Taylor expansion of $\phi$ near $a=0$. We have $\phi(a,\tprime,\theta')=\phi(0,\tprime,\theta')+a\partial_a\phi(0,\tprime,\theta')+O(a^2)$,
 where $\alpha_{c|a=0}=\tprime^2$ and $s_{c|a=0}=-\tprime$, and then, with $(\cdots)=(0,0,\theta')$, 
\begin{gather}
\phi(0,\tprime,\theta') =-\Phi(\cdots,\alpha_{c|a=0},s_{c|a=0})+\tprime \alpha_{c|a=0}=\big(-(\frac{s^3}3+s\alpha)+\tprime \alpha\big)_{\bigl|s=-\tprime,\alpha=\tprime^2} 
    =  \frac{\tprime^3}3\,,\\
\partial_a\phi(0,\tprime,\theta') 
\label{derivphia=0}
 =
    \begin{aligned}[t]
    -\partial_a\Phi(\cdots,\tprime^2,-\tprime)-\partial_a s_{c|a=0}\partial_{s}\Phi(\cdots,\tprime^2,-\tprime) &\\
 {} +\partial_a \alpha_{c|a=0}(\tprime-\partial_{\alpha}\Phi(\cdots,\tprime^2,-\tprime))\,.&  
    \end{aligned}
    \end{gather}
As $s_c,\alpha_c$ are critical points for $-\Phi(\cdot,\alpha,s)+\tprime\alpha$, the last two terms in \eqref{derivphia=0} vanish.
We have
\begin{align*}
\partial_a\phi(0,\tprime,\theta') & =-\partial_a\Phi(a,0,\theta',\alpha,s)|_{(0,0,\theta',\tprime^2,-\tprime)}        
  =\tprime q^{1/3}(\theta')-\tau_q(\tprime^2,\theta')A_{\Gamma}(0,0,-\tprime q^{1/3}(\theta')/{\qtau},\theta'/{\qtau})\,.
\end{align*}
For $(\Xi,\Theta)=(-\tprime q^{1/3}(\theta')/{\qtau},\theta'/{\qtau})$, with ${\qtau}=\tau_q(\tprime^2,\theta')$, the term homogeneous of degree $1$ of $A_{\Gamma}(0,0,\Xi,\Theta)$ is $\Xi \Lp(0,\Theta)=0$ and the term homogeneous  of degree $2$ equals $\mu(0,\Theta)(\Xi^2+|\Theta|^2-1)=0$ as $(\Xi^2+|\Theta|^2-1)=\frac{(s^2-\alpha)}{{\qtau}^2}|_{s=-\tprime,\alpha=\tprime^2}=0$. The terms homogeneous of higher order $j\geq 3$ of $A_{\Gamma}(0,0,\Xi,\Theta)$ are powers of $\Xi^j\sim \tprime^j$ (as we can replace $|\Theta|^2-1$ by $\Xi^2$), hence  $A_{\Gamma}(0,0,\Xi,\Theta)=O(\tprime^3)$. Therefore, stationary phase in $s,\alpha$ yields, for some new symbol $r_h(\tprime,\theta')$ of order $1/3$,
\begin{equation}
  \label{eq:ghaintegral}
  \tilde g_{h,a}(y',{\tprime})=h^{-d}\int e^{\frac i h(\psi_a({\tprime},\theta')+y'\cdot \theta')}
  r_h(\tprime,\theta')q^{-1/6}(\theta') \,d\theta'+O_{C^\infty}(h^\infty)\,,
\end{equation}
where we set $\psi_a(\tprime,\theta'):=\phi(a,\tprime,\theta')$ that is indeed the required \eqref{eq:gosc}. One has $\WF_h(\tilde g_{h,a})\subset \{(y',\tprime,\theta',\alpha'), y'=-\nabla_{\theta'}\psi_a(\tprime,\theta'), \alpha'=\partial_{\tprime}\psi_a(\tprime,\theta')\}$. Using \eqref{eq:phia} and \eqref{gl-2},  we now set $g_{h,a}(y',\tprime):=\chi(\tprime)\tilde g_{h,a}(y',\tprime)$, such that \eqref{eq:Jdirac} holds and this completes the proof of Lemma \ref{lemp}.
\end{proof}
\begin{dfn}\label{defparamPoisson}
    Let $g_{h,a}$ be defined in \eqref{eq:gosc}: using \eqref{eq:AiryPoisson}, we define $\Prond_{h,a}$ equivalently as
    \begin{align}
      \label{eq:Prond}
      \Prond_{h,a}(t,x,y) & = \langle \sum_{N\in \Z} e^{-iN L(\omega)}
    , K_\omega(g_{h,a})(t,x,y)\rangle_{\omega}\\
      \label{eq:Prond2}
      \Prond_{h,a}(t,x,y) & =2\pi  \sum_{k\in \N^*} \frac 1
      {L'(\omega_k)}  K_{\omega_k}(g_{h,a})(t,x,y)\,.
    \end{align}
    \end{dfn}
    We are abusing notation here: one should consider $\Prond^{\pm}$ depending on the sign on $t$ and then obtain $2 \Prond$ from Definition \ref{dfnparametrix} as $\Prond^{+}+\Prond^{-}$. Considering $\Prond^{+}$ is enough by time symetry and we therefore drop the $+$.

We now recall (see the discussion after \eqref{eq:Kequiv}) that, using both localizations in $\eta$ and $\tau_q(\omega,\eta)$, \eqref{eq:Prond2} may be reduced to a finite sum over $k\lesssim h^{-1}$:  support considerations on $K_{\omega}$ (as a function of $\omega$) provide $|\omega|\leq \ceps_0 h^{-2/3}$; after Airy-Poisson summation, this translates into $|\omega_{k}|\leq \ceps_0 h^{-2/3}$. The zeroes $\{-\omega_k\}_{k\geq 1}$ of the Airy function have asymptotic $\omega_k\sim (3\pi k/2)^{2/3}$. We therefore introduce a cut-off in the sum over $k$, $\cutoffchi^{\flat}_{\epsilon_{0}}(h^{2/3}\omega_{k}):={\cutoffchi^{\flat}}(h^{2/3}\omega_k/\ceps_{0})$. Then, \eqref{eq:Prond2} may be rewritten as a finite sum,
 \begin{equation}
      \label{eq:Prond2cut}
      \Prond_{h,a}(t,x,y):=2\pi  \sum_{k\in \N^*} \frac{{\cutoffchi^{\flat}_{\epsilon_{0}}}(h^{2/3}\omega_k
        )}
      {L'(\omega_k)}  K_{\omega_k}(g_{h,a})(t,x,y)\,.
    \end{equation}
    From $\omega_1\geq 2.33$, we remark that the cut-off function ${\cutoffchi^{\sharp}(\omega)}$ that was introduced in the definition \eqref{eq:Kequiv} of $K_{\omega}$ is no longer needed when restricting $\omega$ to the set $\{\omega_{k}\}_{k\in \N^{*}}$.  
    But it will help on the other sum \eqref{eq:Prond}, in estimating how many $N$'s contribute significantly.  Again with \eqref{eq:AiryPoisson}, we also have
 \begin{equation}
      \label{eq:Prondcut}
      \Prond_{h,a}(t,x,y)= \langle \sum_{N\in \Z} e^{-iN L(\omega)}
    , {\cutoffchi^{\flat}}(h^{2/3}\omega/\ceps_{0})K_\omega(g_{h,a})(t,x,y)\rangle_{\omega}\,.
    \end{equation}
The sum $\sum_{N\in\mathbb{Z}}$ converges in $\mathcal{D}'_{\omega}$ and ${\cutoffchi^{\flat}}(h^{2/3}\omega/\ceps_0)K_\omega(g_{h,a})(t,x,y)$ is smooth in $(t,x,y)$ in a neighborhood $W$ of $(0,0,0)$ and smooth and compactly supported in $\omega$. For the moment we use $g_{h,a}$ as expressed from \eqref{ghaF} and \eqref{Fha} (integral over $\alpha,s,\tprime$).
We can however replace the cut-off $\chi(\tprime)$, introduced in the proof of Lemma \ref{lemp}, by $\chi(s)$;  as $s_{c}=-\tprime(1+O(a))$, defining $g_{h,a}$ without the factor $\chi(\tprime)$ but with $\chi(s)$ inside the integral provides the same contribution modulo $O(h^{\infty})$ but allows to immediately obtain $\hat{g}_{h,a}(\theta/h,\alpha/h)$, which is useful 
in the formula for $K_{\omega}(g_{h,a})$:
\[
\hat{g}_{h,a}(\theta/h,\alpha/h)=h^{-1}q^{-1/6}(\theta)\int e^{-\frac ih \Phi(a,0,\theta,\alpha,s)}\chi(s)\tilde q_h(\theta,\alpha,s) ds,
\]
and substitution in \eqref{eq:Kom} yields
\begin{multline}\label{Kghauseful}
K_{\omega}(g_{h,a})(t,x,y)=\frac{h^{2/3}}{2\pi h^{d+1}}\int e^{\frac ih (t\tau_q(h^{2/3}\omega,\theta)+\Phi(x,y,\theta,h^{2/3}\omega,\sigma)-\Phi(a,0,\theta,h^{2/3}\omega,s))}\\
\times p_h(x,y,\theta,h^{2/3}\omega,\sigma)\chi(s)\tilde q_h(\theta,h^{2/3}\omega,s)dsd\theta d\sigma\,.
\end{multline}
We set   $  \Prond_{h,a}(t,x,y)=\sum_{N\in \Z} V_N(t,x,y)$, where $V_{N}$ is defined as
\begin{align}
  \label{eq:newVN}
  \quad V_N(t,x,y) := & \int e^{-iNL(\omega)}{\cutoffchi^{\flat}}(h^{2/3}\omega/\ceps_0 )K_{\omega}(g_{h,a})(t,x,y)d\omega\\
 = & \begin{multlined}[t] \frac 1 {2\pi h^{d+1}} \int e^{\frac i h (
    t\tau_q(\alpha,\theta)+\Phi(x,y,\theta,\alpha,\sigma)-\Phi(a,0,\theta,\alpha,s)-Nh L(h^{-2/3}\alpha)
    )} \\
\quad\quad\quad{}\times {\cutoffchi^{\flat}}(\alpha/\ceps_0)\chi(s)p_h(x,y,\theta,\alpha,\sigma)\tilde q_h (\theta,\alpha,s) ds \,d\theta  d{\sigma}d\alpha\,.\end{multlined}
\end{align}
The symbol ${\cutoffchi^{\flat}}(\alpha/\ceps_0) \chi(s)p_h \tilde q_h$ of $V_N$  is the same for every $N$,  is of order $0$ and is given by an asymptotic expansion with small parameter $h$ and main term equal to $1$ (indeed, since $\tilde q_h$ has been obtained by inverting $J$, whose symbol is $p_h$). Note that we do not have a finite sum over $N$: convergence should be understood in the distributional sense. The cut-off in $\alpha$ is redundant but we will leave it there to emphasize compact support in $\alpha$. In the forthcoming Lemma \ref{sommeNfinie}, we prove that for a generic function $f_{h}$ replacing $g_{h,a}$ the sum over $N$ converges and is $O(h^{\infty})$ for $N>h^{-1/3}$, provided $f_{h}$ is of moderate growth with respect to $h$. Practically, $f_{h}$ is an oscillatory integral with an Airy type phase and with a smooth rapidly decaying or compactly supported symbol and as such, is of moderate growth.

As such, we may indeed replace $g_{h,a}$ modulo $O(h^{\infty})$ as it will only concern a (large but) finite number of terms. Our main result in this section is the following proposition:
\begin{prop}\label{propsommefinie}
Let $a\in (h^{\frac{2}{3}-\ceps},a_0]$ with small $a_0,\ceps>0$.
\begin{itemize}[leftmargin=5.5mm]
\item
For $|t|\lesssim 1$, $\Prond_{h,a}$ is essentially a finite sum in $N$ at any given time,
\begin{equation}
  \label{eq:newProndcut}
  \Prond_{h,a}(t,x,y)=\sum_{|N|\lesssim |t| a^{-1/2}} V_N(t,x,y)+O_{C^{\infty}}(h^{\infty})\,. 
  \end{equation}
  Moreover we can introduce a cut-off ${\cutoffchi}^{\sharp}(4\alpha/a)$ in the definition of $V_N$ without changing its main contribution modulo $O(h^{\infty})$ terms. 
\item At $t=0$, we have $\mathcal{P}_{h,a}(0,x,y)=\chi_0(hD_{x}){\varkappa}(hD_y)\delta_{(a,0)}+O_{C^\infty}(h^\infty)$.
  \item $\Prond_{h,a}$ is a parametrix in the sense of the Definition \ref{dfnparametrix}.
\end{itemize}
\end{prop}
\begin{rem}\label{rmqalpha>a}
  The cut off ${\cutoffchi^{\sharp}}(h^{-2/3}\alpha)$ from $K_{\omega}$ restricts to $1\leq h^{-2/3}\alpha$. The last statement in the first part of Proposition \ref{propsommefinie} translates into the contribution of the integrals defining $V_N$ being irrelevant for small values $\alpha\leq a/2$: for $a>h^{2/3 -\ceps}$ we can further restrict to $\alpha>a/2$. This follows right away from the expression of $G(x,y,\theta/h,\omega)$ appearing in the definition of $K_{\omega}(g_{h,a})$ (recall \eqref{eq:Kequiv}): using \eqref{eq:defG}, $G$ reads as a sum of Airy functions computed at $-\zeta=x|\eta|^{2/3}e_0(x,y,\eta,\omega)-\omega$ with an elliptic $e_0$, close to $1$. These Airy functions are exponentially decreasing for $-\zeta>0$; hence, if $a>h^{2/3-\ceps}$, $\eta=\theta/h$ with $|\theta|\sim 1$ and $\omega=h^{-2/3}\alpha$, we must have $a\lesssim \alpha$ since otherwise the contribution from $G$ is $O(h^{\infty})$. Note that $h^{2/3-\ceps}\lesssim \alpha$ is required to perform stationary phase arguments; for $\alpha\leq h^{2/3-\ceps}$ (to be dealt with if $a\leq h^{2/3-\ceps}$ !), rescaling no longer provides a large parameter.
\end{rem}
\begin{proof}
  We start with the easiest part: from Theorem \ref{thmMelrose} $G(0,y,\eta,\omega_k)=e^{i\psi(0,y,\eta,\omega_k)}
  p_0Ai(-\omega_k)$, which immediately yields $G(0,y,\eta,\omega_k)=0$. Therefore, from
\eqref{eq:Prond2cut} being a finite sum, we get
  $\Prond_{h,a}(t,x,y)_{{\textstyle |}\partial\Omega} =0$,
 which is to say, the Dirichlet boundary condition holds for $\Prond_{h,a}$. From \eqref{eq:AiryPoisson}, we get that the distribution $\sum_{N\in\mathbb{Z}}e^{-iNL(\omega)}\in \mathcal{S}'(\R)$. Moreover, from upcoming Lemma \ref{sommeNfinie}, for $|N|>h^{-1/3}$, the sum is $O(h^{\infty})$ irrespective of $g_{h,a}$. As such, we are reduced to a finite number of $N$'s, and  from \eqref{eq:Prond} and \eqref{eq:Kom1}, it follows that, taking $W$ smaller if needed, and uniformly in $a<a_0$, one has
$(\partial^2_{t}-\Delta)\mathcal{P}_{h,a}\in O_{C^{\infty}(W)}(h^{\infty})$, not only for $x>0$ but in the full neighborhood $W$ of $(0,0,0,1)$.  Both statements on $\mathcal{P}_{h,a}$ are independent on the particular choice of the function $g_{h,a}$ such that \eqref{eq:Jdirac} holds. It remains to check that, with our choice of $g_{h,a}$ given in \eqref{eq:Jdirac}, $\mathcal{P}_{h,a}(0,x,y)$ is the right initial value, which turns out to be the most difficult part of the proof. We first prove that the sum over $N$ is (large but) finite and that at $t=0$, in the sum over $N$ in \eqref{eq:newProndcut}, all the oscillatory integrals $V_N(0,x,y)$ for $|N|\geq 1$ provide a $O(h^{\infty})$ contribution, while $V_0(0,x,y)=J(g_{h,a})(x,y)$ which, by design, is our initial data. The fact that the number of $N$ is finite will allow to deduce that $\mathcal{P}_{h,a}(0,x,y)=V_0(0,x,y)=J(g_{h,a})(x,y)=\chi_0(hD_x)\varkappa(hD_y)\delta_{(a,0)}$ and conclude.
\begin{lemma}
  \label{sommeNfinie}
  Let $f_{h}$ be a smooth function of $(y',\tprime)$, with compact support in $\tprime$ and of moderate growth in $h$, and $K_{\omega}(f_{h})$ be defined by \eqref{eq:Kom}. Then
\begin{equation}\label{eq:suminN}
  \langle \sum_{|N|\gtrsim h^{-1/3}}e^{-iN L(\omega)}
    , {\cutoffchi^{\flat}}(h^{2/3}\omega/\ceps_{0})K_\omega(f_{h})(t,x,y)\rangle_{\omega} = O(h^{\infty})\,.
\end{equation}
\end{lemma}
\begin{proof}
We consider the sum over all $N$: all phases in the sum are linear in $t$ and $N$ and given by
\begin{equation}\label{phaKom}
t\tau_q(\alpha,\theta)+\Phi(x,y,\theta,\alpha,\sigma)-N h L(\alpha h^{-2/3}) -y'\theta-\tprime \alpha,
\end{equation}
with large parameter $1/h$ as a factor; it follows from \eqref{eq:propL} that
\[
NL(\alpha h^{-2/3})=N\frac{\pi}{2}+\frac 1h \Big(\frac 43 N\alpha^{3/2} -Nh B_L(\alpha^{3/2}/h)\Big).
\]
Integration variables are $\sigma,\alpha$, $\theta$ and also $y'$ and $\tprime\,$;  only stationary points with respect to $\alpha$ and $\sigma$ will be required for the sum in $N$ to be finite.
Critical points in $\alpha$ are such that
\begin{equation}\label{eq:critalphaNN}
t\partial_{\alpha}{\qtau}(\alpha,\theta) +\partial_{\alpha}\Phi(x,y,\theta,\alpha,\sigma)=\tprime+2N\alpha^{1/2}\Big(1-\frac 34 B'_L(\alpha^{3/2}/h)\Big)\,,
\end{equation}
while those with respect to $\sigma$ are such that 
$\partial_{\sigma}\Phi(x,y,\theta,\alpha,\sigma)=0$. We used $1\leq h^{-2/3}\alpha$ to expand $L(h^{-2/3}\alpha)$ with \eqref{eq:propL}. 
Recall that
\[
\Phi(x,y,\theta,\alpha,\sigma)=y\cdot\theta+\sigma^3/3+\sigma(xq^{1/3}(\theta)-\alpha)+\tau_q(\alpha,\theta)\Gamma(x,y,\sigma q^{1/3}(\theta)/\tau_q(\alpha,\theta),\theta/\tau_q(\alpha,\theta)),
\]
where $\Gamma(x,y,\Xi,\Theta)=B_{\Gamma}(y,\Theta)+xA_{\Gamma}(x,y,\Xi,\Theta)$ from Proposition \ref{propimpformgamma} in the Appendix. In addition to properties of $A_{\Gamma}$ and $B_{\Gamma}$ listed in Proposition \ref{propimpformgamma}, we will use Lemma \ref{lemB0}. We start with computing derivatives of ${\qtau}(\alpha,\theta)\Gamma(x,y,\sigma sq^{1/3}(\theta)/{\qtau},\theta/{\qtau})$ which depends on $\alpha$ only through ${\qtau}$.
Take $\Theta=\theta/\tau_q(\alpha,\theta)$, $\vartheta=\theta/|\theta|$, then
\[
{\qtau} B_{\Gamma}(y,\theta/{\qtau})={\qtau} \bigl(B_0(y,\vartheta)+(|\theta|/{\qtau}-1)B_2(y,\vartheta)+...+(|\theta|/{\qtau}-1)^jB_{2j}(y,\vartheta)+...\bigr)
\]
and, writing $\rho=|\theta|/{\qtau}$, $\theta/{\qtau}=\rho {\vartheta}$,
\begin{multline}\label{partialBGamma}
\partial_{\alpha}\Big({\qtau} B_{\Gamma}(y,\theta/{\qtau})\Big)=\partial_{\alpha}{\qtau}\partial_w\Big(wB_{\Gamma}(y,\theta/w)\Big)_{{\textstyle |}w={\qtau}}\\
=\partial_{\alpha}{\qtau}\Big(B_{\Gamma}(y,\theta/{\qtau})-\frac{|\theta|}{{\qtau}}\partial_{\rho}B_{\Gamma}(y,\rho{\vartheta})_{{\textstyle |}\rho=|\theta|/{\qtau}}\Big)\\
=\partial_{\alpha}{\qtau}\Big[ B_0(y,\vartheta)-B_2(y,\vartheta)-\Big(\frac{|\theta|^2}{{\qtau}^2}-1\Big)B_4(y,\vartheta)+\mathcal{H}_{j\geq 4}\Big].
\end{multline}
\begin{multline}\label{partialAGamma}
\partial_{\alpha}\Big({\qtau} A_{\Gamma}(x,y,sq^{1/3}(\theta)/{\qtau},\theta/{\qtau})\Big)=\partial_{\alpha}{\qtau}\partial_w\Big(wA_{\Gamma}(x,y,sq^{1/3}/w,\theta/w)\Big)_{{\textstyle |}w={\qtau}}\\
=\partial_{\alpha}{\qtau}\Big(A_{\Gamma}(x,y,sq^{1/3}/{\qtau},\theta/{\qtau})-sq^{1/3}(\theta)\Lp(y,\vartheta)/{\qtau}\\
-2\mu(y,\vartheta)(s^2q^{2/3}(\theta)/{\qtau}^2+|\theta|^2/{\qtau}^2)+\mathcal{H}_{j\geq 3}\Big)\\
=\partial_{\alpha}{\qtau}\Big[-\mu(y,\vartheta)\Big(s^2q^{2/3}(\theta)/{\qtau}^2+|\theta|^2/{\qtau}^2+1\Big)+ \mathcal{H}_{j\geq 3}\Big],
\end{multline}
where in the second to last line we used that the terms in $\mathcal{H}_{j\geq 3}$ are powers of $x,sq^{1/3}(\theta)/w, \theta/w$ and therefore $w\partial_w\mathcal{H}_{j\geq 3{\textstyle |} w={\qtau}}=\mathcal{H}_{j\geq 3}$. 
We have $\partial_{\alpha}{\qtau}(\alpha,\theta)=\frac{q^{2/3}(\theta)}{2\tau_q(\alpha,\theta)}$ and using \eqref{partialBGamma}, \eqref{partialAGamma} (which comes with a factor $x\in \mathcal{H}_{j\geq 2}$), \eqref{eq:critalphaNN} becomes
\begin{multline}\label{critalphaNNf}
\partial_{\alpha}{\qtau}\Big[ t+B_0(y,\vartheta)-B_2(y,\vartheta)-\Big(\frac{|\theta|^2}{{\qtau}^2}-1\Big)B_4(y,\vartheta)+\mathcal{H}_{j\geq 4}\Big]+\sigma=\tprime\\
+2N\alpha^{1/2}\Big(1-\frac 34 B'_L(\alpha^{3/2}/h)\Big).
\end{multline}
Using now \eqref{partialAGamma} and that $B_{\Gamma}(y,\theta/{\qtau})$ does not depend on $\sigma$, $\partial_{\sigma}\Phi(\cdots)=0$ 
becomes
\begin{equation}\label{sigmacrit}
\sigma^2+xq^{1/3}(\theta)-\alpha+\tau_q(\alpha,\theta)x\partial_{\sigma}A_{\Gamma}(x,y,\sigma q^{1/3}(\theta)/\tau_q(\alpha,\theta),\theta/\tau_q(\alpha,\theta))=0.
\end{equation}
Recalling \eqref{AGam}, let $\Xi=\frac{\sigma q^{1/3}(\theta)}{\tau_q(\alpha,\theta)}$ and $\Theta=\frac{\theta}{\tau_q(\alpha,\theta)}$.
Taking a derivative with respect to $\sigma$ of $A_{\Gamma}$ always provides a factor $q^{1/3}(\theta)/{\qtau}$: $A_{\Gamma}(x,y,\Xi,\Theta)$ depends on $\sigma$ only through $\Xi=\frac{\sigma q^{1/3}(\theta)}{\tau_q(\alpha,\theta)}$, hence 
\[
\partial _{\sigma}A_{\Gamma}(x,y,\sigma q^{1/3}(\theta)/{\qtau},\theta/{\qtau})=\partial_{\sigma}\Xi \times \partial_{\Xi}A_{\Gamma}(x,y,\Xi,\Theta), 
\] 
and $\partial_{\Xi}A_{\Gamma}(x,y,\Xi,\Theta)=\Lp(y,\Theta/|\Theta)+2\Xi \mu(y,\vartheta)+\mathcal{H}_{j\geq 2}$. This yields
\begin{equation}\label{partialsigmaAGamma}
\partial _{\sigma}A_{\Gamma}(x,y,\sigma q^{1/3}(\theta)/{\qtau},\theta/{\qtau})=\frac{q^{1/3}(\theta)}{\tau_q(\alpha,\theta)}\Big(\Lp(y,\vartheta)+2\frac{\sigma q^{1/3}(\theta)}{\tau_q(\alpha,\theta)}\mu(y,\vartheta)+\mathcal{H}_{j\geq 2}\Big),
\end{equation}
where $\mathcal{H}_{j\geq 2}$ contains weights $x$, $\frac{\sigma^2q^{2/3}(\theta)}{{\qtau}^2}$ and $\frac{|\theta|}{{\qtau}}-1=\frac{|\theta|^2-{\qtau}^2}{{\qtau}(|\theta|+{\qtau})}=-\frac{\alpha q^{2/3}(\theta)}{{\qtau} (|\theta|+{\qtau})}$.
Using \eqref{sigmacrit} and \eqref{partialsigmaAGamma},
\begin{equation}\label{critsigma}
\sigma^2+xq^{1/3}(\theta)\Big(1+\Lp(y,\vartheta)+2\frac{\sigma q^{1/3}(\theta)}{{\qtau}(\alpha,\theta)}\mu(y,\vartheta)+\mathcal{H}_{j\geq 2}\Big)=\alpha\,.
\end{equation}
Recall that $\Lp$ was chosen after \eqref{eqorder2} and depends  on the curvature at the boundary near $y=0$:
\[
1+\Lp(y,\vartheta)=\Big(R_1(y,\vartheta+\nabla_y B_0(y,\vartheta))/q(\vartheta)\Big)^{1/3},  \quad R_1(0,\vartheta)=q(\vartheta),
\]  
where from \eqref{eqB0} we have $\vartheta+\nabla_yB_0(y,\vartheta)=\vartheta(1+O(y))$. By finite speed of propagation of the wave flow, for bounded time $|t|$ we must have $|y|$ bounded (see Lemma \ref{derriere}); hence, there exists $T_0<1$ sufficiently small such that if $|t|\leq T_0$ then 
\begin{equation}\label{R1q>12}
\Big(R_1(y,\vartheta+\nabla_y B_0(y,\vartheta))/q(\vartheta)\Big)^{1/3}>1/2.
\end{equation}
Recall that $\alpha\leq \ceps_{0}$, small. We have $B_0=O(|y|^2)$, $B_2=O(|y|^2)$, $B_{2j}=O(|y|)$, $\forall j\geq 2$ and $\frac{|\theta|}{{\qtau}}-1=\frac{|\theta|^2-{\qtau}^2}{{\qtau}(|\theta|+{\qtau})}=-\frac{\alpha q^{2/3}(\theta)}{{\qtau} (|\theta|+{\qtau})}=O(\alpha)$, therefore the coefficient of $\partial_{\alpha}{\qtau}$ in \eqref{partialBGamma} is like 
\[
B_0(y,\vartheta)-B_2(y,\vartheta)+O(\alpha y)=y(O(y)+O(\alpha)).
\]
For $|t|\leq T_0$, taking $T_0$ smaller if necessary, we assume $|B_0(y,\vartheta)-B_2(y,\vartheta)|=O(y^2)\leq |t|/8$. As $y \lesssim t$ by finite speed of propagation, taking $\ceps_0$ smaller if necessary, we assume $O(\alpha y)\leq |t|/8$ for $\alpha\leq \ceps_0$. Therefore, the contribution from $\partial_{\alpha}({\qtau} B_{\Gamma}(y,\theta/{\qtau}))/\partial_{\alpha}{\qtau}$ in \eqref{critalphaNNf} is $O(|t|/4)$ and the coefficient of $\partial_{\alpha}{\qtau}$ in \eqref{critalphaNNf} behaves like $t\leq T_0<1$. As $\tprime$ is bounded ($f$ has compact support in $\tprime$), it remains to compare $\sigma$ and $2N\alpha^{1/2}$ in \eqref{critalphaNNf}. Going back to \eqref{critsigma}, using that $0\leq x<1$ and that the terms in $\mathcal{H}_{j\geq 2}$ come with the factors $\sigma^2, x,\alpha$, it follows that
\[
(\sigma+x\mu q^{2/3}(\theta)/{\qtau})^2+x q^{1/3}(1+\Lp +\mathcal{H}_{j\geq 2})=\alpha+ x^2\mu^2 q^{4/3}(\theta)/{\qtau}^2,
\]
which implies that $\sigma^2$ is bounded; hence for $|\sigma|>C$ for some constant $C$, repeated integrations by parts in $\sigma$ provide a contribution $O(h^{\infty})$ in every integral in the sum in $N$ in \eqref{eq:suminN}. We obtain that the phase functions in \eqref{eq:suminN} may be stationary in $\alpha$ only for 
\begin{equation}\label{nrN}
2|N|\alpha^{1/2}\leq 2(\frac{c}{2} |t|+|\tprime|+|\sigma|), \quad c:=\sup_{|\theta|\sim 1, \alpha\leq \ceps_0}\frac{q^{2/3}(\theta)}{{\qtau}(\alpha,\theta)},
\end{equation}
where, as $h^{2/3-\ceps}\lesssim \alpha$, we used $|B_L'(\alpha^{3/2}/h)|= \frac{h^2}{\alpha^3}\Big(b_1+\sum_{j\geq 2} jb_j(\frac{h}{\alpha^{3/2}})^{j-1}\Big)\leq h^{\ceps}$.
As the righthand side from \eqref{nrN} is bounded, phase functions in \eqref{eq:suminN} are stationary in $\alpha$ only for 
$|N|\leq \tilde C\alpha^{-1/2}$ for some constant $\tilde C:=\frac c2 T_0+1+C$. Again, as $h^{2/3-\ceps}\lesssim \alpha$ (which is crucial here),  we get for $N$'s that may provide non-trivial contributions
 $|N|\lesssim \alpha^{-1/2}\lesssim  h^{-1/3+\ceps/2}$.
For $|N|$ larger than $h^{-1/3}$ we perform non stationary phase with respect to $\alpha$:  each integration by parts provides a factor $h$ and
"loses" a factor $h^{-2/3}$ corresponding to taking a derivative on $p_h$ which has been defined in \eqref{eq:qh} in terms of $p(x,y,\theta/h,\alpha/h,\sigma/h^{1/3}$), together with a negative power of the derivative of the phase of $K_{\omega}(f)$ with respect to $\alpha$, which depends on $N$ (through the term $2\sqrt{\alpha}N$). Therefore, if $|N|\geq h^{-1/3}$,  we get, after $M\geq 1$ integrations by parts,
\begin{equation*}
\Big|  \langle \sum_{|N|\geq h^{-\frac 13}}e^{-iN L(\omega)}
    , {\cutoffchi^{\flat}}(h^{\frac 23}\omega/\ceps_{0})K_\omega(f)(t,x,y)\rangle_{\omega}\Big|\leq  C_{M} \sum_{|N|\geq h^{-\frac 1 3}} \Big(\frac{h^{1-\frac 2 3}}{\sqrt{\alpha}N}\Big)^{M}
    \leq  C_{M} h^{M\frac \ceps 2}\,.
\end{equation*}
as the sum in $N$ is bounded for $M\geq 2$, and therefore this provides a contribution $O(h^{\infty})$. We just proved that, for any smooth $f$, the sum over $N$ \eqref{eq:suminN} is essentially finite over $|N|\lesssim h^{-1/3}$.
\end{proof}
In the following we introduce $g_{h,a}$ provided by Lemma \ref{lemp} in this finite sum (for $|N|\lesssim h^{-1/3}$) and prove that for $a>h^{2/3-\ceps}$, $N>|t| a^{-1/2}$, $V_N$ provides an $ O_{C^{\infty}}(h^{\infty})$ contribution.
Let $K_{\omega}(g_{h,a})$ be given by \eqref{Kghauseful}. The phase function of $V_N(t,x,y)$ defined in \eqref{eq:newVN} is 
\[
t{\qtau}(\alpha,\theta)+\Phi(x,y,\theta,\alpha,\sigma)-\Phi(a,0,\theta,\alpha,s)
  -N h L(\alpha h^{-2/3})\,,
\]
with large parameter $1/h$ in front. This phase function is stationary with respect to $\alpha$ if
\begin{multline}\label{critalphaNN}
  \partial_{\alpha}{\qtau}\Big[ t+B_0(y,\vartheta)-B_2(y,\vartheta)-\Big(\frac{|\theta|^2}{{\qtau}^2}-1\Big)B_4(y,\vartheta)+\mathcal{H}_{j\geq 4}\Big]\\
  {}+\sigma-s=2N\alpha^{1/2}\Big(1-\frac 34 B'_L(\alpha^{3/2}/h)\Big)\,.
\end{multline}
Using now \eqref{partialAGamma},
the phase is stationary in $\sigma$ and $s$ when
\begin{gather}
\sigma^2+xq^{1/3}(\theta)-\alpha+{\qtau}(\alpha,\theta)x\partial_{\sigma}A_{\Gamma}(x,y,\sigma q^{1/3}(\theta)/{\qtau}(\alpha,\theta),\theta/{\qtau}(\alpha,\theta))=0,\\
s^2+aq^{1/3}(\theta)-\alpha+{\qtau}(\alpha,\theta)a\partial_{s}A_{\Gamma}(a,0,s q^{1/3}(\theta)/{\qtau}(\alpha,\theta),\theta/{\qtau}(\alpha,\theta))=0\,.
\end{gather}
Using \eqref{partialsigmaAGamma} for $\partial_{\sigma}A_{\Gamma}$, we obtain as in \eqref{critsigma}
\begin{gather}\label{critssigma}
\sigma^2+xq^{1/3}(\theta)\Big(1+\Lp(y,\vartheta)+2\frac{\sigma q^{1/3}(\theta)}{{\qtau}(\alpha,\theta)}\mu(y,\vartheta)+\mathcal{H}_{j\geq 2}\Big)=\alpha,\\
s^2+aq^{1/3}(\theta)\Big(1+2\frac{s q^{1/3}(\theta)}{{\qtau}(\alpha,\theta)}\mu(0,\vartheta)+\mathcal{H}_{j\geq 2}\Big)=\alpha\,.
\end{gather}
Both $a\leq a_0$ and $\alpha\leq \ceps_0$ being small, the second equation in \eqref{critssigma} yields 
\[
\Big(s(1+O(a)+aq^{2/3}(\theta)/{\qtau}\Big)^2+aq^{1/3}(\theta)(1+O(a))\sim \alpha,
\]
and therefore $|s|=\sqrt{\alpha}+O(a)$. Moreover, we must have $a\lesssim\alpha$, otherwise non stationary phase in $s$ provides an $O(h^{\infty})$ contribution.Therefore we introduce a cut-off $(1-\chi)(4\alpha/a)$ supported for $\alpha>a/4$ in the symbol of $V_N$ without changing its contribution modulo $O(h^{\infty})$ (see Remark \ref{rmqalpha>a}).
For $T_0$ sufficiently small such that \eqref{R1q>12} to hold, the first equation in \eqref{critssigma} yields, 
\[
\Big(\sigma(1+O(x))+x\mu q^{2/3}(\theta)/{\qtau}\Big)^2+x q^{1/3}(1+\Lp(y,\vartheta) +O(x))\sim\alpha.
\]
Note that $x$ remains small (comparable to $\alpha$), otherwise non stationary phase with respect to $\sigma$ will provide an $O(h^{\infty})$ contribution.  We also obtain $|\sigma|=\sqrt{\alpha}+O(x)$. Moreover, both $\sigma$ and $s$ are small and $x\geq 0$, so that \eqref{critssigma} implies that $|\sigma|,|s|\leq \sqrt{\alpha}$ for the phase of $V_N$ to be stationary in $\sigma,s$; for $|\sigma|,|s|\geq 2\sqrt{\alpha}$ we can apply the non-stationary phase theorem, and from \eqref{critalphaNN}, for $T_0$ sufficiently small such that \eqref{R1q>12} to hold, the phase of $V_N(t,x,y)$ is stationary in $\alpha$ only for
\begin{equation}\label{statN}
2|N|\sqrt{\alpha}\leq |t|+|\sigma|+|s|\leq |t|+2\sqrt{\alpha}\,.
\end{equation}
We used again here that for $T_0$ sufficiently small and by finite speed of propagation, the coefficient of $\partial_{\alpha}{\qtau}$ in the lefthand side of \eqref{critalphaNN} is $t+ y(O(y)+O(a))\sim t$ and that $\partial_{\alpha}{\qtau} \sim 1/2$. Moreover, for values $2|N|\sqrt{\alpha}\geq 2 |t|+4\sqrt{\alpha}$, non stationary phase in $\alpha$ provides an $O(h^{\infty})$ contribution from all $V_N$ with $|N|\geq 2+|t|/\sqrt{\alpha}$. As  both $a\lesssim \alpha$ and the number of $V_N$ is (large but) finite ($|N|\leq h^{-1/3}$),  non-trivial contributions in $\mathcal{P}_{h,a}(t,x,y)$ may only be provided by the sum over $|N|\lesssim |t|/\sqrt{\alpha}\lesssim |t|/\sqrt{a}$. We have thus proved the second point in Proposition \ref{propsommefinie}. 

Finally, we turn to the data: by design, $V_{0}(0,x,y)=\chi_0(hD_{x}){\varkappa}(hD_y)\delta_{(a,0)}+O_{C^\infty}(h^\infty)$, and we are left to proving that, for $0<|N|\lesssim h^{-1/3}$, $V_N (0,x,y)\in
O_{C^\infty}(h^\infty)$. In this part we consider $g_{h,a}$ as provided by Lemma \ref{lemp}, of the form \eqref{eq:gosc}
with phase function $\psi_a({\tprime},\theta')={\tprime}^3/3+a({\tprime}q^{1/3}(\theta')+O({\tprime}^3))+O(a^2)\,$.
Then $K_{\omega}(g_{h,a})$ is of the form \eqref{eq:Kom} with $f$ replaced by $g_{h,a}$, that we re-write here
\begin{multline}
  \label{eq:Komnew}
  K_\omega(g_{h,a})(t,x,y)=\frac{h^{2/3}}{2\pi h^d} \int e^{\frac i
    h(t{\qtau}(h^{2/3} \omega,\theta)+\Phi(x,y,\theta,h^{2/3}\omega,\sigma)-y'\cdot \theta-{\tprime}h^{2/3} \omega)}
  p_h(x,y,\theta,h^{2/3}\omega,\sigma)q^{1/6}(\theta)\\
  \times h^{-d} e^{\frac i h (\psi_a({\tprime},\theta')+y'\cdot \theta')}
   r_h(\tprime, \theta')q^{-1/6}(\theta') \,d\theta'
 \,dy'd{\tprime}d\theta d\sigma\,.
\end{multline}
For all $N$, $V_N$ is an oscillatory
integral, that we rewrite, using \eqref{eq:newVN} and Proposition \ref{propsommefinie},
\begin{align}
  \label{eq:newwNN}
V_N(t,x,y)  = & h^{-2/3} \int e^{-iNL(h^{-2/3}\alpha)}{\cutoffchi^{\flat}}(\alpha/\ceps_{0}){\cutoffchi}^{\sharp}(4 \alpha/a)K_{h^{-2/3}\alpha}(g_{h,a})(t,x,y)d\alpha\\
 = & \frac{1}{2\pi h^{d}} \int e^{-iNL(h^{-2/3}\alpha)}{\cutoffchi^{\flat}}(\alpha/\ceps_0)\cutoffchi^{\sharp} (4\alpha/a) e^{\frac i
    h(t{\qtau}(\alpha,\theta)+\Phi(x,y,\theta,\alpha,\sigma)-y'\cdot \theta-{\tprime}\alpha)}\\
 & \times   p_h(x,y,\theta,\alpha,\sigma)q^{1/6}(\theta)\nonumber h^{-d} e^{\frac i h (\psi_a({\tprime},\theta')+y'\cdot \theta')}
    r_h(\tprime, \theta')q^{-1/6}(\theta') \,d\theta' 
 \,dy'd{\tprime}d\theta d\sigma\,.
\end{align}
We can write $q_h:={\cutoffchi^{\flat}}(\alpha/\ceps_{0}) {\cutoffchi^{\sharp}}(4\alpha/a)p_h r_h$, which is an elliptic symbol of order $0$; indeed, recall that $p_h$ comes with a factor $h^{-1/3}$ while $r_h$ comes with a factor $h^{1/3}$. At $t=0$, the stationary points of $V_{N}$ with respect to $\alpha,s,\tprime, y',\theta',\theta$ are solutions to the following  equations
\begin{equation}\label{systt=0}
  \left\{ \begin{array}{l}
\partial_{\alpha}\Phi(x,y,\theta,\alpha,\sigma)=\tprime+2N\alpha^{1/2}\Big(1-\frac 34 B'_L(\alpha^{3/2}/h)\Big)\,,\\
\sigma^2+xq^{1/3}(\theta)+{\qtau}(\alpha,\theta)x\partial_{\sigma}A_{\Gamma}(x,y,\sigma q^{1/3}(\theta)/{\qtau}(\alpha,\theta),\theta/{\qtau}(\alpha,\theta))=\alpha\,,\\
\partial_{\tprime}\psi_{a}(\tprime,\theta')=\alpha\,,\\
\theta'=\theta\,,\,\,\,\,\nabla_{\theta'}\psi_a(\tprime,\theta')=y'\,,\,\,\,\,\nabla_{\theta}\Phi(x,y,\theta,\alpha,\sigma)=y'\,.
\end{array}\right.
\end{equation}
From the first two equations we get \eqref{statN} (at $t=0$), which allows to conclude that if $|N|\geq 2$ the phase is non-stationary in $\alpha$. Hence we are left with $|N|=1$. 
The equations from the last line in  \eqref{systt=0} give $\nabla_{\theta}\Phi(x,y,\theta,\alpha,\sigma)=\nabla_{\theta}\psi_a(\tprime,\theta)$ and therefore
\begin{equation}\label{derivthet}
a\tprime \nabla_{\theta}(q^{1/3}(\theta))+O(a\tprime^2,a^2)=y+\sigma x\nabla_{\theta}(q^{1/3}(\theta))+\nabla_{\theta}({\qtau}(\alpha,\theta)\Gamma).
\end{equation}
For ${\qtau}={\qtau}(\alpha,\theta)$ we have $\nabla_{\theta}{\qtau}=(\theta+\alpha\frac{\nabla_{\theta}(q^{1/3}(\theta))}{3q^{1/3}(\theta)})/{\qtau}$ and
\begin{multline}\label{partialthetatauGamma}
\nabla_{\theta}\Big({\qtau}\Gamma(x,y,\sigma q^{1/3}(\theta)/{\qtau},\theta/{\qtau})\Big)=(\nabla_{\theta}{\qtau})\Gamma(x,y,\sigma q^{1/3}(\theta)/{\qtau},\theta/{\qtau})\\{}+{\qtau} \Big(\nabla_{\theta}(B_{\Gamma}(y,\theta/{\qtau}))+x\nabla_{\theta}(A_{\Gamma}(x,y,\sigma q^{1/3}(\theta)/{\qtau},\theta/{\qtau}))\Big).
\end{multline}
From $B_0=O(|y|^2)$, $B_2=O(|y|^2)$, $B_{2j}=O(|y|)$ we get
\begin{align*}
  \partial_{\theta}({\qtau}(\alpha,\theta)\Gamma) & = \partial_{\theta}{\qtau}\Big(O(y^3)+O(y^2\alpha)+O(y\alpha^2)+x(O(\sigma)+O(\alpha))\Big) +{\qtau}(O(y\alpha)+xO(\sigma))\\
   & =O(y^3)+O(y\alpha)+O(x\sigma)\,.
\end{align*}
Combining this with \eqref{derivthet} yields $y=O(a\tprime)+ O(x\sigma)+O(a^2)$, at the stationary phase points $y',\theta',\theta$ of $V_{N}$. Rescale $x=aX$, 
$\alpha=aq^{1/3}(\theta)\Arond$, $\sigma=\sqrt{a}q(\theta)^{1/6}\Srond$ and $\tprime=\sqrt{a}q(\theta)^{1/6}\Trondprime$; from the first three equations of \eqref{systt=0} (and $\theta'=\theta$) we get
\begin{gather}
  \label{eq:111}
  \Srond+\Trondprime+2{\Arond}^\frac 1 2= O(a^{1/2})\,,\quad    \Srond^2+(1+\Lp)X+O(a^{1/2})={\Arond}={\Trondprime}^2+1+O(a^{1/2})\,.
   \end{gather}
As $a\leq a_0$ is small (taking $a_0$ even smaller if necessary), $\Trondprime^2\leq \Arond-\frac12$; then, using that $\Lp=O(y)=O(a^{3/2})$, there exists $\varepsilon>0$ such that for $X\geq - \varepsilon$, $\Srond^2\leq \Arond+\frac{1}{10}$. For these values, the first equation cannot hold and non stationary phase in $\alpha$ allows to conclude. Again, that we can integrate by parts in $V_{\pm 1}(0,x,y)$ relies on $a\geq h^{2/3-\ceps}$.
\end{proof}
\subsection{Parametrix construction for $a< h^{2/3-\ceps}$, $0<\ceps<1/12$}\label{lowparam}
In formulas \eqref{eq:Prond} and \eqref{eq:Prond2}, $G(\cdots,\omega)$ can be written in terms of the Airy function $Ai(-\zeta)$ and its derivative, as stated in Theorem \ref{thmMelrose}. Using the explicit form of $\zeta$, for $\zeta<0$ (hence for $\omega\leq a/(4h^{2/3})$), these Airy factors in $G$ are exponentially decreasing, so the main contribution comes from values $a\lesssim 4\omega h^{2/3}$. For $a\gtrsim h^{2/3-\ceps}$, this implies $\omega\gtrsim h^{-\ceps}$ and one may perform stationary phase arguments in the integrals from \eqref{eq:Prond}, as long as we pick any $\ceps>0$. In this section, $a$ is much smaller and very different issues arise compared to the previous one; we need a different way to construct a suitable $g_{h,a}$ to recover the initial data. The explicit upper bound on $\ceps$ will be of use in this section and later required in Section \ref{secdispapetit}; there is quite an overlap between both parametrix constructions for $h^{2/3}<a<h^{2/3-1/12}$, but we made not attempt at enlarging it; the reader is advised to think $\ceps$ to be really small.

Besides the lack of a large parameter, which forces us to work with \eqref{eq:Prond2}, the regime $a< h$ has its own difficulties: even deciding how the initial data should be chosen in order the Dirichlet condition to be satisfied on the boundary becomes a non trivial issue. Indeed, \eqref{datachi0var} as initial data  provides a non-trivial contribution on the boundary. Taking $\varkappa(-h^2\Delta)\varkappa(hD_y)\delta_{(a,0)}$ would be a natural choice: it may be expanded on eigenfunctions of the Laplace operator on the compact set $\Omega$, but we know very little on them. Instead, we use the spectral theory for the model Laplace operator \eqref{eq:LapM} in order to expand
$\varkappa(-h^2\Delta_M)\varkappa(hD_y)\delta_{(a,0)}$
in terms of the eigenfunctions of $-\Delta_M$, as they have been used extensively in our previous work \cite{Annals}, \cite{ILP3}. This will turn out to be sufficient for our later purposes. We now recall some properties of $-\Delta_M$.
\subsubsection{Spectral theory for $-\Delta_M$: the initial data in terms of model gallery modes}\label{sss231}
Let $-\Delta_M$ be the Friedlander model operator introduced in \eqref{eq:LapM},
and recall $q$ is a positive definite quadratic form.
Taking the Fourier transform in the $y$ variable, the operator $-\Delta_{M}$ becomes
$-\partial^2_x+|\eta|^2+xq(\eta)$. For  $\eta\neq 0$, this operator is a positive self-adjoint operator 
on $L^2(\mathbb{R}_+)$, with compact resolvent. The next Lemma is proved in \cite{ILP4} (with $q(\eta)=|\eta|^2$, but only using $q(\eta)\neq 0$):
\begin{lemma}\label{lemorthog}(see \cite[Lemma 2]{ILP4})
There exists an Hilbert basis of $L^{2}(\mathbb{R}_{+})$ where $\{e_k(x,\eta)\}_{k\geq 0}$ are eigenfunctions of $-\partial^2_x+|\eta|^2+xq(\eta)$, with eigenvalues $\lambda_k(\eta)=|\eta|^2+\omega_k q(\eta)^{2/3}=\tau_{q}^2(\omega_k,\eta)$. These eigenfunctions are translated and rescaled Airy functions:
\begin{equation}\label{eig_k}
 e_k(x,\eta)=\frac{\sqrt{2\pi} q(\eta)^{1/6}}{\sqrt{L'(\omega_k)}}
Ai\Big(xq(\eta)^{1/3}-\omega_k\Big),
\end{equation}
where $L'(\omega_k)$ (from \eqref{eq:propL2}) normalizes $\|e_k(.,\eta)\|_{L^2(\mathbb{R}_+)}=1$.
\end{lemma}
For $a>0$, the Dirac distribution $\delta_{x=a}$ on $\mathbb{R}_+$ may be decomposed as 
 $\delta_{x=a}=\sum_{k\geq 1} e_k(x,\eta)e_k(a,\eta)$. Eigenfunctions of 
$-\partial^2_{x}+xq(\eta)$ are $e_k(x,\eta)$ with eigenvalue $\lambda_k(\eta)-|\eta|^2=\tau_{q}^2(\omega_k,\eta)-|\eta|^2=\omega_k q^{2/3}(\eta)$,  and for cut-offs $\varkappa$, $\chi_0$ to be chosen, the following spectral decomposition holds
\begin{equation}\label{diracmoddeltaM<} 
\chi_0(-h^2\partial^2_{x}+xq(h\partial_y))\varkappa(hD_y)\delta_{(a,0)}=\sum_{k\geq 1}\int
e^{iy\cdot \eta} \\ 
\chi_0\bigl(h^{\frac 23}\omega_kq^{\frac 23}(h\eta)\bigr) \varkappa(h\eta)e_k(x,\eta)e_k(a,\eta)d\eta\,.
\end{equation}
With $\chi_0\in C^{\infty}_0(-2\ceps_0, 2\ceps_0)$, $\chi_{0}=1$ on $[-\ceps_0,\ceps_0]$, \eqref{diracmoddeltaM<} is a finite sum with $O(\ceps_0/h)$ terms. 
Setting $\theta=h\eta$, from support considerations, taking the support of $\chi_0$ smaller if necessary, we can assume that $\chi^{\flat}(\alpha/\ceps_{0})\chi_0(\alpha q^{2/3}(\theta))\varkappa(\theta)=\chi_0(\alpha q^{2/3}(\theta))\varkappa(\theta)$,
where $\alpha=h^{2/3}\omega$ and $\chi^{\flat}(h^{2/3}\omega/\ceps_{0})$ is the cut-off introduced in \eqref{eq:Prond2cut} (which restricts the support of $K_{\omega}$ to values $\omega\leq \ceps_0/h^{2/3}$).

For $a\leq h^{2/3-\ceps}$, the easiest way to define an initial data is to chose the lefthand side term in \eqref{diracmoddeltaM<}, which does vanish on the boundary ($e_k(0,\eta)=0$ for every $k\geq 1$). Using both \eqref{eq:Prondcut} and \eqref{eq:Prond2cut}, we are left to obtain a smooth function $g_{h,a}$ such that for $\Prond_{h,a}$ as in \eqref{eq:Prond2cut} to have
\begin{equation}\label{dataamic}
\Prond_{h,a}(0,x,y)=\chi_0(-h^2\partial^2_{x}+xq(h\partial_y))\varkappa(hD_y)\delta_{(a,0)}+O_{C^{\infty}}(h^{\infty}).
\end{equation}
We proceed as follows : for a given $K_{\ceps}$ such that $h^{-\ceps}\ll h^{-2\ceps}\lesssim K_{\ceps}\lesssim h^{-1/4+\ceps}\ll h^{-1/4}$, define $\mathcal{E}_{M}$ and split the sum over $k$ in \eqref{diracmoddeltaM<},
\begin{gather}
\label{defEmathcalM}\,\,\quad\begin{multlined}[t]
\mathcal{E}_M(x,y,a,\omega_k):=\int e^{iy\cdot \eta} 
\chi_0(h^{2/3}\omega_kq^{2/3}(h\eta)) \varkappa(h\eta)e_k(x,\eta)e_k(a,\eta)d\eta\\
=\frac{2\pi}{L'(\omega_k)}\int e^{iy\cdot \eta}\chi_0(h^{2/3}\omega_kq^{2/3}(h\eta)) \varkappa(h\eta)q^{1/3}(\eta)Ai(xq^{1/3}(\eta)-\omega_k)Ai(aq^{1/3}(\eta)-\omega_k)d\eta.
\end{multlined}\\
\label{dataEMEM}
  \sum_{k\geq 1}(\cdots)=
\sum_{k\geq1}\cutoffchi^{\flat}\Big(\frac{\omega_k}{\omega_{K_{\ceps}}}\Big)\mathcal{E}_M(x,y,a,\omega_k)+\sum_{k\geq 1}\chi^{\flat}(h^{2/3}\omega_k/\ceps_{0})\cutoffchi^{\sharp}\Big(\frac{\omega_k}{\omega_{K_{\ceps}}}\Big)\mathcal{E}_M(x,y,a,\omega_k)\,.
\end{gather}
\begin{prop}\label{propdataapetitpetit1}
For all $a\in (0,h^{2/3-\ceps})$ there exists a smooth function $g_{h,a,1}$ such that
\begin{equation}\label{trucmuche}
\sum_{k\geq 1}\frac{2\pi}{L'(\omega_k)}\cutoffchi^{\flat}(h^{2/3}\omega_k/\ceps_{0})K_{\omega_k}(g_{h,a,1})(0,x,y)=\sum_{k\geq1}\cutoffchi^{\flat}\Big( \frac{\omega_k}{\omega_{K_{\ceps}}}\Big)\mathcal{E}_M(x,y,a,\omega_k)+O(h^{\infty})\,.
\end{equation}
\end{prop}
The proof of Proposition \ref{propdataapetitpetit1} is postponed to Section \ref{sectproofs}, as it requires arguments from section \ref{sectpseudocalcul}; we will introduce the cutoff $\cutoffchi^{\flat}( {\omega_k}/{\omega_{K_{\ceps}}})$ in the (LHS) term of \eqref{trucmuche} as well, due to how $g_{h,a,1}$ is obtained.
\begin{prop}\label{propdataapetitpetit2}
  For all $a\in (0,h^{2/3-\ceps})$, there exists a smooth function $g_{h,a,2}$ such that
  \begin{gather}
\label{eq:g1}\begin{multlined}[t]
\langle \sum_{N\in \Z} e^{-iN L(\omega)}
    , {\cutoffchi^{\flat}}(h^{2/3}\omega/\ceps_{0}){\cutoffchi^{\sharp}}(h^{2\ceps}\omega)K_\omega(g_{h,a,2})(0,x,y)\rangle_{\omega}\\
\quad\quad\quad\quad    {}=\sum_{k\geq 1}\cutoffchi^{\sharp}\Big(\frac{\omega_k}{\omega_{K_{\ceps}}}\Big)\chi^{\flat}(h^{2/3}\omega_k/\ceps_{0})\mathcal{E}_M(x,y,a,\omega_k)+O(h^{\infty}),
\end{multlined}\\
\label{eq:g12}
\langle \sum_{N\in \Z} e^{-iN L(\omega)}
    , {\cutoffchi^{\flat}}(h^{2/3}\omega/\ceps_{0}){\cutoffchi^{\sharp}}(\omega){\cutoffchi^{\flat}}( h^{2\ceps}\omega)K_\omega(g_{h,a,2})(0,x,y)\rangle_{\omega}=O(h^{\infty}).
\end{gather}
\end{prop}
We introduced a new cut-off ${\cutoffchi^{\sharp}}(h^{2\ceps}\omega)$ in \eqref{eq:g1} and removed ${\cutoffchi^{\sharp}}(\omega)$ (which is supported for $\omega\geq 2$ and identically $1$ on the support of ${\cutoffchi}^{\sharp}(h^{2\ceps}\omega)$). The proof of Proposition \ref{propdataapetitpetit2} will be provided in Section \ref{proofprop2}.
The sum of \eqref{eq:g1} and \eqref{eq:g12} yields the second term in \eqref{dataEMEM}. Finally, we have
\begin{prop}\label{propparama<}
For all $a\in (0,h^{2/3-\ceps})$, let $g_{h,a}:=g_{h,a,1}+g_{h,a,2}$, then $\Prond_{h,a}$, defined in Definition \ref{defparamPoisson} is a parametrix in the sense of Definition \ref{dfnparametrix} and \eqref{dataamic} holds.
\end{prop}
Proposition \ref{propparama<} follows easily from Propositions \ref{propdataapetitpetit1} and \ref{propdataapetitpetit2} using $K_{\omega}(g_{h,a,1}+g_{h,a,2})=K_{\omega}(g_{h,a,1})+K_{\omega}(g_{h,a,2})$ together with \eqref{eq:Prond2}.
\begin{rmq}
For later purposes (dispersion for small $a$, Section \ref{secdispapetit}), taking $K_{\ceps}\sim h^{-2\ceps}$ would be enough. However, in the next subsection, we aim at obtaining gallery modes for $k$ as large as possible, which turns out to be up to $K_{\ceps}\sim h^{-1/4+\ceps}$. This is of independent interest and will prove useful to deal with the Schr\"odinger operator as well as generalize \cite{ILP3} from the model case to the general case; both will be adressed elsewhere.
\end{rmq}
\subsubsection{Pseudo-differential calculus; construction of gallery modes.
}\label{sectpseudocalcul}
\begin{dfn}
Let $G$ be defined in \eqref{eq:Gosc}.
We set
\begin{equation}\label{defeomkgen}
e(x,y,\eta,\omega)=\frac{\sqrt{2\pi} q(\eta)^{1/6}}{\sqrt{L'(\omega)}}e^{-iy\cdot \eta}G(x,y,\eta,\omega).
\end{equation}
Replacing $G(\cdots)$ with $G_M(\cdots):=e^{i y\cdot \eta}Ai(xq(\eta)^{1/3}-\omega)$  in \eqref{defeomkgen} yields \eqref{eig_k} instead of $e(x,y,\eta,\omega)$.
\end{dfn}
\begin{dfn}
Let $\varkappa$ be  like in Definition \eqref{dfnparametrix}. For $g\in L^2(\mathbb{R}^{d-1})$, we define
\begin{equation}\label{defFk}
F_{\omega_k}(g)(x,y): =\frac{1}{(2\pi )^{d-1}}\int e^{i(y-y')\cdot \eta}e(x,y,\eta,\omega_k)(\varkappa(hD_{y'})g)(y')d\eta dy'\,. 
\end{equation}
\end{dfn}
\begin{dfn}\label{blabla}
Let $\tilde \varkappa\in C^{\infty}_0(\mathbb{R}^{d-1})$ be such that $\tilde\varkappa=1$ on the support of $\varkappa$ and vanishing outside a neighborhood of $\mathbb{S}^{d-2}$. Let also $\chi\in C^{\infty}_0$ be a smooth cutoff supported in the ball of center $0$ and radius $1/16$ of $\mathbb{R}^{d-1}$.  
For $f\in L^2(\mathbb{R}^{d-1})$ we define an operator $\Lo$ as
\begin{equation}\label{def:L}
\Lo(f)(y):=\int e^{i(y-y')\cdot\eta-i|\eta|B_0(y',\eta/|\eta|)}\tilde\varkappa(h\eta)\chi(y')f(y')dy'd\eta,
\end{equation}
where $B_0$ is the first term in the development of $B_{\Gamma}$ in \eqref{formalseriesBA} and is homogeneous of degree $0$.
\end{dfn}
To define $g_{h,a,1}$ and $g_{h,a,2}$, we need to "invert" $F_{\omega_k}$, which requires estimating derivatives of $e(x,y,\eta,\omega_k)$ with respect to $(y,\eta)$: in the Friedlander model case the corresponding mode $e_k(x,\eta)$ from \eqref{eig_k} does not depend  on $y$, but here, deriving with respect to $y$ yields a large factor due to $\mathcolor{red}{\mathbf{\tau}} B_0$ in the phase $\mathcolor{red}{\mathbf{\tau}} \Gamma(x,y,\sigma q^{1/3}(\eta)/\mathcolor{red}{\mathbf{\tau}},\eta/\mathcolor{red}{\mathbf{\tau}})$. 
In order to get rid of the homogeneous term of degree zero $|\eta|B_0(y,\eta/|/\eta|)$ in the phase function of $e(x,y,\eta,\omega_k)$, we set $\tilde F_{\omega_k}:=F_{\omega_k}\circ \Lo$ to obtain a new operator whose phase function does not include the contribution $|\eta|B_0(y,\eta/|\eta|)$; we need to prove that, at least for $k\ll h^{-1/4}$, these operators can be inverted and this will be our main result in this section:
\begin{prop}\label{propmatrix}
Let $0<\ceps_{1}<1/4$ be small. The operators $\tilde F_{\omega_k}^*\circ \tilde F_{\omega_k}:L^2(\mathbb{R}^{d-1})\rightarrow L^2(\mathbb{R}^{d-1})$ are pseudo-differential operators that are uniformly elliptic with respect to $1\leq k\leq h^{-1/4+\ceps_{1}}$.
\end{prop}
Note that we would like to construct quasi-modes for as many $k$ as possible in the next section, and therefore $\ceps_{1}$ should really be seen as very small. We momentarily postpone the proof of Proposition \ref{propmatrix} to perform some preliminary steps.

We compute $F_{\omega_k}\circ \Lo(f)$ for $f\in L^2(\mathbb{R}^{d-1})$, setting $\tilde{e}(x,y,\eta,\omega_k)=e^{-i|\eta|B_0(y,\eta/|\eta|)}e(x,y,\eta,\omega_k)$:
\begin{align*}
F_{\omega_k}\circ \Lo(f) (x,y) 
&=\frac{1}{(2\pi )^{d-1}} \int_{\eta} e^{i y\cdot \eta}e(x,y,\eta,\omega_k)\varkappa(h\eta)\widehat{\Lo(f)}(\eta)d\eta \\
&=\frac{1}{(2\pi )^{d-1}} \int_{y'}\int_{\eta} e^{i (y-y')\cdot\eta+i|\eta|(B_0(y,\eta/|\eta|)-B_0(y',\eta/|\eta|))}\tilde{e}(x,y,\eta,\omega_k)\varkappa(h\eta)\chi(y')f(y')d\eta dy'.
\end{align*}
\begin{rmq}
As $y'$ is small on the support of $\chi(y')$,  there exists a cut-off $\tilde\chi$, supported in a ball of radius $1/8$ in $\mathbb{R}^{d-1}$, such that $\tilde\chi=1$ on the support of $\chi$ and 
\[
F_{\omega_k}\circ \Lo(f) (x,y)=\tilde\chi(y) F_{\omega_k}\circ \Lo(f) (x,y)+O(h^{\infty}).
\]
When it will be necessary to emphasize the small size of the support in $y$ of the operators we will work with, we will add the cutoff $\tilde\chi(y)$.
\end{rmq}
Setting $\tilde F_{\omega_k}=F_{\omega_k}\circ \Lo$, we compute the adjoint operator $\tilde F_{\omega_k}^*$ using that for every $f\in L^2(\mathbb{R}^{d-1})$ and $\mathcal{E}\in L^2(\Omega)$ we have $<\tilde F_{\omega_k}(f),\mathcal{E}>_{L^2(\Omega)}=<f,\tilde F_{\omega_k}^*(\mathcal{E})>_{L^2(\mathbb{R}^{d-1})}$. This yields
\begin{multline*}
<\tilde F_{\omega_k}(f),\mathcal{E}>_{L^2(\Omega)} =\int_{\Omega}\Big( \frac{1}{(2\pi )^{d-1}} \int_z\int_{\eta} e^{i (y-z)\cdot \eta}e(x,y,\eta,\omega_k)\varkappa(h\eta)\Lo(f)(z)d\eta dz\Big) \overline{\mathcal{E}}(x,y)dxdy\\
=\int_{y'} f(y')\Big(\overline{\frac{1}{(2\pi )^{d-1}}\int_{\Omega}\int_{\eta} e^{i(y'-y)\cdot \eta+i|\eta|(B_0(y',\eta/|\eta|)-B_0(y,\eta/|\eta|))}\overline{\tilde e(x,y,\eta,\omega_k)}\varkappa(h\eta)\chi(y')\mathcal{E}(x, y)d\eta dx dy}\Big)dy'\\
=\int_z f(z)\overline{\tilde F_{\omega_k}^*(\mathcal{E})}(z)dz=<f,\tilde F_{\omega_k}^*(\mathcal{E})>_{L^2(\mathbb{R}^{d-1})},
\end{multline*}
which yields $\tilde F_{\omega_k}^*(\mathcal{E})$.
We can now explicitly compute $\tilde F_{\omega_k}^*\circ \tilde F_{\omega_k}(f)(z)$ for $f\in L^2(\mathbb{R}^{d-1})$ and $z\in\mathbb{R}^{d-1}$:
\begin{align}\label{FistarFj}
\nonumber
\tilde F_{\omega_k}^*\circ \tilde F_{\omega_k}(f)(z)&=\frac{\chi(z)}{(2\pi )^{d-1}}\int_{\Omega}\int_{\eta} e^{i(z-y)\cdot \eta+i|\eta|(B_0(z,\eta/|\eta|)-B_0(y,\eta/|\eta|))}\varkappa(h\eta)\overline{\tilde e(x,y,\eta,\omega_k)}  \tilde F_{\omega_k}(f)(x,y)d\eta dxdy\\
&=\int_{z'}M_{k}(z,z')f(z')dz' +O(h^{\infty}),
\end{align}
where we have introduced the cut-off $\tilde\chi(y)$ and have set
\begin{multline*}
M_{k}(z,z')=\frac{1}{(2\pi )^{d-1}}\int_{y,\eta,\eta'} e^{i\big((z-y)\cdot \eta+|\eta|(B_0(z,\eta/|\eta|)-B_0(y,\eta/|\eta|))\big)-i\big((z'-y)\cdot \eta'+|\eta'|(B_0(z',\eta'/|\eta'|)-B_0(y,\eta'/|\eta'|))\big)}
\\
\times \tilde\chi(y) \chi(z)\chi(z'){\varkappa}(h\eta) {\varkappa}(h\eta')\int_0^{\infty}\overline{\tilde{e}(x,y,\eta,\omega_k)} \tilde{e}(x,y,\eta',\omega_k) dx d\eta' d\eta dy. 
\end{multline*}
We let $\tiTheta=\tiTheta(\eta,z,y):=\eta+|\eta|\int_0^1\nabla_yB_0(\varpi y+(1-\varpi)z,\eta/|\eta|)d\varpi$ (resp. $\tiTheta'=\tiTheta(\eta',z',y)$) such that 
\begin{gather}
(z-y)\tiTheta =(z-y)\cdot \eta+|\eta|(B_0(z,\eta/|\eta|)-B_0(y,\eta/|\eta|))\,,\\
(z'-y)\tiTheta'=(z'-y)\cdot \eta'+|\eta'|(B_0(z',\eta'/|\eta'|)-B_0(y,\eta'/|\eta'|))\,.
\end{gather}
$\nabla_{\eta}\tiTheta = \mathbb{I}_{d-1}+O(|\nabla_yB_0|)$ being close to the identity matrix $\mathbb{I}_{d-1}$ for small $y,z$, $|\nabla_{\tiTheta}\eta|= 1+O(|\nabla_yB_0|)$ and, in the same way $|\nabla_{\tiTheta'}\eta'|= 1+O(|\nabla_yB_0|)$. Denoting $\eta(\tiTheta,z,y)=\tiTheta+|\tiTheta|O(y,z)$ the inverse function, we have $\eta=\eta(\tiTheta,z,y)$ and $\eta'=\eta(\tiTheta',z',y)$, respectively, and therefore
\begin{multline}\label{thetajk}
M_{k}(z,z')=\frac{1}{(2\pi )^{d-1}}\int_{y,\tiTheta,\tiTheta'} e^{i(z-y)\cdot\tiTheta-i(z'-y)\cdot\tiTheta'}|\nabla_{\tiTheta}\eta||\nabla_{\tiTheta'}\eta'|\times  \tilde\chi(y) \chi(z)\chi(z')\\
\times {\varkappa}(h\eta(\tiTheta,z,y)){\varkappa}(h\eta(\tiTheta',z',y)) \int_0^{\infty}\overline{\tilde{e}(x,y,\eta(\tiTheta,z,y),\omega_k)} \tilde{e}(x,y,\eta(\tiTheta',z',y),\omega_k) dx d\tiTheta' d\tiTheta dy.
\end{multline}
Taking $y=z'+w$, $\tiTheta'=\tiTheta-\zeta$  in \eqref{thetajk}, yields 
\begin{gather}\label{mjk}
M_{k}(z,z')=\int_{\tiTheta}e^{i(z-z')\cdot\tiTheta} m_{k}(z,z',\tiTheta) d\tiTheta\,,\\
m_{k}(z,z',\tiTheta) = \begin{multlined}[t]
\frac{1}{(2\pi )^{d-1}}\int _{\zeta}\int_{w} \int_0^{\infty} e^{-iw\cdot\zeta}|\nabla_{\tiTheta}\eta|(\tiTheta,z,z'+w)|\nabla_{\tiTheta'}\eta|(\tiTheta-\zeta,z',z'+w)\\ 
\times {\varkappa}(h\eta(\tiTheta,z,z'+w)){\varkappa}(h\eta'(\tiTheta-\zeta,z',z'+w))\tilde\chi(z'+w)\chi(z)\chi(z')\\
\times \overline{\tilde{e}(x,z'+w,\eta(\tiTheta,z,z'+w),\omega_k)} \tilde{e}(x,z'+w,\eta(\tiTheta-\zeta,z',z'+w),\omega_k) dx dwd\zeta.
\end{multlined}
\end{gather}
On the support of the two cut-offs $\varkappa$, $h(\tiTheta+O(z,z'+w))\in [1/2,2]$,  $h(\tiTheta-\zeta+O(z',z'+w))\in[1/2,2]$, and on the support of $\chi(z)\chi(z'+w)\tilde\chi(z'+w)$, $|z,z'|<1/16, |z'+w|<1/8$; then, set $\tiTheta=\frac{\theta}{h}$, $\zeta=\frac{\varrho}{h}$ where $\theta\in[1/4,5/4]$ and $|\varrho|\leq 2$. In \eqref{mjk} we may replace $m_{k}(z,z',\tiTheta)$ by $m_{k}(z',z',\tiTheta)=:\tilde{m}_{k}(z',\tiTheta)$ without changing the integral modulo $O(\tiTheta^{-\infty})=O(h^{\infty})$. In the new variables (and modulo $O(h^{\infty})$ terms), the symbol of \eqref{mjk} becomes
\begin{equation}\label{eq:mktilde}
  \tilde{m}_{k}(z',\theta/h) =\frac{1}{(2\pi h )^{d-1}}\int _{\varrho}\int_{w} 
  e^{-\frac ih w\cdot \varrho}
a_{k}((z',\theta); (w,\varrho);h)
  dw d\varrho\,,
\end{equation}
where
\begin{multline}\label{defajk}
a_{k}((z',\theta); (w,\varrho);h):=|\nabla_{\tiTheta}\eta|(\theta/h,z',z'+w)|\nabla_{\tiTheta'}\eta|((\theta-\varrho)/h,z',z'+w)\\\times {\varkappa}(h\eta(\theta/h,z',z'+w)){\varkappa}(h\eta((\theta-\varrho)/h,z',z'-w))\tilde\chi(z'+w)\chi(z)\chi(z')\\
\times \int_0^{\infty}\overline{\tilde{e}(x,z'+w,\eta(\theta/h,z,z'+w),\omega_k)} \tilde{e}(x,z'+w,\eta((\theta-\varrho)/h,z',z'+w),\omega_k) dx\,.
\end{multline}
 Define
\begin{equation}
  \label{eq:20}
  \mathcal{S}_{\gamma}:=\{a\in C^{\infty} \,\,\text{such that}\,\, |\partial_{w}^{\beta_{1}}\partial_{\varrho}^{\beta_{2}}a(w,\varrho,h)|\leq C_{\beta} h^{-\gamma (|\beta_{1}|+|\beta_{2}|)}\}\,.
\end{equation}
We now prove
\begin{prop}\label{propsymbole}
Let $0<\ceps_{1}<1/4$ be small. The symbols $a_k((z',\theta);(w,\varrho);h)$ are in the class $\mathcal{S}_{\frac{1}{2}-\frac 2 3\ceps_{1}}$, uniformly with respect to $1\leq k\leq h^{-\frac 1 4+\ceps_{1}}$.
\end{prop}
\begin{proof}
We check that there exists $\ceps_{1}>0$ such that for every $|\beta|\geq 1$,
\begin{equation}\label{akderb}
h^{|\beta|}\Big|\partial^{\beta}_w\partial^{\beta}_{\varrho}a_{k}(z',\theta;w,\varrho;h)|_{w=0,\varrho=0}\Big|\lesssim h^{\frac 4 3\ceps_{1} |\beta|}\,.
\end{equation}
This easily follows from following lemma whose proof is postponed to the Appendix:
\begin{lemma}\label{lemestimderivekk}
Uniformly for $h^{2/3}\omega_k\ll 1$, we have:
\begin{equation}\label{estimderivekk}
\|\partial^{\beta}_{(y,\theta)}\tilde{e}(.,y,\theta/h,\omega_k)\|_{L^2(x\geq 0)}\lesssim \big(\omega_k/h^{1/3}\big)^{|\beta|}\,.
\end{equation}
\end{lemma}
Using \eqref{defajk}, \eqref{estimderivekk}, and $\omega_k\sim k^{2/3}$, we get that for every $|\beta|\geq 1$, as $k\leq h^{-1/4+\ceps_{1}}$.
\[
  h^{|\beta|}\Big|\partial^{\beta}_w\partial^{\beta}_{\varrho}a_{k}((z',\theta);(w,\varrho);h)|_{w=0,\varrho=0}\Big|
  \lesssim C_{\beta}h^{|\beta|}\big(\omega_k/h^{1/3}\big)^{2|\beta|}\lesssim C_{\beta}\big(h^{4\ceps_{1}/3}\big)^{|\beta|}\,,
\]
which proves the Lemma and completes the proof of Proposition \ref{propsymbole}.
\end{proof}
We now turn our attention to $\tilde m_{k}$ and recall the classical expansion:
\begin{lemma}\label{propstationaryphase}
Let $a(w,\varrho,h):\mathbb{R}^{2n}\rightarrow \mathbb{C}$ be $C^{\infty}$ with $a\in \mathcal{S}_{\frac 12-\ceps_{1}}$, with small $\ceps_{1}>0$. Then 
\[
  \frac{1}{(2\pi h)^n}\int _{\varrho}\int_{w} e^{-\frac ih w\cdot \varrho}a(w,\varrho,h)dw d\varrho=
  \sum_{\beta} \frac{h^{|\beta|}}{i^{|\beta|}|\beta|!}\partial^{\beta}_w \partial_{\varrho}^{\beta} a(w,\varrho,h)\Big|_{(w,\varrho)=(0,0)} \,.
\]
\end{lemma}
Using Lemma \ref{propstationaryphase} for $n=d-1$ with $a_k$ defined in \eqref{defajk} and Proposition \ref{propsymbole}, we get that $\tilde{m}_{k}(z',\theta/h)$ may be written as
\begin{equation}\label{mtildeh}
\tilde{m}_{k}(z',\theta/h)= \sum_{\beta}\frac{h^{|\beta|}}{i^{|\beta|}|\beta|!}\partial^{\beta}_w\partial^{\beta}_{\varrho}a_{k}((z',\theta);(w,\varrho);h)|_{w=0,\varrho=0}\,.
\end{equation}
We may now return to the proof of our main Proposition.
\begin{proof}(of Proposition \ref{propmatrix})
  From \eqref{FistarFj} and \eqref{mjk},
  \begin{equation}
    \label{eq:21}
        \tilde F_{\omega_k}^*\circ \tilde F_{\omega_k}(f)(z)=\int e^{\frac ih (z-z')\cdot\theta}\tilde m_k(z',\theta/h)f(z')dz'\,,
      \end{equation}
 where the symbol $\tilde m_k(z',\theta/h)$ is given by \eqref{mtildeh} and where, for every $k\leq h^{-1/4-\ceps_{1}}$, \eqref{akderb} holds true. Moreover, $\tilde m_k$ is elliptic. Indeed, from \eqref{defajk}, it follows that
\[
a_{k}((z',\theta);(0,0);h)={\sigma}(z',\theta/h;0,0) \|\tilde{e}(.,z',\theta/h,\omega_k) \|^2_{L^2(x>0)},
\]
where for $\tiTheta=\theta/h$ and $\zeta=\varrho/h$ like before, we define
\begin{multline}\label{sigmaandtildes}
\sigma((z',\tiTheta);(w,\zeta))=|\nabla_{\tiTheta}\eta|(\tiTheta,z',z'+w)|\nabla_{\tiTheta'}\eta|(\tiTheta-\zeta,z',z'+w)\\\times {\varkappa}(h\eta(\tiTheta,z',z'+w)){\varkappa}(h\eta(\tiTheta-\zeta,z',z'-w))\tilde\chi(z'+w)\chi(z)\chi(z').
\end{multline} 
Using Lemma \ref{lemequivL2normstildee} in the Appendix,
the $e(.,y,\eta,\omega_k)$ are almost $L^2$-normalized in $x>0$ and $\|\tilde e(\cdot,y,\eta,\omega_k)\|_{L^2(x>0)}\sim 1$. On the other hand the symbol $\sigma$ defined in \eqref{sigmaandtildes} is elliptic, therefore $a_{k}$ is elliptic and $a_{k}((z',\theta);(0,0);h)\sim 1$ for $\theta$ close to $1$; using \eqref{akderb}, $\tilde m_k$ is therefore elliptic; $\tilde F_{\omega_k}^*\circ \tilde F_{\omega_k}$ are pseudo-differential operators, uniformly elliptic for $k\leq h^{-1/4+\ceps_{1}}$.
\end{proof}
\subsubsection{Construction of quasi-modes $k\ll h^{-1/4}$ (proof of Proposition \ref{propdataapetitpetit1}).}\label{sectproofs}
It suffices to  construct a smooth function $g_{h,a,1}$ such that, for all $k$ such that $\chi_{\flat}(\omega_{k}/\omega_{K_{\ceps}})\neq 0$, 
\begin{equation}\label{eqKE}
\frac{2\pi}{L'(\omega_k)}K_{\omega_k}(g_{h,a,1})(0,x,y)=\mathcal{E}_M(x,y,a,\omega_k) +O(h^{\infty}).
\end{equation}
Indeed, for a function $g(y',\tprime)$, let $\hatdeux{g}(y',\alpha/h)$ be its Fourier transform w.r.t. $\tprime$ at $\alpha/h$. 
Using the definition of $K_{\omega}(g)$, we only need $\hatdeux{g}_{h,a,1}(.,\omega/h^{1/3})$ for $\omega\in\{\omega_k, 1\leq k\ll h^{-1/4}\}$.
\begin{lemma}\label{lemfkhatgk}
Let $h^{-2\ceps}\lesssim K_{\ceps}\lesssim h^{-1/4+\ceps}$, $\ceps>0$ small. For $1\leq k\leq 4 K_{\ceps}$, define
\begin{equation}
  \label{eq:13}
  f_{\omega_k}(y):=\Big(\tilde F^*_{\omega_k}\circ \tilde F_{\omega_k}\Big)^{-1}\Big(\tilde F^*_{\omega_k}(\mathcal{E}_M(.,\omega_k))\Big)(y)\,,
\end{equation}
then define $g_{h,a,1}$ with: $\forall k$ such that $\chi^{\flat}(\omega_k/\omega_{K_{\ceps}})\neq 0$,
\begin{equation}
  \label{eq:14}
  \hatdeux{g}_{h,a,1}(.,\omega_k/h^{1/3}) :=\frac{\sqrt{L'(\omega_k)}}{\sqrt{2\pi}}\Lo(f_{\omega_k})\,,\\
\end{equation}
and $\forall k$ such that $\chi^{\flat}(\omega_k/\omega_{K_{\ceps}})=0$,
$\hatdeux{g}_{h,a,1}(.,\omega_k/h^{1/3}) :=0$. Then Proposition \ref{propdataapetitpetit1} 
holds for $g_{h,a,1}$.
\end{lemma}
\begin{proof}
  As $k\leq 4 K_{\ceps}\lesssim h^{-1/4+\ceps}\ll h^{-1/4}$, we can use Proposition \ref{propmatrix} and define $f_{\omega_k}$ by \eqref{eq:13}. For such $f_{\omega_k}$ and $\chi_{\flat}(\omega_{k}/\omega_{ K_{\ceps}})\neq 0$, we now define $\hatdeux{g}_{h,a,1}(y,\omega_k/h^{1/3})$ by \eqref{eq:14} (and zero for larger $k$'s).

  By construction, on the support of $\chi_{\flat}(\omega_{k}/\omega_{K_{\ceps}})$, we have $F_{\omega_k}(\Lo(f_{\omega_k}))(x,y)=\tilde{F}_{\omega_k}(f_{\omega_k})(x,y)=\mathcal{E}_M(x,y,a,\omega_k)$, as we chose $f_{\omega_k}$ such that $\tilde F^*_{\omega_k}\circ \tilde F_{\omega_k}(f_{\omega_k})=\tilde F^*_{\omega_k}(\mathcal{E}_M(.,\omega_k))$. In turn, we have
\begin{equation}\label{eqKEfgfg}
F_{\omega_k}\Big(\frac{2\pi}{\sqrt{L'(\omega_k)}}\hatdeux{g}_{h,a,1}(.,\omega_k/h^{1/3})\Big)(x,y)=\mathcal{E}_{M}(x,y,a,\omega_k)+O(h^{\infty})\,,
\end{equation}
(inverting $\tilde F^{*}_{\omega_{k}}$); using \eqref{eq:Kequiv} and \eqref{defFk}, $\frac{2\pi}{L'(\omega_k)}K_{\omega_k}(g_{h,a,1})(0,x,y)=\mathcal{E}_{M}(x,y,a,\omega_k)+O(h^\infty)$. 
\end{proof}
We obtain the explicit form of $\hatdeux{g}_{h,a,1}$ as a corollary of Lemma \ref{lemfkhatgk}:
\begin{cor}\label{lemgomegak}
We keep the notations from the proof of Lemma \ref{lemfkhatgk}.
Let 
\begin{equation}\label{Io}
I_a(\eta,\omega_k):=\int_{x,y} e^{-iy\cdot \eta}\overline{e(x,y,\eta,\omega_k)}\mathcal{E}_M(x,y,a,\omega_k)dxdy.
\end{equation}
For $1\leq k\leq K_{\ceps}$, $g_{h,a,1}$ (from Lemma \ref{lemfkhatgk}) may be rewritten
\[
\hatdeux{g}_{h,a,1}(y',\omega_k/h^{1/3})=\frac{\sqrt{L'(\omega_k)}}{\sqrt{2\pi}}\int e^{iy'\cdot\eta}\varkappa(h\eta)r(\eta,\omega_k)I_a(\eta,\omega_k)d\eta,
\]
where $r(.,\omega_k)$ is an elliptic symbol of order $0$ and main contribution $1/\tilde m_k(y,\eta+|\eta|\partial_yB_0(y,\eta/|\eta|))$ with $\tilde m_k$ defined in \eqref{mtildeh}. 
\end{cor}
\begin{proof}
We compute explicitly 
\[
\tilde F^*_{\omega_k}(\mathcal{E}_M)(z)=\frac{1}{(2\pi)^{d-1}}\int e^{i(z\cdot \eta+|\eta|B_0(z,\eta/|\eta|))}\varkappa(h\eta)\int_{x,y} e^{-iy\cdot \eta}\overline{e(x,y,\eta,\omega_k)}\mathcal{E}_M(x,y,a,\omega_k)dxdy d\eta.
\]
Moreover, using \eqref{FistarFj} and \ref{mjk}, there exists an elliptic symbol $\tilde r_k(y',\tiTheta)$ of order $0$ with main contribution $1/\tilde m_k$ and $\tilde\varkappa$ supported for $\tiTheta\sim \frac 1h$ and equal to $1$ on the support of $\varkappa$ such that $(\tilde F^*_{\omega_k}\circ \tilde F_{\omega_k})^{-1}(F)(y')=\int e^{i(y'-z)\cdot \eta}\tilde\varkappa(h\eta)\tilde r_k(y',\eta)F(z)dzd\eta$. Taking $F=\tilde F^*_{\omega_k}(\mathcal{E}_M)$ yields
\[
f_{\omega_k}(y')=\frac{1}{(2\pi)^{d-1}}\int e^{i(y'-z)\cdot \eta}\tilde r_k(y',\eta)\varkappa(h\eta')\int e^{i(z\cdot \eta'+|\eta'|B_0(z,\eta'/|\eta'|))} I_a(\eta',\omega_k) d\eta' dz d\eta.
\]
Applying stationary phase with respect to $z,\eta$, critical points are $z=y'$, $\eta=\eta'+|\eta'|\partial_zB_0(z,\eta'/|\eta'|)|_{z=y'}$ and
$$
f_{\omega_k}(y')=\frac{1}{(2\pi)^{d-1}}\int e^{i(y'\cdot \eta+|\eta|B_0(y',\eta/|\eta|))}r_k(y',\eta) I_a(\eta,\omega_k)d\eta
$$
where the new symbol $r_k(y',\eta)$ is obtained from $\tilde r_k$ and has main contribution $\tilde r_k(y',\eta)=1/\tilde m_k$ and where $I$ remains unchanged since depended only on $\eta'$ and not on $x,\eta$. We therefore get
\[
\Lo(f_{\omega_k})(y)=\frac{1}{(2\pi)^{d-1}}\int e^{i(y-y')\cdot\tilde\eta-|\tilde \eta|B_0(y',\tilde\eta/|\tilde\eta|)}\varkappa(h\tilde\eta)\int e^{i(y'\cdot\eta+| \eta|B_0(y',\eta/|\eta|))}r_k(y',\eta) I_a(\eta,\omega_k)d\eta dy'd\tilde \eta
\]
and integrating in $y',\tilde \eta$ give $y'=y$, $\tilde \eta=\eta$ achieves the proof.
\end{proof}

\subsubsection{Proof of Proposition \ref{propdataapetitpetit2}}\label{proofprop2}
Our goal is to obtain $g_{h,a,2}$ such that \eqref{eq:g1} holds. Observe that the sum in the second line of \eqref{eq:g1}, involving $\mathcal{E}_M$ for large $\omega_k$, is, using \eqref{eq:AiryPoisson},
\begin{multline}\label{EMomegamare}
\sum_{k\geq 1}\cutoffchi^{\sharp}\Big( \frac{\omega_k}{\omega_{K_{\ceps}}}\Big)\cutoffchi^{\flat}(h^{2/3}\omega_k/\ceps_{0})\mathcal{E}_M(x,y,a,\omega_k)
\\
=\langle \sum_{N\in \Z} e^{-iN L(\omega)}
    , {\cutoffchi^{\flat}}(h^{2/3}\omega/\ceps_{0})\cutoffchi^{\sharp}\Big(\frac{ \omega}{\omega_{K_{\ceps}}}\Big)\frac{L'(\omega)}{2\pi}\mathcal{E}_M(x,y,a,\omega))\rangle_{\omega},
\end{multline}
and non-stationary phase (with respect to $\omega$) easily applies for all $|N|\geq 2$ in the second line of \eqref{EMomegamare}, providing an $O(h^{\infty})$ contribution (this is just the model case). Therefore, we are left to obtain $g_{h,a,2}$ such that
\begin{multline}\label{tosolve}
\int {\cutoffchi^{\flat}}(h^{2/3}\omega)\cutoffchi^{\sharp}( h^{2\ceps} \omega)\Big(1+\sum_{\pm} e^{\pm i L(\omega)}\Big)K_\omega(g_{h,a,2})(0,x,y)d\omega\\+\sum_{|N|\geq 2}\int e^{- iN L(\omega)}{\cutoffchi^{\flat}}(h^{2/3}\omega/\ceps_{0})\cutoffchi^{\sharp}( h^{\ceps}\omega)K_\omega(g_{h,a,2})(0,x,y)d\omega \\
=\int  {\cutoffchi^{\flat}}(h^{2/3}\omega/\ceps_{0})\cutoffchi^{\sharp}\Big(\frac{ \omega}{\omega_{K_{\ceps}}}\Big)\Big(1+\sum_{\pm} e^{\pm i L(\omega)}\Big)\frac{L'(\omega)}{2\pi}\mathcal{E}_M(x,y,a,\omega)d\omega+O(h^{\infty}).
\end{multline}
Let us analyze the last line of \eqref{tosolve}, corresponding to the sum over model gallery modes. Here $\mathcal{E}_M$ is a product of two Airy functions $e^{iy\cdot \eta}\mathrm{Ai}(-\zeta_M(x,\eta,\omega))\mathrm{Ai}(-\zeta_M(a,\eta,\omega))$, where the phases $\zeta_M=\omega-xq^{1/3}(\eta)$ and $\psi_M(y,\eta)=y\cdot \eta$ are such that \eqref{systeikeq} holds with $<.,.>$ replaced by the scalar product obtained by polarization of the principal symbol $\xi^2+|\eta|^2+xq(\eta)$ of the model Laplace operator $\Delta_M$. Using the definition of $L$ in Lemma \ref{lemL},
$
1+\sum_{\pm}e^{\pm iL(\omega)}=1-\Big(\frac{A_+}{A_-}\Big)(\omega)-\Big(\frac{A_-}{A_+}\Big)(\omega)
$. As $h^{2/3}\omega_{K_{\ceps}}\sim (hK_{\ceps})^{2/3}\gtrsim (h^{1-2\ceps})^{2/3}\gg h^{2/3-\ceps}$, it follows that $h^{2/3}\omega$ is much larger than $a$ on the support of the symbol of the integral in the last line of \eqref{tosolve}, and we can use \eqref{eq:Apm} to write $Ai(-\zeta_M(a,\eta,\omega))=\sum_{\pm}A_{\pm}(\zeta_M(a,\eta,\omega))$. The phase of $\mathcal{E}_M$ is now
\begin{equation}\label{phaEm}
y\cdot \eta+\xi^3/3+\xi(xq^{1/3}(\eta)-\omega)\pm \frac 23 (\omega-aq^{1/3}(\eta))^{3/2}.
\end{equation}
\begin{lemma}
In the integral defining $\mathcal{E}_M(.,\omega)$, the usual stationary phase in $\xi$ applies. Moreover, for the phase corresponding to $N=0$ in the second line of \eqref{EMomegamare}, we have
\[
\phi_{M,\pm,\mp}(x,y,\eta,\omega):= y\cdot \eta\pm\frac 23 \Big((\omega-xq^{1/3}(\eta))^{3/2}-(\omega-aq^{1/3}(\eta))^{3/2}\Big).
\]
In the same way, the phases corresponding to $N=\pm1 $ in the second line of \eqref{EMomegamare} are
\[
\phi_{M,\pm,\pm}(x,y,\eta,\omega)\mp\frac 43 \omega^{3/2}:= y\cdot \eta\pm \frac 23\Big((\omega-xq^{1/3}(\eta))^{3/2}+(\omega-aq^{1/3}(\eta))^{3/2}-2\omega^{3/2}\Big).
\]
Moreover, for $x>2h^{2/3-\ceps}$, the integral in the second line of \eqref{tosolve} is $O(h^{\infty})$.
\end{lemma}
\begin{proof}
Let $N=0$. The derivative with respect to $\omega$ of the phase \eqref{phaEm} of $\mathcal{E}_M$ vanishes when $\xi=\pm \sqrt{\omega-aq^{1/3}(\eta)}$. As $a\leq h^{2/3-\ceps}$ and $\omega\geq \omega_{K_{\ceps}}\gtrsim h^{-2\ceps}\gg a/h^{2/3}$, we introduce a cut-off in $\xi$ localizing for $|\xi|\in[\frac 12 \sqrt{\omega}, 2\sqrt{\omega}]$ without changing the contribution of the integral corresponding to $N=0$ modulo $O(h^{\infty})$ (for $|\xi|\leq \frac 12 \sqrt{\omega}$ the phase \eqref{phaEm} is non-stationary in $\omega$). The second derivative of \eqref{phaEm} with respect to $\xi$ is $2\xi\sim \pm \sqrt{\omega}$ and for $\xi$ such that $|\xi|\sim \sqrt{\omega}\geq \sqrt{\omega_{K_{\ceps}}}\gg h^{-\ceps}$ on the support of the symbol, stationary phase yields the following phase for $N=0$,
\[
\phi_{M,\pm,\pm}(x,y,\eta,\omega):=y\cdot \eta\pm \frac 23 (\omega-x q(\eta)^{1/3})^{3/2}\pm \frac 23(\omega-aq(\eta)^{1/3})^{3/2}\,.
\]
We further notice that the phases $\phi_{M,+,+}$ and $\phi_{M,-,-}$ are non-stationary in $\omega$: indeed, $\partial_{\omega}\phi_{M,+,+}\sim 2\sqrt{\omega}$ and we get $O(h^{\infty})$ by integrations by parts. Moreover, taking switching signs, the derivative with respect to $\omega$ becomes 
\begin{equation}\label{derivphaEMN=0}
\sqrt{\omega-xq(\eta)^{1/3}}-\sqrt{\omega-aq(\eta)^{1/3}}\sim \frac{(a-x)}{2\sqrt{\omega}} q(\eta)^{1/3}\,,
\end{equation}
and from \eqref{derivphaEMN=0}, for $x-a\geq 2h^{2/3-\ceps}-a\geq h^{2/3-\ceps}$ and $\omega_{K_{\ceps}}\lesssim \omega\leq \ceps_0h^{-2/3}$, integrations by parts in $\omega$ provide a $O(h^{\infty})$ contribution in the integral corresponding to $N=0$ in the second line of \eqref{EMomegamare}. In fact, with $\mathcal{E}_M$ as in \eqref{defEmathcalM}, the symbol of the integral in the second line of \eqref{EMomegamare} depends on $\omega$ only through $\chi^{\flat}(h^{2/3}\omega/\ceps_{0})\chi^{\sharp}({\omega}/{\omega_k})$, and therefore, in order to integrate by parts in $\omega$, it remains to check that, for some $\tilde\ceps>0$, we have
\begin{equation}\label{derivtocheck}
\frac{2\sqrt{\omega}}{(x-a)q(\eta)^{1/3}}\partial_{\omega}\Big(\cutoffchi^{\sharp}\Big(\frac{\omega}{\omega_{K_{\ceps}}}\Big)\Big)\ll h^{\tilde\ceps}\,.
\end{equation}
As $\cutoffchi^{\sharp}({\omega}/{\omega_{K_{\ceps}}})$ is constant everywhere but for $\omega\sim \omega_{K_{\ceps}}$, \eqref{derivtocheck} vanishes everywhere but for $\omega\sim\omega_{K_{\ceps}}$; from $\eta\sim \frac 1h$ and $\omega_{K_{\ceps}}\sim K_{\ceps}^{2/3}\gg h^{-4\ceps/3}$, the lefthand side in \eqref{derivtocheck} is at most $\frac{h^{2/3}}{h^{2/3-\ceps}\sqrt{\omega_{K_{\ceps}}}}$ ; 
we obtain \eqref{derivtocheck} with $\tilde \ceps=5\ceps/3$ and we can integrate by parts infinitely many times. 

For $N=\pm1 $ we proceed in a similar manner. Let for instance $N=1$, then the phase of $e^{-iL(\omega)}\mathcal{E}_M$ is just \eqref{phaEm}$-\frac 43\omega^{3/2}$; it is stationary with respect to $\omega$ for $\xi=2\sqrt{\omega}-\sqrt{\omega-aq^{1/3}(\eta)}$. As $\omega\geq \omega_{K_{\ceps}}\gtrsim h^{-4\ceps/3}\gg a/h^{2/3}$ we again introduce a cut-off, supported for $|\xi|\in [\frac 12 \sqrt{\omega},2\sqrt{\omega}]$ without changing the contribution of the integral modulo $O(h^{\infty})$. Stationary phase then applies in $\xi$ and provides the phase function $\phi_{M,\pm,\pm}-\frac 43\omega^{3/2}$. We easily see that $\phi_{M,\pm,\mp}$ and $\phi_{M,-,-}$ are non-stationary with respect to $\omega$ and provide a $O(h^{\infty})$ contribution. We are left with $\phi_{M,+,+}-\frac 43 \omega^{3/2}$ whose derivative with respect to $\omega$ is
\begin{equation}\label{derivphaEMN=1}
\sqrt{\omega-xq(\eta)^{1/3}}+\sqrt{\omega-aq(\eta)^{1/3}}-2\sqrt{\omega}\sim -(x+a)q(\eta)^{1/3}/(2\sqrt{\omega})\,.
\end{equation}
From $a\lesssim h^{2/3-\ceps}$ we obtain that, for $x\geq 2h^{2/3-\ceps}$, the phase $\phi_{M,+,+}-\frac 43 \omega^{3/2}$ is non-stationary in $\omega$ and yields an $O(h^{\infty})$ contribution. The exact same line of reasoning applies to $N=-1$. 
\end{proof}
\begin{rmq}
For $|x,a|\leq \frac{h^{2/3}}{\sqrt{\omega_{K_{\ceps}}}}$, we cannot get an  $O(h^{\infty})$ contribution for $N=\pm 1$: \eqref{derivtocheck} does not hold anymore even though for $x\geq 0$ the derivative with respect to $\omega$ does not vanish; for such small values of $a$ we cannot perform integrations by parts. Specifically, as $K_{\ceps}\ll h^{-1/4}$, we are to deal with this case for all $|x,a|\leq h^{2/3+1/12}=h^{3/4}(\ll\frac{h^{2/3}}{\sqrt{\omega_{K_{\ceps}}}})$.
\end{rmq}
Putting all this together, the integral in the last line in \eqref{tosolve} reads as 
\[
\int  {\cutoffchi^{\flat}}(h^{2/3}\omega/\eps_{0})\cutoffchi^{\sharp}\Big(\frac{ \omega}{\omega_{K_{\ceps}}}\Big)\Big(1+\sum_{\pm} e^{\pm i L(\omega)}\Big)\frac{L'(\omega)}{2\pi}\mathcal{E}_M(x,y,a,\omega)d\omega=E_{M,+}(x,y,a)+E_{M,-}(x,y,a),
\]
where we have set, modulo $O(h^{\infty})$,
\begin{multline}\label{EM1}
E_{M,+}(x,y,a):=\int  {\cutoffchi^{\flat}}(h^{2/3}\omega/\eps_{0})\cutoffchi^{\sharp}({\omega}/{\omega_{K_{\ceps}}}) \int q(\eta)^{1/6}\varkappa(h\eta)\chi_0(h^{2/3}\omega q^{2/3}(h\eta))\\
\times e^{iy\cdot \eta}\Big(A_+(\zeta_M(x,\eta,\omega))-\Big(\frac{A_+}{A_-}\Big)(\omega)A_-(\zeta_M(x,\eta,\omega))\Big)A_-(\zeta_M(a,\eta,\omega))d\eta d\omega,
\end{multline}
\begin{multline}\label{EM2}
E_{M,-}(x,y,a):=\int  {\cutoffchi}^{\flat}(h^{2/3}\omega/\eps_{0})\cutoffchi^{\sharp}({\omega}/{\omega_{K_{\ceps}}})\int q(\eta)^{1/6}\varkappa(h\eta)\chi_0(h^{2/3}\omega q^{2/3}(h\eta))\\
\times e^{iy\cdot \eta}\Big(A_-(\zeta_M(x,\eta,\omega))-\Big(\frac{A_-}{A_+}\Big)(\omega)A_+(\zeta_M(x,\eta,\omega))\Big)A_+(\zeta_M(a,\eta,\omega))d\eta d\omega,
\end{multline}
where the phase functions of the Airy terms in the second line of \eqref{EM1} are $\phi_{M,+,-}$ and $\phi_{M,-,-}+L(\omega)$, while the phase functions of the Airy terms in the second line of \eqref{EM2} are $\phi_{M,-,+}$ and $\phi_{M,+,+}-L(\omega)$. Moreover, for $x>2h^{2/3-\ceps}$, $E_{M,\pm}(x,y,a)=O(h^{\infty})$ and $E_{M,\pm}(0,y,a)=0$. This means that we can introduce a smooth cut-off $\chi_1( x/h^{2/3-\ceps})$ with $\chi_1\in C^{\infty}_0$ equal to $1$ on $[-1,1]$ and equal to $0$ for $x\geq 2h^{2/3-\ceps}$ such that $E_{M,\pm}(x,y,a)=\chi_1(x/h^{2/3-\ceps})E_{M,\pm}(x,y,a)+O(h^{\infty})$, and therefore we need to construct $g_{h,a,2}$ such that \eqref{tosolve} holds with the last line replaced by $\chi_1(x/h^{2/3-\ceps})(E_{M,+}(x,y,a)+E_{M,-}(x,y,a))$ (instead of $E_{M,+}(x,y,a)+E_{M,-}(x,y,a)$). 

We now go back to \eqref{tosolve}: the symbol of its left hand side has support in $\omega\gtrsim h^{-2\ceps}$, while the right hand side is essentially supported for $x,a\lesssim h^{2/3-\ceps}$. For such values of $x$ and $\omega$ we have $\zeta(x,y,\eta,\omega)=\omega-x|\eta|^{2/3}e_0(x,y,\eta/|\eta|,\omega_k/|\eta|^{2/3})\geq \omega/2$; using \eqref{eq:defG} we write
\[
G(x,y,\eta,\omega)=e^{i\psi}\sum_{\pm}\Big(p_0A_{\pm}(\zeta)+i  |\eta|^{-1/3}p_1A'_{\pm}(\zeta)\Big)=:G_{\pm}(x,y,\eta,\omega).
\]
\begin{prop}\label{propapetitg2}
There exists smooth functions $g_{h,a,2,\pm}$ such that, with $g_{h,a,2}:=\sum_{\pm} g_{h,a,2,\pm}$,
\begin{multline}\label{g2-}
\int \cutoffchi^{\flat}(h^{2/3}\omega/\ceps_{0})\cutoffchi^{\sharp}( h^{2\ceps}\omega)\Big(G_+(x,y,\eta,\omega)-\Big(\frac{A_+}{A_-}\Big)(\omega)G_-(x,y,\eta,\omega)\Big)\\
q(\eta)^{1/6}\varkappa(h\eta)\varkappa(h\tau_{q}(\omega,\eta))\hat{g}_{h,a,2,-}(\eta,\omega/h^{1/3})d\eta d\omega
=E_{M,+}(x,y,a)+O(h^{\infty}),
\end{multline}
\begin{multline}\label{g2+}
\int\cutoffchi^{\flat}(h^{2/3}\omega/\ceps_{0})\cutoffchi^{\sharp}( h^{2\ceps}\omega) \Big(G_-(x,y,\eta,\omega)-\Big(\frac{A_-}{A_+}\Big)(\omega)G_+(x,y,\eta,\omega)\Big)\\
\times q(\eta)^{1/6}\varkappa(h\eta)\varkappa(h\tau_{q}(\omega,\eta))\hat{g}_{h,a,2,+}(\eta,\omega/h^{1/3})d\eta d\omega
=E_{M,-}(x,y,a)+O(h^{\infty}).
\end{multline}
\begin{multline}\label{g2+0}
\int\cutoffchi^{\flat}(h^{2/3}\omega/\ceps_{0})\cutoffchi^{\sharp}( h^{2\ceps}\omega) \Big(G_+(x,y,\eta,\omega)-\Big(\frac{A_+}{A_-}\Big)(\omega)G_-(x,y,\eta,\omega)\Big)\\
q(\eta)^{1/6}\varkappa(h\eta)\varkappa(h\tau_{q}(\omega,\eta))\hat{g}_{h,a,2,+}(\eta,\omega/h^{1/3})d\eta d\omega
=O(h^{\infty}),
\end{multline}
\begin{multline}\label{g2-0}
\int\cutoffchi^{\flat}(h^{2/3}\omega/\ceps_{0})\cutoffchi^{\sharp}( h^{2\ceps}\omega) 
\Big(G_-(x,y,\eta,\omega)-\Big(\frac{A_-}{A_+}\Big)(\omega)G_+(x,y,\eta,\omega)\Big)\\
\times q(\eta)^{1/6}\varkappa(h\eta)\varkappa(h\tau_{q}(\omega,\eta))\hat{g}_{h,a,2,-}(\eta,\omega/h^{1/3})d\eta d\omega
=O(h^{\infty}),
\end{multline}
\begin{equation}\label{toshowsecond}
  \sum_{|N|\geq 2 }\int e^{- iN L(\omega)}\int\cutoffchi^{\flat}(h^{2/3}\omega/\ceps_{0})\cutoffchi^{\sharp}( h^{2\ceps}\omega)
  K_\omega(g_{h,a,2})(0,x,y)d\omega =O(h^{\infty}).
\end{equation}
\end{prop}
\begin{proof}
Proving that both operators in the first lines of \eqref{g2+} and \eqref{g2-} are invertible is sufficient: once we define $g_{h,a,2,\pm}$ it will be clear that \eqref{g2+0}, \eqref{g2-0} and \eqref{toshowsecond} hold (using non-stationary phase arguments). In fact, we notice from \eqref{g2-} that $g_{h,a,2,-}$ has to have a phase function whose derivative with respect to $\omega$ should equal $ 2\sqrt{\omega}+O(a)$, since otherwise the phase of \eqref{g2-} is non-stationary in $\omega$. Introducing such a function $\hat{g}_{h,a,2,-}$ in the integral in the first line of \eqref{g2-0} yields a phase function (for \eqref{g2-0}) whose derivative with respect to $\omega$ behaves like $4\sqrt{\omega}$ and since $\omega$ is large on the support of the symbol this allows to perform repeated integrations by part with respect to $\omega$ to obtain a $O(h^{\infty})$ contribution. In the same way we prove \eqref{toshowsecond}, since for $|N|\geq 2$ all the phase functions will be non-stationary in $\omega$ and after each integration by parts we obtain a factor $(N\sqrt{\omega})^{-1}$, which will allow to sum up over $N$ to conclude. We are reduced to proving that we can define $g_{h,a,2,-}$ satisfying \eqref{g2-} (solving \eqref{g2+} follows in exactly the same way).
\begin{prop}\label{proptildeJ}
Let $\tilde J_+:=J_++R_+$, with
\begin{multline}
  J_+(f)(x,y)=\int \cutoffchi^{\flat}(h^{2/3}\omega/\ceps_{0})\cutoffchi^{\sharp})(h^{2\ceps}\omega)
  G_+(x,y,\eta,\omega)\chi_1(x/h^{2/3-\ceps}) \\
\times q(\eta)^{1/6}\varkappa(h\eta)\varkappa(h\tau_{q}(\omega,\eta))e^{i(y'\cdot\eta+\tprime \omega/h^{1/3})} f(y',\tprime )d\eta d\omega dy'd\tprime.
\end{multline}
\begin{multline}
  R_+(f)(x,y)=\int \cutoffchi^{\flat}(h^{2/3}\omega/\ceps_{0})\cutoffchi^{\sharp})(h^{2\ceps}\omega)
  \Big(\frac{A_+}{A_-}\Big)(\omega)G_-(x,y,\eta,\omega)\chi_1(x/h^{2/3-\ceps})  \\
\times q(\eta)^{1/6}\varkappa(h\eta)\varkappa(h\tau_{q}(\omega,\eta))e^{i(y'\cdot\eta+\tprime \omega/h^{1/3})} f(y',\tprime )d\eta d\omega dy'd\tprime.
\end{multline}
The operator $\tilde J_+$ is well defined from $\mathcal{S}'_{y',\tprime}$ into the space of functions of $(x,y)$ near $(0,0)$, and with $h$ as small parameter, $J_+$ is an elliptic semi-classical Fourier integral operator. Moreover,  $\|J_+^{-1}\circ R_+\|_{\mathcal{L}(L^{2})}=O(h^{\infty})$, hence $\tilde J_+$ is invertible and $\tilde J^{-1}_+=\Big(I+J^{-1}_+\circ R_+\Big)^{-1}\circ J^{-1}_+$.
\end{prop}
If we now chose $g_{h,a,2,-}(y',\tprime):=\tilde J_+^{-1}(E_{M,+})$, this achieves the proof of Proposition \ref{propapetitg2}.
\end{proof}
\begin{proof}(of Proposition \ref{proptildeJ})
The operator $J_+$ is easily elliptic and invertible with phase function $\psi+\frac 23\zeta^{3/2}$ with $\psi$ and $\zeta$ defined in Theorem \ref{thmMelrose}. The phase function of $R_+$ is $\psi-\frac 23\zeta^{3/2}+\frac 43 \omega^{3/2}$. Therefore the phase function of $J^{-1}_+\circ R_+$ is given by
\begin{equation}\label{phaBR}
-\psi(x,y,\eta,\omega)-\frac 23\zeta^{3/2}(x,y,\eta,\omega)+\psi(x,y,\tilde \eta,\tilde\omega)-\frac 23 \zeta^{3/2}(x,y,\tilde\eta,\tilde \omega)+\frac 43 \tilde\omega^{3/2},
\end{equation}
where $x,y$ are now integration variables. The derivative of \eqref{phaBR} with respect to $x$ is
\[
-\partial_x\zeta(x,y,\eta,\omega)\sqrt{\zeta(x,y,\eta,\omega)}-\partial_x\zeta(x,y,\tilde \eta,\tilde\omega)\sqrt{\zeta(x,y,\tilde\eta,\tilde\omega)}-\partial_x\psi(x,y,\eta,\omega)+\partial_x\psi(x,y,\tilde\eta,\tilde\omega),
\]
where $|\eta|,|\tilde \eta|\sim 1/h$, $\omega,\tilde\omega\geq h^{-2\ceps}$ and $xq^{1/3}(\eta)\leq h^{2/3-\ceps-2/3}=h^{-\ceps}$ on the support of the symbol. As $\zeta=\omega-x|\eta|^{2/3}e_0(x,y,\eta/|\eta|,\omega/|\eta|^{2/3})$ with $e_0$ elliptic and close to $1$, the derivatives of the two terms involving $\zeta$ in \eqref{phaBR} are such that 
\[
\Big|\partial_x\zeta(x,y,\eta,\omega)\sqrt{\zeta(x,y,\eta,\omega)}+\partial_x\zeta(x,y,\tilde \eta,\tilde\omega)\sqrt{\zeta(x,y,\tilde\eta,\tilde\omega)}\Big|\gtrsim \frac 12 (\sqrt{\omega}|\eta|^{2/3}+\sqrt{\tilde \omega}|\tilde\eta|^{2/3}).
\]
On the other hand, using Corollary \ref{corUpsi} and \eqref{eq:Ups} in particular, the derivative with respect to $x$ of $\psi$ is of the form
\[
\partial_x\psi(x,y,\eta,\omega)= \mu(y,\eta/|\eta|)q^{2/3}(\eta)/\tau_{q}(\omega,\eta)\Big(\omega+2xq^{1/3}(1+\Lp(y,\eta/|\eta|))+\mathcal{H}_{j\geq 4}\Big),
\]
and therefore
\[
\Big|\partial_x\psi(x,y,\eta,\omega)-\partial_x\psi(x,y,\tilde \eta,\tilde\omega)\Big|\lesssim  \omega |\eta|^{1/3}+\tilde\omega |\tilde\eta|^{1/3},
\]
where we have used that $q^{2/3}(\eta)/\tau_{q}(\omega,\eta)=|\eta|^{4/3}q^{2/3}(\eta/|\eta|)/(|\eta|\tau_{q}(\omega/|\eta|^{2/3},\eta/|\eta|))\sim |\eta|^{1/3}$. On the support of the symbol $\chi(h^{2/3}\omega)\chi(h^{2/3}\tilde\omega)$ we have $|\omega|,|\tilde\omega|\leq \ceps_0h^{-2/3}$, which means that the main contribution of the derivative of \eqref{phaBR} comes from the terms involving $\zeta$ and behaves like $\simeq (\sqrt{\omega}|\eta|^{2/3}+\sqrt{\tilde \omega}|\tilde\eta|^{2/3})$, as for $|\eta|,|\tilde\eta|\simeq 1/h$ we have $\sqrt{\omega}|\eta|^{2/3}\gg \omega|\eta|^{1/3}$  ($\omega h^{2/3}\ll 1$). To perform non stationary phase and obtain an $O(h^{\infty})$ contribution, we check that taking one derivative with respect to $x$ of the symbol provides a factor $O(h^{\ceps})$. Indeed, ${h^{2/3}}\partial_x(\chi(x/h^{2/3-\ceps}))/{\sqrt{\omega}}\sim h^{\ceps}/\sqrt{\omega}\leq h^{\frac 53 \ceps}$, which completes the proof.
\end{proof}
To complete the proof of Proposition \ref{propdataapetitpetit2}, it remains to prove that, for $g_{h,a,2,-}:=\tilde J^{-1}_+(E_{M,+})$, \eqref{eq:g12} holds: but then in \eqref{eq:g12} one obtains a vanishing symbol as $\chi^{\flat}(h^{2\ceps}\omega)\chi^{\sharp}( h^{2\ceps}\omega)=0$.\qed
\section{Dispersion estimates when $a\geq h^{2/3-\ceps}$}
\label{sec:dispersion-estimates}
Here again one should think of $\ceps$ as being very small: we may set $0<\ceps<1/12$ to be consistent with the parametrix construction we just did in the opposite regime (subsection \ref{lowparam}), to have an overlap between both regimes where we get dispersion estimates by different arguments. We now use the parametrix as a sum over $N$ to obtain the following dispersion estimates, restricting to positive times for the sake of simplicity.
\begin{thm}\label{dispintermediaire}
There exist $a_{0}$, $c$, $C$, $\ceps$ such that for all $|(t,x,y,h,a)|<a_{0}$, one has
\begin{itemize}
\item for $t\leq c \sqrt a$,
  \begin{equation}
    \label{eq:1}
|  \Prond_{h,a}(t,x,y)|\leq C h^{-d}  \min \left(1, (h/t)^{\frac{d-1} 2}\right) \,;
  \end{equation}
\item For $a\geq h^{2/3-\ceps}$, $t>c\sqrt a$,
  \begin{equation}
    \label{eq:2}
|  \Prond_{h,a}(t,x,y)|\leq C h^{-d}  \left(\frac h t\right)^{\frac{d-2} 2} \left(  (\max(a,x))^{\frac 1 4} \left(\frac h t \right)^{\frac 1 4} +h^{\frac 1 3}\right)\,;
  \end{equation}
\item For $a\leq h^{1/3+\ceps}$,
  \begin{equation}
    \label{eq:3}
|  \Prond_{h,a}(t,x,y)|\leq C h^{-d}  \left(\frac h t\right)^{\frac{d-2} 2+\frac 1 3} \,.
  \end{equation}
\end{itemize}
\end{thm}
The first estimate, \eqref{eq:1}, is just the (short time) dispersion for a free wave. On this timescale, the wave has at most one reflection and singularities have not appeared yet. One should point out that (a suitable version of) such dispersion is already proved in \cite{blsmso08}, for a more general boundary.

The second estimate, \eqref{eq:2}, is proved using the parametrix as a sum over reflected waves. The first term is due to swallowtail singularities (and always larger than the corresponding factor in the free dispersion) and the second term is due to the presence of cusps appearing after each swallowtail singularity, between two consecutive reflections; notice that here we use the parametrix construction in an extended region $a\geq h^{2/3-\ceps}$ (when compared to the previous parametrix construction in \cite{Annals}, where it was obtained as a superposition of waves only for $a\geq h^{4/7-\ceps}$).

The third estimate, \eqref{eq:3}, will be proved using the parametrix as a sum over quasi-modes, and we postpone its proof to the last section, where we deal with decay of such quasi-modes.

We start with a lemma which allows to deal with the parametrix ``behind the wave front''.
\begin{lemma}
\label{derriere}
There exist $c_{0}$ and $T_{0}$ such that , with $\mathcal{B}=\{ 0\leq x\leq a, |y|\leq c_{0} t, 0<h\leq t \}$, 
  \begin{equation}
    \label{eq:2bis}
\forall t\in [0, T_{0}]\,,\quad\quad   \sup_{x,y,t\in \mathcal{B}}  |  \Prond_{h,a}(t,x,y)|\leq C h^{-d}  O(( h/ t)^{\infty})\,.
  \end{equation}
\end{lemma}
\begin{proof}
The lemma follows from classical propagation of singularities: for $t\leq a/C_{1}$ with a large $C_{1}$, we apply Melrose-Sjöstrand's theorem in the interior of $\Omega$.
Now, let $T_{0}$ be small enough and consider a given $s\in [h,T_{0}]$. Rescale with $(t',x',y')$ as new variables and $\hbar$ a new parameter: $t=st'$, $x=sx'$, $y=sy'$, $\hbar=h/s$, let $v_{s}(t',x',y')=v(st',sx',sy')$ for any function $v$, then
\begin{equation}
  \label{eq:5}
  (\Box v)_{s} = \frac 1 {s^{2}} \Box_{s} v_{s}\,\,\text{with }\,\, \Box_{s} = - \partial^{2}_{t'}+\partial^{2}_{x'}+s^{2} R(sx', sy', s^{-1} D_{y'})\,.
\end{equation}
Set $b=a/s$, $0< b\leq c_{1}$ and $\Prond_{b,s,\hbar}=s^{d} (\Prond_{a,h})_{s}$; then one may apply the Melrose-Sj\"ostrand theorem to $\Prond_{b,s,\hbar}$ to obtain $\Prond_{b,s,\hbar}\in O(\hbar^{\infty})$ for $0\leq x'\leq b$, $|y'|\leq c_{0}$ and $t'=1$, when $s\leq T_{0}$ and $T_0$ is small enough. In fact, $\Prond_{b,s,\hbar}$ is a parametrix for $\Box_{s}$, and $\Box_{s}$ is smooth in $s$ and for $s$ small enough, $\Box_{s}$ is close to the usual wave operator $-\partial^{2}_{t'}+\partial^{2}_{x'}+\Delta_{y'}$ (recall that we picked boundary normal coordinates, and $R_{0}(y,
\eta)=|\eta|^{2}+O(|y|)$). If we denote by $\Grond_{b,s}$ the Green function of $\Box_{s}$,
\begin{equation}
  \label{eq:6}
  \Box_{s} \Grond_{b,s}=0 \text{ in  } \Omega,\quad \Grond_{b,s}=0 \text{ on } \partial\Omega,
\end{equation}
 with $\Grond_{b,s}|_{t'=0}=\delta_{(b,0)}$ and $\partial_{t'}  \Grond_{b,s}|_{t'=0}=0$, then we have $\Prond_{b,s,\hbar}={\varkappa}(\hbar D_{t'})\mathcal{Q}(sx',sy', \hbar D_{y'}) \Grond_{b,s}$ and its wave front set is described by the Melrose-Sj\"ostrand theorem. To complete our proof, we need to check the following property of an optical ray $\sigma\in [0,1]\ra \mathcolor{red}{\mathbf{\tau}}_{s}(\sigma)$: if it starts at $\sigma=0$ from $(b,0,\xi_{0},\eta_{0})$ with $\xi^{2}_{0}+R(sb,0;\eta_{0})=1$ and  if $|\eta_{0}|\geq c_{1}>0$ for some constant $c_1$, then, if $\mathcolor{red}{\mathbf{\tau}}_{s}(\sigma)=(x'(\sigma),y'(\sigma),\xi'(\sigma),\eta'(\sigma))$, there exists $c_0$ such that $|y'(1)|\geq c_{0}>0$. Note that $(y',\eta')(\sigma)$ are solutions to the Hamilton-Jacobi equations $\partial_{\sigma} y'=\partial_{\eta'} R_{s}$, and $\partial_{\sigma} \eta'=- \partial_{y'} R_{s}$, with $y'(0)=b$, $\eta'(0)=\eta_{0}$, $R_{s}(x',y',\eta')=s^{2} R(sx',xy',\eta'/s)$. For some $c>0$ small, $|y'(\sigma)|\geq 2 c_{1} \sigma-cs\sigma$ for all $\sigma \in [0,1]$, and if $T_{0}$ is small enough, we get  the lower bound $|y'(1)|\geq c_{0}>0$ for some $c_0>0$.
\end{proof}
\subsection{Number of waves that contribute in $\Prond_{h,a}$}\label{sectcardN1}
We further localize our parametrix $\Prond_{h,a}$: let $\phi\in C^{\infty}_{0}(\R)$ be even, $\phi=1$ on $[-1,1]$ and $\phi=0$ outside $(-3/2,3/2)$, set $\chi_1=\phi-\phi(2\cdot)$, then define (see \eqref{eq:Kequiv}), for any dyadic $\gamma$ (i.e. $1/\gamma \in 2^{\N}$) such that $\gamma \leq \ceps_{0}$,
\begin{equation}\label{KKKgam}
K_{\omega,\gamma}(f)(t,x,y):=\int e^{it\tau_{q}(\omega,\eta)}G(x,y,\eta,\omega)\cutoffchi^{\sharp}(\omega) q^{1/6}(\eta)\varkappa(h\eta)\varkappa(h\tau_{q}(\omega,\eta))\chi_1\Big(\frac\omega{\gamma |\eta|^{2/3}}\Big)\hat{f}\Big(\eta,\frac \omega{h^{1/3}}\Big)\,d\eta
\end{equation}
as well as $\Prond_{h,a,\gamma}$ by replacing $K_{\omega}$ with $K_{\omega,\gamma}$ in \eqref{eq:Prond2cut}.
We just reduced the sum over $k$ to $k$'s such that $k\sim \gamma^{3/2}/h$, where $h^{2/3-\ceps}\lesssim \gamma\leq \ceps_{0}$. Then,
$\Prond_{h,a}(t,x,y)=\sum_{\gamma}\Prond_{h,a,\gamma}(t,x,y)$, where the sum is intended to be dyadic $\gamma's$ with $\gamma<\ceps_{0}$.
For $\gamma\ll a$, the corresponding $\Prond_{h,a,\gamma}$ is irrelevant, as the phase of $g_{h,a}$ is non-stationary. Using  \eqref{eq:Prondcut} or \eqref{eq:AiryPoisson},
\begin{equation}\label{defProndgamma}
\Prond_{h,a,\gamma}(t,x,y)= \langle \sum_{N\in \Z} e^{-iN L(\omega)}
    , {\cutoffchi^{\flat}}(h^{2/3}\omega/\ceps_0)K_{\omega,\gamma}(g_{h,a})(t,x,y)\rangle_{\omega}.
\end{equation}
We will deal successively with $\gamma\sim a$ and $4a\leq \gamma<1$. Heuristically, the first case contains "tangent" initial directions (worst scenario); the second case contains all "almost transverse" directions and provide less important contributions (even after summing up in $\gamma$). We write
\begin{align}
  \label{eq:newVNgam}
  V_{N,\gamma}(t,x,y) & =\int e^{-iNL(\omega)}\cutoffchi^{\flat}(h^{2/3}\omega/\ceps_{0})K_{\omega,\gamma}(g_{h,a})(t,x,y)d\omega\\
  & = \begin{multlined}[t]\frac 1 {2\pi h^{d+1}} \int e^{\frac i h (
    t\tau_{q}(\alpha,\theta)+\Phi(x,y,\theta,\alpha,\sigma)-\Phi(a,0,\theta,\alpha,s)-Nh L(h^{-2/3}\alpha)
    )}\chi_1(\alpha/(|\theta|^{2/3}\gamma)) \\
  {}\times {\cutoffchi^{\flat}}(\alpha/\ceps_{0}){\cutoffchi^{\sharp}}(\alpha/h^{2/3})\chi(s)p_h(x,y,\theta,\alpha,\sigma)\tilde q_h (\theta,\alpha,s) ds \,d\theta  d{\sigma}d\alpha\,,\end{multlined}\\
    \label{eq:newProndgam}
  \Prond_{h,a,\gamma}(t,x,y) & =\sum_{N\in \Z} V_{N,\gamma}(t,x,y)\,,
\end{align}
where the symbol of $V_{N,\gamma}$ (the same for every $N$) is of order $0$.
\begin{lemma}
  At fixed $|t|\lesssim T_0$, the significant contributions in the sum \eqref{eq:newProndgam} defining $\Prond_{h,a,\gamma}$ come from  $|N|\lesssim |t|/\sqrt \gamma$:
  \begin{equation}
    \label{eq:24}
    \sum_{|N|\geq 4 |t|/\sqrt \gamma} V_{N,\gamma}(t,x,y)=O(h^{\infty})\,.
  \end{equation}
\end{lemma}
The proof of the lemma reproduces that of Proposition \ref{propsommefinie}; the maximum number of integrals that provide non-trivial contributions in the sum over $N$ is $1/\sqrt{a}$ when $\gamma\sim a$. Observe that Proposition \ref{propsommefinie} tells us that for $t=0$, only the $N=0$ term may contribute. 

At fixed $t\geq C\gamma^{1/2}$, we can further bound the cardinal of those $N$ that contribute significantly among the $C|t|\gamma^{-1/2}$ which are left. We introduce a few notations before stating a sharp bound of the number of waves that can overlap when $t/\sqrt{\gamma}$ is large : let $\mathcal{N}(x,y,t)$ be the set of $N$ with significant contributions (e.g., we have a stationary point for the phase in all variables),
\begin{equation}\label{Ncal}
\mathcal{N}(t,x,y)=\{N\in\mathbb{Z}, (\exists)(\sigma,s,\alpha,\theta) \text{ such that } \nabla_{(\sigma,s,\alpha,\theta)}\Phi_{N,a,\gamma}=0 \},
\end{equation}
where $\Phi_{N,a,\gamma}$, the phase function of $V_{N,\gamma}$ for the large parameter $1/h$, is defined as follows
\begin{equation}\label{PhiNagamma}
\Phi_{N,a,\gamma}:= t\tau_{q}(\alpha,\theta)+\Phi(x,y,\theta,\alpha,\sigma)-\Phi(a,0,\theta,\alpha,s)-\frac 43N\alpha^{3/2}+NhB_L(\alpha^{3/2}/h).
\end{equation}
Let $(t,x,y)$ such that $\mathcal{N}(t,x,y)\neq\emptyset$ and assume without loss of generality that $t>0$ and $t/\sqrt{\gamma}$ is large. As we shall see below, at a critical point $(\sigma,s,\alpha, \theta)$ of $\Phi_{N,a,\gamma}$,
 \[
\frac{y+\nabla B_0(y,\vartheta)}{|y+\nabla B_0(y,\vartheta)|}=-{\vartheta} +O(\alpha),\quad  \frac{|y+\nabla B_0(y,\vartheta)|}{t}=1+O(\alpha),
 \] 
and on the support of the symbol $\chi_{1}$, $\alpha \sim\gamma \ll 1$. Therefore, as $B_0(y,\omega)=O(|y|^2)$ and $\nabla B_0(y,\omega)=O(|y|^2)$,  if $\mathcal{N}(t,x,y)\neq\emptyset$ then we must have $\frac{|y|}{t}= 1+O(|y|)$.

Consider first $d>2$. Let $c$ satisfy $4\gamma <c<\frac{1}{16}$ and such that $1-c\leq \frac{|y|}{t}\leq 1+c$ (notice that such $c$ does exist since otherwise $\mathcal{N}(t,x,y)=\emptyset$). As $B_2(y,\vartheta)=O(|y|^2)$, $|y|\lesssim T_0$ with $T_0$ sufficiently small, there exists $\tilde \vartheta=\tilde\vartheta(t,y)\in\mathbb{R}^{d-1}\setminus \{0\}$ such that, with $\anabla_{\theta}=\nabla_{\theta}-\frac \theta{|\theta|}(\frac\theta {|\theta|}\cdot\nabla)  $,
\begin{equation}\label{tildeomega}
\tilde \vartheta= -\frac{y+\nabla B_0(y,\tilde\vartheta)}{|y+\nabla B_0(y,\tilde\vartheta)|}\\{}-2\Big(\frac{|y+\nabla B_0(y,\tilde\vartheta)|}{t}-1\Big)
\Big[\frac{\anabla q(\tilde\vartheta)}{q(\tilde\vartheta)}-\frac{3\anabla B_2(y,\tilde\vartheta)}{2t(1-B_2(y,\tilde\vartheta)/t)}\Big]\,,
\end{equation}
as a fixed point of a continuous map from a ball to a ball. In fact, as $B_0(y,\theta)$ (resp. $B_2(y,\theta)$) is homogeneous of degree $1$ (resp. $0$) in $\theta$, so is $\nabla B_0$ (of degree $0$), hence the right hand side in \eqref{tildeomega} does not depend on $|\tilde\vartheta|$. Uniqueness follows by taking differences and using smallness of $B_{0}$. Moreover, for sufficiently small $c$ and $T_0$ we have $|\tilde \vartheta|\in [\frac 12,\frac 32]$. 
As we shall see below (in \eqref{eqNcontYmodgennewnew} from Lemma \ref{lemdetails}), $\tilde\vartheta(t,,y)$ is an approximation (modulo $O(\gamma^{3/2}/t)$ terms) of the critical point $\vartheta_{\theta}$ of the phase $\Phi_{N,a,\gamma}$.
Let $(t',x',y')$ be  such that
$\mathcal{N}(t',x',y')\neq \emptyset$, then $1-c\leq \frac{|y'|}{t'}\leq 1+c$ and there exists $\tilde\vartheta(t',y')$ solution to \eqref{tildeomega} with $(t,y)$ replaced by $(t',y')$.
We now define a cylinder $\mathcal{C}_{\gamma}(t,x,y)$ as the set of $(t',x',y')$ such that
\begin{gather}
  \label{eqcondTjYjB2gen}|(t'-B_2(y',\tilde\vartheta(t',y')))-(t-B_2(y,\tilde\vartheta(t,y)))|\leq r_0\sqrt{\gamma},\\
    \label{eqcondTjXjYjgen}|x'(1+\Lp(y,\tilde\vartheta(t',y')))-x(1+\Lp(y,\tilde\vartheta(t,y)))|\leq r_0 \gamma, \\
   \label{eqcondTjYjB0gen}  ||y'+\nabla B_0(y',\tilde\vartheta(t',y'))|-t'-|y+\nabla B_0(y,\tilde\vartheta(t,y))|+t|\leq r_0\gamma^{3/2},
\end{gather}
and let, with $\mathcal{N}$ defined in \eqref{Ncal},
\begin{equation}\label{Ncal1}
 \mathcal{N}^1_{d>2}(t,x,y)=\cup_{\mathcal{C}_{\gamma}(t,x,y)}\mathcal{N}(t',x',y'). 
 \end{equation}
 For $d=2$, replace $\nabla B_{0}(y,\tilde\vartheta(t,y))$ by $B_{0}(y)$, $B_2(y,\tilde\vartheta(t,y))$ by $B_2(y)$ and $\Lp(y,\tilde\vartheta(t,y))$ by $\Lp(y)$ to define $\mathcal{N}^1_{d=2}(t,x,y)$.
\begin{prop}\label{propcardN}
For any $d\geq 2$, the following optimal upper bound holds for $\mathcal{N}^1_{d}$ defined in \eqref{Ncal1}
\begin{equation}
  \label{eq:113}
  \left| \mathcal{N}^1_{d}(t,x,y)\right| \lesssim O(1)+\gamma^{-1/2}{|t|}(\gamma^{3/2}/h)^{-2}\,.
\end{equation}
\end{prop}
\begin{rmq}
A particular case of the two dimensional version of Proposition \ref{propcardN} has been proved in \cite[Lemma 2.17, Lemma 2.18]{Annals} in the case $a\gg h^{4/7}$ (and without the $\gamma$ cut-off), where we proved  $\left|\mathcal{N}^1_{d=2}(t,x,y)\right|$ was bounded by a constant and that $\mathcal{N}(t,x,y)\subset [1,t/(2\sqrt{a})+N_0]$, with $N_0$ being an absolute constant. For $a\sim \gamma \gg h^{4/7}$ and $|t|\lesssim 1$, one easily sees that $|t|/(\gamma^{1/2} (\gamma^{3}/h^{2}))=O(1)$. Indeed, $\gamma^{7/2}/h^2\gg 1$, hence the right hand side term in \eqref{eq:113} is $O(1)$. Therefore, if $\gamma\gg h^{4/7}$,  we only get non-trivial contributions from an uniformly bounded number of waves at a fixed $t$. 
The case $\gamma \gg h^{2/3}$ was recently dealt with in \cite{ILP3} where Propositions \ref{propcardN} and \ref{propcardoutN1} were proved in the 2D Friedlander model domain. Compared to \cite{ILP3}, there are significant additional difficulties with angles in the higher dimensional case (even in the model situation !)
\end{rmq}
\begin{proof}
We first provide a proof in all dimensions $d\geq 2$ for the model case, with $\Delta_M$: a parametrix of the wave equation reads as \eqref{eq:newProndgam}, where $V_{N,\gamma}$ has symbol $\chi_{1}(\alpha/(\gamma |\theta|^{2/3})) \cutoffchi^{\flat}(\alpha/\eps_{0}){\cutoffchi^{\sharp}}(\alpha/h^{2/3})\chi(s)$ and phase function $\Phi^M_{N,a,\gamma}$ given by
\begin{equation*}
\Phi^M_{N,a,\gamma}(t,x,y,\sigma,s,\alpha,\theta):= t\tau_{q}(\alpha,\theta)+\Phi_M(x,y,\theta,\alpha,\sigma)-\Phi_M(a,0,\theta,\alpha,s)-\frac 43N\alpha^{3/2}+NhB_L(\alpha^{3/2}/h),
\end{equation*}
where $\Phi_M(x,y,\theta,\alpha,\sigma)=y\cdot\theta+\frac{\sigma^3}{3}+\sigma(xq^{1/3}(\theta)-\alpha)$; note that the only difference between $\Phi^M_{N,a,\gamma}$ and $\Phi_{N,a,\gamma}$ comes from the additional terms $\tau_{q}(\alpha,\theta)\Gamma(x,y,\frac{\sigma q^{1/3}(\theta)}{\tau_{q}(\alpha,\theta)},\frac{\theta}{\tau_{q}(\alpha,\theta)})-\tau_{q}(\alpha,\theta)\Gamma(a,0,\frac{s q^{1/3}(\theta)}{\tau_{q}(\alpha,\theta)},\frac{\theta}{\tau_{q}(\alpha,\theta)})$ (difference between $\Phi_M$ and $\Phi$). In the model case, without the additional phase function $\Gamma$, we rescale $x=\gamma X$, $t=\sqrt{\gamma}T$, $y=\sqrt{\gamma}Y$, and our cylinder $\mathcal{C}_{\gamma}(t,x,y)$ and $\mathcal{N}^1_{d\geq 2}$ simplify to $\mathcal{N}^{1,M}_d(T,X,Y):=\cup_{\mathcal{C}^{M}_{\gamma}(T,X,Y)}\mathcal{N}^M(T',X',Y')$ and
 \[
\mathcal{C}^{M}_{\gamma}(T,X,Y)=\{(T',X',Y')\,\,: \,\, |Y'-Y|\leq r_0, |X-X'|<r_0 , ||Y'|-T'-|Y|+T|<r_0\gamma\}\,.
\]
Note that $(X',Y',Y') \in \mathcal{C}^{M}_{\gamma}(T,X,Y)$ implies $||Y'|-|Y||\leq  r_0 $, $|Y|\Big|\frac{Y'}{|Y'|}-\frac{Y}{|Y|}\Big|\leq 2 r_0$ and $|T'-T|\leq r_0$.
We also rescale  $\sigma=\sqrt{\gamma}|\theta|^{1/3}\Sigma$, $\alpha=\gamma |\theta|^{2/3} A$, $s=\sqrt{\gamma}|\theta|^{1/3}S$ and we let $\lambda_{\gamma}=\frac{\gamma^{3/2}}{h}$. 
Define our new phase to be (with large parameter $1/h$ replaced by $\lambda_{\gamma}$)
\begin{multline}\label{Psimodel}
\Psi^M_{N,a,\gamma}(T,X,Y,\Sigma,S,A,\theta)=|\theta|\Big(\frac{Y\cdot \vartheta+T\sqrt{1+\gamma A q^{2/3}(\vartheta)}}{\gamma}+\frac{\Sigma^3}{3}+\Sigma(Xq^{1/3}(\vartheta)-A)\\
-\frac{S^3}{3}-S(\frac{a}{\gamma}q^{1/3}(\vartheta)-A)-\frac 43 N A^{3/2}\Big)+\frac{N}{\lambda_{\gamma}}B_L(|\theta|\lambda_{\gamma} A^{3/2})\,.
\end{multline}
The phase $\Psi^M_{N,a,\gamma}$ defines a Lagrangian $\Lambda^M_N$, which is described by $\nabla_{A,S,\Sigma,\theta}\Psi^{M}_{N,a\gamma}=0$:
\begin{gather}
\label{systemICmodelA}  \frac{T q^{2/3}(\vartheta)}{2\sqrt{1+\gamma A q^{2/3}(\vartheta)}}-(\Sigma-S)-2N A^{1/2}(1-\frac 3 4 {B_L}'(|\theta|\lambda_{\gamma}A^{3/2})) =0, \\
\label{systemICmodelSS}\Sigma^{2}+X q^{1/3}(\vartheta)-A=0,\quad 
S^{2}+\frac{a}{\gamma} q^{1/3}(\vartheta)-A=0,\\
\label{systemICmodeltheta}\begin{multlined} Y+T\vartheta\sqrt{1+\gamma A q^{2/3}(\vartheta)}+\gamma \Big(\frac{TAq^{1/3}(\vartheta)}{\sqrt{1+\gamma A q^{2/3}(\vartheta)}}+(\Sigma X-S\frac{a}{\gamma}) \Big)\frac{\anabla q(\vartheta)}{3 q^{2/3}(\vartheta)} \\
{}+\frac 2 3\gamma \vartheta(S^{3}-\Sigma^3)= \frac 43 \gamma N A^{3/2}(1-\frac 3 4 {B_L}'(|\theta|\lambda_{\gamma}A^{3/2}))\vartheta \,,\end{multlined}
\end{gather}
where we used the second and third equation for substitution in the last one.
We recover \cite[formula (2.13) to (2.15)]{Annals} when $d=2$, with small adjustments due to our (more complicated) phase construction there. Assume, without loss of generality, that $T>0$, then eliminating $N$ between \eqref{systemICmodelA} and \eqref{systemICmodeltheta},
\begin{multline}\label{systemICmodelA-theta}
  \frac{Y}{T}=-\Big(1+\frac{\gamma A q^{2/3}(\vartheta)}{1+\sqrt{1+\gamma Aq^{2/3}(\vartheta)}}-\frac{\gamma A q^{2/3}(\vartheta)}{3\sqrt{1+\gamma Aq^{2/3}(\vartheta)}}+\frac {2\gamma} {3T}( A(\Sigma-S)+(S^{3}-\Sigma^3))\Big)\vartheta\\
  {}-\Big(\frac{\gamma A  q^{1/3}(\vartheta)}{\sqrt{1+\gamma Aq^{2/3}(\vartheta)}}+\frac{\gamma}{T}\Big((\Sigma X-S\frac{a}{\gamma})\Big) \frac{\anabla q(\vartheta)}{3 q^{2/3}(\vartheta)}
\end{multline}
Using that $\vartheta\cdot\anabla q(\vartheta)=0$, we compute $|Y|^{2}/T^{2}$ and expand its square root to get
\begin{equation}\label{eqAYT}
  \quad\quad\frac{|Y|}{T}  -  1=\frac \gamma 6 Aq^{2/3}(\vartheta)+\frac{2\gamma}{3T}\big(A(\Sigma-S)+S^{3}-{\Sigma^3}\big)+{\gamma^2}\mathcal{E}\,,
\end{equation}
where $\mathcal{E}=\mathcal{E}(1/T,\Sigma X-S\frac a \gamma,A,S,\Sigma,\vartheta)$  is a smooth function of its arguments (that may be computed explicitly, although irrelevant). We then compute
\begin{equation}
\label{eqomegathetaYT}
\frac{Y}{|Y|}  =  -\vartheta-\frac{\gamma A}{3}\Big(\frac{\anabla q(\vartheta)}{q^{1/3}(\vartheta)}\Big)+\frac{\gamma}{T}(\Sigma X-S\frac{a}{\gamma}) \frac{\anabla q(\vartheta)}{3 q^{2/3}(\vartheta)}+\gamma^{2} \vec{\mathcal{E}}\,,
\end{equation}
where $\vec{\mathcal{E}}=\vec{\mathcal{E}}(1/T,\Sigma X-S\frac a \gamma,A,S,\Sigma,\vartheta)$ is an explicit smooth, vector-valued function. Later we will use $O(\gamma^{2})\lesssim O(\gamma/T)$, as $T\sqrt\gamma=O(1)$. We now estimate the distance between any two elements of $\mathcal{N}^{1,M}_d(T,X,Y)$. Pick $(T, X, Y)$ with $T>0$, and let $N_j\in \mathcal{N}^{1,M}_d(T,X,Y)$, $j\in \{1,2\}$. There exist $(T_j,X_j,Y_j)\in \mathcal{C}^{M}_{\gamma}$ 
and there exist $(\theta_j=|\theta_j|\vartheta_{j},A_j,\Sigma_j,S_j)$, $j\in\{1,2\}$, $A_j$ close to $1$, such that \eqref{systemICmodelA}, \eqref{systemICmodelSS}, \eqref{systemICmodelA-theta} hold.
Taking the difference between \eqref{systemICmodelA} for $j=1,2$ and using \eqref{systemICmodelSS} ($\frac{(\Sigma_j-S_j)}{A_j^{1/2}}=O(1)$) and $\gamma T_j=O(\sqrt{\gamma}t_j)=O(\sqrt{\gamma})$,
\begin{multline}\label{eqN1N2dif}
N_1-N_2=\frac 34 \Big(N_1B_L'(|\theta_1|\lambda_{\gamma}A_1^{3/2})-N_2B_L'(|\theta_2|\lambda_{\gamma}A_2^{3/2})\Big)-\frac{\Sigma_1-S_1}{2\sqrt{A_1}}+\frac{\Sigma_2-S_2}{2\sqrt{A_2}}\\
+\frac{T_1 q^{2/3}(\vartheta_{1})}{4A_1^{1/2}\sqrt{1+\gamma A_1q^{2/3}(\vartheta_{1})}}-\frac{T_2 q^{2/3}(\vartheta_{2})}{4A_2^{1/2}\sqrt{1+\gamma A_2q^{2/3}(\vartheta_{2})}}\\
=O\Big(\frac{N_1+N_2}{\lambda_{\gamma}^2}\Big)+O(1)+\frac{T_1 q^{2/3}(\vartheta_{1})}{4A_1^{1/2}}-\frac{T_2 q^{2/3}(\vartheta_{2})}{4A_2^{1/2}}\,.
\end{multline}
\begin{rmq}\label{rmqN1N2dif}
The first term in the first line of \eqref{eqN1N2dif} behaves like $(|N_1|+|N_2|)/\lambda_{\gamma}^2$, using that $B'(\theta\lambda A^{3/2})\sim -\frac{b_1}{\theta^2\lambda^2A^3}$ and $\theta,A\sim 1$. We cannot take advantage of the difference, since each term $N_jB'(\cdot)$ corresponds to some $\theta_j,A_j$ (close to $1$) and, although the difference between $A_j$ turns out to be $O(1/T)$, we do not have any better information about the difference between $|\theta_j|$ than bounded by a small constant on the support of $\chi_{1}$. Therefore the bound $(N_1+N_2)/\lambda_{\gamma}^2$ for the terms involving $B'_L$ in \eqref{eqN1N2dif} is sharp. As $N_j\sim T_j$, and $|T_j-T|\leq 1$, this contribution is of order $|T|/\lambda^2_{\gamma}$. 
\end{rmq}
We are reduced to proving that the following difference (from \eqref{eqN1N2dif}) is $O(1)$. Write
\begin{multline}\label{estimdifT1T2}
\frac{T_1}{A_1^{1/2}} q^{2/3}(\vartheta_{1})-\frac{T_2 }{A_2^{1/2}}q^{2/3}(\vartheta_{2})=\frac{(T_1-T_2)}{A_1^{1/2}}q^{2/3}(\vartheta_{1})+\frac{T_2}{\sqrt{A_1q^{2/3}(\vartheta_{1})}}(q(\vartheta_{1})-q(\vartheta_{2}))\\
+T_2q(\vartheta_2)(\frac{1}{\sqrt{A_1q^{2/3}(\vartheta_{1})}}-\frac{1}{\sqrt{A_2q^{2/3}(\vartheta_{2})}}).
\end{multline}
As $|T_1-T_2|\leq |T_1-T|+|T_2-T|\leq 2r_0$, it remains to prove
\begin{lemma}\label{derderlemma}
Let $(T,X,Y)$ be fixed, let $(T_j,X_j,Y_j)\in \mathcal{C}^{M}_{\gamma}$ and let $(\theta_j,A_j,\Sigma_j,S_j)$, $j\in\{1,2\}$ with $A_j$ close to $1$ such that \eqref{systemICmodelA}, \eqref{systemICmodelSS} and \eqref{systemICmodelA-theta} hold, then
\begin{equation}\label{eqAjomega_jdif}
T|\vartheta_{1}-\vartheta_{2}|\lesssim 1, \quad T|A_1q^{2/3}(\vartheta_{1})-A_2q^{2/3}(\vartheta_{2})|\lesssim 1.
\end{equation}
\end{lemma}
\begin{proof}
 When $T$ is not too large, \eqref{eqAjomega_jdif} immediately follows. We consider $T$ sufficiently large. For $T_j,Y_j,A_j,\vartheta_{j}$, \eqref{eqAYT} and \eqref{eqomegathetaYT} hold.
Taking difference between \eqref{eqomegathetaYT} for $j=1$ and $j=2$,
\begin{multline}\label{difomega1omega2}
-\frac{Y_1}{|Y_1|}+\frac{Y_2}{|Y_2|}=(\vartheta_{1}-\vartheta_{2})\Big(1-\frac 23 \gamma A_1q^{2/3}(\vartheta_{1})\Big)-\frac 23 \vartheta_{2}\gamma(A_1q^{2/3}(\vartheta_{1})-A_2q^{2/3}(\vartheta_{2}))\\
+\gamma(A_1q^{2/3}(\vartheta_{1})-A_2q^{2/3}(\vartheta_{2}))\frac{\nabla q(\vartheta_{2})}{q(\vartheta_{2})}+\gamma A_1 q^{2/3}(\vartheta_{1})\Big(\frac{\nabla q(\vartheta_{1})}{q(\vartheta_{1})}-\frac{\nabla q(\vartheta_{2})}{q(\vartheta_{2})}\Big)+O(\frac{\gamma}{T}),
\end{multline}
where we used that $O(\frac{\gamma}{T_1})=O(\frac{\gamma}{T_2})=O(\frac{\gamma}{T})$. As $A_j$'s stay close to $1$ and $\vartheta_{j}\in\mathbb{S}^{d-2}$, we get
\begin{equation}\label{difomega1omega2new}
(\vartheta_{1}-\vartheta_{2})(1+O(\gamma))=\Big(-\frac{Y_1}{|Y_1|}+\frac{Y_2}{|Y_2|}\Big)+O(\gamma)\,.
\end{equation}
Using $\big|-\frac{Y_1}{|Y_1|}+\frac{Y_2}{|Y_2|}\big|=\big|-\frac{y_1}{|y_1|}+\frac{y_2}{|y_2|}\big|\leq 2\frac{r_0}{|Y|}$ and  $T|\vartheta_{1}-\vartheta_{2}|\lesssim t(2r_0+O(\sqrt{\gamma})) \lesssim 2r_0$, we obtain the first inequality in \eqref{eqAjomega_jdif}. Taking now the difference between \eqref{eqAYT} for $j=1,2$,
\[
6\Big(\frac{|Y_1|}{T_1}-\frac{|Y_2|}{T_2}\Big)=\gamma(A_1q^{2/3}(\vartheta_{1})-A_2q^{2/3}(\vartheta_{2}))+O(\frac{\gamma}{T}),
\]
and using $|(|Y_j|-T_j)-(|Y|-T)|\leq r_0 \gamma$, $|T_1-T_2|\leq 2r_0$ and \eqref{eqAYT} with $j=2$ yields
\begin{align}
T|A_1q^{2/3}(\vartheta_{1})-A_2q^{2/3}(\vartheta_{2})| & \leq \frac{6T}{\gamma}\Big|\frac{|Y_1|}{T_1}-\frac{|Y_2|}{T_2}\Big|+O(1)\\
 & \leq \frac{6T}{\gamma}\frac{|(|Y_1|-T_1)-(|Y_2|-T_2)|}{T_1}
+ \frac{6T}{\gamma}\Big(\frac{|Y_2|}{T_2}-1\Big)\frac{(T_2-T_1)}{T_1}+O(1)\\
 & \leq 12r_0\frac{T}{T_1}+2r_0\frac{T}{T_1} \Big(A_2q^{2/3}(\vartheta_{2})+O(\frac{1}{T_2})\Big)+O(1)=O(1),
\end{align}
where, again, on the support of the symbol $\chi_{1}$, $A_j$ is close to $1$. The proof of Lemma \ref{derderlemma} is complete and, combined with \eqref{estimdifT1T2} and \eqref{eqN1N2dif}, this yields \eqref{eq:113} for the model operator $\Delta_M$.
\end{proof}
We now proceed with the general case, following the same steps as above. We first deal with the most complicated situation $d>2$. We also rescale 
variables as follows $\sigma=\sqrt{\gamma}|\theta|^{1/3}\Sigma$, $\alpha=\gamma |\theta|^{2/3} A$, $s=\sqrt{\gamma}|\theta|^{1/3}S$ and let $\lambda_{\gamma}=\frac{\gamma^{3/2}}{h}$. With $\vartheta=\theta/|\theta|$ we get $q(\theta)=|\theta|^2q(\vartheta)$, $\tau_{q}(\alpha,\theta)=\tau_{q}(\gamma A|\theta|^{2/3},\theta)=|\theta|\sqrt{1+\gamma A q^{2/3}(\vartheta)}=:|\theta|\tau_{q}(\gamma A,\vartheta)$. We retain  space-time variables $(t,x,y)$ as our phase function is no longer homogeneous in $y$.
Recall that
\begin{equation}
  \label{eq:22}
\Phi(x,y,\theta,\alpha,\sigma)=\Phi_M(x,y,\theta,\alpha,\sigma)+\tau_{q}(\alpha,\theta)(B_{\Gamma}(y,\theta/\tau_{q})+xA_{\Gamma}(x,y,\sigma q^{1/3}(\theta)/\tau_{q},\theta/\tau_{q})\,,
\end{equation}
where $\tau_{q}=\tau_{q}(\alpha,\theta)=\sqrt{|\theta|^2+\alpha q^{2/3}(\theta)}$. Let
\begin{align}
\tilde\Phi_{M,\gamma}(x,y,\vartheta,A,\Sigma):&=|\theta|^{-1}\Phi_M(x,y,\theta,\gamma |\theta|^{2/3}A,\sqrt{\gamma}|\theta|^{1/3}\Sigma)\\
&=y\cdot \vartheta+\gamma^{3/2}\Big(\frac{\Sigma^3}{3}+\Sigma(\frac{x}{\gamma}q^{1/3}(\vartheta)-A)
\Big).
\end{align}
Set $\tilde\Phi_{\gamma}(x,y,\vartheta,A,\Sigma):=|\theta|^{-1}\Phi(x,y,\theta,\gamma |\theta|^{2/3}A,\sqrt{\gamma}|\theta|^{1/3}\Sigma)$, then, using the homogeneity in $|\theta|$,
\begin{gather}
\tilde\Phi_{\gamma}(x,y,\vartheta,A,\Sigma)  := \begin{multlined}[t]\tilde\Phi_{M,\gamma}(x,y,\vartheta,A,\Sigma)+\tau_{q}(\gamma A,\vartheta)\Big[B_{\Gamma}(y,\vartheta/\tau_{q}(\gamma A,\vartheta))\\
\label{tildePhigammadefnew}{}+xA_{\Gamma}\Big(x,y,\frac{\sqrt{\gamma}\Sigma q^{1/3}(\vartheta)}{\tau_{q}(\gamma A,\vartheta)},\frac{\vartheta}{\tau_{q}(\gamma A,\vartheta)}\Big)\Big]\,,\end{multlined}\\
\tilde\Phi_{\gamma}(a,0,\vartheta,A,S) :=\tilde\Phi_{M,\gamma}(a,0,\vartheta,A,S)+\tau_{q}(\gamma A,\vartheta)aA_{\Gamma}\Big(a,0,\frac{\sqrt{\gamma}S q^{1/3}(\vartheta)}{\tau_{q}(\gamma A,\vartheta)},\frac{\vartheta}{\tau_{q}(\gamma A,\vartheta)}\Big)\,,
\end{gather}
where in the last line we used that $B_{\Gamma}(0,\vartheta/\tau_{q})=0$ and where in the new variables we have
\begin{gather}
\tau_{q}(\gamma A,\vartheta) B_{\Gamma}(y,\frac{\vartheta}{\tau_{q}(\gamma A,\vartheta)})=  \begin{multlined}[t]\label{Bgamhomog}
 B_0(y,\vartheta)+(1-\tau_{q}(\gamma A,\vartheta))B_2(y,\vartheta)\\{}+\tau_{q}(\gamma A,\vartheta)\sum_{j\geq 2}(\frac{1}{\tau_{q}(\gamma A,\vartheta)}-1)^jB_{2j}(y,\vartheta)\,,
\end{multlined}\\
\tau_{q}(\gamma A,\vartheta) A_{\Gamma}\Big(x,y,\frac{\sqrt{\gamma}\Sigma q^{1/3}(\vartheta)}{\tau_{q}(\gamma A,\vartheta)},\frac{\vartheta}{\tau_{q}(\gamma A,\vartheta)}\Big)=\begin{multlined}[t]\label{Agamtau}
(\sqrt{\gamma}q^{1/3}(\vartheta))\Sigma \Lp(y,\vartheta)+\mathcal{H}_{j\geq 3}\\{}+\frac{\gamma q^{2/3}(\vartheta)}{\tau_{q}(\gamma A,\vartheta)}(\Sigma^2-A)\mu(y,\vartheta)\,,
\end{multlined}
\end{gather}
where we used \eqref{AGam}; homogeneous terms of order $j$ in $\mathcal{H}_{j\geq 3}$ have weights $\sqrt{\gamma}^j$. We also set
\begin{equation}\label{deftildephiNagamma}
\tilde\Phi_{N,a,\gamma}(t,x,y,\Sigma,S,A,\theta):=\Phi_{N,a,\gamma}(t,x,y,\sqrt{\gamma}|\theta|^{1/3}\Sigma,\sqrt{\gamma}|\theta|^{1/3}S,\gamma |\theta|^{2/3}A,\theta)\,,
\end{equation}
and from \eqref{PhiNagamma},
\begin{multline}
\label{eqtildePhiN}
\tilde\Phi_{N,a,\gamma}(t,x,y,\Sigma,S,A,\theta)=|\theta|\Big(t\sqrt{1+\gamma A q^{2/3}(\vartheta)}
+\tilde\Phi_{\gamma}(x,y,\vartheta,A,\Sigma)-\tilde\Phi_{\gamma}(a,0,\vartheta,A,S)\Big)\\-\frac 43\gamma^{3/2} |\theta| N A^{3/2}+N hB_L((|\theta|\lambda_{\gamma} A^{3/2}).
\end{multline}
The phase function $\tilde\Phi_{N,a,\gamma}$ defines a Lagrangian $\Lambda_N$ and, as in the model case,
we obtain a parametrization of $\pi_N(\Lambda_N)$ by $(d+1)$ parameters $(\rho=|\theta|,\vartheta,\Sigma,A)$ as follows 
\begin{gather}\label{desclambdaN}
\,\,\,\left\{ \begin{array}{l}
  \frac{t \gamma q^{2/3}(\vartheta)}{2\sqrt{1+\gamma A q^{2/3}(\vartheta)}}+\partial_{A}\Big(\tilde\Phi_{\gamma}(x,y,\vartheta,A,\Sigma)-\tilde\Phi_{\gamma}(a,0,\vartheta,A,S)\Big)=2N \gamma^{3/2}A^{1/2}(1-\frac 3 4 {B_L}'(\rho\lambda_{\gamma} A^{3/2})), \\
  \Sigma^{2}+\frac{x}{\gamma} \Big(q^{1/3}(\vartheta)+\frac{\tau_{q}(\gamma A,\vartheta) }{\sqrt{\gamma}}\partial_{\Sigma}\Big(A_{\Gamma}(x,y,\frac{\sqrt{\gamma}\Sigma q^{1/3}(\vartheta)}{\tau_{q}(\gamma A,\vartheta)},\frac{\vartheta}{\tau_{q}(\gamma A,\vartheta)})\Big)\Big)=A,\\
 S^{2}+\frac{a}{\gamma} \Big(q^{1/3}(\vartheta)+\frac{\tau_{q}(\gamma A,\vartheta) }{\sqrt{\gamma}}\partial_{S}\Big(A_{\Gamma}(a,0,\frac{\sqrt{\gamma}S q^{1/3}(\vartheta)}{\tau_{q}(\gamma A,\vartheta)},\frac{\vartheta}{\tau_{q}(\gamma A,\vartheta)})\Big)\Big)=A,\\
t\sqrt{1+\gamma Aq^{2/3}(\vartheta)}+\tilde\Phi_{\gamma}(x,y,\vartheta,A,\Sigma)-\tilde\Phi_{\gamma}(a,0,\vartheta,A,S)=\frac 43 \gamma^{3/2}NA^{3/2}(1-\frac 34 B_L'(\rho\lambda_{\gamma}A^{3/2})),\\
\partial_{{\vartheta}_j}\Big(t\sqrt{1+\gamma Aq^{2/3}(\vartheta)}+\tilde\Phi_{\gamma}(x,y,\vartheta,A,\Sigma)-\tilde\Phi_{\gamma}(a,0,\vartheta,A,S)\Big)=0,\quad \forall 1\leq j\leq d-2,
   \end{array} \right.
\end{gather}
where ${\vartheta}=({\vartheta}_1,...,{\vartheta}_{d-1})\in\mathbb{S}^{d-2}$ and, as $t\neq 0$ implies $|y|\neq 0$, we assumed that $y_{d-1}\neq 0$; then ${\vartheta}_{d-1}=\pm \sqrt{1-\sum_{j=1}^{d-2}{\vartheta}_j^2}$. The last line in the system \eqref{desclambdaN} reads as follows
\begin{equation}\label{ecomegacy}
\partial_{{\vartheta}_j}\tilde\Phi_{0,a,\gamma}-\partial_{{\vartheta}_{d-1}}\tilde\Phi_{0,a,\gamma}\frac{{\vartheta}_j}{{\vartheta}_{d-1}}=0, \quad j\in\{1,...,d-2\}.
\end{equation}
With $\tilde\Sigma:=\frac{\sqrt{\gamma}\Sigma q^{1/3}({\vartheta})}{\tau_{q}(\gamma A,{\vartheta})}$, $\frac{d\tilde\Sigma}{d\Sigma}=\frac{\sqrt{\gamma} q^{1/3}({\vartheta})}{\tau_{q}(\gamma A,{\vartheta})}$ the second equation in the system \eqref{desclambdaN} reads as 
\[
  \Sigma^{2}+\frac{x}{\gamma} q^{1/3}(\vartheta)\Big(1+\partial_{\tilde\Sigma}\Big(A_{\Gamma}(x,y,\tilde \Sigma,\frac{\vartheta}{\tau_{q}(\gamma A,\vartheta)})\Big)\Big)=A.
\]
We further compute, with $\tau_{q}=\tau_{q}(\gamma A,{\vartheta})=\sqrt{1+\gamma Aq^{2/3}({\vartheta})}$,
\begin{multline}
\label{derivAPhi}
\quad  \partial_{A}\tilde\Phi_{\gamma}(x,y,{\vartheta},A,\Sigma)=\partial_{A}\tau_{q}\Big(\sum_{k\geq 2}\partial_{\tau_{q}}\Big(\frac{(1-\tau_{q})^k}{\tau_{q}^{k-1}}\Big)B_{2k}(y,\vartheta)-B_2(y,\vartheta)\Big)\\
  {}+\gamma^{3/2}\Big(-\Sigma+\frac{\partial_{A}\tau_{q}}{\sqrt{\gamma}}\frac{x}{\gamma}\partial_{\tau_{q}}\Big(\tau_{q} A_{\Gamma}(x,y,\sqrt{\gamma}\Sigma q^{1/3}(\vartheta)/\tau_{q},\vartheta/\tau_{q})\Big)\Big)\,,
\end{multline}
where the last term in the second line is small and behaves like $xO(\sqrt{\gamma})$; indeed, it follows from \eqref{Agamtau} that $\partial_{\tau_{q}}(\tau_{q} A_{\Gamma})=-\gamma q^{2/3}({\vartheta})\frac{(\Sigma^2-A)}{\tau_{q}^2}\mu(y,{\vartheta})+\mathcal{H}_{j\geq 3}$, hence its main contribution is $O(\gamma)$ (here we have also used the second and the third equations in \eqref{desclambdaN} which imply that $\Sigma^2,S^2\leq 2A$ where the phase may be stationary in $\Sigma,S$). Using that for $k\geq 2$, all terms in $B_{\Gamma}$ come with factors $O(\gamma^2)$ and $\partial_A\tau_{q}= \gamma q^{2/3}({\vartheta})/(2\tau_{q})$, the main contribution in the first equation in \eqref{desclambdaN} reads 
 \begin{equation}\label{desclambdaNmaincontrib}
  \frac{q^{2/3}(\vartheta)}{2\sqrt{1+\gamma A q^{2/3}(\vartheta)}}\frac{(t-B_2(y,\vartheta))}{\sqrt{\gamma}}-\Sigma+S+O(\sqrt{\gamma}x)+O(\gamma^{3/2}|y|)=2NA^{1/2}(1-\frac 3 4 {B_L}'(\rho\lambda_{\gamma} A^{3/2})).
 \end{equation}
We now turn to critical points with respect to $\vartheta$ and deal with the last equation in \eqref{desclambdaN}.
\begin{lemma}\label{lemomegac}
The last equation in the system \eqref{desclambdaN} has two solutions, denoted ${\vartheta}_{\pm}$, such that
\begin{multline}\label{eqNcontYmodgen}
\pm {\vartheta}_{\pm}=\frac{y+\nabla B_0(y,{\vartheta}_{\pm})}{|y+\nabla B_0(y,{\vartheta}_{\pm})|}
- \frac{\gamma A q^{2/3}({\vartheta}_{\pm})}{|y+\nabla B_0(y,{\vartheta}_{\pm})|}\Big[\frac{\anabla q({\vartheta}_{\pm})}{3q({\vartheta}_{\pm})}\Big(t-B_2(y,{\vartheta}_{\pm})\Big)
-\frac{1}{2}\anabla B_2(y,{\vartheta}_{\pm})\Big]\\
{}-\frac{\gamma^{3/2}}{|y+\nabla B_0(y,{\vartheta}_{\pm})|}\Big[\frac{\anabla q({\vartheta}_{\pm})}{3q^{2/3}({\vartheta}_{\pm})}\Big(\frac{x}{\gamma}\Sigma (1+\Lp(y,{\vartheta}_{\pm}))-\frac{a}{\gamma}S\Big)+\frac{x}{\gamma}\Sigma q^{1/3}({\vartheta}_{\pm}) \anabla (\Lp(y,{\vartheta}_{\pm}))\Big]+\frac{\gamma^{2}}{|y|}\mathcal{E}_{\pm},
\end{multline}
with $\mathcal{E}_{\pm}(t,x,y,\frac{t}{|y+\nabla B_0|},\frac{\sqrt{\gamma}}{t},{\vartheta})$ smooth functions. Moreover, $\Big(\frac{y+\nabla B_0(y,{\vartheta}_{\pm})}{|y+\nabla B_0(y,{\vartheta}_{\pm})|}\Big)\cdot {\vartheta}_{\pm}=\pm 1+\frac{\gamma^2}{|y+\nabla B_0|}\mathcal{E}_{\pm}$.
\end{lemma}
\begin{proof}
Using \eqref{ecomegacy} with ${\vartheta}^2_{d-1}= 1-\sum_{j=1}^{d-2}{\vartheta}_j^2$, $\Big(\sum_{j=1}^{d-2}|\partial_{{\vartheta}_j}\tilde\Phi_{0,a,\gamma}|^2\Big){\vartheta}_{d-1}^2=(1-{\vartheta}_{d-1}^2)|\partial_{{\vartheta}_{d-1}}\tilde\Phi_{0,a,\gamma}|^2$. We eventually obtain
\begin{equation}\label{eqomega}
{\vartheta}_{d-1}=\pm\frac{\partial_{{\vartheta}_{d-1}}\tilde\Phi_{0,a,\gamma}}{|(\sum_{j=1}^{d-1}|\partial_{{\vartheta}_{j}}\tilde\Phi_{0,a,\gamma}|^2|^{1/2}},\quad {\vartheta}_j={\vartheta}_{d-1} \frac{\partial_{{\vartheta}_j}\tilde\Phi_{0,a,\gamma}}{\partial_{{\vartheta}_{d-1}}\tilde\Phi_{0,a,\gamma}}=\pm\frac{\partial_{{\vartheta}_{j}}\tilde\Phi_{0,a,\gamma}}{|\sum_{j=1}^{d-1}|\partial_{{\vartheta}_{j}}\tilde\Phi_{0,a,\gamma}|^2|^{1/2}}.
\end{equation}
For each sign $\pm$ there exists an unique solution as the maps from $\mathbb{S}^{d-2}$ to $\mathbb{S}^{d-2}$: ${\vartheta}\rightarrow \pm\frac{\nabla_{{\vartheta}} \tilde\Phi_{0,a,\gamma}}{|\nabla_{{\vartheta}} \tilde\Phi_{0,a,\gamma}|}$ each have a unique fixed point for small $|y|, t$ (as $B_0=O(|y|^2)$, $B_2=O(|y|^2)$ and $\gamma$ is small).

In the following we approximate the equation satisfied by ${\vartheta}$ up to $O(\gamma^2)$ terms. This will turn out to be useful later on, in the proof of Proposition \ref{propcardoutN1}. Using \eqref{eqtildePhiN} with $N=0$ together with \eqref{tildePhigammadefnew},\eqref{Bgamhomog} and \eqref{Agamtau}, we write
\begin{multline}
\tilde\Phi_{0,a,\gamma}(t,x,y,\Sigma,S,A,\rho{\vartheta})=\rho\Big(t\tau_{q}+\tilde \Phi_{M,\gamma}(x,y,{\vartheta},A,\Sigma)-\tilde\Phi_{M,\gamma}(a,0,{\vartheta},A,S)\\+B_0(y,{\vartheta})+(1-\tau_{q})B_2(y,{\vartheta})+\gamma^{3/2}\frac{x}{\gamma}q^{1/3}({\vartheta})\Sigma \Lp(y,{\vartheta})+\gamma^2\tilde \Gamma_{\gamma}(x,a,y,{\vartheta},\sqrt{\gamma}\Sigma,\sqrt\gamma S,A)\Big),
\end{multline}
where we have set, for $\tau_{q}=\tau_{q}(\gamma A,{\vartheta})$,
\begin{equation}\label{deftildeGammagam}
\tilde \Gamma_{\gamma}:=\gamma^{-2}\Big(\tau_{q} B_{\Gamma}(y,{\vartheta}/\tau_{q})-B_0(y,{\vartheta})-(1-\tau_{q})B_2(y,{\vartheta})+x(\tau_{q} A_{\Gamma}-\sqrt{\gamma}xq^{1/3}({\vartheta}) \Lp(y,{\vartheta})\Sigma)
-a\tau_{q} A_{\Gamma}S\Big).
\end{equation}
From \eqref{Bgamhomog} and \eqref{Agamtau}, $\tilde\Gamma_{\gamma}$ is a smooth bounded function and 
\begin{multline}
\tilde \Gamma_{\gamma}=\frac{(1-\tau_{q})^2}{\gamma^2 \tau_{q}}\Big(B_4(y,{\vartheta})+\sum_{j\geq 3}(\frac{1}{\tau_{q}}-1)^{j-2}B_{2j}(y,{\vartheta})\Big)+\frac{x}{\gamma}\Big(q^{2/3}({\vartheta})(\Sigma^2-A)\mu(y,{\vartheta})+\gamma^{-1}\mathcal{H}_{j\geq 3}\Big)\\
-\frac{a}{\gamma}\Big(q^{2/3}({\vartheta})(S^2-A)\mu(0,{\vartheta})+\gamma^{-1}\mathcal{H}_{j\geq 3}\Big),
\end{multline}
where homogeneous terms of order $j$ in $\gamma^{-1}\mathcal{H}_{j\geq 3}$ come from $A_{\Gamma}$ and have factors $\gamma^{j/2-1}$, $j\geq 3$ and where $(1-\tau_{q})^2/\gamma^2=A^2q^{4/3}({\vartheta})/(1+\tau_{q})^2$ is bounded. We compute explicitly 
\begin{multline}
\label{derivjomtildePhi}
 \frac 1 {\rho}\partial_{{\vartheta}_j}\tilde\Phi_{0,a,\gamma}=y_j+\partial_j B_0(y,{\vartheta})+\frac{\gamma A}{2\tau_{q}}q^{2/3}({\vartheta})\Big[\frac{2\partial_jq({\vartheta})}{3q({\vartheta})}(t-B_2(y,{\vartheta}))-\frac{2}{1+\tau_{q}}\partial_{j}B_2(y,{\vartheta})\Big]\\
 +\gamma^{3/2}\Big[\frac{\partial_j q({\vartheta})}{3q^{2/3}({\vartheta})}\Big(\frac{x}{\gamma}\Sigma (1+\Lp(y,{\vartheta}))-\frac{a}{\gamma}S\Big)+\frac{x}{\gamma}\Sigma q^{1/3}({\vartheta}) \partial_j \Lp(y,{\vartheta})\Big]+\gamma^{2}\partial_{{\vartheta}_j}\tilde\Gamma_{\gamma}.
\end{multline}
With $\nabla F=(\partial_{{\vartheta}_1}F,...,\partial_{{\vartheta}_{d-1}}F)$, let $\Omega_1=\frac{ A}{2\tau_{q}}q^{2/3}({\vartheta})\Big(\frac{2\nabla q({\vartheta})}{3q({\vartheta})}(1-\frac{B_2(y,{\vartheta})}{t})-\frac{2}{1+\tau_{q}}\frac{\nabla B_2(y,{\vartheta})}{t}\Big)$, $\Omega_2=\frac{\nabla q({\vartheta})}{3q^{2/3}({\vartheta})}\Big(\frac{x}{\gamma}\Sigma (1+\Lp(y,{\vartheta}))-\frac{a}{\gamma}S\Big)+\frac{x}{\gamma}\Sigma q^{1/3}({\vartheta}) \nabla \Lp(y,{\vartheta})$, then, using \eqref{derivjomtildePhi} we compute,
\begin{multline}\label{eqomegamodul}
\sum_{j=1}^{d-1}|\partial_{{\vartheta}_{j}}\tilde\Phi_{0,a,\gamma}|^2=\Big|y+\nabla B_0+\gamma t\Omega_1+\gamma^{3/2}\Omega_2+\gamma^{2}\nabla\tilde\Gamma_{\gamma}\Big|^2=|y+\nabla B_0|^2\Big[1+\frac{2\gamma t}{|y+\nabla B_0|}\times \\{} \Omega_1\cdot\Big(\frac{y+\nabla B_0}{|y+\nabla B_0|}\Big)+\frac{2\gamma^{3/2}}{|y+\nabla B_0|}\Omega_2\cdot\Big(\frac{y+\nabla B_0}{|y+\nabla B_0|}\Big)
+\frac{\gamma^2}{|y+\nabla B_0|}\Big(\mathcal{E}_1+\frac{t^2\mathcal{E}_2+t\sqrt{\gamma}\mathcal{E}_3+\gamma \mathcal{E}_{4}}{|y+\nabla B_0|}\Big)\Big],
\end{multline}
where we have set $\mathcal{E}_1:=2\Big(\frac{y+\nabla B_0}{|y+\nabla B_0|}\Big)\cdot \nabla \tilde \Gamma_{\gamma}$, $\mathcal{E}_2:=|\Omega_1|^2$, $\mathcal{E}_3=2\Omega_1\cdot (\Omega_2+\sqrt{\gamma}\nabla \tilde\Gamma_{\gamma})$, $\mathcal{E}_4=|\Omega_2|^2+\sqrt{\gamma}\Omega_2\cdot(\nabla\tilde \Gamma_{\gamma})+\gamma |\nabla \tilde\Gamma_{\gamma}|^2$. As $\frac{\gamma}{|y+\nabla B_0|}=\sqrt{\gamma}\frac{\sqrt{\gamma}}{t}\frac{t}{|y+\nabla B_0|}$ is small when $t/\sqrt{\gamma}$ is sufficiently large, the coefficient of $\mathcal{E}_3$ in \eqref{eqomegamodul} is bounded by $O(\sqrt{\gamma})$. All the three terms $\mathcal{E}_j$ are smooth, bounded functions and we relabel their sum as follows
\[
\mathcal{E}_1+\frac{t^2\mathcal{E}_2+t\sqrt{\gamma} \mathcal{E}_3}{|y+\nabla B_0|}+\frac{\gamma}{|y+\nabla B_0|} \mathcal{E}_4=\mathcal{E}(t,x,y,\frac{t}{|y+\nabla B_0|},\frac{\sqrt{\gamma}}{t},{\vartheta}).
\]
We eventually find 
\begin{multline}
\Big(\sum_{j=1}^{d-1}|\partial_{{\vartheta}_{j}}\tilde\Phi_{0,a,\gamma}|^2\Big)^{1/2}=|y+\nabla B_0|\Big(1+\frac{\gamma t}{|y+\nabla B_0|}<\Omega_1,\frac{y+\nabla B_0}{|y+\nabla B_0|}>\\+\frac{\gamma^{3/2}}{|y+\nabla B_0|}<\Omega_2,\frac{y+\nabla B_0}{|y+\nabla B_0|}>
+\frac{\gamma^2}{|y+\nabla B_0|} \mathcal{\tilde E}\Big),
\end{multline}
where $\mathcal{\tilde E}$ is a smooth function of $(t,x,y,\frac{t}{|y+\nabla B_0|},\frac{\sqrt{\gamma}}{t},{\vartheta})$ obtained by taking the square root in \eqref{eqomegamodul} and using the asymptotic expansion of $\sqrt{1+\kappa}=1+\kappa/2+O(\kappa^2)$.
From \eqref{eqomega} we have 
\[
\pm{\vartheta}=\frac{y+\nabla B_0+\gamma t\Omega_1+\gamma^{3/2}\Omega_2+\gamma^{2}\nabla\tilde\Gamma_{\gamma}}{|y+\nabla B_0|\Big(1+\frac{\gamma t}{|y+\nabla B_0|}\Omega_1\cdot\Big(\frac{y+\nabla B_0}{|y+\nabla B_0|}\Big)+\frac{\gamma^{3/2}}{|y+\nabla B_0|}\Omega_2\cdot\Big(\frac{y+\nabla B_0}{|y+\nabla B_0|}\Big)
+\frac{\gamma^2}{|y+\nabla B_0|}\mathcal{\tilde R}\Big)}.
\]
Setting $\tilde \Omega_j=\frac{\Omega_j}{|y+\nabla B_0|}$ and using $(1+\kappa)^{-1}=1-\kappa+O(\kappa^2)$, we obtain, for another smooth and bounded function $\mathcal{\breve E}(t,x,y,\frac{t}{|y+\nabla B_0|},\frac{\sqrt{\gamma}}{t},{\vartheta})$,
\begin{multline}
\pm{\vartheta}=\Big\{\frac{y+\nabla B_0}{|y+\nabla B_0|}+\gamma t \tilde\Omega_1+\gamma^{3/2}\tilde\Omega_2+\gamma^2\nabla \tilde \Gamma_{\gamma}/|y+\nabla B_0|\Big\}\\\times \Big\{1-\gamma t \tilde \Omega_1\cdot\Big(\frac{y+\nabla B_0}{|y+\nabla B_0|}\Big)-\gamma^{3/2}\tilde\Omega_2\cdot\Big(\frac{y+\nabla B_0}{|y+\nabla B_0|}\Big)+\frac{\gamma^2}{|y+\nabla B_0|}\breve R\Big\}\\
=\frac{y+\nabla B_0}{|y+\nabla B_0|}+\frac{\gamma t}{|y+\nabla B_0|}\Big(\Omega_1-\frac{y+\nabla B_0}{|y+\nabla B_0|}\Omega_1\cdot\Big(\frac{y+\nabla B_0}{|y+\nabla B_0|}\Big)\Big)\\
+\frac{\gamma^{3/2}}{|y+\nabla B_0|}\Big(\Omega_2-\frac{y+\nabla B_0}{|y+\nabla B_0|}\Omega_2\cdot\Big(\frac{y+\nabla B_0}{|y+\nabla B_0|}\Big)\Big)+\frac{\gamma^2}{|y+\nabla B_0|}\mathcal{E},
\end{multline}
and replacing $\frac{y+\nabla B_0}{|y+\nabla B_0|}$ in the last two lines by $\pm{\vartheta}-\gamma t\tilde\Omega_1-\gamma^{3/2}\tilde\Omega_2-\gamma^2\nabla \tilde \Gamma_{\gamma}/|y+\nabla B_0|$ yields \eqref{eqNcontYmodgen} where $\mathcal{E}_{\pm}$ are smooth, bounded functions of $(t,x,y,\frac{t}{|y+\nabla B_0|},\frac{\sqrt{\gamma}}{t},{\vartheta})$.
\end{proof}
\begin{rmq}\label{rmqsigneomegac}
For $t>0$, considering the critical point satisfying \eqref{eqNcontYmodgen} with $+$ sign in the left hand side term produces $O(h^{\infty})$ contributions, as the critical value of the phase is not stationary in $\rho$. Therefore in the following for $t>0$ we pick ${\vartheta}_c(t,x,y,\Sigma,S,A):={\vartheta}_{-}(t,x,y,\Sigma,S,A)$ such that \eqref{eqNcontYmodgen} holds (and for $t<0$ one should take ${\vartheta}_c:={\vartheta}_+$).
\end{rmq}
In Lemma \ref{lemomegac} we only considered critical points of $\tilde\Phi_{N,a,\gamma}$ with respect to ${\vartheta}$. In the next Lemma we deal with critical points with respect to $A$, $\rho$ and ${\vartheta}$ of $\tilde \Phi_{N,a,\gamma}$, where $\rho=|\theta|$, $\theta=\rho{\vartheta}$. 
\begin{lemma}\label{lemdetails}
Let $t>0$. If $\nabla_{A,\rho,{\vartheta}}\tilde \Phi_{N,a,\gamma}=0$ then there exists a smooth, bounded function $\tilde{\mathcal{E}}$ depending on $(t,x,y,\frac{t}{|y+\nabla B_0|},\frac{\sqrt{\gamma}}{t},{\vartheta})$ such that for the critical point with respect to ${\vartheta}$, that we denote ${\vartheta}_c^{\sharp}$ (and for which \eqref{eqNcontYmodgen} holds with ${\vartheta}_-$)
\begin{multline}\label{desclambdaNmaincontribymod}
  \frac{|y+\nabla B_0(y,{\vartheta}_c^{\sharp})|}{t}=1+\frac 16 \gamma Aq^{2/3}({\vartheta}_c^{\sharp})\Big(1-\frac{B_2(y,{\vartheta}_c^{\sharp})}{t}\Big)+\frac{\gamma^{3/2}}{t}\Big(\frac 23 A(\Sigma-S)\\
{}+\frac{\Sigma^3}{3}+\Sigma(\frac{x}{\gamma}q^{1/3}({\vartheta}_c^{\sharp})(1+\Lp(y,{\vartheta}^{\sharp}_{c}))-A)-\frac{S^3}{3}-S(\frac{a}{\gamma}q^{1/3}({\vartheta}_c^{\sharp})-A)\Big)+\frac{\gamma^{2}}{t}\tilde{\mathcal{E}}\,.
\end{multline}
Moreover, the critical point ${\vartheta}_c^{\sharp}$ is such that
\begin{equation}\label{eqNcontYmodgennewnew}
- {\vartheta}_c^{\sharp}=\frac{y+\nabla B_0(y,{\vartheta}_c^{\sharp})}{|y+\nabla B_0(y,{\vartheta}_c^{\sharp})|}
-2\Big(\frac{|y+\nabla B_0(y,{\vartheta}_c^{\sharp})|}{t}-1\Big)\Big[\frac{\anabla q({\vartheta}_c^{\sharp})}{q({\vartheta}_c^{\sharp})}
-\frac{3\anabla B_2(y,{\vartheta}_c^{\sharp})}{2t(1-B_2(y,{\vartheta}^{\sharp}_{c})/t)}\Big]+O(\frac{\gamma^{\frac 32}}{t})\,.
\end{equation}
\end{lemma}
\begin{proof}
From Lemma \ref{lemomegac}, \eqref{eqNcontYmodgen} holds  (critical point w.r.t. ${\vartheta}$). Let $\partial_{A}\tilde\Phi_{N,a,\gamma}=\partial_{\rho}\tilde\Phi_{N,a,\gamma}=0$ (first and 
 second to last equations in \eqref{desclambdaN}): substitution between the two equations yields
\begin{multline}
t\sqrt{1+\gamma Aq^{2/3}({\vartheta}_{})}+\tilde\Phi_{\gamma}(x,y,{\vartheta}_{},A,\Sigma)-\tilde\Phi_{\gamma}(a,0,{\vartheta}_{},A,S)\\
=\frac 23 A\Big(  \frac{t \gamma q^{2/3}({\vartheta}_{})}{2\sqrt{1+\gamma A q^{2/3}({\vartheta}_{})}}+\partial_{A}\Big(\tilde\Phi_{\gamma}(x,y,{\vartheta}_{},A,\Sigma)-\tilde\Phi_{\gamma}(a,0,{\vartheta}_{},A,S)\Big)\Big),
\end{multline}
and using \eqref{tildePhigammadefnew} and \eqref{derivAPhi} we further obtain, with ${\tau_{q}=\tau_{q}(\gamma A,\vartheta)}$, $\partial_A\tau_{q}=\gamma q^{2/3}({\vartheta})/(2\tau_{q})$,
\begin{multline}\label{eq1eq4}
t\tau_{q}+y\cdot {\vartheta}+\gamma^{3/2}\Big(\frac{\Sigma^3}{3}+\Sigma(\frac{x}{\gamma}q^{1/3}(\vartheta)-A)
-\frac{S^3}{3}-S(\frac{a}{\gamma}q^{1/3}(\vartheta)-A)\Big)
+B_0(y,{\vartheta})\\
{}+(1-\tau_{q})B_2(y,{\vartheta})+\gamma^{3/2}\frac{x}{\gamma}q^{1/3}({\vartheta})\Sigma \Lp(y,{\vartheta}) +\gamma^2\tilde\Gamma_{\gamma}=\frac 23 \frac{\gamma Aq^{2/3}}{2\tau_{q}}\Big[t-B_2(y,{\vartheta})+\sum_{k\geq 2}\partial_{\tau_{q}}\Big(\frac{(1-\tau_{q})^k}{\tau_{q}^{k-1}}\Big)B_{2k}(y,{\vartheta}_{})\\
 {}+x\partial_{\tau_{q}}\Big(\tau_{q} A_{\Gamma}(x,y,\sqrt{\gamma}\Sigma q^{1/3}({\vartheta}_{})/\tau_{q},{\vartheta}_{}/\tau_{q})\Big)-a\partial_{\tau_{q}}\Big(\tau_{q} A_{\Gamma}(a,0,\sqrt{\gamma}S q^{1/3}({\vartheta}_{})/\tau_{q},{\vartheta}_{}/\tau_{q})\Big)\Big]-\frac 23\gamma^{3/2}(\Sigma-S)A
\end{multline}
with $\tilde \Gamma_{\gamma}$ defined in \eqref{deftildeGammagam}. As, from \eqref{Agamtau}, we obtain $\partial_{\tau_{q}}(\tau_{q} A_{\Gamma})=-\gamma q^{2/3}({\vartheta})(\Sigma^2-A)\mu(y,{\vartheta})/\tau_{q}^2+\mathcal{H}_{j\geq 3}$, we define a smooth, bounded function $\breve\Gamma_{\gamma}$ 
\begin{multline}
\breve \Gamma_{\gamma}=\frac 13 \frac{ Aq^{1/3}({\vartheta})}{\tau_{q}}\Big[\sum_{k\geq 2}\partial_{\tau_{q}}\Big(\frac{(1-\tau_{q})^k}{\gamma \tau_{q}^{k-1}}\Big)B_{2k}(y,{\vartheta}_{})\\
+\frac{x}{\gamma}\partial_{\tau_{q}}\Big(\tau_{q} A_{\Gamma}(x,y,\sqrt{\gamma}\Sigma q^{1/3}({\vartheta}_{})/\tau_{q},{\vartheta}_{}/\tau_{q})\Big)- \frac{a}{\gamma}\partial_{\tau_{q}}\Big(\tau_{q} A_{\Gamma}(a,0,\sqrt{\gamma}S q^{1/3}({\vartheta}_{})/\tau_{q},{\vartheta}_{}/\tau_{q})\Big].
\end{multline}
Note that $\breve \Gamma_{\gamma}=\frac{Aq({\vartheta})}{3\tau_{q}^2}(\frac{2A}{(1+\tau_{q})}-x(\Sigma^2-A)\mu(y,{\vartheta})+a(S^2-A)\gamma(0,{\vartheta})+\gamma^{-1}\mathcal{H}_{j\geq 3})$, where terms in $\mathcal{H}_{j\geq 3}$ come with wights $\gamma^{j/2}$ and $j\geq 3$. Equation \eqref{eq1eq4} becomes
\begin{multline}
 t+y\cdot {\vartheta}+B_0(y,{\vartheta})= -\gamma Aq^{2/3}({\vartheta})(t-B_2(y,{\vartheta}))(\frac{1}{1+\tau_{q}}-\frac{1}{3\tau_{q}})-\gamma^2(\tilde\Gamma_{\gamma}-\breve \Gamma_{\gamma})\\{}-\gamma^{3/2}\Big(\frac{\Sigma^3}{3}
  +\Sigma(\frac{x}{\gamma}q^{1/3}(\vartheta)(1+\Lp(y,{\vartheta}))-A) -\frac{S^3}{3}-S(\frac{a}{\gamma}q^{1/3}(\vartheta)-A)+\frac 23(\Sigma-S)A\Big)\,.
\end{multline}
As $\tau_{q}^2-1=\gamma Aq^{2/3}/(1+\tau_{q})$ and $(\frac{1}{1+\tau_{q}}-\frac{1}{3\tau_{q}})=\frac 16+\frac{1}{12}\frac{(\tau_{q}^2-1)(1-(\tau_{q}^2-1)/4)}{\tau_{q}^2(1+(\tau_{q}^2-1)/2)}$, we move the part of the coefficient of $(t-B_2(y,{\vartheta}))$ with factor $\gamma^2$ into the next term. Recall $B_0(y,{\vartheta})={\vartheta}\cdot \nabla B_0(y,{\vartheta})$, which eventually yields
\begin{multline}
(y+\nabla B_0(y,{\vartheta}))\cdot{\vartheta}+t=- \gamma Aq^{2/3}({\vartheta})(t-B_2(y,{\vartheta}))\Big(\frac 16+\frac{1}{12}\frac{(\tau_{q}^2-1)(1-(\tau_{q}^2-1)/4)}{\tau_{q}^2(1+(\tau_{q}^2-1)/2)}\Big)\\
{}-\gamma^{3/2}\Big(\frac{\Sigma^3}{3}+\Sigma(\frac{x}{\gamma}q^{1/3}(\vartheta)(1+\Lp(y,{\vartheta}))-A)
-\frac{S^3}{3}-S(\frac{a}{\gamma}q^{1/3}(\vartheta)-A)+\frac 23(\Sigma-S)A\Big)+\gamma^2(\tilde\Gamma_{\gamma}-\breve \Gamma_{\gamma})\,.
\end{multline}
As ${\vartheta}$ is itself a critical point, we obtain, using the last statement of Lemma \ref{lemomegac} and Remark \ref{rmqsigneomegac}, that \eqref{eqNcontYmodgen} holds with $-$ sign on the left, and that 
$\vartheta\cdot (y+\nabla B_0(y,{\vartheta}))=-|y+\nabla B_0|+\gamma^2\mathcal{E}_-$. As a result we obtain \eqref{desclambdaNmaincontribymod} with $\tilde{\mathcal{E}}=\tilde{\mathcal{E}}(t,x,y,\frac{t}{|y+\nabla B_0|},\frac{\sqrt{\gamma}}{t},{\vartheta})$, defined as follows
\[
\tilde{\mathcal{E}}:=\tilde\Gamma_{\gamma}-\breve \Gamma_{\gamma}+\mathcal{E}_-+\frac{1}{12}\frac{A^2q^{4/3}({\vartheta}v)}{1+\tau_{q}}\frac{(1-(\tau_{q}^2-1)/4)}{\tau_{q}^2(1+(\tau_{q}^2-1)/2)}(t-B_{2}(y,\vartheta)\,,
\]
which completes the proof of Lemma \ref{lemdetails}.
 \end{proof}
Pick $(t,x,y)$ with $t>0$ and $\mathcal{N}(t,x,y)\neq\varnothing$. As already noticed, we must have $||y|/t-1|\leq c$ (since otherwise $\mathcal{N}(t,x,y)=\emptyset$ by \eqref{desclambdaNmaincontribymod}). Let now  $N_j\in \mathcal{N}^1_d(t,x,y)$, $j\in \{1,2\}$. Then there exist $(t_j,x_j,y_j)$ such that $N_j\in\mathcal{N}(t_j,x_j,y_j)$ and $(t_{j},x_{j},y_{j})\in \mathcal{C}_{\gamma}(t,x,y)$; there exist $(\theta_j,A_j,\Sigma_j,S_j)$, $j\in\{1,2\}$, $A_j$ close to $1$ 
such that \eqref{desclambdaN} holds with $(t,x,y,\theta,A,\Sigma,S)$ replaced by $(t_j,x_j,y_j,\theta_j,A_j,\Sigma_j,S_j)$. Then we have, with $\vartheta_{j}=\theta_{j}/|\theta_{j}|$,
\begin{gather}\label{eqNjcontgen}
\quad 2N_j (1-\frac 3 4   {B_L}'(|\theta_{j}|\lambda_{\gamma} A_j^{3/2}))=\frac{q^{2/3}(\vartheta_{j})}{2A_j^{1/2}\sqrt{1+\gamma A_jq^{2/3}(\vartheta_{j})}}\frac{1}{\sqrt{\gamma}}\Big(t_j-B_2(y_j,\vartheta_{j})\Big)+O(1)\,,\\
\label{eqNjcontSjgen}
\Sigma_j^2+\frac{x_j}{\gamma}q^{1/3}(\vartheta_{j})\Big(1+\Lp(y_j,\vartheta_{j})+\mathcal{H}_{j\geq 1}\Big)=A_j,
S_j^2+\frac{a}{\gamma}q^{1/3}(\vartheta_{j})(1+\mathcal{H}_{j\geq 1})=A_j\,,\\
\label{eqNjcontYjgen}
\begin{multlined}[t]\frac{y_j+\nabla B_0(y_j,\vartheta_{j})}{t_j}=-\vartheta_{j}+\gamma A_j q^{2/3}(\vartheta_{j})\Big[(1-\frac{B_2(y_j,\vartheta_{j})}{t_j})\Big(\frac 12 \vartheta_{j}-\frac{\nabla q(\vartheta_{j})}{3q(\vartheta_{j})}\Big)\\{}+\frac{1}{2t_j}\anabla B_2(y_j,\vartheta_{j})\Big]
+O(\frac{\gamma^{3/2}}{t_j})\,.\end{multlined}
\end{gather}
Taking the difference between \eqref{eqNjcontgen} for $j=1,2$ yields
\begin{multline}\label{eqN1N2difgen}
N_1-N_2=\frac 34 \Big(N_1B_L'(|\theta_1|\lambda_{\gamma}A_1^{3/2})-N_2B_L'(|\theta_2|\lambda_{\gamma}A_2^{3/2})\Big)+O(1)\\
+\frac{(t_1-B_2(y_1,\vartheta_{1})) q^{2/3}(\vartheta_{1})}{4\sqrt{\gamma}A_1^{1/2}\sqrt{1+\gamma A_1q^{2/3}(\vartheta_{1})}}-\frac{(t_2-B_2(y_2,\vartheta_{2})) q^{2/3}(\vartheta_{2})}{4\sqrt{\gamma}A_2^{1/2}\sqrt{1+\gamma A_2q^{2/3}(\vartheta_{2})}}.
\end{multline}
Using Remark \ref{rmqN1N2dif}, the first term in the RHS of \eqref{eqN1N2difgen} is $O\Big(\frac{N_1+N_2}{\lambda_{\gamma}^2}\Big)$. We are reduced to proving
\begin{equation}\label{bounddeuxTjgen}
\Big|\frac{(t_1-B_2(y_1,\vartheta_{1})) q(\vartheta_{1})}{\sqrt{A_1q^{2/3}(\vartheta_{1})}}-\frac{(t_2-B_2(y_2,\vartheta_{2})) q(\vartheta_{2})}{\sqrt{A_2q^{2/3}(\vartheta_{2})}}\Big|=O(\sqrt{\gamma}).
\end{equation}
This will follow from the next lemma :
\begin{lemma}\label{lemcardN1gen}
Let $(t,x,y)$ be fixed, let $(t_j,x_j,y_j)$ satisfy \eqref{eqcondTjXjYjgen}, \eqref{eqcondTjYjB0gen} and let $(\theta_j=|\theta_j|\vartheta_{j},A_j,\Sigma_j,S_j)$, $j\in\{1,2\}$ with $A_j$ close to $1$ such that \eqref{eqNjcontgen}, \eqref{eqNjcontSjgen} and \eqref{eqNjcontYjgen} hold true, then
\begin{equation}\label{eqAjomega_jdifgen}
t|\vartheta_{1}-\vartheta_{2}|\lesssim \sqrt{\gamma}, \quad t|A_1q^{2/3}(\vartheta_{1})-A_2q^{2/3}(\vartheta_{2})|\lesssim \sqrt{\gamma}.
\end{equation}
\end{lemma}
\begin{proof}
If $\frac{t}{\sqrt{\gamma}}$ is bounded then both inequalities in \eqref{eqAjomega_jdifgen} follow immediately. Suppose $\frac{t}{\sqrt{\gamma}}$ is sufficiently large.
From \eqref{eqNjcontYjgen} and \eqref{eqNcontYmodgen} we get (as in \eqref{desclambdaNmaincontribymod})
\begin{align}\label{desclambdaNmaincontribymodbis}
\frac{|y_j+\nabla B_0(y_j,\vartheta_{j})|-t_j}{t_j-B_2(y_j,\vartheta_{j})} & =\frac 16\gamma A_jq^{2/3}(\vartheta_{j})+O(\frac{\gamma^{3/2}}{t_j})\,.\\
\label{eqNjcontYjmodgen}
\frac{y_j+\nabla B_0(y_j,\vartheta_{j})}{|y_j+\nabla B_0(y_j,\vartheta_{j})|}+\vartheta_{j} & =\begin{multlined}[t]
-\gamma A_jq^{2/3}(\vartheta_{j})\Big[\frac{\anabla q(\vartheta_{j})}{3q(\vartheta_{j})}\Big(1-\frac{B_2(y_j,\vartheta_{j})}{t_j}\Big)
\\{}-\frac{1}{2t_j}\anabla B_2(y_j,\vartheta_{j})\Big]+O(\frac{\gamma^{3/2}}{t_j})\,.
\end{multlined}
  \end{align}
Taking the difference between \eqref{eqNjcontYjmodgen} for $j=1$ and $j=2$ yields 
\[
|\vartheta_{1}-\vartheta_{2}|=\Big|\frac{y_1+\nabla B_0(y_1,\vartheta_{1})}{|y_1+\nabla B_0(y_1,\vartheta_{1})|}-\frac{y_2+\nabla B_0(y_2,\vartheta_{2})}{|y_2+\nabla B_0(y_2,\vartheta_{2})|}\Big|+O(\gamma).
\]
As $\nabla B_0(y,\vartheta)=O(|y|^2)$, $\Big|\frac{y_1}{|y_1|}-\frac{y_2}{|y_2|}\Big|\leq \frac{2r_0\sqrt{\gamma}}{|y|}$ from \eqref{eqcondTjXjYjgen} and $\frac{t}{|y|}\sim 1$, the first inequality in \eqref{eqAjomega_jdifgen} holds true. We proceed with the second one, which is, as for the model, more delicate to handle.
\begin{lemma}\label{lemomegcrit}
Let ${\tilde\vartheta}_j={\tilde \vartheta}_j(t_j,y_j)$ be the solution to \eqref{eq:omegatildethetaj} below, then $|\vartheta_{j}-{\tilde \vartheta}_j|\lesssim \frac{\gamma^{3/2}}{t_j}$.
\begin{equation}\label{eq:omegatildethetaj}
{\tilde\vartheta}_{j}+\frac{y_j+\nabla B_0(y_j,{\tilde \vartheta}_{j})}{|y_j+\nabla B_0(y_j,{\tilde\vartheta}_{j})|}=-2\Big(\frac{|y_j+\nabla B_0(y_j,{\tilde\vartheta}_{j})|}{t_j}-1\Big)\Big[\frac{\anabla q({\tilde \vartheta}_{j})}{q({\tilde\vartheta}_{j})}-\frac{3\anabla B_2(y_j,{\tilde\vartheta}_{j})}{2t_{j}(1-\frac{B_2(y_j,{\tilde\vartheta}_{j})}{t_j})}\Big]\,.
\end{equation}
\end{lemma}
\begin{proof}
The coefficient of $\gamma A_j$ in \eqref{desclambdaNmaincontribymodbis} does not vanish as $B_2(y_j,\vartheta_{j})=O(|y_j|^2)$ and $|y_j|/t_j\sim 1$; we replace $\gamma A_j$ in \eqref{eqNjcontYjmodgen} by its first approximation given in \eqref{desclambdaNmaincontribymodbis}. We obtain
\begin{multline}\label{eq:omegathetaj}
  \vartheta_{j}=-\frac{y_j+\nabla B_0(y_j,\vartheta_{j})}{|y_j+\nabla B_0(y_j,\vartheta_{j})|}\\-2\Big(\frac{|y_j+\nabla B_0(y_j,\vartheta_{j})|}{t_j}-1\Big)\Big[\frac{\anabla q(\vartheta_{j})}{q(\vartheta_{j})}-\frac{3\anabla B_2(y_j,\vartheta_{j})}{2t_j{(1-\frac{B_2(y_j,\vartheta_{j})}{t_j})}}\Big]+O(\frac{\gamma^{3/2}}{t_j}).
\end{multline}
Taking the difference between \eqref{eq:omegathetaj} and \eqref{eq:omegatildethetaj}, using $\nabla B_2(y_j,\vartheta)=O(|y_j|^2)$ and smallness of $\Big|\frac{|y_j+\nabla B_0(y_j,{\tilde\vartheta}_{j})|}{t_j}-1\Big|$, completes the proof of Lemma \ref{lemomegcrit}.
\end{proof}
Using \eqref{desclambdaNmaincontribymodbis}, in order to achieve the proof of Lemma \ref{lemcardN1gen} we are reduced to proving  that
\begin{equation}\label{remainstoprove}
\frac{t}{\gamma^{3/2}}\Big|\frac{|y_1+\nabla B_0(y_1,\vartheta_{1})|-t_1}{t_1-B_2(y_1,\vartheta_{1})}-\frac{|y_2+\nabla B_0(y_2,\vartheta_{2})|-t_2}{t_2-B_2(y_2,\vartheta_{2})}\Big|=O(1).
\end{equation}
Set $\tilde y_{j}=y_{j}+\nabla B_{0}(y_{j},\vartheta_{{j}})$ and write 
\begin{multline}
  \Big|\frac{|{\tilde y_1}|-t_1}{t_1-B_2(y_1,\vartheta_{1})}-\frac{|{\tilde y_2}|-t_2}{t_2-B_2(y_2,\vartheta_{2})}\Big|\leq \Big|\frac{|{\tilde y_1}|-t_1+t_2-|{\tilde y_2}|}{t_1-B_2(y_1,\vartheta_{1})}\\
{}\quad\quad\quad\quad\quad\quad\quad\quad\quad\quad\quad\quad\quad\quad\quad  {}-(t_2-|{\tilde y_2}|)\Big|\frac{1}{t_1-B_2(y_1,\vartheta_{1})}-\frac{1}{t_2-B_2(y_2,\vartheta_{2})}\Big|\\
\leq  \Big|\frac{|y_1+\nabla B_0(y_1,{\tilde\vartheta}_{1})|-t_1+t_2-|y_2+\nabla B_0(y_2,{\tilde\vartheta}_{2})|}{t_1-B_2(y_1,\vartheta_{1})}\Big|+\Big|\frac{|{\tilde y_1}|-|y_1+\nabla B_0(y_1,{\tilde \vartheta}_{1})|}{t_1-B_2(y_1,\vartheta_{1})}\Big|\\ +\Big|\frac{|{\tilde y_2}|-|y_2+\nabla B_0(y_2,{\tilde \vartheta}_{2})|}{t_1-B_2(y_1,\vartheta_{1})}\Big|
+\Big|\frac{|{\tilde y_2}|-t_2}{t_2-B_2(y_2,\vartheta_{2})}\Big| \Big| \frac{t_2-B_2(y_2,\vartheta_{2})}{t_1-B_2(y_1,\vartheta_{1})}-1\Big|
\end{multline}
and using \eqref{eqcondTjYjB2gen} and \eqref{eqcondTjYjB0gen}, Lemma \ref{lemomegcrit} and \eqref{desclambdaNmaincontribymodbis} for $j=2$ yields
\begin{multline}
\frac{t}{\gamma^{3/2}}\Big|\frac{|{\tilde y_1}|-t_1}{t_1-B_2(y_1,\vartheta_{1})}-\frac{|{\tilde y_2}|-t_2}{t_2-B_2(y_2,\vartheta_{2})}\Big|\leq \frac{t}{\gamma^{3/2}}\times \frac{r_0\gamma^{3/2}}{t_1} \\ +\frac{t}{\gamma^{3/2}}\times \Big(O(\frac{|y_1|^2}{t_1})|\vartheta_{1}-{\tilde \vartheta}_1|+O(\frac{|y_2|^2}{t_1})|\vartheta_{2}-{\tilde \vartheta}_2|\Big)\\
+\frac{t}{\gamma^{3/2}}\times O(\gamma)\times \Big(\frac{|t_2-B_2(y_2,{\tilde \vartheta}_{2})-t_1+B_2(y_1,{\tilde\vartheta}_{1})|}{t_1-B_2(y_1,\vartheta_{1})}+O(\frac{|y_1|^2}{t_1})|\vartheta_{1}-{\tilde \vartheta}_1|+O(\frac{|y_2|^2}{t_1})|\vartheta_{2}-{\tilde\vartheta}_2|\Big)
=O(1).
\end{multline}
For $d=2$ the proof is much simpler (no angle $\vartheta$). In this case Lemma \ref{lemcardN1gen} reduces to obtaining suitable estimates for $|A_1-A_2|$ ; this follows from the $2D$ equivalent of
\eqref{desclambdaNmaincontribymodbis} which reads 
\[
\frac{|y_j+ B_0(y_j)|}{t_j}=1+\frac 16 \gamma A_j\Big(1-\frac{B_2(y_j)}{t_j}\Big)+O(\gamma^{3/2}/t)\,,
\]
and we are done with Lemma \ref{lemcardN1gen}.\end{proof}
This completes the proof of Proposition \ref{propcardN} in the general case.
\end{proof}
\begin{prop}\label{propcardoutN1}
Let $(t,x,y)$ such that $\mathcal{N}(t,x,y)\neq\emptyset$. Then,
\[
\sum_{N\notin \mathcal{N}^1_d(t,x,y), |N|\lesssim \frac{1}{\sqrt{\gamma}}} V_{N,\gamma}(t,x,y)=O(h^{\infty}).
\]
\end{prop}
\begin{proof}
We, again, consider first the model case:  in rescaled variables $T,X,Y$, the Lagrangian $\Lambda^M_N$ is defined by \eqref{systemICmodelA}, \eqref{systemICmodelSS}, \eqref{systemICmodeltheta}. 
We start with integration in $\vartheta$ : $\Psi^M_{N,a,\gamma}=\Psi^M_{0,a,\gamma}-
\frac 43 |\theta|NA^{3/2}+\frac{N}{\lambda_{\gamma}}B_L((|\theta|\lambda_{\gamma}A^{3/2})$ and the last two terms do not depend on $\vartheta$;  $\nabla^{2}_{\vartheta}\Psi^M_{N,a,\gamma}=\nabla_{\vartheta}^{2}\Psi^M_{0,a,\gamma}=\frac{|Y|}{\gamma}\times O(1)=\frac{|y|}{\gamma^{3/2}}\times O(1)$. Using Lemma \ref{derriere}, we have $|y|>c_0 |t|$, hence critical points with respect to $\vartheta$ are non-degenerate and stationary phase applies (providing a factor $(\lambda_{\gamma}\frac{|y|}{\gamma^{3/2}})^{-(d-2)/2}\lesssim (\frac{h}{t})^{(d-2)/2}$). Let $t>0$, then using Lemma \ref{lemomegac} with $\Gamma=0$ (hence $B_0=B_2=l=0$), there exists a smooth, bounded function $\mathcal{E}_{M,-}$ such that the critical point (associated to the $-$ sign in \eqref{eqNcontYmodgen}) is
\begin{equation}\label{omegacformmodel}
\vartheta_c=-{\vartheta_Y}-\frac{\gamma A}{3}\frac{T}{|Y|}\frac{\anabla q(-{\vartheta_Y})}{q^{1/3}(-{\vartheta_Y})}+\frac{\gamma}{|Y|}(\Sigma X-S\frac{a}{\gamma})\frac{\anabla q(-{\vartheta_Y})}{3q^{2/3}(-{\vartheta_Y})}+\frac{\gamma^{3/2}}{|Y|}\mathcal{E}_{M,-}\,,
\end{equation}
where we set $\vartheta_{Y}=Y/|Y|$.
The critical point (up to its sign) is unique, and from Lemma \ref{lemdetails} (in the model case) we see that we recover \eqref{eqomegathetaYT}. Let $\rho:=|\theta|$ and denote $\tilde\Psi^M_{N,a,\gamma}(T,X,Y,\Sigma,S,A,\rho):=\Psi^M_{N,a,\gamma}(T,X,Y,\Sigma,S,A,\rho\vartheta_c)$ the critical value of the phase at $\vartheta_c$ given above, then
\begin{multline}\label{defnewPsitildemodel}
\tilde\Psi^M_{N,a,\gamma}(T,X,Y,\Sigma,S,A,\rho)=\rho\Big(\frac{Y\cdot \vartheta_{c}+T\sqrt{1+\gamma A q^{2/3}(\vartheta_{c})}}{\gamma}+\frac{\Sigma^3}{3}+\Sigma(Xq^{1/3}(\vartheta_{c})-A)\\
-\frac{S^3}{3}-S(\frac{a}{\gamma}q^{1/3}(\vartheta_{c})-A)-\frac 43 N A^{3/2}\Big)+\frac{N}{\lambda_{\gamma}}B_L(\rho\lambda_{\gamma} A^{3/2})\,.
\end{multline}
Define sets of integers related to stationary points of this phase:
\begin{gather}
  \mathcal{N}^{M}_{\tilde\Psi}(T,X,Y):=\{N\in \mathbb{Z}, \exists (\Sigma,S,A,\rho) \text{ such that } \nabla_{(\Sigma,S,A,\rho)}\tilde\Psi^M_{N,a,\gamma}(T,X,Y,\Sigma,S,A,\rho)=0\}\,,\\
  \mathcal{N}^{1,M}_{\tilde\Psi}(T,X,Y):=\cup_{\mathcal{C}^M_{\gamma}(T,X,Y)}\mathcal{N}^M_{\tilde\Psi}(T',X',Y')\,.
\end{gather}
One sees that $\mathcal{N}^M_{\tilde \Psi}(T,X,Y)=\mathcal{N}^M(t,x,y)$, which implies $ \mathcal{N}^{1,M}_{\tilde \Psi}(T,X,Y)=\mathcal{N}^{1,M}_d(t,x,y)$, where $(T,X,Y)=(\frac{t}{\sqrt{\gamma}},\frac{x}{\gamma},\frac{y}{\sqrt{\gamma}})$: indeed, if $N$ is such that there exists a critical point $(\Sigma,S,A,\rho)$ for $\tilde\Psi^M_{N,a,\gamma}$ at $(T,X,Y)$, then $(\Sigma,S,A,\rho,\vartheta_c(T,X,Y,\Sigma,S,A,a,\gamma))$ is a critical point for $\Psi^M_{N,a,\gamma}$ and the converse also holds. We now  need to prove that 
\[
\sum_{N\notin\mathcal{N}^{1,M}_{\tilde\Psi}(T,X,Y),  |N|\lesssim \frac{1}{\sqrt{\gamma}}} V^M_{N,\gamma}(T,X,Y)=O(h^{\infty}),
\]
where $V^M_{N,\gamma}(T,X,Y)$ had phase $\Psi^M_{N,a,\gamma}$ that became $\tilde \Psi^M_{N,a,\gamma}$ after the stationary phase in $\vartheta$. 

Let first $4a\lesssim \gamma$: stationary phase applies in $S$. Indeed, $\tilde \Psi^M_{N,a,\gamma}$ is stationary in $S$ when $S^2+\frac{a}{\gamma}q^{1/3}(\vartheta_c)=A$ and for $\frac{a}{\gamma}$ small enough there are two non-degenerate critical points $S_{\pm}$ (with main contributions $\pm \sqrt{A}$). We denote 
by $\tilde \Psi^{M,\pm}_{N,a,\gamma}$ the critical values of the phase $\tilde \Psi^{M}_{N,a,\gamma}$ at $S_{\pm}$. For $\cveps\in\{\pm\}$, we define
\begin{gather*}
  \mathcal{N}^{M,\cveps}_{\tilde\Psi}(T,X,Y):=\{N\in \mathbb{Z}\,:\,\, \exists (\Sigma,A,\rho) \text{ such that } \nabla_{(\Sigma,A,\rho)}\tilde\Psi^{M,\cveps}_{N,a,\gamma}(T,X,Y,\Sigma,A,\rho)=0\}\,,\\
 \mathcal{N}^{1,M,\cveps}_{\tilde\Psi}(T,X,Y):=\cup_{\mathcal{C}^M_{\gamma}(T,X,Y)}\mathcal{N}^{M,\cveps}_{\tilde\Psi}(T',X',Y')
\,\,\text{ and }\,\,\cap_{\pm} \mathcal{N}^{M,\pm}_{\tilde\Psi}(T,X,Y)=\mathcal{N}^M_{\tilde \Psi}(T,X,Y)\,.
\end{gather*}
 If $N\in \mathcal{N}^M_{\tilde \Psi}(T,X,Y)$, then there exists a critical point $\Sigma_c,S_c,A_c,\rho_c$ for the phase $\tilde\Psi^M_{N,a,\gamma}$; $S_c$ satisfies $S^2+\frac{a}{\gamma}q^{1/3}(\vartheta_c)=A$, hence $S_c\in \{S_{\pm}\}$. Therefore $(\Sigma_c,A_c,\rho_c)$ is a critical point for $\tilde\Psi^M_{N,a,\gamma}(T,X,Y,\Sigma,S_{\pm},A,\rho\vartheta_c|_{S_{\pm}})=\tilde\Psi^{\pm}_{N,a,\gamma}$. Conversely, let $\Sigma_{\pm},A_{\pm},\rho_{\pm}$ be a critical point for $\tilde\Psi^{M,\pm}_{N,a,\gamma}$; for each sign $\cveps\in\{\pm\}$, let $S_{\cveps}^{\pm}$ denote both solutions to $S^2+\frac{a}{\gamma}q^{1/3}(\vartheta_c|_{A_{\cveps},\Sigma_{\cveps},S})=A_{\cveps}$, then $(\Sigma_{\cveps},S_{\cveps}^{\pm}, A_{\cveps},\rho_{\cveps})$ are critical points for $\tilde\Psi^{M}_{N,a,\gamma}$ and both inclusions hold. Using  $\cap_{\pm} \mathcal{N}^{M,\pm}_{\tilde\Psi}(T,X,Y)=\mathcal{N}^M_{\tilde \Psi}(T,X,Y)$ we obtain $\cap_{\pm} \mathcal{N}^{1,M,\pm}_{\tilde\Psi}(T,X,Y)=\mathcal{N}^{1,M}_{\tilde \Psi}(T,X,Y)$ and therefore 
$(\mathcal{N}^{1,M}_{\tilde \Psi}(T,X,Y))^c=\cup_{\pm} (\mathcal{N}^{1,M,\pm}_{\tilde\Psi}(T,X,Y))^c$, where $(\mathcal{N}^{1,M}_{\tilde \Psi})^c$ is the complement set. Hence, the proof of Proposition \ref{propcardoutN1} for $\Delta_M$ (for $4a\lesssim \gamma$) rests on proving
\[
\sum_{N\notin\cup_{\cveps\in \{\pm\}}\mathcal{N}^{1,M,\cveps}_{\tilde\Psi}(T,X,Y), |N|\lesssim \frac{1}{\sqrt{\gamma}}} V^{M,\cveps}_{N,\gamma}(T,X,Y)=O(h^{\infty}),
\]
where now $V^{M,\cveps}_{N,\gamma}$ has phase function $\tilde\Psi^{M,\cveps}_{N,a,\gamma}$ and its symbol is obtained from the one of $V^{M}_{N,\gamma}$ after stationary phase in $S$. Using \eqref{defnewPsitildemodel}, we obtain, with $\vartheta_c$ satisfying \eqref{eqomegathetaYT}
\begin{multline}\label{defnewPsitildepmmodel}
\tilde\Psi^{M,\pm}_{N,a,\gamma}(T,X,Y,\Sigma,A,\rho)=\rho\Big(\frac{Y\cdot \vartheta_{c}+T\sqrt{1+\gamma A q^{2/3}(\vartheta_{c})}}{\gamma}+\frac{\Sigma^3}{3}+\Sigma(Xq^{1/3}(\vartheta_{c})-A)\\
\mp\frac 23(\frac{a}{\gamma}q^{1/3}(\vartheta_{c})-A)^{3/2}-\frac 43 N A^{3/2}\Big)+\frac{N}{\lambda_{\gamma}}B_L(\rho\lambda_{\gamma} A^{3/2}).
\end{multline}
\begin{lemma}\label{rmqIC}
There exists a uniform constant $c>0$ such that, if $N\notin\cup_{\cveps\in \{\pm\}}\mathcal{N}^{1,M,\cveps}_{\tilde\Psi}(T,X,Y)$ then, for all $(\Sigma, A,\rho)$ on the support of the symbol of $V^{M,\pm}_{N,\gamma}(T,X,Y)$,
\begin{equation}\label{infboundforPsimodel}
|\nabla_{(\Sigma,A,\rho)}\tilde\Psi^{\cveps}_{N,a,\gamma}(T,X,Y,\Sigma,A,\rho)|\geq c.
\end{equation}
\end{lemma}
\begin{rmq}
The lemma allows to conclude the proof of Proposition \ref{propcardoutN1} in the case $4a\lesssim \gamma$ as, using \eqref{infboundforPsimodel}, we apply non stationary phase with respect to $(\Sigma,A,\rho)$ and obtain a contribution $O(h^{\infty})$ for each such $V^{M,\cveps}_{N,a,\gamma}$; the sum over $N$ is finite (up to $|N|\lesssim\frac{1}{\sqrt{\gamma}}$) and we can sum up.
\end{rmq}
Before dealing with the proof of  Lemma \ref{rmqIC}, we go back to \eqref{systemICmodelA}, \eqref{systemICmodelSS} and \eqref{systemICmodeltheta}:
 we set $S=S_{\pm}(A,\vartheta)$ is such that $S^2+\frac{a}{\gamma}q^{1/3}(\vartheta)=A$. For each $\cveps\in\{\pm\}$, we explicitly obtain the integral curves $(T_{\cveps}, X_{\cveps}, Y_{\cveps})(A,\Sigma,\theta)$
depending on the parameters $(A,\Sigma,\theta)$,
\begin{gather}\label{courbesint}
\,\,\quad  \left\{ \begin{array}{l}
 X_{\cveps}(A,\Sigma,\vartheta)q^{1/3}(\vartheta)=A-\Sigma^2,\quad S_{\cveps}=\cveps\sqrt{A-\frac{a}{\gamma}q^{1/3}(\vartheta)},\\
T_{\cveps}(A,\Sigma,|\theta|,\vartheta)\frac{q^{2/3}(\vartheta)}{2\sqrt{1+\gamma Aq^{2/3}(\vartheta)}}=2NA^{1/2}(1-\frac 34 B'_L(|\theta|\lambda_{\gamma}A^{3/2}))+(\Sigma-S_{\cveps}(A,\vartheta)),\\
\begin{multlined} Y_{\cveps}(A,\Sigma,|\theta|,\vartheta)= -\frac{ T_{\cveps}(A,\Sigma,|\theta|,\vartheta)q^{2/3}(\vartheta)}{\sqrt{1+\gamma Aq^{2/3}(\vartheta)}}\Big(\frac{\vartheta}{q^{2/3}(\vartheta)}(1+\gamma A q^{2/3}(\vartheta))+\frac{\gamma A}{3}\frac{\anabla q(\vartheta)}{q(\vartheta)}\Big)\\
 {} -\gamma \Big(\frac{\Sigma^3}{3}+\Sigma(X_{\cveps}q^{1/3}(\vartheta_{c})-A)-\cveps\frac{2}{3}S^3_{\cveps}\Big)\vartheta+\frac 43 \gamma NA^{3/2}(1-\frac 34 B'_L(|\theta|\lambda_{\gamma}A^{3/2}))\vartheta\\
{} -\gamma \Big(\Sigma  X_{\cveps}(A,\Sigma,\vartheta)-S_{\cveps}\frac{a}{\gamma}\Big)q^{1/3}(\vartheta)\frac{\anabla q(\vartheta)}{3q(\vartheta)}
 \,.\end{multlined}
\end{array} \right.
v\end{gather}
Taking the scalar product with $\vartheta$ in the last equation and using that $\vartheta\cdot \anabla q(\vartheta)=0$ yields
\[
 Y_{\cveps}\cdot \vartheta+ T_{\cveps}\sqrt{1+\gamma Aq^{\frac 23}(\vartheta)}+\gamma \Big(\frac{\Sigma^3}{3}+\Sigma(X_{\cveps}q^{\frac  13}(\vartheta_{c})-A)-\cveps\frac{2}{3}S^3_{\cveps}\Big)=\frac 43 \gamma NA^{\frac 32}(1-\frac 34 B'_L(|\theta|\lambda_{\gamma}A^{\frac 32})),
\]
which is nothing but the derivative of the phase with respect to $\rho=|\theta|$. 
\begin{proof} (of \eqref{infboundforPsimodel})
  Fix $(T,X,Y)$. Let $N\notin\cup_{\cveps\in \{\pm\}}\mathcal{N}^{1,M,\cveps}_{\tilde\Psi}(T,X,Y)$, hence $N\notin\cup_{\cveps\in \{\pm\}}\mathcal{N}^{M,\cveps}_{\tilde\Psi}(T',X',Y')$ for all $(T',X',Y')\in \mathcal{C}^M_{\gamma}(T,X,Y)$, which is equivalent to $\nabla_{(\Sigma,A,\rho)}\tilde\Psi^{M,\cveps}_{N,a,\gamma}(T',X',Y',\Sigma,A,\rho)\neq 0$ for all $(T',X',Y')\in \mathcal{C}^M_{\gamma}(T,X,Y)$.  Pick a sign $\cveps\in\{\pm\}$, let $(T',X',Y')\in \mathcal{C}^M_{\gamma}(T,X,Y)$, and let $(A,\Sigma,\rho)$ be a given point (on the support of the symbol of $V^{M,\cveps}_{N,a,\gamma}$); let $S_{\cveps}(A,\vartheta_c')$ and $\vartheta_c'$ be solutions to $S^2+\frac{a}{\gamma}q^{1/3}(\vartheta_c')=A$ and \eqref{omegacformmodel}, where in \eqref{omegacformmodel} we replace $(T,X,Y)$ by $(T',X',Y')$ and where the sign of $S_{\cveps}$ is $\cveps$. We may compute $\nabla_{(A,\Sigma,\rho)}\tilde \Psi^{M,\cveps}_{N,a,\gamma}(T',X',Y',\Sigma,A,\rho)$,
taking advantage of $\vartheta_c'$ and $S_{\cveps}(A,\vartheta_c')$ being stationary points for $\Psi^{M}_{N,a,\gamma}(\cdots)$ to cancel their derivatives. Let now $\vartheta\in\mathbb{S}^{d-2}$,
using \eqref{courbesint} and
$S_{\cveps}(A,\vartheta)=\cveps\sqrt{A-\frac{a}{\gamma}q^{1/3}(\vartheta)}$), we get
\begin{gather} \label{derAetc}
  \,  \left\{ \begin{array}{l}
  \partial_{A}\tilde\Psi^{M,\cveps}_{N,a,\gamma}(T',X',Y',\cdot)  =\frac{T'q^{2/3}(\vartheta'_{c})}{2\sqrt{1+\gamma Aq^{2/3}(\vartheta'_{c})}}-\frac{T_{\cveps}q^{2/3}(\vartheta)}{2\sqrt{1+\gamma Aq^{2/3}(\vartheta)}}+S_{\cveps}(A,\vartheta'_c)-S_{\cveps}(A,\vartheta),\\
\partial_{\Sigma}\tilde\Psi^{M,\cveps}_{N,a,\gamma}(T',X',Y',\cdot)  =X'q^{1/3}(\vartheta'_c)-X_{\cveps}q^{1/3}(\vartheta),\\
                \partial_{\rho}\tilde\Psi^{M,\cveps}_{N,a,\gamma}(T',X',Y',\cdot)  = \begin{multlined}[t]\gamma^{-1}\bigl( {Y'\cdot\vartheta'_{c}+ T'\sqrt{1+\gamma Aq^{\frac 23}(\vartheta'_{c})}}
                  -{Y_{\cveps}\cdot \vartheta 
                    - T_{\cveps}\sqrt{1+\gamma Aq^{\frac 23}(\vartheta)}}\bigr) \\
  {}+\Sigma\Big(X'q^{\frac 13}(\vartheta'_c)-X_{\cveps}q^{\frac 13}(\vartheta)\Big)-\frac 23\Big(S^3_{\cveps}(A,\vartheta'_c)-S^3_{\cveps}(A,\vartheta)\Big)
  \,.\end{multlined}
\end{array} \right.
\end{gather}
Recall that $\vartheta'_c$ is provided by \eqref{omegacformmodel} with $(T,X,Y)$ replaced by $(T',X',Y')$; the same formula holds for $\vartheta$ with $(T,X,Y)$ replaced by $(T_{\cveps},X_{\cveps},Y_{\cveps})$. For such $\vartheta'_c$ (resp. $\vartheta$) we have $Y'\cdot \vartheta'_c=-|Y'|+O(\gamma^2)$ (resp. $Y_{\cveps}\cdot \vartheta=-|Y_{\cveps}|+O(\gamma^2)$). Setting $\gamma Z':={-|Y'|+T'}$ and $\gamma Z_{\cveps}:= {-|Y_{\cveps}|+T_{\cveps}}$, $\partial_{\rho}\tilde\Psi^{M,\cveps}_{N,a,\gamma}$ may be rewritten as
\begin{multline}\label{derrho}
\quad\quad\partial_{\rho}\tilde\Psi^{M,\cveps}_{N,a,\gamma}(T',X',Y',\cdots)=A\partial_A\tilde\Psi^{M,\cveps}_{N,a,\gamma}(T',X',Y',\cdots)+\Sigma \partial_{\Sigma}\tilde\Psi^{M,\cveps}_{N,a,\gamma}(T',X',Y',\cdots)\\
{}+Z'-Z_{\cveps}-A(S_{\cveps}(A,\vartheta'_c)-S_{\cveps}(A,\vartheta))
- \Big(\frac{S^3_{\cveps}(A,\vartheta'_c)}{3}-\frac{S^3_{\cveps}(A,\vartheta)}{3}\Big)+O(\gamma),
\end{multline}
where all the small terms $O(\gamma)$ come with differences $T'q^{2/3}(\vartheta'_{c})-T_{\cveps}q^{2/3}(\vartheta)$, $q^{1/3}(\vartheta'_c)-q^{1/3}(\vartheta)$, $X'q^{1/3}(\vartheta'_c)-X_{\cveps}q^{1/3}(\vartheta)$ etc. Using \eqref{derAetc} we bound the gradient of $\tilde\Psi^{M,\cveps}_{N,a,\gamma}$
\begin{multline}\label{inegtildePsi} 
(|\partial_{A}\tilde\Psi^{M,\cveps}_{N,a,\gamma}|+|\partial_{\Sigma}\tilde\Psi^{M,\cveps}_{N,a,\gamma}|+|\partial_{\rho}\tilde\Psi^{M,\cveps}_{N,a,\gamma}|)(T',X',Y',\Sigma,A,\rho)\leq 4 (|T'q^{2/3}(\vartheta'_{c})-T_{\cveps}q^{2/3}(\vartheta)|\\
+|q^{1/3}(\vartheta'_c)-q^{1/3}(\vartheta)|+|X'q^{1/3}(\vartheta'_c)-X_{\cveps}q^{1/3}(\vartheta)| +|Z'-Z_{\cveps}|)\,.
\end{multline}
As $N\notin \mathcal{N}^{1,M,\cveps}_{\tilde\Psi}(T,X,Y)$ implies that $\nabla_{(\Sigma,A,\rho)}\tilde\Psi^{M,\cveps}_{N,a,\gamma}(T',X',Y',\Sigma,A,\rho)\neq 0$ for all $(T',X',Y')\in\mathcal{C}^M_{\gamma}(T,X,Y)$, it follows that the right hand side of \eqref{inegtildePsi} doesn't vanish for any $(A,\Sigma,\rho,\vartheta)$. Hence, for all $(A,\Sigma,\rho)$ on the support of the symbol and for every $\vartheta\in\mathcal{S}^{d-2}$,
\begin{equation}\label{notintheballr0}
(T_{\cveps}(A,\Sigma,\rho,\vartheta),X_{\cveps}(A,\Sigma,\rho,\vartheta),Z_{\cveps}(A,\Sigma,\rho,\vartheta))\notin B_{r_0}(T,X,Z),\quad \gamma Z={-|Y|+T}\,.
\end{equation}
In fact, if the last statement does not hold, then there exist $(A,\Sigma,\rho,\vartheta)$ such that
\[
(|T-T_{\cveps}(A,\Sigma,\rho,\vartheta)|
+|X-X_{\cveps}(A,\Sigma,\rho,\vartheta)|+|Z-Z_{\cveps}(A,\Sigma,\rho,\vartheta)|\leq  r_0.
\]
Taking $(T',X',Y')=(T_{\cveps},X_{\cveps},Y_{\cveps})(A,\Sigma,\rho,\vartheta)$, $\gamma Z'={-|Y'|+T'}$, then $\vartheta'_c=\vartheta$ and therefore the right hand side term in \eqref{inegtildePsi} vanishes which contradicts  $(T',X',Y')\in \mathcal{C}_{\gamma}^M(T,X,Y)$.
As \eqref{notintheballr0} holds true for all $\vartheta$, it also holds for $\vartheta={\vartheta_c}$ from \eqref{omegacformmodel} (corresponding to $(T,X,Y,\Sigma,A,\rho)$), and 
\[
|T-T_{\cveps}(A,\Sigma,\rho,{\vartheta_c})|
+|X-X_{\cveps}(A,\Sigma,\rho,{\vartheta_c})|+|Z-Z_{\cveps}(A,\Sigma,\rho,{\vartheta_c})|\geq  r_0.
\]
Moreover, at $(T,X,Y,\Sigma,A,\rho)$ and $\vartheta={\vartheta_c}$ we obtain from \eqref{derAetc}
\begin{gather*}
|\partial_{A}\tilde\Psi^{M,\cveps}_{N,a,\gamma}|(T,X,Y,\Sigma,A,\rho)= \frac{q^{2/3}(\vartheta_{c})}{2\sqrt{1+\gamma Aq^{2/3}({\vartheta_c})}}|T-T_{\cveps}(A,\Sigma,\rho,{\vartheta_c})|,\\
|\partial_{\Sigma}\tilde\Psi^{M,\cveps}_{N,a,\gamma}|(T,X,Y,\Sigma,A,\rho)=q^{1/3}({\vartheta_c})|X-X_{\cveps}|,\\
|\partial_{\rho}\tilde\Psi^{M,\cveps}_{N,a,\gamma}|(T,X,Y,\Sigma,A,\rho)\geq |Z-Z_{\cveps}(A,\Sigma,\rho,{\vartheta_c})|-\frac{Aq^{2/3}({\vartheta_c})|T-T_{\cveps}|}{1+\sqrt{1+\gamma A q^{2/3}({\vartheta_c})}}-|\Sigma|q^{1/3}({\vartheta_c})|X'-X_{\cveps}|,
\end{gather*}
and therefore there exists a uniform constant (depending only on $q$) such that
\begin{multline}\label{inegtildePsilowerbound} 
(|\partial_{A}\tilde\Psi^{M,\cveps}_{N,a,\gamma}|+|\partial_{\Sigma}\tilde\Psi^{M,\cveps}_{N,a,\gamma}|+|\partial_{\rho}\tilde\Psi^{M,\cveps}_{N,a,\gamma}|)(T,X,Y,\Sigma,A,\rho)
\geq C(|T-T_{\cveps}(A,\Sigma,\rho,{\vartheta_c})|\\
+|X-X_{\cveps}(A,\Sigma,\rho,{\vartheta_c})| +|Z'-Z_{\cveps}(A,\Sigma,\rho,{\vartheta_c})|)\geq Cr_0.
\end{multline}
which allows to conclude the proof of Lemma \ref{rmqIC}. \end{proof}
Let now $a\sim \gamma$. We rescale variables as follows $s=\sqrt{a}|\theta|^{1/3}S$, $\sigma\sqrt{a}|\theta|^{1/3}\Sigma$, $\alpha=a|\theta|^{2/3}A$, let $\lambda=\lambda_a=\frac{a^{3/2}}{h}$ and  $\Psi^M_{N,a,a}(T,X,Y,\Sigma,S,A,\theta):=\Phi^M_{N,a,a}(\sqrt{a}T,aX,\sqrt{a}Y,\sqrt{a}|\theta|^{1/3}\Sigma,\sqrt{a}|\theta|^{1/3}S,a|\theta|^{2/3}A)$. The saddle points of $\Psi^M_{N,a,a}$ satisfy $S^2=A-q^{1/3}({\vartheta_c})$ and they undergo coalescence when $A=q^{1/3}({\vartheta_c})$. Let $\chi_0\in C^{\infty}$ be a smooth cutoff, equal to $1$ on $[0,\infty]$ and equal to $0$ on $[-\infty,-2]$. Then $(\chi_0Ai)(-(|\theta|\lambda)^{2/3}(A-q^{1/3}({\vartheta_c})))$ is a symbol of order $2/3$ supported for values $(|\theta|\lambda)^{2/3}(A-q^{1/3}({\vartheta_c}))\leq 2$ and $(1-\chi_0)Ai(-(|\theta|\lambda)^{2/3}(A-q^{1/3}({\vartheta_c})))$ is supported for $A\geq q^{1/3}({\vartheta_c})$ and is equal to $1$ on $(|\theta|\lambda)^{2/3}(A-q^{1/3}({\vartheta_c}))\geq 2$. Notice that on the support of $(1-\chi_0)$ the Airy function can be written as a sum of two contributions $A_{\pm}((|\theta|\lambda)^{2/3}(A-q^{1/3}({\vartheta_c})))$ corresponding to the saddle points $S_{\pm}$. We split the symbol of $V^M_N$ in two parts using $\chi_0+(1-\chi_0)=1$ and notice that on the support of $\chi_0$ the Airy function behaves as a symbol of order $2/3$.
Therefore we can write each integral $V^M_{N,a}(t,x,y)$ as a sum $V^M_{N,a}(t,x,y)=\sum_{\cveps\in \{0,\pm\}}V^{M,\cveps}_{N,a}(t,x,y)$ where for $\cveps\in\{\pm\}$, $V^{M,\cveps}_N$ has phase  $\tilde \Psi^{M,\cveps}_{N,a,a}(T,X,Y,\Sigma,A,\rho)$ while $V^{M,0}_{N,a}$ has phase function $\Psi_{N,a,a}(t,x,y,\Sigma, 0,A,\rho)$. 
We are left to prove that for ever $\cveps\in\{0,\pm\}$ the following holds
\begin{equation}\label{VNtangent}
\sum_{N\notin \mathcal{N}^{1,M,\cveps}_{\Psi}(T,X,Y),|N|\lesssim 1/\sqrt{a}}V^{M,\cveps}_{N,a}(T,X,Y)=O(h^{\infty}).
\end{equation}
For $\cveps\in\{\pm\}$ we act exactly like in the transverse case $4a\leq\gamma$ since on the support of $(1-\chi_0)$ we obtain two distinct saddle points $S_{\pm}$. For $\cveps=0$ we use the fact that we got rid of variable $S$ hence $\partial_S \Psi_{N,a,a}(T,X,Y,\Sigma,0,A,\rho)=0$ and act again as in the previous case, completing the proof of Proposition \ref{propcardoutN1} for the model Laplace $\Delta_M$.

We can now proceed with the proof of Proposition \ref{propcardoutN1} in the general situation.
Recall that the phase of $V_N$ is $\Phi_{N,a,\gamma}(t,x,y,\sigma,s,\alpha,\theta)=t\tau_{q}(\alpha,\theta)+\Phi(x,y,\theta,\alpha,\sigma)-\Phi(a,0,\theta,\alpha,s) -NhL(\alpha^{3/2}/h)$ (see \eqref{PhiNagamma}). Rescaling $s=\sqrt{\gamma}|\theta|^{1/3}S$, $\alpha=\gamma |\theta|^{2/3}A$ 
we obtain a new phase $\tilde \Phi_{N,a,\gamma}(t,x,y,\Sigma,S,A,\theta):=\Phi_{N,a,\gamma}(t,x,y,\sigma,s,\alpha,\theta)$ whose saddle points $S_{\pm}$ are solutions to the third equation in \eqref{desclambdaN} that we re-write here, using \eqref{partialsigmaAGamma},
\begin{equation}\label{thirdeqdesclambdaN}
S^2+\frac{a}{\gamma}q^{1/3}(\vartheta)\Big(1+\mathcal{H}_{j\geq 1}\Big)=A,
\end{equation}
where terms in $\mathcal{H}_{j}$ come with factors $O(\gamma^{j/2})$. As this equation is independent of $|\theta|$, we recover that $S_{\pm}=S_{\pm}(a,\gamma,\vartheta,A)$ are independent of $|\theta|$ or $N$.

Let first $4a\lesssim \gamma$; as $A$ is close to $1$ on the support of the symbol, critical points $S_{\pm}$ are non-degenerate and have opposite signs.
Stationary phase in $S$ yields $\tilde\Phi_{N,a,\gamma}\in\{\tilde\Phi^{\cveps}_{N,a,\gamma}\}$, where for $\cveps\in\{\pm\}$ we  set $\tilde\Phi^{\cveps}_{N,a,\gamma}(t,x,y,\Sigma,A,\theta):=\tilde \Phi_{N,a,\gamma}(t,x,y,\Sigma,S_{\cveps},A,\theta)$. Assume $t>0$. For $\cveps\in\{\pm\}$, let
\[
\mathcal{N}^{\cveps}(t,x,y):=\{(N\in\mathbb{Z}, \exists (\Sigma,A,\theta) \text{ such that } \nabla_{(\Sigma,A,\theta)}\tilde\Phi^{\cveps}_{N,a,\gamma}(t,x,y,\Sigma,A,\theta)=0\},
\]
then $\mathcal{N}(t,x,y)=\cap_{\cveps=\pm}\mathcal{N}^{\cveps}(t,x,y)$. Indeed, if $N\in \mathcal{N}(t,x,y)$, then there exists a critical point $\Sigma_c,S_c,A_c,\theta_c$ for the phase $\tilde\Phi_{N,a,\gamma}$; $S_c$ satisfies \eqref{thirdeqdesclambdaN}, hence $S_c\in \{S_{\pm}\}$. Therefore $(\Sigma_c,A_c,\theta_c)$ is a critical point for $\tilde \Phi_{N,a,\gamma}(t,x,y,\Sigma,S_{\pm}(a,\gamma,\vartheta,A),A,\theta)=\tilde\Phi^{\pm}_{N,a,\gamma}$. Conversely, if $\Sigma_{\pm},A_{\pm},\theta_{\pm}$ is a critical point for $\tilde\Phi^{\pm}_{N,a,\gamma}$, then $(\Sigma_{\pm},S_{\pm}(a,\gamma,\vartheta_{{\pm}},A_{\pm}), A_{\pm},\theta_{\pm})$ are critical points for $\tilde\Phi_{N,a,\gamma}$. 
For $\cveps\in \{\pm\}$, let $\mathcal{N}^{1,\cveps}_d(t,x,y)=\cup_{\mathcal{C}_{\gamma}(t,x,y)}\mathcal{N}^{\cveps}(t',x',y')$, 
then we have $\mathcal{N}^1_d=\cap_{\pm}\mathcal{N}^{1,\pm}_d$ and therefore, $(\mathcal{N}_1)^c=\cup_{\cveps\in \{\pm 1\}}(\mathcal{N}^{1,\cveps}_d)^c$. Proposition \ref{propcardoutN1} will follow from proving
\begin{equation}\label{WNdecaytoprove}
\forall \cveps\in\{\pm\}\,,\quad\quad\sum_{N\notin \cup_{\pm}\mathcal{N}^{1,\pm}_d(t,x,y),|N|\lesssim \frac{1}{\sqrt{\gamma}}}V^{\cveps}_{N,\gamma}(t,x,y)=O(h^{\infty})\,,
\end{equation}
where $V^{\cveps}_{N,\gamma}(t,x,y)$ has phase function $\tilde\Phi^{\cveps}_{N,a,\gamma}$ and symbol obtained from the one of $V_{N,\gamma}$ after stationary phase in $S$. 
For every $N$ the phase $\tilde \Phi^{\cveps}_{N,a,\gamma}$ has two critical points with respect to $\vartheta$, which are non-degenerate. Indeed, $\tilde \Phi^{\cveps}_{N,a,\gamma}$ depends on $N$ through $-NhL(|\theta|\lambda_{\gamma} A^{3/2})$ and therefore its dependence on $\vartheta$ comes only through $\tilde\Phi^{\cveps}_{0,a,\gamma}=\tilde \Phi_{0,a,\gamma}|_{S=S_{\cveps}}$. Moreover, $\nabla^2_{\vartheta}\tilde\Phi_{0,a,\gamma}=|y|O(1)$, $(\nabla^2_{\vartheta}\tilde\Phi_{0,a,\gamma})^{-1}=|y|^{-1}O(1)$ and using Lemma \ref{derriere} it follows that $|y|>c_0 |t|$ and in particular $y\neq 0$. Lemma \ref{lemomegac} provide the explicit form of the critical point ${\vartheta_c}$ and for $t>0$ we set ${\vartheta_c}(t,x,y,\Sigma,A)=\vartheta_-(t,x,y,\Sigma,S_{\cveps},A)$. Let $\rho:=|\theta|$ and set
\begin{gather*}
  \Psi^{\cveps}_{N,a,\gamma}(t,x,y,\Sigma,A,\rho):=\tilde \Phi^{\cveps}_{0,a,\gamma}(t,x,y,\Sigma,A,\rho\vartheta_{c})-\frac 43 N\rho \gamma^{3/2}A^{3/2}+NhB_L(\rho \lambda_{\gamma}A^{3/2})\,,\\
\mathcal{N}_{\Psi^{\cveps}}(t,x,y):=\{(N\in\mathbb{Z}, \exists (\Sigma,A,\rho) \text{ such that } \nabla_{(\Sigma,A,\rho)}\Psi^{\cveps}_{N,a,\gamma}(t,x,y,\Sigma,A,\rho)=0\}\,.
\end{gather*}
Remark again that $\mathcal{N}_{\Psi^{\cveps}}(t,x,y)=\mathcal{N}^{\cveps}(t,x,y)$: indeed, if $N$ is such that there exists a critical point $(\Sigma,A,\rho)$ for $\Psi^{\cveps}_{N,a,\gamma}$ at $(t,x,y)$, then $(\Sigma,A,\rho, {\vartheta_c}(t,x,y,\Sigma,A))$ is a critical point for $\tilde\Phi^{\cveps}_{N,a,\gamma}$ and the reverse statement holds as well. Let $\cveps\in\{\pm\}$ and $N\notin \mathcal{N}^{1,\cveps}_d$. Then $N\notin \mathcal{N}^{\cveps}(t',x',y')$ for all $(t',x',y')\in\mathcal{C}_{\gamma}(t,x,y)$ and therefore for all such $(t',x',y')$ we have
\begin{equation}\label{criPsiNout}
\nabla_{(\Sigma,A,\theta)}\tilde\Phi^{\cveps}_{N,a,\gamma}(t',x',y',\Sigma,A,\theta)\neq 0,\quad \forall (A,\Sigma,\rho).
\end{equation}
This translates into $N\notin \mathcal{N}_{\Psi^{\cveps}}(t',x',y')$ and $\nabla_{(\Sigma,A,\rho)}\Psi^{\cveps}_{N,a,\gamma}(t',x',y',\Sigma,A,\rho)\neq 0$
for all $(t',x',y')\in\mathcal{C}_{\gamma}(t,x,y)$ and all $\Sigma,A,\rho$ on the support of the symbol of $V^{\cveps}_{N,\gamma}(t,x,y)$. 
\begin{lemma}\label{borneinfgradcg}
There exists a uniform constant $c>0$ such that, for all $N\notin\cup_{\cveps\in \{\pm\}}\mathcal{N}^{1,\cveps}_{\tilde\Psi}(t,x,y)$,
\begin{equation}
  \label{eq:4}
|\nabla_{(\Sigma,A,\rho)}\Psi^{\cveps}_{N,a,\gamma}(t,x,y,\Sigma,A,\rho)|\geq c \gamma^{3/2}
\end{equation}
for all $(\Sigma, A,\rho)$ on the support of the symbol of $V^{\cveps}_{N,\gamma}$.
\end{lemma}
This lemma allows to perform non stationary phase in $V^{\cveps}_{N,\gamma}(t,x,y)$ with
large parameter $\lambda_{\gamma}$ and conclude that \eqref{WNdecaytoprove} holds true.  
\begin{rmq}\label{rmqICgen}
We need to define the integral curves of $\tilde\Phi^{\cveps}_{N,a,\gamma}$ :  recall that \eqref{desclambdaN} allows to parametrize the Lagrangian $\Lambda_N$; the projection $\pi_N(\Lambda_N)$   can be parametrized by $(\rho,\vartheta,\Sigma, A)$, with $\rho\sim 1$, $\vartheta\in\mathbb{S}^{d-2}$, $A,\Sigma,S$ on the support of the symbol of $V_{N,\gamma}$ and $S\in \{S_{\pm}(a,\gamma,\vartheta,A)\}$. We define the integral curves $\tilde t_{\cveps}=\tilde t_{\cveps}(\Sigma,A,\rho, \vartheta, a,\gamma)$, $\tilde x_{\cveps} =\tilde x_{\cveps}(\Sigma,A,\rho,\vartheta,a, \gamma)$, $\tilde y_{\cveps}=\tilde y_{\cveps}(\Sigma,A,\rho,\vartheta,a,\gamma)$ (as we did in Remark \ref{rmqIC} in the model case) such that \eqref{desclambdaN} holds at $(\tilde t_{\cveps},\tilde x_{\cveps},\tilde y_{\cveps})$. Therefore, at a given point $(\rho,\vartheta,\Sigma,A)$, $(\tilde t_{\cveps},\tilde x_{\cveps},\tilde y_{\cveps})$ is defined such that \eqref{desclambdaN} holds with $(t,x,y)=(\tilde t_{\cveps},\tilde x_{\cveps},\tilde y_{\cveps})$, $S=S_{\cveps}(a,\gamma,\vartheta,A)$ and $\theta=\rho \vartheta$. 
Using Lemma \ref{lemomegac} with $\vartheta\in \{\vartheta_{\pm}\}$ and $(t,x,y)$ replaced by $(\tilde t_{\cveps},\tilde x_{\cveps},\tilde y_{\cveps})$, we obtain 
$(\tilde y_{\cveps}+\nabla B_0(\tilde y_{\cveps},\vartheta))\vartheta=-|\tilde y_{\cveps}+\nabla B_0(\tilde y_{\cveps},\vartheta)|+\gamma^{2}O(\frac{|\tilde y_{\cveps}|}{\tilde t_{\cveps}})$. 
\end{rmq}
\begin{proof}
Let $(\Sigma,A,\rho)$ be a point on the support of the symbol of $V^{\cveps}_{N,\gamma}(t,x,y)$, let $\vartheta\in \mathbb{S}^{d-2}$ and define $S_{\cveps}(a,\gamma,\vartheta,A)$ to be the solution to the third equation in \eqref{desclambdaN} (where $\vartheta$ is replaced by $\vartheta$) with sign $\cveps$. Let $(\tilde t_{\cveps},\tilde x_{\cveps},\tilde y_{\cveps}):=(\tilde t_{\cveps},\tilde x_{\cveps},\tilde y_{\cveps})(\Sigma,A,\rho,\vartheta,a,\gamma)$ denote the corresponding point on the integral curve (as in Remark \ref{rmqICgen}); then, by construction of the integral curves,
\[
\nabla_{(A,\Sigma,\rho)}\tilde \Psi^{\cveps}_{N,a,\gamma}(\tilde t_{\cveps},\tilde x_{\cveps},\tilde y_{\cveps},\Sigma, A,\rho)=\nabla_{(A,\Sigma,\rho)}\tilde \Phi^{\cveps}_{N,a,\gamma}((\tilde t_{\cveps},\tilde x_{\cveps},\tilde y_{\cveps})(\Sigma,\cdots,\gamma),\Sigma, A,\rho\vartheta)=0\in\mathbb{R}^3\,.
\]
Using the explicit form of \eqref{desclambdaN}, this translates into
\begin{gather*}
  \frac{\gamma q^{2/3}(\vartheta)}{2\sqrt{1+\gamma A q^{2/3}(\vartheta)}}(\tilde t_{\cveps}-B_2(\tilde y_{\cveps},\vartheta))+\gamma^{\frac 32}(\Sigma-S_{\cveps}(a,\gamma,\vartheta,A))=2N \gamma^{\frac 32}A^{1/2}(1-\frac 34 B'_L(\rho \lambda_{\gamma}A^{\frac 32}))+O(\gamma^2)\,,\\
\tilde x_{\cveps}(1+\Lp(\tilde y_{\cveps},\vartheta)+O(\sqrt{\gamma}))=\gamma(A-\Sigma^2)\,,\\
\begin{multlined}
\Big((\tilde y_{\cveps}+\nabla B_0(\tilde y_{\cveps},\vartheta))\vartheta+\tilde t_{\cveps}\Big)-(1-\tau_{q}(A,\vartheta))(\tilde t_{\cveps}-B_2(\tilde y_{\cveps},\vartheta))\\
+\gamma^{\frac 32}\Big[\frac{\Sigma^3}{3}+\Sigma\Big(\frac{\tilde x_{\cveps}}{\gamma}q^{1/3}(\vartheta)(1+\Lp(\tilde y_{\cveps},\vartheta)+O(\sqrt{\gamma}))-A\Big)-\frac{S_{\cveps}^3}{3}-S_{\cveps}\Big(\frac{a}{\gamma}q^{1/3}(\vartheta)(1+O(\sqrt{\gamma}))-A\Big)\Big]\\
=\frac 43 \gamma^{\frac 32}N A^{\frac 32}(1-\frac 34 B'_L(\rho\lambda_{\gamma}A^{\frac 32}))+O(\gamma^2)\,,
\end{multlined}
\end{gather*}
where the last equation is obtained using the second to last equation in the system \eqref{desclambdaN}, as satisfied by $(\tilde t_{\cveps},\tilde x_{\cveps},\tilde y_{\cveps})(\Sigma,A,\rho,\vartheta,a,\gamma),\Sigma,S_{\cveps}(a,\gamma,\vartheta,A),A,\rho,\vartheta$. Let $(t',x',y')\in\mathcal{C}_{\gamma}(t,x,y)$ and let $\vartheta'_c:={\vartheta_c}(t',x',y',\Sigma,A)$ be the critical point of $\tilde \Phi^{\cveps}_{N,a,\gamma}(t',x',y',\Sigma,A,\rho\vartheta)$, then, for any $\vartheta\in\mathbb{S}^{d-2}$,
\begin{gather} \quad\begin{multlined}[t]\label{A}
  \partial_A\Psi^{\cveps}_{N,a,\gamma}=\frac{\gamma q^{2/3}(\vartheta'_c)}{2\sqrt{1+\gamma q^{2/3}(\vartheta'_c)}}(t'-B_2( y',\vartheta'_c))\\
  {}-\frac{\gamma A q^{2/3}(\vartheta)}{2\sqrt{1+\gamma A q^{2/3}(\vartheta)}}(\tilde t_{\cveps}-B_2(\tilde y_{\cveps},\vartheta))-\gamma^{3/2}(S_{\cveps}(a,\gamma,\vartheta'_c,A)-S_{\cveps}(a,\gamma,\vartheta,A))+O(\gamma^2)\,,
\end{multlined}\\
\label{Sigma}
\partial_{\Sigma}\Psi^{\cveps}_{N,a,\gamma}=\gamma^{\frac 32}\Bigl(q^{1/3}(\vartheta'_c)\frac{x'}{\gamma}(1+\Lp(y',\vartheta'_c)+O({\gamma}^{\frac 1 2})\Bigr)-q^{1/3}(\vartheta)\frac{\tilde x_{\cveps}}{\gamma}\bigr(1+\Lp(\tilde y_{\cveps},\vartheta)+O({\gamma}^{\frac 1 2}))\bigl)\,,
\end{gather}
\begin{multline}
\partial_{\rho}\Psi^{\cveps}_{N,a,\gamma}=((y'+\nabla B_0(y',\vartheta_c'))\cdot \vartheta_c'+t')
-((\tilde y_{\cveps}+\nabla B_0(\tilde y_{\cveps},\vartheta))\cdot\vartheta+\tilde t_{\cveps})\\
\label{rho}{}+\frac{\gamma A q^{2/3}(\vartheta'_c)}{2\sqrt{1+\gamma A q^{2/3}(\vartheta'_c)}}(t'-B_2( y',\vartheta'_c))-\frac{\gamma A q^{2/3}(\vartheta)}{2\sqrt{1+\gamma A q^{2/3}(\vartheta)}}(\tilde t_{\cveps}-B_2(\tilde y_{\cveps},\vartheta))\\
{}+\gamma^{3/2}\Sigma\Big(\frac{x'}{\gamma}q^{1/3}(\vartheta'_c)(1+\Lp(y',\vartheta'_c)+O(\sqrt{\gamma}))-\frac{\tilde x_{\cveps}}{\gamma}q^{1/3}(\vartheta)(1+\Lp(\tilde y_{\cveps},\vartheta)+O(\sqrt{\gamma}))\Big)\\
{}-\frac 23\gamma^{3/2}\Big(S^3_{\cveps}(a,\gamma,\vartheta'_c,A)-S^3_{\cveps}(a,\gamma,\vartheta,A)\Big)+O(\gamma^2)\,,
\end{multline}
where in the last line we used $\frac{S_{\cveps}^3}{3}+S_{\cveps}\Big(\frac{a}{\gamma}q^{1/3}(\vartheta)(1+O(\sqrt{\gamma}))-A\Big)=-\frac 23 S^{3/2}_{\cveps}+O(\sqrt{\gamma})$.
\begin{rmq}\label{rmqgamsmalldiff}
All terms $O(\gamma^2)$ are smooth functions of differences: $\vartheta'_c-\vartheta$, $q^{2/3}(\vartheta'_c)(t'-B_2(y',\vartheta'_c))-q^{2/3}(\vartheta)(\tilde t_{\cveps}-B_2(\tilde y_{\cveps},\vartheta))$, $q^{1/3}(\vartheta'_c)\frac{x'}{\gamma}-q^{1/3}(\vartheta)\frac{\tilde x_{\cveps}}{\gamma}$ as well as $(f(y',\vartheta'_c)-f(\tilde y_{\cveps},\vartheta))$ where $f$ is either $B_2(y,\vartheta)$ or $\nabla B_2(y,\vartheta)$ and such that $f(y,\vartheta)=O(|y|^2)$ or $f\in \{B_{2k}(y,\vartheta),\nabla B_{2k}(y,\vartheta)\}$ is such that $f(y,\vartheta)=O(|y|)$ and coefficients $O(\gamma^{2+k})$ instead of $O(\gamma^2)$, or $f\in\{\alpha_j(y,\vartheta),\gamma_j(y,\vartheta)\}$ (where $\alpha_j, \gamma_j$ are coefficients of homogeneous terms of degree $2j$ in $A_{\Gamma}$) with coefficients $O(\gamma^{2+j/2})$. 
\end{rmq}
We then obtain and upper bound for $\nabla_{A,\Sigma,\rho}\tilde\Psi^{\cveps}_{N,a,\gamma}$ (similar to \eqref{inegtildePsi})
 \begin{multline}\label{inegtildePsigen} 
\gamma^{-3/2}(|\partial_{A}\tilde\Psi^{\cveps}_{N,a,\gamma}|+|\partial_{\Sigma}\tilde\Psi^{\cveps}_{N,a,\gamma}|+|\partial_{\rho}\tilde\Psi^{\cveps}_{N,a,\gamma}|)(t',x',y',\Sigma,A,\rho)\\
\leq C\Big(\Big|q^{2/3}(\vartheta'_{c})\frac{(t'-B_2( y',\vartheta'_c))}{\sqrt{\gamma}}-q^{2/3}(\vartheta)\frac{(\tilde t_{\cveps}-B_2( \tilde y_{\cveps},\vartheta))}{\sqrt{\gamma}}\Big|\\
+|q^{1/3}(\vartheta'_c)-q^{1/3}(\vartheta)|(1+O(\sqrt{\gamma}))+|q^{1/3}(\vartheta'_c)\frac{x'}{\gamma}(1+\Lp(y',\vartheta'_c))-q^{1/3}(\vartheta)\frac{\tilde x_{\cveps}}{\gamma}(1+\Lp(\tilde y_{\cveps},\vartheta))|\\
+\Big|\frac{(y'+\nabla B_0(y',\vartheta_c'))\cdot\vartheta_c'+t'}{\gamma^{3/2}}
-\frac{(\tilde y_{\cveps}+\nabla B_0(\tilde y_{\cveps},\vartheta))\cdot\vartheta+\tilde t_{\cveps}}{\gamma^{3/2}}\Big|+O(\sqrt{\gamma})|y'-\tilde y_{\cveps}|\Big),
\end{multline}
where the term $|q^{1/3}(\vartheta'_c)-q^{1/3}(\vartheta)|$ comes from $\Big(S_{\cveps}(a,\gamma,\vartheta'_c,A)-S_{\cveps}(a,\gamma,\vartheta,A)\Big)$ and the term $O(\sqrt{\gamma})|y'-\tilde y_{\cveps}|$ comes from differences involving $B_{2k}$, $\alpha_j$, $\gamma_j$. The constant $C>0$ depends on $q$.
\begin{lemma}
Let $(t,x,y)$ be fixed, let $(\Sigma,A,\rho,\vartheta)$ belong to the support of the symbol of $V^{\cveps}_{N,\gamma}$, then $(\tilde t_{\cveps},\tilde x_{\cveps},\tilde y_{\cveps})(\Sigma,A,\rho,\vartheta,a,\gamma)\notin \mathcal{C}_{\gamma}(t,x,y)$.
\end{lemma}
\begin{proof}
If not, then taking $(t',x',y'):=(\tilde t_{\cveps},\tilde x_{\cveps},\tilde y_{\cveps})(\Sigma,A,\rho,\vartheta)$ we have $\vartheta_c'=\vartheta$, $(t',x',y')\in\mathcal{C}_{\gamma}$ and $\nabla_{A,\Sigma,\rho}\tilde\Psi^{\cveps}_{N,a,\gamma}(t',x',y',\Sigma,A,\rho)=0$, which cannot happen for $N\notin \mathcal{N}_{\Psi^{\cveps}}(t',x',y')$.
\end{proof} 
We now prove \eqref{eq:4}: from the previous lemma, for all $A,\Sigma,\rho$ we have $(\tilde t_{\cveps},\tilde x_{\cveps},\tilde y_{\cveps})(\Sigma,A,\rho,{\vartheta_c},a,\gamma)\notin \mathcal{C}_{\gamma}(t,x,y)$, where ${\vartheta_c}={\vartheta_c}(t,x,y,\Sigma,A)$ is the critical point of $\tilde \Phi^{\cveps}_{N,a,\gamma}(t,x,y,\Sigma,A,\rho\vartheta)$ given in Lemma \ref{lemomegac} (with ${\vartheta_c}=\vartheta_-(t,x,y,\Sigma,A)$ for $t>0$).
This yields
\begin{multline}\label{lowerinegtildePsigentilde} 
 \Big|\frac{t-B_2( y,\tilde \vartheta(t,y))}{\sqrt{\gamma}}-\frac{\tilde t_{\cveps}-B_2( \tilde y_{\cveps},\tilde \vartheta(\tilde t_{\cveps},\tilde y_{\cveps})}{\sqrt{\gamma}}\Big|
+|\frac{x}{\gamma}(1+\Lp(y,\tilde\vartheta(t,y)))-\frac{\tilde x_{\cveps}}{\gamma}(1+\Lp(\tilde y_{\cveps},\tilde \vartheta(\tilde t_{\cveps},\tilde y_{\cveps})))|\\
+\Big|\frac{|y+\nabla B_0(y,\tilde \vartheta(t,y))|-t}{\gamma^{3/2}}
-\frac{|\tilde y_{\cveps}+\nabla B_0(\tilde y_{\cveps},\tilde \vartheta(\tilde t_{\cveps},\tilde y_{\cveps}))|-\tilde t_{\cveps}}{\gamma^{3/2}}\Big|> r_0.
\end{multline}
As \eqref{A},\eqref{Sigma} and \eqref{rho} do hold for any $\vartheta\in\mathbb{S}^{d-2}$, taking $(t',x',y')=(t,x,y)$ and $\vartheta:={\vartheta_c}(=\vartheta_c')={\vartheta_c}(t,x,y,\Sigma,A)$ yields $S_{\cveps}(a,\gamma,\vartheta'_c,A)=S_{\cveps}(a,\gamma,\vartheta,A)$ and therefore, for all $(A,\Sigma,\rho)$ on the support of $V^{\cveps}_{N,\gamma}(t,x,y)$ and $(\tilde t_{\cveps},\tilde x_{\cveps},\tilde y_{\cveps})=(\tilde t_{\cveps},\tilde x_{\cveps},\tilde y_{\cveps})(\Sigma,A,\rho,{\vartheta_c})$ we have
\begin{align*}
  \gamma^{-3/2}|\partial_{A}\tilde\Psi^{\cveps}_{N,a,\gamma}|(t,x,y,\Sigma,A,\rho) & =\frac 12 q^{2/3}(\vartheta_{c})\Big|\frac{(t-B_2( y,{\vartheta_c}))}{\sqrt{\gamma}}-\frac{(\tilde t_{\cveps}-B_2( \tilde y_{\cveps},{\vartheta_c}))}{\sqrt{\gamma}}\Big|+O(\sqrt{\gamma})\,,\\
\gamma^{-3/2}|\partial_{\Sigma}\tilde\Psi^{\cveps}_{N,a,\gamma}|(t,x,y,\Sigma,A,\rho) & =q^{1/3}({\vartheta_c})|\frac{x}{\gamma}(1+\Lp(y,{\vartheta_c}))-\frac{\tilde x_{\cveps}}{\gamma}(1+\Lp(\tilde y_{\cveps},{\vartheta_c}))|+O(\sqrt{\gamma})\,,\\
\gamma^{-3/2}|\partial_{\rho}\tilde\Psi^{\cveps}_{N,a,\gamma}|(t,x,y,\Sigma,A,\rho) & \geq 
\begin{multlined}[t] \Big|\frac{(y+\nabla B_0(y,{\vartheta_c}))\cdot{\vartheta_c}+t)}{\gamma^{3/2}}
-\frac{(\tilde y_{\cveps}+\nabla B_0(\tilde y_{\cveps},{\vartheta_c}))\cdot{\vartheta_c}+\tilde t_{\cveps})}{\gamma^{3/2}}\Big|\\
-\frac 12 q^{2/3}(\vartheta_{c})\Big|\frac{(t-B_2( y,{\vartheta_c}))}{\sqrt{\gamma}}-\frac{(\tilde t_{\cveps}-B_2( \tilde y_{\cveps},{\vartheta_c}))}{\sqrt{\gamma}}\Big|\\
-q^{1/3}({\vartheta_c})|\frac{x}{\gamma}(1+\Lp(y,{\vartheta_c}))-\frac{\tilde x_{\cveps}}{\gamma}(1+\Lp(\tilde y_{\cveps},{\vartheta_c}))|-O(\sqrt{\gamma}),
\end{multlined}
\end{align*}
and therefore there exists a uniform constant $C>0$ (depending only on $q$) such that
\begin{multline}\label{lowerinegtildePsigen} 
C\gamma^{-3/2}(|\partial_{A}\tilde\Psi^{\cveps}_{N,a,\gamma}|+|\partial_{\Sigma}\tilde\Psi^{\cveps}_{N,a,\gamma}|+|\partial_{\rho}\tilde\Psi^{\cveps}_{N,a,\gamma}|)(t,x,y,\Sigma,A,\rho)\\
\geq \Big[\Big|\frac{(t-B_2( y,{\vartheta_c}))}{\sqrt{\gamma}}-\frac{(\tilde t_{\cveps}-B_2( \tilde y_{\cveps},{\vartheta_c}))}{\sqrt{\gamma}}\Big|
+|\frac{x}{\gamma}(1+\Lp(y,{\vartheta_c}))-\frac{\tilde x_{\cveps}}{\gamma}(1+\Lp(\tilde y_{\cveps},{\vartheta_c}))|\\
+\Big|\frac{(y+\nabla B_0(y,{\vartheta_c}))\cdot{\vartheta_c}+t)}{\gamma^{3/2}}
-\frac{(\tilde y_{\cveps}+\nabla B_0(\tilde y_{\cveps},{\vartheta_c}))\cdot{\vartheta_c}+\tilde t_{\cveps})}{\gamma^{3/2}}\Big|-O(\sqrt{\gamma})\Big].
\end{multline}
As $(y+\nabla B_0(y,{\vartheta_c}))\cdot{\vartheta_c}=-|y+\nabla B_0(y,{\vartheta_c})|+O(\gamma^2)$ by Lemma \ref{lemomegac}, $(\tilde y_{\cveps}+\nabla B_0(\tilde y_{\cveps},{\vartheta_c}))\cdot{\vartheta_c}=-|\tilde y_{\cveps}+\nabla B_0(\tilde y_{\cveps},{\vartheta_c})|+O(\gamma^2)$ by construction and Lemma \ref{lemomegac}, the last line in \eqref{lowerinegtildePsigen} may be rewritten as
\[
\Big|\frac{|y+\nabla B_0(y,{\vartheta_c})|-t}{\gamma^{3/2}}
-\frac{|\tilde y_{\cveps}+\nabla B_0(\tilde y_{\cveps},{\vartheta_c})|-\tilde t_{\cveps}}{\gamma^{3/2}}\Big|-O(\sqrt{\gamma}).
\]
As $(\tilde t_{\cveps},\tilde x_{\cveps},\tilde y_{\cveps})=(\tilde t_{\cveps},\tilde x_{\cveps},\tilde y_{\cveps})(\Sigma,A,\rho,{\vartheta_c})$ is the integral curve associated to $(\Sigma,A,\rho,{\vartheta_c}(t,x,y,\Sigma,A))$, it follows from the definition of $\tilde\vartheta(t,y)$ in \eqref{tildeomega} and Lemma \ref{lemomegcrit} (applied to ${\vartheta_c}$ and $\tilde\vartheta(\tilde t_{\cveps}, \tilde y_{\cveps})$) that 
\begin{equation}\label{inftildeCI}
|\tilde\vartheta(\tilde t_{\cveps}, \tilde y_{\cveps})-{\vartheta_c}(t,x,y,\Sigma,A)|=O(\gamma^{3/2}/\tilde t_{\cveps})\,,
\end{equation}
which further yields
\[
|\tilde y_{\cveps}+\nabla B_0(\tilde y_{\cveps},{\vartheta_c})|-|\tilde y_{\cveps}+\nabla B_0(\tilde y_{\cveps},\tilde \vartheta(\tilde t_{\cveps},\tilde y_{\cveps}))|=O(\frac{|\tilde y_{\cveps}|^2}{\tilde t_{\cveps}}\gamma^{3/2})\,.
\]
We now consider the difference between ${\vartheta_c}$ and $\tilde\vartheta(t,y)$ : the assumption of Proposition \ref{propcardoutN1} is $\mathcal{N}(t,x,y)\neq \emptyset$, hence there exists $N_0\in\mathcal{N}(t,x,y)$ for which $\tilde\Phi_{N_0,a,\gamma}$ is stationary with respect to $A,\rho,\vartheta$ and from Lemma \ref{lemdetails} it follows that \eqref{desclambdaNmaincontribymod} must hold. Then $\tilde\vartheta(t,y)$ is an approximation modulo $O(\gamma^{3/2}/t)$ of the critical point $\vartheta_{}$ (when we consider all variables $(A,\rho,\vartheta_{})$) satisfying \eqref{eqNcontYmodgennewnew}. From the formulas \eqref{eqNcontYmodgen} for ${\vartheta_c}$ and \eqref{eqNcontYmodgennewnew} for $\vartheta_{c}^{\sharp}$, we always have $|{\vartheta_c}-{\vartheta_c}^{\sharp}|\lesssim \gamma$, but this is not small enough to conclude. If $|{\vartheta_c}-{\vartheta_c}^{\sharp}|\lesssim \gamma^{3/2}/t$, then, using \eqref{lowerinegtildePsigentilde}, \eqref{lowerinegtildePsigen}, \eqref{inftildeCI}, we obtain 
\[
|(y+\nabla B_0(y,{\vartheta_c}))\cdot{\vartheta_c}-(y+\nabla B_0(y,\tilde\vartheta(t,y)))\cdot \tilde\vartheta(t,y)|= O(|y|^2\gamma^{3/2}/t)=O(|y|\gamma^{3/2}).
\]
Suppose that, with ${\vartheta_c}^{\sharp}$ given in \eqref{eqNcontYmodgennewnew} from Lemma \ref{lemdetails}, we have 
$\gamma\gtrsim |{\vartheta_c}-{\vartheta_c}^{\sharp}|\gtrsim \gamma^{3/2}/t$ such that
\begin{equation}\label{assumptomegactilde}
|(y+\nabla B_0(y,{\vartheta_c}))\cdot{\vartheta_c}-(y+\nabla B_0(y,{\vartheta_c}^{\sharp}))\cdot{\vartheta_c}^{\sharp}|\geq r_0\gamma^{3/2}\,.
\end{equation}
Let $C_{q}\geq 2 $ be such that $\|q(\vartheta)\|_{\infty}\leq C_{q}$. If the first line in \eqref{lowerinegtildePsigentilde} is bounded from below by $r_{0}/(10C_{q})$, then using $|{\vartheta_c}-\tilde\vartheta(t,y)|\leq |{\vartheta_c}-{\vartheta_c}^{\sharp}|+|{\vartheta_c}^{\sharp}-\tilde \vartheta(t,y)|\lesssim \gamma$, the second line in \eqref{lowerinegtildePsigen} is bounded from below by $r_{0}/(20C_{q})$ (if $\sqrt{\gamma}$ is small compared to $r_0$) and we are done. Therefore we are reduced to considering the first line in \eqref{lowerinegtildePsigentilde} bounded from above by $r_{0}/(10C_{q})$, which is the same as assuming that $|\partial_A\tilde\Psi^{\cveps}(t,x,y,\Sigma,A,\rho)|\leq \gamma^{3/2}{r_0}/{10}$ and $|\partial_{\Sigma}\tilde\Psi^{\cveps}_{N,a,\gamma}(t,x,y,\Sigma,A,\rho)|\leq \gamma^{3/2}{r_0}/{10}$. 
We prove that, if, moreover, \eqref{assumptomegactilde} holds, then $\gamma^{-3/2}|\partial_{\rho}\tilde\Psi^{\cveps}_{N,a,\gamma}(t,x,y,\Sigma,A,\rho)|\geq r_0/5$. As $\gamma^{-3/2}|\partial_A\tilde\Psi^{\cveps}(t,x,y,\Sigma,A,\rho)|\leq \frac{r_0}{10}$, 
we find
\begin{equation}
\Big|\frac{q^{2/3}({\vartheta_c})}{2}\frac{(t-B_2(y,{\vartheta_c}))}{\sqrt{\gamma}}+(\Sigma-S_{\cveps}(a,\gamma,{\vartheta_c},A))+O(\sqrt{\gamma})-2NA^{1/2}(1-\frac 34 B'_L(\rho \lambda_{\gamma}A^{3/2}))\Big|\leq r_0/10\,.
\end{equation}
Recall that 
\begin{multline}
\gamma^{-3/2}\partial_{\rho}\tilde\Psi^{\cveps}_{N,a,\gamma}=\gamma^{-3/2}\Big((y+\nabla B_0(y,{\vartheta_c}))\cdot{\vartheta_c}+t\Big)-\frac{Aq^{2/3}({\vartheta_c})}{2}\frac{(t-B_2(y,{\vartheta_c}))}{\sqrt{\gamma}}\\
+\Big(\Sigma^3/3+\Sigma(\frac{x}{\gamma}q^{1/3}({\vartheta_c})(1+\Lp(y,{\vartheta_c}))+\frac 23 S_{\cveps}(a,\gamma,{\vartheta_c},A)\Big)-\frac 43N A^{3/2}(1-B'_L(\rho\lambda_{\gamma}A^{3/2}))+O(\sqrt{\gamma})
\end{multline}
and therefore, by substitution in the $NA^{3/2}$ term,
\begin{multline}
\left|\gamma^{-3/2}\partial_{\rho}\tilde\Psi^{\cveps}_{N,a,\gamma}-\left(\gamma^{-3/2}\Big((y+\nabla B_0(y,{\vartheta_c}))\cdot{\vartheta_c}+t\Big)-\frac{Aq^{2/3}({\vartheta_c})}{6}\frac{(t-B_2(y,{\vartheta_c}))}{\sqrt{\gamma}}\right.\right.\\{}+
\left.\left.\Big(\Sigma^3/3+\Sigma\frac{x}{\gamma}q^{1/3}({\vartheta_c})(1+\Lp(y,{\vartheta_c}))+\frac 23 S_{\cveps}(a,\gamma,{\vartheta_c},A)-\frac 23 (\Sigma-S_{\cveps}(a,\gamma,{\vartheta_c},A))\Big)\right)\right| \leq r_0/10\,.
\end{multline}
With ${\vartheta_c}^{\sharp}$ given by \eqref{eqNcontYmodgennewnew} in Lemma \ref{lemdetails} (recall that it is independent of $N$), we have
\begin{multline}
\gamma^{-3/2}\Big((y+\nabla B_0(y,{\vartheta_c}^{\sharp}))\cdot{\vartheta_c}^{\sharp}+t\Big)-\frac{A  q^{2/3}({\vartheta_c}^{\sharp})}{6}\frac{(t-B_2(y,{\vartheta_c}^{\sharp}))}{\sqrt{\gamma}}\\
+\Big(\Sigma^3/3+\Sigma\frac{x}{\gamma}q^{1/3}({\vartheta_c}^{\sharp})(1+\Lp(y,{\vartheta_c}))+\frac 23 S_{\cveps}(a,\gamma,{\vartheta_c}^{\sharp},A)-\frac 23 (\Sigma-S_{\cveps}(a,\gamma,{\vartheta_c}^{\sharp},A))\Big)+O(\sqrt{\gamma})=0.
\end{multline}
Taking the difference between the last two equations and using that $|{\vartheta_c}-{\vartheta_c}^{\sharp}|\lesssim \gamma$ and that differences always provide functions of ${\vartheta_c}-{\vartheta_c}^{\sharp}$, we get, using \eqref{assumptomegactilde}
\[
  \gamma^{-3/2}|\partial_{\rho}\tilde\Psi^{\cveps}_{N,a,\gamma}|=\gamma^{-3/2}\Big|(y+\nabla B_0(y,{\vartheta_c}))\cdot{\vartheta_c}-(y+\nabla B_0(y,{\vartheta_c}^{\sharp}))\cdot{\vartheta_c}^{\sharp}\Big|-O(\sqrt \gamma)- r_0/10
 \geq \frac{r_0}{5}\,.
\]
Therefore we always have $\Big|\nabla_{(\Sigma,A,\rho)}\tilde \Psi^{\cveps}_{N,a,\gamma}\Big|\gtrsim r_0/5$, which completes the proof in the case $4a\lesssim \gamma$.\end{proof}
Let now $\gamma\sim a$. We write $\Phi_{N,a,a}:=\Phi_{N,a,\gamma\sim a}$ and rescale variables as follows $s=\sqrt{a}|\theta|^{1/3}S$, $\sigma\sqrt{a}|\theta|^{1/3}\Sigma$, $\alpha=a|\theta|^{2/3}A$ and let $\lambda=\lambda_a=\frac{a^{3/2}}{h}$. We define
\begin{equation}\label{deftildePhi}
\tilde \Phi_{N,a,a}(t,x,y,\Sigma,S,A,\theta):=\Phi_{N,a,a}(t,x,y,\sqrt{a}|\theta|^{1/3}\Sigma,\sqrt{a}|\theta|^{1/3}S,a|\theta|^{2/3}A,\theta).
\end{equation}
Let $s_0(a,\theta,\alpha)$ be the unique solution to $\partial^2_{s}\Phi(a,0,\theta,\alpha,s)=0$ and let $s_{\pm}(a,\theta,\alpha)$ be the critical points of $\Phi(a,0,\theta,\alpha)$ (see Lemma \ref{lemphaseG} from the Appendix for details on critical points $s_{\pm}$ of $\Phi(x,y,\theta,\alpha,\sigma)$), then after rescaling variables we obtain at most two critical points such that
\[
S_{\pm}(a,\vartheta,A)-S_0(a,\vartheta,A)= \pm \sqrt{\frac{\zeta(a,0,\vartheta,a A)}{a}}(1+O(\sqrt{\zeta(a,0,\vartheta,a A)})),
\] 
where $\zeta(a,0,\vartheta,a A)/a=A-e_0(a,0,\vartheta,a A)$ is the phase function introduced in Theorem \ref{thmMelrose}, with $e_0$ elliptic and close to $1$. In this case $\zeta(a,0,\vartheta,a A)/a$ is close to $0$ and $S_{\pm}$ are real only for $\zeta(a,0,\vartheta,a A)\geq 0$. We now repeat the argument from the model case : let $\chi_0\in C^{\infty}$ be a smooth cutoff, equal to $1$ on $[0,\infty]$ and equal to $0$ on $[-\infty,-2]$. Then $(\chi_0Ai)(-(|\theta|\lambda)^{2/3}\zeta(a,0,\vartheta,aA))$ is a symbol of order $2/3$ supported for values $(|\theta|\lambda)^{2/3}\zeta(a,0,\vartheta,aA)\leq 2$ and $(1-\chi_0)Ai(-(|\theta|\lambda)^{2/3}\zeta(\cdots)))$ is supported on $A\geq e_0(a,0,\vartheta,aA)$ with value $1$ on $(|\theta|\lambda)^{2/3}\zeta(a,0,\vartheta,aA)\geq 2$. On the support of $(1-\chi_0)$ the Airy function may be written as a sum of two contributions $A_{\pm}((|\theta|\lambda)^{2/3}\zeta(a,0,\vartheta,aA)))$ corresponding to critical points $S_{\pm}$. We split the symbol of $V_N$ in two parts using $\chi_0+(1-\chi_0)=1$ and on the support of $\chi_0$ the Airy function behaves as a symbol of order $2/3$. Therefore we write each integral $V_{N,a}(t,x,y)$ as a sum $V_{N,a}(t,x,y)=\sum_{\cveps\in \{0,\pm\}}V^{\cveps}_{N,a}(t,x,y)$ where for $\cveps\in\{\pm\}$, $V^{\cveps}_N$ has phase $\tilde\Phi_{N,a,a}(t,x,y,\Sigma, S_{\cveps},A,\theta)$ while $V^0_{N,a}$ has phase $\tilde\Phi_{N,a,a}(t,x,y,\Sigma, S_0(a,\vartheta,A),A,\theta)$. It remains to prove that, for each $\cveps\in\{0,\pm\}$,
\begin{equation}\label{VNtangentbis}
\sum_{N\notin \mathcal{N}^{1}_d(t,x,y),|N|\lesssim 1/\sqrt{a}}V^{\cveps}_{N,a}(t,x,y)=O(h^{\infty}).
\end{equation}
For $\cveps\in\{\pm\}$ we act exactly like in the transverse case $4a\leq\gamma$, as on the support of $(1-\chi_0)$ we have two separate critical points $S_{\pm}$. For $\cveps=0$ we use that we got rid of variable $S$ hence $\partial_S\tilde \Phi_{N,a,a}(t,x,y,\Sigma, S_0(a,\vartheta,A),A,\theta)$ and act again as in the previous case, finally completing the proof of Proposition \ref{propcardoutN1}.
\end{proof}
In the last part of this section we prove Theorem \ref{dispintermediaire} : writing $\Prond_{h,a}$ as the sum over $\gamma$, 
we evaluate each $\Prond_{h,a,\gamma}$ and then sum up in $a\lesssim \gamma\ll 1$. We deal separately with the cases $\gamma>4a$ and $\gamma\sim a$ when use the notation $\Prond_{h,a,a}$ for $\Prond_{h,a,\gamma\sim a}$ and $V_{N,a}$ for $V_{N,\gamma\sim a}$. From \eqref{eq:newVN} and \eqref{eq:newVNgam} we have
\[
V_N=\sum_{a\lesssim\gamma\ll 1}V_{N,\gamma}=V_{N,a}+\sum_{4a\leq\gamma}V_{N,\gamma}\,.
\]
\begin{rem}
  In \eqref{eq:newVNgam} we set $\omega=h^{-2/3}\alpha$, and then $\alpha=\gamma A$, which yields $\omega=\lambda_{\gamma}^{2/3}A$, with $\lambda_{\gamma} =\gamma^{3/2}/h$ and $A\sim 1$ on the support of $\psi$. Since the "main" contribution in the parametrix \eqref{eq:Prond} comes from values $\omega\sim \lambda^{2/3}$ with $\lambda=\lambda_a=\frac{a^{3/2}}{h}$ (or, in terms of $\mathcal{P}_{h,a}$ written as a sum \eqref{eq:Prond2}  with $\omega_k\sim k^{2/3}$, the main contribution comes from $k\sim \lambda$), the part $\Prond_{h,a,a}$ will provide the "worst" case scenario.
\end{rem}
Before stating estimates for $V_{N,\gamma}$, when $|N|\geq 1$ and $a\lesssim \gamma$, we recall that we have the free space dispersion for $V_{0}$, assuming that $t>h$ so that the dispersive effect takes over:
\begin{equation}
  \label{eq:36}
  |V_{0}(t,x,y)| \leq C h^{-d} \left(\frac h t \right)^{\frac{d-1}2}\,.
\end{equation}
\begin{prop}\label{propN=1}
Assume $\ceps>0$, $a\in [h^{2/3-\ceps},a_{0}]$, and $h\in (0,1)$. Then there exists $C(\ceps,a_{0})$ such that, for $t\gtrsim h$,
  \begin{equation}
    \label{eq:37}
      |V_{\pm 1}(t,x,y)| \leq C h^{-d} \left(\frac h t \right)^{\frac{d-2}2} \left(\left(\frac h t\right)^{1/2}+ a^{1/4} \left(\frac{h}{t}\right)^{1/4}+h^{\frac 1 3}\right)\,.
  \end{equation}
\end{prop}
\begin{prop}\label{proptransv}
Assume $\ceps>0$, $a\in [h^{2/3-\ceps},a_{0}]$, $\gamma\geq 4a$, $h\in (0,1)$ and let $\lambda_{\gamma}=\gamma^{3/2}/h$. Then there exists $C(\ceps,a_{0})$ such that, for $t\gtrsim \sqrt \gamma$,
  \begin{equation}
    \label{eq:37bis}
      |\sum_{|N|\geq 2 }V_{N,\gamma}(t,x,y)| \leq C h^{-d} \left(\frac h t \right)^{\frac{d-2}2} h^{1/3}\left(\frac{\gamma^{1/4}}{t^{1/2}}+\frac 1 {\lambda_{\gamma}^{3/2}}\right)\,.
  \end{equation}
\end{prop}
\begin{prop}\label{proptang}
Assume $\ceps>0$, $a\in [h^{2/3-\ceps},a_{0}]$, $h\in (0,1)$  and let $\lambda=a^{3/2}/h$. Then there exists $C(\ceps,a_{0})$ such that for $0\leq x\leq 2a$,
  \begin{equation}
    \label{eq:37ter}
      |\sum_{|N|\geq 2 }V_{N,a}(t,x,y)| \leq C h^{-d} \left(\frac h t \right)^{\frac{d-2}2}  \left( a^{1/4} \left(\frac h t \right)^{1/4}+\frac{h^{1/3}}{\lambda^{4/3}}\right)\,.
  \end{equation}
\end{prop}
Our main estimate, \eqref{eq:2} from Theorem \ref{dispintermediaire}, follows at once from the previous propositions, in the regime $a\geq h^{2/3-\ceps}$. Recall that for $a\lesssim \gamma$, $V_{N,\gamma}$ was  defined in \eqref{eq:newVNgam} as
\begin{multline}
  V_{N,\gamma}(t,x,y) = \frac 1 {2\pi h^{d+1}} \int e^{\frac i h (
    t\tau_{q}(\alpha,\theta)+\Phi(x,y,\theta,\alpha,\sigma)-\Phi(a,0,\theta,\alpha,s)-Nh L(h^{-2/3}\alpha)
    )}\psi(\alpha/(|\theta|^{2/3}\gamma)) \\
{}\times {\cutoffchi^{\flat}}(\alpha/\ceps_{0}){\cutoffchi^{\sharp}}(\alpha/h^{2/3})\chi(s)p_h(x,y,\theta,\alpha,\sigma)\tilde q_h (\theta,\alpha,s) ds \,d\theta  d{\sigma}d\alpha\,,
\end{multline}
where the symbol of $V_{N,\gamma}$ (the same for every $N$) is of order $0$. Let $|N|\geq 1$ and $a\lesssim \gamma$, and set $\lambda_{\gamma}=\frac{\gamma^{3/2}}{h}$. We rescale variables as follows 
\begin{itemize}
\item If $\gamma\sim a$ we set $s=\sqrt{a}|\theta|^{1/3}S$, $\sigma\sqrt{a}|\theta|^{1/3}\Sigma$, $\alpha=a|\theta|^{2/3}A$ and let $\lambda=\frac{a^{3/2}}{h}$;
\item If $4a\leq \gamma$ we set
 $s=\sqrt{\gamma}|\theta|^{1/3}S$, $\sigma\sqrt{\gamma}|\theta|^{1/3}\Sigma$, $\alpha=\gamma|\theta|^{2/3}A$.
\end{itemize}
We let $\tilde\Phi_{N,a,a}(t,x,y,\Sigma,S,A,\theta)$ be given by \eqref{deftildePhi} when $\gamma \sim a$ and $\tilde\Phi_{N,a,\gamma}(t,x,y,\Sigma,S,A,\theta)$ be given by \eqref{deftildephiNagamma} when $4a\leq \gamma$. We have (including the case where $\gamma$ is replaced by $a$)
\[
\tilde\Phi_{N,a,\gamma}(t,x,y,\Sigma,S,A,\theta)=\tilde\Phi_{0,a,\gamma}(t,x,y,\Sigma,S,A,\theta)-\frac 43 N\gamma^{3/2}|\theta|A^{3/2}+NhB_L(|\theta|\lambda_{\gamma} A^{3/2}).
\] 
The phase $\tilde\Phi_{0,a,\gamma}(t,x,y,\Sigma,S,A,\theta)$, with $\theta=\rho \vartheta$,  $|\vartheta|=1$, has two critical points $\vartheta_{\pm}$ and they are non-degenerate. Indeed (see the proof of Proposition \ref{propcardN}), $\nabla^2_{\vartheta}\tilde\Phi_{0,a,\gamma}=|y|O(1)$ for any $a\lesssim\gamma$. From Lemma \ref{derriere} we have $|y|\geq c_0|t|$ and stationary phase in $\vartheta\in\mathbb{S}^{d-2}$ yields a decay factor
\begin{equation}
  \label{eq:46}
  \left(\frac{h}{|y|}\right)^{\frac{d-2}2} \leq C   \left(\frac{h}{|t|}\right)^{\frac{d-2}2},\text{ for } |y|\geq c_{0} |t|.
\end{equation}
Recall that the critical points (in $\vartheta$) of $\Phi_{N,a,\gamma}$ are \eqref{eqNcontYmodgen} (Lemma \ref{lemomegac}). According to Remark \ref{rmqsigneomegac}, if we fix a sign for $t$, only one of these critical points provide non-trivial contributions. Let $t>0$ and denote $\vartheta_{c}=\vartheta_-(t,x,y,\Sigma,S,A)$ the critical point in $\vartheta$,
and let
\[
  \Psi_{N,a,\gamma}(t,x,y,\Sigma,S,A,\rho)=\tilde\Phi_{N,a,\gamma}(t,x,y,\Sigma,S,A,\rho\vartheta_{c})\,.
\]
Recall from Lemma \ref{lemomegac} that ${\vartheta_c}$ doesn't depend on $N$, neither on $\rho=|\theta|$ (since $\tilde\Phi_{0,a,\gamma}(t,x,y,\Sigma,S,A,\theta)$ is linear in $|\theta|$). 
\begin{lemma}
The critical points ${\vartheta_c}$ are such that
\begin{equation}\label{derivomegac}
\partial_{\Sigma}{\vartheta_c}=O(\gamma^{3/2}/t),\quad \partial_{S}{\vartheta_c}=O(\gamma^{3/2}/t).
\end{equation}
\end{lemma}
\begin{proof}
We evaluate its derivatives in $\Sigma$ and $S$ of ${\vartheta_c}$ using the equation \eqref{eqNcontYmodgen} from Lemma \ref{lemomegac}. 
\end{proof}
\begin{rmq}\label{rmqBL}
For $|N|\leq \lambda_{\gamma}$, the factor $e^{iNB_L(\rho \lambda_{\gamma}A^{3/2})}$ does not oscillate. Indeed $NB_L(\rho \lambda_{\gamma}A^{3/2})\sim N/\lambda_{\gamma}\lesssim 1$, and moving this factor to the symbol, the phase becomes linear in $\rho$. On the other hand, as soon as $N>\lambda_{\gamma}$ we can take  advantage of the stationary phase in $\rho$, as we shall see below. Applying the stationary phase in $\rho$ turns out to be of particular interest for $N\geq \lambda_{\gamma}^2$, since in this case for a given $t$ such that $t/\sqrt{\gamma}\geq \lambda_{\gamma}^2$, the cardinal of the set $\mathcal{N}_1(t,x,y)$ is not uniformly bounded anymore but starts to increase like $\frac{t}{\sqrt{\gamma}\lambda_{\gamma}^2}$.
\end{rmq}
\subsection{The tangent part $\gamma\sim a$. } 
Let first $a\sim \gamma$ and set $\lambda=\lambda_a=a^{3/2}/h$. We apply stationary phase to $\Psi_{N,a,a}$ with respect to $A$.
\begin{lemma}
The equation $\partial_A\Psi_{N,a,a}=0$ has at most one solution on the support of the symbol $\psi(A)$, that we denote $A_c$. Then $A_c$ is a non-degenerate critical point and $\partial^2_{A}\Psi_{N,a,a}|_{A=A_c}\sim Na^{3/2}$.
\end{lemma}
\begin{proof}
The phase is stationary when $\partial_A\Psi_{N,a,a}=\partial_A\tilde\Phi_{N,a,a}|_{\vartheta={\vartheta_c}}=0$. The equation $\partial_A\tilde\Phi_{N,a,a}=0$ is the first in \eqref{desclambdaN} and we easily see that for any $N\neq 0$ there exists a solution $A^{1/2}_c\sim \frac{t}{4N\sqrt{a}}$. For $N\sim t/(4\sqrt{a})$ we have $A_c\sim 1$. Moreover,
\[
\partial^2_{A}\Psi_{N,a,\gamma}=\partial^2_{A}\tilde\Phi_{N,a,\gamma}|_{\vartheta={\vartheta_c}}+\partial_A{\vartheta_c}\partial_A\nabla_{\vartheta}\tilde\Phi_{N,a,\gamma}|_{\vartheta={\vartheta_c}}\,.
\]
The derivative with respect to $A$ of the first term in the left hand side of \eqref{desclambdaN} is $t\,O(a^2)$. Using \eqref{derivAPhi}, $\partial^{2}_{A}(\Phi(x,y,\theta,a|\theta|^{2/3}A,\sqrt{a}|\theta|^{1/3}\Sigma)-\Phi(a,0,\theta,a|\theta|^{2/3}A,\sqrt{a}|\theta|^{1/3}S))=|y|O(a^2)$. As the derivative of the coefficient of $N$ in \eqref{desclambdaN} is close to $Na^{3/2}$ and $|N|\sim |t|/4\sqrt{a}$, $|t|\sim |y|$, the main contribution of $\partial^2_{A}\tilde\Phi_{N,a,a}|_{\vartheta={\vartheta_c}}$ is also $Na^{3/2}$. On the other hand, 
\begin{equation}\label{estimderivAomegac}
\partial_A\nabla_{\vartheta}\tilde\Phi_{N,a,a}|_{\vartheta={\vartheta_c}}=\partial_A\nabla_{\vartheta}\tilde\Phi_{0,a,a}|_{\vartheta={\vartheta_c}}\sim tO(a).
\end{equation}
Moreover, taking the derivative with respect to $A$ of $\nabla_{\vartheta}\tilde\Phi_{0,a,a}|_{\vartheta={\vartheta_c}}=0$ gives $\partial_A{\vartheta_c}\nabla^2_{\vartheta}\tilde\Phi_{0,a,a}=-\partial_A\nabla_{\vartheta} \tilde\Phi_{0,a,a}$ ; using $\Big(\nabla^2_{\vartheta}\tilde\Phi_{0,a,a}\Big)^{-1}=|y|^{-1}O(1)$, $|y|\geq c_0|t|$ and \eqref{estimderivAomegac} eventually yields 
\[
\partial_A{\vartheta_c}\partial_A\nabla_{\vartheta} \tilde\Phi_{0,a,a}|_{\vartheta={\vartheta_c}}\sim \Big| \partial_A\nabla_{\vartheta} \tilde\Phi_{0,a,a}\Big|^2\times |y|^{-1}O(1)\sim \frac{t^2}{|y|}O(a^2)\lesssim \frac{|t|}{c_0}O(a^2)\,.
\]
Therefore the main contribution of $\partial^2_{A}\Psi_{N,a,a}$ is $Na^{3/2}$.
\end{proof}
Hence we may apply stationary phase in $A$ and get another decay factor
\begin{equation}
  \label{eq:169}
\frac{1}{\sqrt {a^{3/2}|N|}} \times \frac{1}{\sqrt{1/ h}}=  \frac 1{(\lambda |N|)^{1/2}}\,.
\end{equation}
\begin{lemma}
Assume $N\sqrt a \lesssim 1$. The critical point $A_c$ is
\begin{gather}
  \label{eq:146bis}
A_{c}=\Ab_0^2-2\Ab_0\Ab_1+\Big(\frac{\Sigma}{2N}-\frac {S} {2N}\Big)^2+ O\Big(a (\frac{\Sigma}{N},\frac{S}{N})^2\Big)+\frac {f_0} {N\lambda^{2}}\,,\\
\label{eq:defFbGb}
\Ab_0=\frac{q^{2/3}({\vartheta_c})|_{\Sigma=S=0}}{4N\sqrt{a}}(t+E_0(y,a)),\quad \Ab_1=\frac{\Sigma}{2N}(1-xE_1)-\frac {S} {2N}(1-aE_2)\,,
\end{gather}
and $f_0$ is an asymptotic expansion in $\lambda^{-1}$,  $E_0=O(|y|^2,a |y|)$, $E_{1,2}=O(1)$, $E_{0,1,2}$ are independent of $\Sigma$ and $S$.
\end{lemma}
\begin{proof}
  At the critical point $A_c$, \eqref{desclambdaN} holds (replacing $\vartheta$ with ${\vartheta_c}$ and $\gamma$ with $a$). Using \eqref{derivAPhi},
\begin{multline}\label{eqAcritgen}
A^{1/2}(1-\frac 34 B_L'(\rho\lambda A^{3/2}))=\frac{q^{2/3}({\vartheta_c})\Big(t-B_2(y,{\vartheta_c})+\sum_{k\geq 2}\partial_{\tau}\Big(\frac{(1-\tau)^k}{\tau^{k-1}}\Big)B_{2k}(y,{\vartheta_c})\Big)}{4N\sqrt{a}\sqrt{1+a Aq^{2/3}({\vartheta_c})}}|_{\tau=\tau_{q}(a A,{\vartheta_c})}\\+\frac{(-\Sigma+S)}{2N}
+\frac{{a}q^{2/3}({\vartheta_c})}{2N\sqrt{1+a Aq^{2/3}({\vartheta_c})}}\Big[\frac{x}{a^{3/2}}\partial_{\tau}\Big(\tau A_{\Gamma}(x,y,\sqrt{a}\Sigma q^{1/3}({\vartheta_c})/\tau,{\vartheta_c}/\tau)\Big)|_{\tau=\tau_{q}(a A,{\vartheta_c})}\\
-\frac{a}{a^{3/2}}\partial_{\tau}\Big(\tau A_{\Gamma}(a,0,\sqrt{a}S q^{1/3}({\vartheta_c})/\tau,{\vartheta_c}/\tau)\Big)|_{\tau=\tau_{q}(a A,{\vartheta_c})}\Big]\,,
\end{multline}
where we recall that $B_{2k}(y,{\vartheta_c})=O(|y|)$ for all $k\geq 2$ and $1-\tau_{q}(a A,{\vartheta_c})=\frac{a Aq^{2/3}({\vartheta_c})}{(1+\tau_{q}(a A,{\vartheta_c}))}$. 
\begin{rmq}
When $|N|\lesssim \lambda$, according to Remark \ref{rmqBL} we get rid of the factor $(1-\frac 34 B_L'(\rho\lambda A^{3/2}))$ in the LHS of \eqref{eqAcritgen}, without which the equation satisfied by $A_c$ is independent of $\rho,\lambda$.
\end{rmq}
At $\Sigma=S=0$, we get for $A_c^{1/2}|_{\Sigma=S=0}$
\begin{multline}\label{eqAcritgenzero}
A^{\frac 12}(1-\frac 34 B_L'(\rho\lambda A^{\frac 32}))=\frac{q^{2/3}({\vartheta_c})/(4N\sqrt{a})}{\sqrt{1+a Aq^{\frac 23}({\vartheta_c})}}\Big[t-B_2(y,{\vartheta_c})+\sum_{k\geq 2}\partial_{\tau}\Big(\frac{(1-\tau)^k}{\tau^{k-1}}\Big)B_{2k}(y,{\vartheta_c})\\
+2a \Big(\frac{x}{a}\partial_{\tau}\Big(\tau A_{\Gamma}(x,y,0,{\vartheta_c}/\tau)\Big)
-\frac{a}{a}\partial_{\tau}\Big(\tau A_{\Gamma}(a,0,0,{\vartheta_c}/\tau)\Big)\Big)\Big]_{|\tau=\tau_{q}(a A,{\vartheta_c}),\Sigma=S=0}\,.
\end{multline}
Using that $\partial_{\tau}(\tau A_{\Gamma})=O(1)$ and the expansion \eqref{eq:propL} of $B_L$, by the implicit function theorem applied to \eqref{eqAcritgenzero}, we obtain that $A^{1/2}_c|_{\Sigma=S=0}$ is of the form $\Ab_0$ given in \eqref{eq:defFbGb}. 
Taking the derivative of \eqref{eqAcritgen} with respect to $\Sigma$ yields 
\begin{multline}
\partial_{\Sigma}( A_c^{1/2})\Big(1-\frac 34 B_L'(z)-\frac 94zB_L''(z)\Big)|_{z=\rho\lambda A^{3/2}}=\frac{q^{2/3}({\vartheta_c})}{4N\sqrt{a}\sqrt{1+a A_cq^{2/3}({\vartheta_c})}}\Big[\partial_{\Sigma}{\vartheta_c}\cdot\Big(-\nabla B_2(y,{\vartheta_c})\\
+\sum_{k\geq 2}\partial_{\tau}\Big(\frac{(1-\tau)^k}{\tau^{k-1}}\Big)\nabla B_{2k}(y,{\vartheta_c})\Big)+(\partial_{\Sigma}A_c\partial_A\tau+\partial_{\Sigma}{\vartheta_c}\cdot\nabla_{\vartheta}\tau)\sum_{k\geq 2}\partial^2_{\tau,\tau}\Big(\frac{(1-\tau)^k}{\tau^{k-1}}\Big)B_{2k}(y,{\vartheta_c})\Big]\Big|_{\tau=\tau_{q}(a A_c,{\vartheta_c})}\\ 
+\partial_{\Sigma}\Big(\frac{q^{2/3}({\vartheta_c})}{\sqrt{1+a A_cq^{2/3}({\vartheta_c})}}\Big)\frac{\Big(t-B_2(y,{\vartheta_c})+\sum_{k\geq 2}\partial_{\tau}\Big(\frac{(1-\tau)^k}{\tau^{k-1}}\Big)B_{2k}(y,{\vartheta_c})\Big)}{4N\sqrt{a}}\Big|_{\tau=\tau_{q}(a A_c,{\vartheta_c})}\\
-\frac{1}{2N}+\frac{\sqrt{a}}{2N}\partial_{\Sigma}\Big(\frac{q^{2/3}({\vartheta_c})}{\sqrt{1+a A_cq^{2/3}({\vartheta_c})}}\Big)\frac{x}{a} \partial_{\tau}\Big(\tau A_{\Gamma}(x,y,\sqrt{a}\Sigma q^{1/3}({\vartheta_c})/\tau,{\vartheta_c}/\tau)\Big)|_{\tau=\tau_{q}(a A,{\vartheta_c})}\\
+\frac{\sqrt{a}}{2N}\frac{q^{2/3}({\vartheta_c})}{\sqrt{1+a A_cq^{2/3}({\vartheta_c})}}\frac{x}{a} \partial_{\Sigma}\Big[\partial_{\tau}\Big(\tau A_{\Gamma}(x,y,\sqrt{a}\Sigma q^{1/3}({\vartheta_c})/\tau,{\vartheta_c}/\tau)\Big)|_{\tau=\tau_{q}(a A,{\vartheta_c})}\Big]\,,
\end{multline}
where $\partial_A\tau_{q}=a q^{2/3}(\vartheta)/(2\tau_{q})$, $B_{2k}(y,{\vartheta_c})=O(|y|)$, $\nabla_{\vartheta}\tau_{q}=a A\nabla(q^{2/3}(\vartheta))/(2\tau_{q})$ and $\partial_{\Sigma} (\partial_{\tau}(\tau A_{\Gamma}))=O(a)$. At $\Sigma=S=0$ we find
\[
\partial_{\Sigma}(A_c^{1/2})|_{\Sigma=S=0}=\frac 23 \partial_{\Sigma}{\vartheta_c}\nabla q({\vartheta_c})q^{-1/3}({\vartheta_c})\frac{(t+O(|y|^2)+O(a |y|))}{4N\sqrt{a}}-\frac{1}{2N}(1+O(a^{3/2})+O(\lambda^{-2}))\,,
\]
which, together with \eqref{eqAcritgen} and \eqref{derivomegac}, allows to obtain the explicit form of the derivative of $A_c$ with respect to $\Sigma$ (similar computations hold for $S$).
\end{proof}

\begin{rmq}\label{}
When $|N|\geq \lambda^2$, stationary phase in $\rho$ turns out to be of particular interest: for a given $t$ such that $t/\sqrt{a}\geq \lambda^2$, the cardinal of the set $\mathcal{N}_1(t,x,y)$ is large. Obtaining a bound of the integral defining $V_{N,a}$ better than the one given by integrating only with respect to $\Sigma,S$ turns out to be crucial in order to obtain the desired dispersive estimates. 
\end{rmq}

Let us consider the case $|N|>C\lambda^2$ for some constant $C>1$. Let
\begin{equation}
  \label{eq:23}
  \phi_{N,a}(t,x,y,\Sigma,S,\rho):=\Psi_{N,a,a}(t,x,y,\Sigma,S,A_c,\rho)
\end{equation}
denote the critical value of $\Psi_{N,a,a}$. We compute derivatives of $\phi_{N,a}$ with respect to $\rho$. 
Since $\vartheta_{c}$ is independent of $\rho$, 
the phase $\phi_{N,a}$ is stationary in $\rho$ when
\[
0=\partial_{\rho}\phi_{N,a}=(\partial_{\rho}\Psi_{N,a,a}+\partial_{\rho}A_c\partial_A\Psi_{N,a,a})|_{A_c}=\partial_{\rho}\Psi_{N,a,a}|_{A_c}\,,
\]
with $A_c$ provided in \eqref{eq:146bis}, where the coefficients of $t,\Sigma,S$ (in the first three terms of the sum) are independent of $\rho$ (notice that in the equation \eqref{eqAcritgen} verified by $A_c$, the only term containing $\rho$ is $B_L'$ which amounts for the terms with factor $O(\lambda^{-2})$ in \eqref{eq:146bis}). Since $\Psi_{N,a,a}=\Psi_{0,a,a}-\frac 43 a^{3/2}\rho NA^{3/2}+N hB_L(\rho\lambda A^{3/2})$, and since $\Psi_{0,a,a}-\frac 43 a^{3/2}\rho NA^{3/2}$ is linear in $\rho$, the second derivative is
\[
\partial^2_{\rho}\phi_{N,a}=\Big(\partial^2_{\rho}\Psi_{N,a,a}+\partial_{\rho} A_c(\partial_{\rho}\partial_{A}\Psi_{N,a,a}+\partial_{\rho}A_c\partial^2_{A}\Psi_{N,a,a})\Big)|_{A_c}\,.
\]
For $|N|>\lambda^2\gg \lambda$, the only part that matters here is the contribution from $\frac{N}{\lambda}B_L(\rho\lambda A^{3/2})$. Taking into account that the support of the symbol of $V_{N,a}$ in $A$ is a fixed, compact set of $(0,\infty)$, the contribution of $\partial^2_{\rho}\Psi_{N,a,a}$ will be $hN/\lambda$.
The main contribution of $\partial_{\rho}\partial_{A}\Psi_{N,a,a}|_{A_c}$ is also of  size $hN/\lambda$, but it comes with a factor $\partial_{\rho}A_c$, which, from \eqref{eq:146bis}, can be at most $O(1/N)$, since the first three terms in the sum defining $A_c$ do not depend on $\rho$. Eventually we obtain, 
 \begin{equation}
   \label{eq:54}
 \frac 1h  |\partial^{2}_{\rho} \phi_{N,a}| \sim \frac{|N|}{\lambda}\,,
 \end{equation}
 hence the stationary phase in $\rho$ will produce a factor $(\frac{|N|}{\lambda})^{-1/2}$
which is (of interest only in the regime $|N|\geq C \lambda$ and) particularly useful for $\frac{|N|}{\lambda}>\lambda$. 
Let $\sigma_{V,h,a}(\Sigma,S,\rho)$ be obtained from $$\psi(A)p_h(x,y,\theta,a |\theta|^{2/3}A,\sqrt{a}|\theta|^{1/3}\Sigma)\tilde q_h(\theta,a |\theta|^{2/3}A,\sqrt{a}|\theta|^{1/3}S)$$ after applying the stationary phase in both $\vartheta$ and $A$. We note that $\sigma_{V,h,a}$ is independent of $N$, of order zero and has compact support in $\Sigma,S$. We are left with estimating the oscillatory integral $\int e^{\frac ih \phi_{N,a}}\sigma_{V,h,a}d\Sigma dS d\rho$.
  Following the approach in \cite{ILP3} (see also \cite{Annals}), we deal separately with $|N|>\lambda^{1/3}$ and $|N|\leq \lambda^{1/3}$, where we recall that $\lambda=\lambda_a=a^{3/2}/h$.
\subsubsection{Large number of reflections: $|N|>\lambda^{1/3}$}
\begin{lemma}\label{lemestimVN2}
There exists $C$ (independent of $N$) such that,  if $|N|\geq \lambda^{1/3}$,
  \begin{equation}
    \label{eq:64}
\left|      \int e^{\frac i h \phi_{N,a}} \sigma_{V,h,a}(\Sigma,S,\rho) d\Sigma dS d\rho \right|  \leq \frac{C}{\lambda^{2/3}}\,.
  \end{equation}
\end{lemma}
\begin{rem}
The proof will be split in two parts: first, we consider $\lambda^{1/3}\leq |N|\leq \lambda^2$ where we integrate with respect to $S,\Sigma$ and then bound the remaining integral in $\rho$ owning to its compact support; then for $|N|>\lambda^2$ we start with stationary phase in $\rho$, which provides additional decay; then we integrate with respect to $S,\Sigma$ following closely the approach for $\lambda^{1/3}\leq |N|\leq \lambda^2$. 
\end{rem}
\begin{proof}
We compute derivatives of $\phi_{N,a}$ with respect to $\Sigma,S$. Start with $\lambda^{1/3}\leq |N|\leq \lambda^2$ and follow \cite[Proposition 4]{ILP3} (see also \cite[Lemma 2.24]{Annals}, where only the case $N<\lambda^2$ is considered).
Using $\tilde\Phi_{0,a,a}$ as provided in \eqref{eqtildePhiN} for $N=0$ and $\gamma$ replaced by $a$, we compute
\begin{multline}\label{derivSigmaphizero}
\partial_{\Sigma} \phi_{N,a}=\partial_{\Sigma}\tilde\Phi_{0,a,a}|_{{\vartheta_c},A_c,\rho=|\theta|}= a^{3/2}\rho\Big[\Sigma^2+\frac{x}{a}q^{1/3}(\vartheta_{c})-A_c\\
+\frac{x}{a}\frac{\tau_{q}(a A,\vartheta_{c})}{\sqrt{a}} \partial_{\Sigma}A_{\Gamma}\Big(x,y,\frac{\sqrt{a}\Sigma q^{1/3}(\vartheta_{c})}{\tau_{q}(a A,\vartheta_{c})},\frac{\vartheta_{c}}{\tau_{q}(a A,\vartheta_{c})}\Big)
\Big]\\
= a^{3/2}\rho\Big[\Sigma^2+\frac{x}{a}q^{1/3}(\vartheta_{c})-A_c+\frac{x}{a}q^{1/3}({\vartheta_c})\partial_{\Xi}A_{\Gamma}(x,y,\Xi,\frac{\vartheta_{c}}{\tau_{q}(a A,\vartheta_{c})})|_{\Xi=\frac{\sqrt{a}\Sigma q^{1/3}(\vartheta_{c})}{\tau_{q}(a A,\vartheta_{c})}}\Big].
\end{multline}
In the same way we have
\begin{multline}
\partial_{S} \phi_{N,a}=\partial_{S}\tilde\Phi_{0,a,a}|_{{\vartheta_c},A_c,\rho=|\theta|}= -a^{3/2}\rho\Big[S^2+q^{1/3}(\vartheta_{c})-A_c\\
+\frac{\tau_{q}(a A,\vartheta_{c})}{\sqrt{a}} A_{\Gamma}\Big(a,0,\frac{\sqrt{a}S q^{1/3}(\vartheta_{c})}{\tau_{q}(a A,\vartheta_{c})},\frac{\vartheta_{c}}{\tau_{q}(a A,\vartheta_{c})}\Big)
\Big]\\
= -a^{3/2}\rho\Big[S^2+q^{1/3}(\vartheta_{c})-A_c+q^{1/3}({\vartheta_c})\partial_{\Xi}A_{\Gamma}(a,0,\Xi,\frac{\vartheta_{c}}{\tau_{q}(a A,\vartheta_{c})})|_{\Xi=\frac{\sqrt{a}S q^{1/3}(\vartheta_{c})}{\tau_{q}(a A,\vartheta_{c})}}\Big]\,.
\end{multline}
The critical points are such that
\begin{gather}\label{systsecderivSigS}
\Sigma^2+\frac xa q^{1/3}({\vartheta_c})\Big(1+\partial_{\Xi}A_{\Gamma}(x,y,\Xi,\frac{\vartheta_{c}}{\tau_{q}(a A,\vartheta_{c})})|_{\Xi=\frac{\sqrt{a}\Sigma  q^{1/3}(\vartheta_{c})}{\tau_{q}(a A,\vartheta_{c})}}\Big)=A_c,\\
S^2+q^{1/3}({\vartheta_c})\Big(1+\partial_{\Xi}A_{\Gamma}(a,0,\Xi,\frac{\vartheta_{c}}{\tau_{q}(a A,\vartheta_{c})})|_{\Xi=\frac{\sqrt{a}S q^{1/3}(\vartheta_{c})}{\tau_{q}(a A,\vartheta_{c})}}\Big)=A_c\,,
\end{gather}
where $\partial_{\Xi}A_{\Gamma}(x,y,\Xi,\Theta)=\Lp(y,{\vartheta_c})+2\Xi\mu(y,{\vartheta_c})+\mathcal{H}_{j\geq 2}$, $\Lp(0,{\vartheta_c})=0$, and where homogeneous terms of order $j$ come with small factors $a^{j/2}$.
We will prove that, although the determinant of the matrix of second derivatives may vanish, we can still use degenerate stationary phase with critical point of order at most $2$ to conclude.
We compute the second derivative with respect to $\Sigma$ using \eqref{derivSigmaphizero} : 
\begin{equation}\label{secderivSigSig}
  \partial_{\Sigma}^{2} \phi_{N,a}=a^{3/2}\rho\Big[2\Sigma-\partial_{\Sigma}A_c+\frac xa \partial_{\Sigma}\Big(q^{1/3}({\vartheta_c})\partial_{\Xi}A_{\Gamma}(x,y,\Xi,\frac{{\vartheta_c}}
   {\tau_{q}
    (aA_c,{\vartheta_c})})|_{\Xi=\frac{\sqrt{a}\Sigma q^{1/3}(\vartheta_{c})}{\tau_{q}(a A_c,\vartheta_{c})}}\Big)\Big]\,,
\end{equation}
where homogeneous terms of order $j$ in $\partial_{\Xi}A_{\Gamma}$ come with small factors $a^{j/2}$ and 
\begin{multline}\label{formderivXiAGam}
  \partial_{\Sigma}\Big(q^{1/3}({\vartheta_c})\partial_{\Xi}A_{\Gamma}(x,y,\Xi,\frac{{\vartheta_c}}{
    \tau_{q}(aA_c,{\vartheta_c})})|_{\Xi=\frac{\sqrt{a}S q^{1/3}(\vartheta_{c})}{\tau_{q}(a A_c,\vartheta_{c})}}\Big)=\partial_{\Sigma}\Big(
q^{1/3}({\vartheta_c})(\Lp(y,{\vartheta_c})\\+2\mu(y,{\vartheta_c})\frac{\sqrt{a}\Sigma q^{1/3}(\vartheta_{c})}{\tau_{q}(a A_c,\vartheta_{c})}+\mathcal{H}_{j\geq 2})\Big)=\partial_{\Sigma}{\vartheta_c}\nabla_{\vartheta}\Big(\Lp(y,\vartheta)q^{1/3}(\vartheta)+2\sqrt{a}\Sigma\mu(y,\vartheta)q^{2/3}(\vartheta)+O(a)\Big)|_{\vartheta={\vartheta_c}}\\
+2\sqrt{a}\mu(y,{\vartheta_c})q^{2/3}({\vartheta_c})+O(a)\,.
\end{multline}
Using that $\sqrt{a}\lesssim \frac 1N$, $x\lesssim 2a$, we obtain an estimation of the second derivatives of $\phi_{N,a}$
  \begin{align}\label{eq:secderivSTprime}
  \partial^2_{\Sigma}\phi_{N,a} & =a^{3/2}\rho(2\Sigma+O( a^{3/2}/|y|)+O(1/N)),\\
    \partial^2_{S}\phi_{N,a} & =a^{3/2}\rho(2S+O( a^{3/2}/|y|)+O(1/N))\,,\\
      \partial^2_{\Sigma,S}\phi_{N,a} & =O(a^{3/2}/N)\,.
  \end{align}
We rescale variables $(\Sigma,S)=(\lambda^{-1/3} \xx,\lambda^{-1/3}\yy)$, so that we are left with proving
\begin{equation}
  \label{eq:65}
  \left|   \int e^{\frac i h \phi_{N,a}(\lambda^{-1/3}\xx,\lambda^{-1/3}\yy,\rho)} \sigma_{V,h,a}(\lambda^{-1/3}\xx,\lambda^{-1/3}\yy,\rho) \,d\xx d\yy\right|  \leq {C}\,.
\end{equation}
From the compact support of $\sigma_{V,h,a}$ we obviously have, for any multi-index $ \nu$,
\begin{equation}
  \label{eq:66}
|  \partial^{\nu}_{(\xx,\yy)} \sigma_{V,h,a}(\lambda^{-1/3}\xx,\lambda^{-1/3}\yy,\rho)|\leq C_{\nu} (1+|\xx|+|\yy|)^{-|\nu|}\,.
\end{equation}
Let $\Ab_0$ be the main term of $A_c^{1/2}$ defined \eqref{eq:defFbGb}; we define $(\Af,\Bf)$ as follows
\begin{equation}
  \label{eq:70}
-  \Af \lambda^{-2/3}=\frac{x}{a} q^{1/3}({\vartheta_c})(1+\Lp(y,{\vartheta_c}))-\Ab_0^2,\,\,\,\,\, -\Bf\lambda^{-2/3}=q^{1/3}({\vartheta_c})-\Ab_0^2
\end{equation}
and we also write $(\Af,\Bf)=r (\sin \tf,\cos \tf)$, where $r^2=\Af^2+\Bf^2$. In the new variables, since $\frac{\partial \Sigma}{\partial \xx}=\frac{\partial S}{\partial \yy}=\lambda^{-1/3}$ and using $\partial_{\Xi} A_{\Gamma}$ as provided in \eqref{formderivXiAGam}, we find
 \begin{multline}
 \frac 1h  \partial_{\xx}\phi_{N,a}=\frac{a^{3/2}}{h}\lambda^{-1/3}\rho\Big[\lambda^{-2/3}(\xx^{2}-\Af)+2\frac{\lambda^{-1/3}\xx}{2N}\Big(\Ab_0(1-xE_1)+\frac xa (N\sqrt{a})\mu(y,{\vartheta_c})q^{2/3}({\vartheta_c})\Big)\\
 -2\frac{\lambda^{-1/3}\yy}{2N}\Big(\Ab_0(1-aE_2)\Big)+\frac{\lambda^{-1/3}}{N}O(\frac{\yy}{N},\frac{\xx}{N})+O(a)\Big]\\
 = \rho\Big((\xx^2-\Af)+\frac{\lambda^{1/3}}{N} \Big(\xx(\Ab_0(1-xE_1)+O(N\sqrt{a}))-\yy\Ab_0(1-aE_2)+O(\frac{\yy}{N},\frac{\xx}{N}))+O(\lambda^{2/3}a)\Big)\\
=\rho(\xx^{2}-\Af + O(\xx ,\yy )+O(1))\,,
\end{multline}
where the term $O(a)$ in the second comes from the terms homogeneous of order $j\geq 2$ in the term $\partial_{\Xi}A_{\Gamma}(x,y,\Xi,\frac{{\vartheta_c}}{\tau_{q}(aA_c,{\vartheta_c})})|_{\Xi=\frac{\sqrt{a}S q^{1/3}(\vartheta_{c})}{\tau_{q}(a A_c,\vartheta_{c})}}$ (see \eqref{formderivXiAGam}). We have written $O(\lambda^{2/3}a)=O((\lambda^{1/3}/N\times N\sqrt{a})^2)=O(1)$ since ${\lambda^{1/3}}/{N}<1$ and $(Na^{1/2})<1$.
Similarly, derivatives with respect to $\xx$ and $\yy$ are,
\begin{equation}
 \frac 1h \partial_{\xx }\phi_{N,a}  =\rho(\xx ^{2}-\Af +O(\xx ,\yy )+O(1))\,,\quad\quad\frac 1h  \partial_{\yy }\phi_{N,a}  =\rho(\yy ^{2}-\Bf+O(\xx ,\yy )+O(1))\,.
\end{equation}
From \eqref{eq:secderivSTprime}, we obtain, in the new coordinates
\begin{gather}
  \label{eq:74}
 \frac 1h \partial_{\xx}^{2}\phi_{N,a}=\lambda^{1/3} \rho(2\lambda^{-1/3} \xx+O(N^{-1})) =\rho(2\xx+O(1))\,,\\
  \frac 1h \partial_{\yy}^{2}\phi_{N,a}=\lambda^{1/3} \rho(2\lambda^{-1/3} \yy+O(N^{-1})) =\rho(2\yy+O(1))\,,
\end{gather}
where we used $|N|>\lambda^{1/3}$ in the last step together with $|y|\geq c_0 t$ which yields $a^{3/2}/|y|\lesssim a^{3/2}/t\sim a/N$, as $t/\sqrt{a}\sim 4N$  when the phase is stationary in $A$.
Now, going back to \eqref{eq:65}, the integral is bounded for $0\leq r \leq r_{0}$,  for some $r_0>0$, by integration by parts for large $(\xx,\yy)$ (recall that the support of the integrand is now of radius $\lambda^{1/3}$). For $r_{0}< r \lesssim \lambda^{2/3}$, set $(\xx,\yy)=r^{1/2}(\xx',\yy')$, $\phi_{N,a}=r^{3/2} \tilde{\phi}_{N,a}$, $\tilde \sigma_{V,h,a}(\xx',\yy',.)= \sigma_{V,h,a}(r^{1/2}\lambda^{-1/3} \xx',r^{1/2}\lambda^{-1/3} \yy',\rho)$,  and as $r^{1/2}\lambda^{-1/3}$ is bounded, we retain the decay
$|  \partial^{\nu}_{(\xx',\yy')}\tilde\sigma_{V,h,a}(\xx',\yy',.)|\leq C_{\nu} (1+|\xx'|+|\yy'|)^{-|\nu|}$.
It remains to prove
\begin{equation}
  \label{eq:73}
 r |\int e^{i\frac{r^{3/2}}{h}\tilde \phi_{N,a}} \tilde\sigma_{V,h,a} \,d\xx' d\yy' | \leq C \,.
\end{equation}
To begin with, notice that
\begin{align}
  \partial_{\xx' }(\frac{r^{3/2}}{h} \tilde\phi_{N,a}) & = r^{3/2} \rho\Big(\xx'^{2}-\sin \tf + r^{-1/2}O(\xx' ,\yy' )+ r^{-1} O(1)\Big)\,,\\
  \partial_{\yy' }(\frac{r^{3/2}}{h} \tilde\phi_{N,a}) & =r^{3/2} \rho\Big(\yy'^{2}-\cos \tf + r^{-1/2}O(\xx' ,\yy' )+r^{-1 }O(1)\Big)\,.
\end{align}
If $|\sin \tf|\geq 1/100$, then $\frac{r^{3/2}}{h} \tilde\phi_{N,a}$ has two critical points in $\xx'$, namely $\xx'_{\pm}=\pm |\sin \tf|^{1/2}+O(r^{-1/2})$,
 and these critical points are non degenerate as 
$|\partial^{2}_{\xx'} (\frac{1}{h} \tilde\phi_{N,a})|\geq \rho/10 $. By stationary phase, we get
\begin{equation}
  \label{eq:75}
  \int e^{i\frac{r^{3/2}}{h} \tilde\phi_{N,a}} \tilde\sigma_{V,h,a} \,d\xx'd\yy'= (r^{3/2})^{-1/2} (\int  e^{i\frac{r^{3/2}}{h} \tilde\phi_{N,h,a,+}} \tilde\sigma_{V,h,a,+} \,d\yy'+\int  e^{i\frac{r^{3/2}}{h} \tilde\phi_{N,a,-}} \tilde\sigma_{V,h,a,-} \,d\yy')\,,
 \end{equation}
for some symbols $\tilde \sigma_{V,h,a,\pm}$ of order $0$.
Using \eqref{eq:74}, one may check that $|\partial^{3}_{\yy',\yy',\yy'} (\frac{1}{h} \tilde\phi_{N,a})|\geq c>0$, and by degenerate stationary phase (or Van der Corput lemma) we get
\begin{equation}
  \label{eq:76}
|  \int  e^{i\frac{r^{3/2}}{h} \tilde\phi_{N,a,\pm}} \tilde\sigma_{V,h,a,\pm} \,d\yy'| \leq C (r^{3/2})^{-1/3}=Cr^{-1/2}\,,
\end{equation}
and for $r\geq 1$, we get the desired decay, and even an extra $r^{-1/4}$ on the right hand side. If $|\cos \tf|\geq 1/100$, we proceed in the same way, exchanging $\xx' $ and $\yy'$.

Let us now deal with $|N|\geq \lambda^2$. Stationary phase applies in $\rho$ and from \eqref{eq:54},
 \begin{equation}
 \label{eq:54a}
 \frac 1h  |\partial^{2}_{\rho} \phi_{N,a}| \sim \frac{|N|}{\lambda}\,,
 \end{equation}
 hence the stationary phase in $\rho$ yields a factor $(\frac{|N|}{\lambda})^{-1/2}$. In the following we apply exactly the same method as in the previous step; the only difference is that now we have an additional function $\rho_c$ which depends on the variables $\Sigma,S$ and whose derivatives with respect to these variables are small only for $|N|>\lambda^2$ (which explains why we did not perform stationary phase with respect to $\rho$ earlier, for any $|N|>\lambda^{1/3}$). Indeed, from $\partial_{\rho}\phi_{N,a}=\partial_{\rho}\Psi_{N,a,a}|_{A_c}=0$ we get
\begin{equation}\label{eq:eqr}
\Psi_{0,a,a}-\frac 43 a^{3/2}NA_c^{3/2}=Nh\lambda A_c^{3/2}B'_L(\rho\lambda A_c^{3/2})= -\frac{Nh A_c^{3/2}}{\rho^2\lambda A_c^3}(b_1+O(1/\lambda)),\quad b_1\neq 0\,,
\end{equation}
where the term in the left hand side doesn't depend on $\rho$. Since $\rho=|\theta|\sim 1$ and since the support of the symbol in $A_c$ is a given compact set of $(0,\infty)$, we have
\[
\rho^2_c=-{\frac{Nh}{\lambda A_c^{3/2}}(b_1+O(1/\lambda))}\Bigl({\Psi_{0,a,a}(t,x,y,\Sigma,S,1)-\frac 43 a^{3/2}NA_c^{3/2}}\Bigr)^{-1}\,.
\]
Taking the derivatives with respect to $S,\Tprime$ of \eqref{eq:eqr}, with $A_c$ bounded, provides a factor $\lambda^2/N$ which is small in the regime we consider here, $|N|>C\lambda^{2}$ for some $C>1$ sufficiently large (indeed, the terms containing $\Sigma,S$ in \eqref{eq:eqr} come with a factor $a^{3/2}=h\lambda$). Since $\partial_{\Sigma}\rho_c$ and $\partial_{S}\rho_c$ are now sufficiently small, we can follow the same earlier steps to estimate the integral in the remaining variables $\Sigma,S$.
\end{proof}
\subsubsection{Moderately large $0<|N|\leq \lambda^{1/3}$}
In this case the contribution from the integral in $\rho$ is uniformly bounded due to its compact support. According to Remark \ref{rmqBL}, we bring the factor $e^{iNB_L}$ into the symbol and work with the phase $\Psi_{0,a,a}-\frac 43 a^{3/2}\rho NA^{3/2}$, which is linear in $\rho$. Therefore, the critical point $A_c$ satisfies an equation similar to \eqref{eqAcritgen}, but without the factor $1-\frac 34 B_L'(\rho\lambda A^{3/2})$, which leads to an explicit expression of the form \eqref{eq:146bis} where $f_0=0$. We start in the same way as in the proof of Lemma \ref{lemestimVN2}, replacing $\lambda^{1/3}$ by $|N|$: we rescale variables with $(\Sigma,S)=(\frac{\xx}{|N|},\frac{\yy}{|N|})$. 
\begin{lemma}
  For $|N|\geq 1$, set $\Lambda=\lambda |N|^{-3}$ and assume $\Lambda \geq 1$, then we have
  \begin{equation}
    \label{eq:77}
    \left| \int e^{\frac i h \phi_{N,a}} \sigma_{V,h,a} (\frac \xx {|N|},\frac \yy {|N|},\rho) \,d\xx d\yy\right | \leq C \Lambda^{-3/4}\,.
  \end{equation}
\end{lemma}
\begin{proof}
  Here we need all the first three terms in the formula \eqref{eq:146bis}. In our new variables, $\Ab_0$, as defined in \eqref{eq:defFbGb}, does not change and $\Ab_1=\frac{1}{2N^{2}}(\xx(1-xE_1)-\yy(1-aE_2))$, hence, using \eqref{eq:146bis}
  \begin{align}
\nonumber
 \frac 1h  \partial_{\xx}\phi_{N,a} & \! \begin{multlined}[t] = \Lambda r\Big(\xx^{2}+N^2(\frac xa q^{1/3}({\vartheta_c})(1+\Lp(y,{\vartheta_c}))-\Ab_0^2))+\Ab_0\Big(\xx(1-xE_1)-\yy(1-aE_2)\Big)\\+2N\sqrt{a}\xx\frac xa\mu(y,{\vartheta_c})q^{2/3}({\vartheta_c}) -\frac{1}{4N^2}(\xx-\yy)^2+O(\frac{\yy^2}{N^2},\frac{\xx^2}{N^2})+N^2O(a)\Big),
\end{multlined}\\
\nonumber
\frac 1h  \partial_{\yy}\phi_{N,a} & \! \begin{multlined}[t] = -\Lambda r\Big(\yy^{2}+N^2(q^{1/3}({\vartheta_c})-\Ab_0^2)+\Ab_0\Big(\xx(1-xE_1)-\yy(1-aE_2)\Big)\\-2N\sqrt{a}\yy\mu(0,{\vartheta_c})q^{2/3}({\vartheta_c})-\frac{1}{4N^2}(\xx-\yy)^2+O(\frac{\yy^2}{N^2},\frac{\xx^2}{N^2})+N^2O(a)\Big),
\end{multlined}
      \end{align}
where $N^2O(a)$ in the last two formulas are homogeneous terms of order $j\geq 2$ in the expression of $\partial_{\Xi}A_{\Gamma}$.
Next, we compute the second derivatives : notice that \eqref{derivomegac} yields $|\nabla_{\Sigma,S}{\vartheta_c}|\sim O(a^{3/2}/|y|)\sim O(a/N)$ since $|y|\geq c_0 t$ and $t/\sqrt{a}\sim 4N$.
Using \eqref{secderivSigSig} and \eqref{eq:defFbGb} we find
\begin{align*}
   \frac 1h  \partial^2_{\xx}\phi_{N,a} & =\Lambda \rho\Big(2\xx+2N\sqrt{a}\frac xa \mu(y,{\vartheta_c})q^{2/3}({\vartheta_c}) +\Ab_0(1-xE_1)-\frac{1}{2N^2}(\xx-\yy)+O(\frac{\yy^2}{N^2},\frac{\xx}{N^2})\Big),\\
  \frac 1h  \partial^2_{\yy}\phi_{N,a} & \begin{multlined}[t] =-\Lambda \rho\Big(2\yy+2N\sqrt{a}\mu(0,{\vartheta_c})q^{2/3}({\vartheta_c})(1+O(\sqrt{a}))+\Ab_0(1-aE_2)-\frac{1}{2N^2}(\xx-\yy)\\
    {}+O(\frac{\yy}{N^2},\frac{\xx^2}{N^2}) \Big)
  \end{multlined}\\
    \frac 1h  \partial_{\xx}\partial_{\yy}\phi_{N,a} & =\Lambda \rho\Big(\Ab_0-\frac{(\xx-\yy)}{2N^2}+O(\frac{\yy}{N^2},\frac{\xx}{N^2}) \Big)\,.
 \end{align*}
In fact, we infer that we can write
 \begin{gather}\label{eq:derivePhiN2}
 \frac 1h  \partial_{\xx}\phi_{N,a}=\Lambda \rho\Big(\xx^{2}-\Af+(b \xx-d \yy)-\frac{1}{4N^2}(\xx-\yy)^2+O(\frac{\xx^2}{N^2},\frac{\yy^2}{N^2})\Big),\\
 \frac 1h  \partial_{\yy}\phi_{N,a}=-\Lambda \rho\Big(\yy^{2}-\Bf+(d \xx-c\yy)-\frac{1}{4N^2}(\xx-\yy)^2+O(\frac{\xx^2}{N^2},\frac{\yy^2}{N^2}) \Big),
\end{gather}
where $(\Af,\Bf)$ and $b$, $c$, $d$ are defined as follows
\begin{gather*}
  -\frac{\Af}{N^{2}}=\frac xa q^{1/3}({\vartheta_c})(1+\Lp(y,{\vartheta_c}))-\Ab_0^2 \,,\quad 
  -\frac{\Bf}{N^{2}}  =q^{1/3}({\vartheta_c})-\Ab_0^2\,,\\
  \begin{aligned}
    b  & =\Ab_0(1+O(x))+N\sqrt{a}\frac xa \mu(y,{\vartheta_c})q^{2/3}({\vartheta_c})(1+O(\sqrt{a})) \,,\\
c & =\Ab_0(1+O(a))-N\sqrt{a}\mu(0,{\vartheta_c})q^{2/3}({\vartheta_c})(1+O(\sqrt{a}))\\
d & =\Ab_0(1+O(x;a))\,.  \end{aligned}
\end{gather*}
This follows easily from writing the Taylor development for the first order derivatives of $\phi_{N,a}$ with respect to $\xx,\yy$ and use that $\partial_{\xx}\partial_{\yy}\phi_{N,a}=\partial_{\yy}\partial_{\xx}\phi_{N,a}$. We also have
\begin{align*}
  \label{eq:160bis}
   \frac 1h  \partial^2_{\xx}\phi_{N,a} & =\Lambda \rho\Big(2\xx+b-\frac{1}{2N^2}(\xx-\yy)+O(\frac{\yy^2}{N^2},\frac{\xx}{N^2})\Big)\,,\\
 \frac 1h  \partial^2_{\yy}\phi_{N,a} & =-\Lambda \rho\Big(2\yy-c+\frac{1}{2N^2}(\xx-\yy)+O(\frac{\yy}{N^2},\frac{\xx^2}{N^2}) \Big)\,,\\
  \frac 1h  \partial_{\xx}\partial_{\yy}\phi_{N,a} & =-\Lambda \rho\Big(d-\frac{(\xx-\yy)}{2N^2}+O(\frac{\yy}{N^2},\frac{\xx}{N^2}) \Big)\,.
 \end{align*}
Let $\Lambda \rho M(\xx,\yy)$ denote the matrix of second order derivatives (i.e. the Hessian of $\frac 1h \phi_{N,a}$); we compute its determinant
\begin{equation}
  \label{eq:80}
  \det M(\xx,\yy)=-\Lambda ^2 \rho^{2}
  \Big(4\xx\yy+2(b\yy-c\xx)+d^2-bc+\frac{(\xx-\yy)^{2}}{N^{2}}+O(\frac{\xx}{N^2},\frac{\yy}{N^2})\Big)\,.
\end{equation}
Set  $(\Af,\Bf)=r(\sin \tf,\cos \tf)$. We will again deal separately with $r\geq r_{0}$, for some large $r_{0}>1$, and then $r< r_{0}$. We start with large $r\geq r_0>1$ and we prove
\begin{equation}
  \label{eq:81}
  \left|  \int e^{\frac ih \phi_{N,a}} \sigma_{V,h,a} (\frac \xx {|N|},\frac \yy {|N|}) \,d\xx d\yy\right | \leq C \Lambda^{-5/6}\,,
\end{equation}
which has better decay  than required. First, observe that from the hypothesis on $r$ and \eqref{eq:derivePhiN2}, we may integrate by parts in $(\xx,\yy)$ in a region $|(\xx,\yy)|\leq cr^{1/2}$ with $c$ small enough: for any $k\geq 1$,
\begin{equation}
  \label{eq:82}
    \left|  \int_{|(\xx,\yy)\leq c r^{1/2}} e^{\frac ih \phi_{N,a}} \sigma_{V,h,a} (\frac \xx {|N|},\frac \yy {|N|}) \,d\xx d\yy\right | \leq C_{k} (r\Lambda)^{-k}\,.
\end{equation}
Thus we are left with the region $|(\xx,\yy)|\geq c r^{1/2}$. We rescale again $(\xx,\yy)=r^{1/2}(\xx',\yy')$ and, in the new variables, we prove the following
\begin{equation}
  \label{eq:83}
      \left| r  \int_{|(\xx',\yy')\geq c } e^{\frac ih r^{3/2} \tilde\phi_{N,a}} \tilde\sigma_{V,h,a}(\xx',\yy',.) \,d\xx' d\yy' \right | \leq C \Lambda^{-5/6}\,,
\end{equation}
where we set, like in the previous section, $\phi_{N,a}(\xx,\yy,.)=r^{3/2}\tilde\phi_{N,a}(\xx',\yy',.)$ and $\tilde\sigma_{V,h,a}(\xx',\yy')=\sigma_{V,h,a}(r^{1/2}\xx/N,r^{1/2}\yy/N,.)$.
Taking $b=r^{1/2}b'$, $d=r^{1/2}d'$ and $c=r^{1/2}c'$, we have
\begin{align}
  \partial_{\xx'} \Big(\frac{r^{3/2}}{h}\tilde\phi_{N,a}\Big)& = \Lambda \rho r^{3/2} \Big( \xx'^{2}-\sin\tf+2(b' \xx'-d'\yy')-\frac{(\xx'-\yy')^{2}}{4N^{2}}+O(1/N^2)\Big)\\
  \partial_{\yy'} \Big(\frac{r^{3/2}}{h}\tilde\phi_{N,a}\Big)& = -\Lambda \rho r^{3/2} \Big( \yy'^{2}-\cos \tf+2(d' \xx'-c'\yy')-\frac{(\xx'-\yy')^{2}}{4N^{2}}+O(1/N^2)\Big)\,.
\end{align}
From these expressions of $\nabla_{(\xx',\yy')}(\frac{r^{3/2}}{h}\tilde\phi_{N,a})$, if $(\xx',\yy')$
is large, we get decay by integrations by parts, as $|N|\geq 2$. Therefore, we
can assume that $c\leq |(\xx',\yy')|\leq C$ where $C$ is a large, fixed
constant (we avoid a neighbourhood of $(0,0)$ since we have assumed $|(\xx,\yy)|>cr^{1/2}$).
If $|N|\geq 2$, $a$ is small and $r_{0}$ is large, the set $\{\det M(\xx,\yy)=0\}$ is a smooth curve which does not intersect the origin. Away from this set, we may use stationary phase, which will provide $r(r^{3/2}\Lambda)^{-1}$ decay on the left of \eqref{eq:83}. In the region close to $\{\det M(r^{1/2}\xx',r^{1/2}\yy')=0\}$, we can apply \cite[Lemma 2.21 (a)]{Annals}, i.e. degenerate stationary phase along a curve, to obtain
\begin{equation}
  \label{eq:84}
        \left|  \int_{c\leq |(\xx',\yy')\leq C } e^{\frac ih r^{3/2}\tilde\phi_{N,a}} \sigma_{V,h,a} (\xx',\yy',.) \,d\xx' d\yy' \right | \leq C (r^{3/2}\Lambda)^{-5/6}\,,
\end{equation}
and therefore we get \eqref{eq:83} as the extra factor $r$ on the lefthand side is canceled by the $r^{-5/4}$ on the righthand side (recall $r\geq r_{0}>1$).

We can now focus on $r=|(\Af,\Bf)|\leq r_0$. Notice that we may further restrict support again, this time to $|(\xx,\yy)|\leq 2r^{1/2}\leq 2r^{1/2}_0$ as we get $\Lambda^{-\infty}$ decay by integration by parts if $|(\xx,\yy)|$ is larger than this value, as we did before. We now aim at proving
\begin{equation}
  \label{eq:85}
  \left | \int_{|(\xx,\yy)|\leq 2r^{1/2}} e^{\frac ih \phi_{N,a}} \sigma_{V,h,a}(\frac{\xx}{|N|},\frac{\yy}{|N|}) \,d\xx d\yy \right| \leq C \Lambda^{-3/4}\,,
\end{equation}
for which we will apply
\cite[Lemma 2.21 (b)]{Annals}. The determinant of the Hessian $M(\xx,\yy)$ of
$\frac 1h \phi_{N,a}$ is given in \eqref{eq:80} and the set $\{\det M(\xx,\yy)=0\}$ is a smooth curve in
$\{|(\xx,\yy)|\leq 2r^{1/2}\}$, at least for small $a$: it will be close to $(2\xx+b)(2\yy-c)+d^2=-(\xx-\yy)^{2}$ when $|N|=1$ and $(2\xx+b)(2\yy-c)+d^2=-(\xx-\yy)^{2}/N^{2}$ for
$|N|\geq 2$ (and using $b\sim c\sim d\sim \Ab_0(1+O(a))$, we see that it is close to the hyperbola $4\xx\yy+2\Ab_0(\xx-\yy)=-(\xx-\yy)^2/N^2$ for $|N|>1$ and to the parabola $2\Ab_0(\xx-\yy)=-(\xx+\yy)^2$ when $|N|=1$). 

Let $(\xx_{0},\yy_{0})$ such that $|(\xx_0,\yy_0)|\leq 2r^{1/2}$: if $\det M(\xx_{0},\yy_0)\neq 0$, then
the usual stationary phase applies, unless we are in the special
condition where \cite[Lemma 2.21]{Annals} takes over and we recall it now: let  $H(\xi)$ be a smooth function defined in a neighborhood of $(0,0)$ in $\mathbb{R}^2$, such that $H(0)=0$ and $\nabla H(0)=0$.
We assume that the Hessian $H''$ satisfies $\text{rank} (H''(0))= 1$ and $\nabla \det (H'')(0)\not= 0$.
Then $\det (H'')(q,p)=0$ defines a smooth curve $\mathcal C$ near $0\in \R^2$ with $0\in \mathcal C$.
Let $s\rightarrow \xi(s)$ be a smooth parametrization of $\mathcal C$, with $\xi(0)=0$, and define the curve $\Xi(s):=H'(\xi(s))$ in $\R^2$.
\begin{lemma}\label{lemA}(\cite[Lemma 2.21]{Annals})
Let $K=\{\xi \in \mathbb{R}^2,  |\xi|\leq R_0\}$ , and $a(\xi,\Lambda)$  a classical symbol of 
order $0$ in $\Lambda\geq 1$ with $a(\xi,\Lambda)=0$ for $\xi\notin K$.
Set for $(P,Q)\in \mathbb{R}^2$ close to $0$
\begin{equation}\label{gls1}
I(\cdot,\Lambda): = \int e^{i\Lambda (<\xi,\cdot>-H(\xi))}a(\xi,\Lambda)d\xi\,.
\end{equation}
Then for $R_0>0$ small enough, the following holds true:
\begin{itemize}
\item[(a)] If $\Xi'(0)\not= 0$, there exists $C$ such that for all $(P,Q)$ close to $0$, $| I(\cdot,\Lambda) | \leq C\Lambda^{-5/6}$.
\item[(b)] If $\Xi'(0)= 0$ and $\Xi''(0)\not= 0$
there exists $C$ such that for all $\cdot$ close to $0$, $| I(\cdot,\Lambda) | \leq C\Lambda^{-3/4}$.
Moreover, if $a$ is elliptic at $\xi=0$, there exists $C'$ such that $| I((0,0),\Lambda) | \geq C'\Lambda^{-3/4}$.
\end{itemize}
\end{lemma}
Let $(\xx_0,\yy_0)$ be such $|(\xx_0,\yy_0)|\leq 2 r_0^{1/2}$ and $\det M(\xx_{0},\yy_0)= 0$.
For $(\xx,\yy)$ near $(\xx_0,\yy_0)$, $|(\xx,\yy)|\leq 2r^{1/2}_0$ let $\xi=(\xx-\xx_0,\yy-\yy_0)$ and let $R_0=4r_0^{1/2}$, then $|\xi|\leq R_0$.
We set
\[
H(\xi)=\Phi_N(\xx,\yy)-\Phi_N(\xx_{0},\yy_0)-\nabla \phi_{N,a}(\xx_{0},\yy_0)\cdot \xi, \quad (\xx,\yy)=(\xx_0,\yy_0)+\xi.
\]
We see that $H(0)=0$, $H'(0)=0$ and $H''(0)=\nabla^{2}
\phi_{N,a}(\xx_{0},\yy_0)=M(\xx_0,\yy_0)$. The matrix $M$ has two
  eigenvalues, $0$ with normalized eigenvector
  $v(\xx_{0},\yy_0)$ and $\lambda_1=\text{tr} (M(\xx_{0},\yy_0))$ with
 normalized eigenvector $u(\xx_{0},\yy_0)$. Let $Q$ be the
  matrix formed with column vectors $u$ and $v$, then $^{t} Q
  M(\xx_{0},\yy_0) Q=\text{diag}(\lambda_{1}/2,0)$, and $\lambda_{1}\neq 0$
  as the rank of $M(\xx_{0},\yy_0)\geq 1$ (thanks to $|N|\geq 2$ and
  $(t,x,y)$ close to $0$). Using \eqref{eq:80}, a simple computation yields
  \begin{equation}
    |\nabla \text{det} H''(0)|^{2}=16\Big(\xx_0+b/2-\frac{(\xx_0-\yy_0)}{2N^2}\Big)^2+16\Big(\yy_0-c/2+\frac{(\xx_0-\yy_0)}{2N^2}\Big)^2+O(\frac{\xx_0}{N^2},\frac{\yy_0}{N^2}).
  \end{equation}
We need to prove that there exists a positive constant $C_0>0$ such that 
$| \nabla \text{det} H''(0)|^{2}\geq C_0>0$. Since $b=\Ab_0(1+O(a))+O(a^{1/2}N)$, $c=\Ab_0(1+O(a))+O(a^{1/2}N)$ it follows that for $(\xx_0,\yy_0)\in \{M(\xx,\yy)=0\}$ we can write
  $|\nabla \text{det} H''(0)|^{2}=8\Ab_0^2+16(\xx_0-\yy_0)^2(1-\frac{1}{2N^2})+O(a^{1/2}N)$,
 provided $a^{1/2}N$ is small
enough and therefore $ |\nabla \text{det} H''(0)|\geq C_0>0$ for some positive constant $C_0$.
This is enough to apply Lemma \ref{lemA} (case [b]) and get the desired bounds.
When $N=1$, one may inspect the previous proofs to check that, knowing from support conditions that $(\xx,\yy)$ is bounded, for large $(\Af,\Bf)$ integrations by parts provide decay. On the other hand, in the range $(\Af,\Bf)$ bounded, one may proceed as before, using Lemma \ref{lemA}.
\end{proof}
We moreover remark that the last statement in Lemma \ref{lemA}, together with ellipticity of the symbol $\sigma_{V,h,a}$, provides, at fixed $\theta\in \mathbb{S}^{d-2}$, a sequence $(t_{N},x_{N},y_{N})$ where the bound \eqref{eq:77} saturates, which is exactly at the swallowtail singularity in space-time. This is a key point in proving Theorem \ref{disperoptimal}: for now, we have $V_{N}(t_{N},x_{N},y_{N}) \sim \|V_{N}\|_{{\infty}}$, and at $(t_{N},x_{N},y_{N})$, we have $\Af\sim \Bf\sim 0$, i.e. $\mathbb{A}_{0}\sim q^{1/6}(\theta)$, implying $t_{N}\sim 4N q^{-1/2}(y/|y|)\sqrt a$, $x_{N}\sim a$ and $y_{N}/|y_{N}\sim \vartheta_{c}\sim \theta$. Moreover, at $(t_{M},x_{N},y_{N})$ with $M\neq N$ and $|M-N|\leq M_{0}$, $\Bf\sim C N$ and $|V_{N}(t_{M},x_{N},y_{N})|$ decays much faster ($\Lambda^{-5/6}$ rather than $\Lambda^{-3/4}$). Therefore when summing over $M$ at fixed $N$, only $V_{N}$ contributes to saturating the bound and Theorem \ref{disperoptimal} holds.
\subsubsection{The sum over $2\leq N\lesssim \sqrt{a}$: completing the proof of Proposition \ref{proptang}}
We now prove \eqref{eq:37ter}.
For $N\not\in \Nrond_{1}(t,x,y)$, Proposition \ref{propcardoutN1} provides an $O_{C^{\infty}}(h^{\infty})$ contribution.  Recall that if $(\theta=(\rho,\vartheta),A,\Sigma,S)$ is a critical point
  in the phase integral defining $V_{N,a}$, then $t/(4\sqrt{a}|N|)\sim
  A^{1/2}$ and on the support of the symbol $\psi(A)$ we have $A\sim 1$. Set $\sharp N=|\Nrond_{1}(t,x,y)|$. If $\sharp N\leq 2C$, the sum reduces to a finite sum, discarding all $N\not\in \Nrond_{1}(x,y,t)$. For those $N\in \Nrond_{1}(t,x,y)$, if $|N|\leq \lambda^{1/3}$, then, collecting \eqref{eq:46} (stationary phase in $\vartheta$), \eqref{eq:169} (stationary phase in $A$) and \eqref{eq:77} (degenerate stationary phase in $(\Sigma,S)$),
  \begin{equation}
    \label{eq:87der}
    |V_{N,a}(t,x,y)|\lesssim \frac {1}{h^{d}} \left(\frac h t\right)^{\frac{d-2}2} \frac{a^{2}} h \frac{1}{ |N|^{1/2} \lambda^{1/2}} |N|^{1/4} \lambda^{-3/4}
  \end{equation}
from which, with $t\sim 4a^{1/2} N$ we get the desired result:
\begin{equation}
  \label{eq:88der}
      |V_{N,a}(t,x,y)|\lesssim \frac {1}{h^{d}} \left(\frac h t\right)^{\frac{d-2}2}  \frac{a^{1/8} h^{1/4}}{ |N|^{1/4}}\leq \frac {C_{0}}{h^{d}} \left(\frac h t\right)^{\frac{d-2}2}  \frac{a^{1/4} h^{1/4}}{ t^{1/4}}\,.
\end{equation}
If $|N|\geq \lambda^{1/3}$, we collect the same bounds but with \eqref{eq:77} replaced by \eqref{eq:64}:
  \begin{equation}
    \label{eq:89der}
        |V_{N,a}(t,x,y)|\lesssim \frac {1}{h^{d}} \left(\frac h t\right)^{\frac{d-2}2}  \frac{a^{2}} h \frac 1 {|N|^{1/2} \lambda^{1/2}}\frac 1 {\lambda^{2/3}}
  \end{equation}
and one easily checks that $a^{1/2} N^{-1/2}\lambda^{-1/6}\leq a^{1/8}
h^{1/4}N^{-1/4}$ is equivalent to $|N|\geq \lambda^{1/3}$ (using $1/N\leq \lambda^{-1/3}$ we get $h^{1/3}$ instead of $a^{1/8} h^{1/4}/|N|^{1/4}$.) Moreover, if only a finite number of $V_{N,a}(t,x,y)$ contributes, we have $a> h^{4/7}$.
We proceed with large $\sharp N(\geq 2C)$: as $\sharp
N\leq C+C \frac{t}{\sqrt{a} \lambda^{2}}$, we have $t\geq c \sqrt{a}\lambda^{2}$. We also have $|N|\sim t/(4\sqrt{a})$ and from $t \lesssim 1$, we have $a\lesssim h^{4/7}$.
 For those $N$, 
\begin{equation}
  \label{eq:90der}
\left|  \sum_{N\in \Nrond_{1}(t,x,y),|N|\sim t/\sqrt{a}} V_{N,a} \right| \lesssim  \frac {1}{h^{d}} \left(\frac h t\right)^{\frac{d-2}2} a^{1/2} \lambda^{1/3} \frac 1 {|N|} \sharp N  \lesssim \frac {1}{h^{d}} \left(\frac h t\right)^{\frac{d-2}2} h^{1/3} \lambda^{-4/3}\,,
\end{equation}
and one checks that $h^{1/3} \lambda^{-4/3} \leq h^{1/3}$ for $a\gtrsim h^{2/3}$. This completes the proof of Proposition \ref{proptang}.
\subsection{The transverse part $4a\leq \gamma\ll1$} 
We go back to $\Psi_{N,a,\gamma}(t,x,y,\Sigma,S,A,\rho)$, which is the critical value of $\tilde\Phi_{N,a,\gamma}$ after the stationary phase in $\vartheta$: when $4a\leq \gamma$ we start by applying stationary phase in $S$, as in the proof of Proposition \ref{propcardoutN1}. 
The phase has two distinct saddle points $S_{\pm}$, 
\begin{equation}\label{critScpm}
S_{\pm}^{2}+ \frac{a}{\gamma}q^{1/3}(\vartheta_{c})\Big(1+\partial_{\Xi}\Big(A_{\Gamma}(a,0,\Xi,\frac{\vartheta_{c}}{\tau_{q}(\gamma A,\vartheta_{c})})\Big)|_{\Xi=\frac{\sqrt{\gamma}S_{\pm} q^{1/3}(\vartheta_{c})}{\tau_{q}(\gamma A,\vartheta_{c})}}\Big)=A.
\end{equation}
Using \eqref{derivomegac}, ${\vartheta_c}$ (function of $\Sigma,S,A$) does not affect the second derivative in $S$ significantly and we get $\partial^2_{S}\Psi_{N,a,\gamma}|_{S_{\pm}}\sim 2$. Moreover, the critical points $S_{\pm}$ depend on $\Sigma$ only through ${\vartheta_c}$ and we have $S_{\pm}=S_0\pm\sqrt{A-O(a/\gamma)}$, where $S_0$ is the unique solution to $\partial^2_{S}\Psi_{N,a,\gamma}(t,x,y,\Sigma,S,A,\rho)=0$ and it satisfies $S_0=O(a/\sqrt{\gamma})$ (indeed, from Corollary \ref{corUpsi}, the unique solution $s_0$ to $\partial^2_{s}\tilde\Phi_{N,a,\gamma}$ satisfies $s_0=O(a)$; making the change of variables $s=\sqrt{\gamma}\rho^{1/3}S$ gives $S_0=O(a/\sqrt{\gamma})$). We are left with an integral with respect to $A,\Sigma,\rho$ and we pick a factor $\lambda_{\gamma}^{1/2}$.

 When $1\leq |N|\lesssim \lambda_{\gamma}^2$, we ignore $\rho$ : we first perform stationary phase with respect to $A$, and then prove that the remaining phase has critical points that may be degenerate of order at most two in $\Sigma$. For each $S_{\pm}$, the critical point with respect to $A$, denoted $A_{\pm}$, is such that, with $\tau=\tau_{q}(\gamma A,\vartheta_{c})$,
\begin{multline}\label{AccritScpm}
2N A^{1/2}(1-\frac 34 B_L'(\rho\lambda_{\gamma}A^{3/2}))=\frac{q^{2/3}({\vartheta_c})\Big(t-B_2(y,{\vartheta_c})+\sum_{k\geq 2}\partial_{\tau}\Big(\frac{(1-\tau)^k}{\tau^{k-1}}\Big)B_{2k}(y,{\vartheta_c})\Big)}{2\sqrt{\gamma}\sqrt{1+\gamma Aq^{2/3}({\vartheta_c})}}\\{}-\Sigma+S_{\pm}
+\frac{{\gamma}q^{2/3}({\vartheta_c})}{\sqrt{1+\gamma Aq^{2/3}({\vartheta_c})}}\Big[\frac{x}{\gamma^{3/2}}\partial_{\tau}\Big(\tau A_{\Gamma}(x,y,\sqrt{\gamma}\Sigma q^{1/3}({\vartheta_c})/\tau,{\vartheta_c}/\tau)\Big)\\
{}-\frac{a}{\gamma^{3/2}}\partial_{\tau}\Big(\tau A_{\Gamma}(a,0,\sqrt{\gamma}S_{\pm} q^{1/3}({\vartheta_c})/\tau,{\vartheta_c}/\tau)\Big)\Big].
\end{multline}
The second derivative with respect to $A$ behaves like $N/\sqrt{A}$ and with $A\sim 1$ on the support of the symbol $\psi$, stationary phase provides a factor $\lambda_{\gamma}^{-1/2}|N|^{-1/2}$.

 When $|N|>\lambda_{\gamma}^2$ we apply stationary phase in $A$ just like we did, and then stationary phase in $\rho$. We then prove that the remaining integral can be degenerate of order at most two. To apply (additional) stationary phase in $\rho$ we act exactly like in the case $\gamma \sim a$ and obtain, as in \eqref{eq:54a}
\[
\frac 1h |\partial^2_{\rho}\Psi_{N,a,\gamma}(t,x,y,\Sigma,S_{+},A_{+},\rho)|\sim \frac 1h |\partial^2_{\rho}\Psi_{N,a,\gamma}(t,x,y,\Sigma,S_{+},A_{+},\rho)|\sim \frac{|N|}{\lambda_{\gamma}}\,.
\]

Let $\phi_{N,a,\gamma,+}(t,x,y,\Sigma,\rho)=\Psi_{N,a,\gamma}(t,x,y,\Sigma,S_{+},A_{+},\rho)$ and, respectively, $\phi_{N,a,\gamma,-}(t,x,y,\Sigma,\rho)=\Psi_{N,a,\gamma}(t,x,y,\Sigma,S_{-},A_{-},\rho)$, denote the critical values of $\Psi_{N,a,\gamma}$ after the the stationary phase in $S$ and then in $A$ (notice that $A_+$ (and $A_-$, respectively) is the critical point of the phase $\Psi_{N,a,\gamma}(t,x,y,S_+,A,\rho)$ (and $\Psi_{N,a,\gamma}(t,x,y,S_-,A,\rho)$, respectively), so the critical points $S$ and $A$ are paired with same sign, either $+$ or $-$).
Let $\sigma_{V,h,\gamma,\pm}(\Sigma,\rho)$ be the symbol obtained from $$\psi(A)p_h(x,y,\theta,\gamma |\theta|^{2/3}A,\sqrt{\gamma}|\theta|^{1/3}\Sigma)\tilde q_h(\theta,\gamma |\theta|^{2/3}A,\sqrt{\gamma}|\theta|^{1/3}S)$$ after applying stationary phase in $\vartheta$, $A$ and $S$, evaluated at $(S,A)_{\pm}$; $\sigma_{V,h,\gamma}$ is independent of $N$, of order zero and has compact support in $\Sigma$. We are left with estimating 
$
\int e^{\frac ih \phi_{N,a,\gamma,\pm}}\sigma_{V,h,\gamma,\pm}d\Sigma d\rho.
$
 Following \cite{ILP3}, we state different estimates for $|N|\geq 2$, $N=0$ and $N=\pm 1$.
\begin{lemma}\label{LemN>2}
For $|N|\geq 2$ and a given point $(t,x,y)$ with $t\sim 4N\sqrt{\gamma}$, the phase $\phi_{N,a,\gamma,\pm}(t,x,y,\Sigma,\rho)$ has at most one degenerate critical point, which is of order two.
\end{lemma}
\begin{proof}
The first derivative of the phase $\phi_{N,a,\gamma,\pm}$ is 
\begin{align}\label{derivSigmaphiN}
\partial_{\Sigma}\phi_{N,a,\gamma,\pm}& =\partial_{\Sigma}\Psi_{N,a,\gamma}(t,x,y,\Sigma,S,A,\rho)|_{A_{\pm},S_{\pm}}\\
\nonumber
&=\Sigma^2+ \frac{x}{\gamma}q^{1/3}(\vartheta_{c})\Big(1+\partial_{\Xi}\Big(A_{\Gamma}(x,y,\Xi,\frac{\vartheta_{c}}{\tau_{q}(\gamma A,\vartheta_{c})})\Big)|_{\Xi=\frac{\sqrt{\gamma}S q^{1/3}(\vartheta_{c})}{\tau_{q}(\gamma A,\vartheta_{c})}}\Big)-A_c
\end{align}
Using \eqref{AccritScpm} and then \eqref{derivSigmaphiN}, we estimate derivatives w.r.t. $\Sigma$ as follows
\begin{gather}
\partial_{\Sigma} (A^{1/2}_{\pm})=-\frac{1}{2N\mp 1}(1+O(\sqrt{\gamma})+O(\lambda^{-2})+O(\partial_{\Sigma}\omega_c))=-\frac{1}{2N\mp 1}(1+O(\sqrt{\gamma}))\,,\\
\partial^2_{\Sigma}\phi_{N,a,\gamma,\pm}=2\Sigma(1+O(\sqrt{\gamma}))+\frac{A_{\pm}^{1/2}}{2N\mp 1}(1+O(\sqrt{\gamma}))\,,\quad \partial^3_{\Sigma}\phi_{N,a,\gamma,\pm}= 2+O(1/N)+O(\sqrt{\gamma})\,.
\end{gather}
For a given $t,x,y, N$, the equation $\partial^2_{\Sigma}\phi_{N,a,\gamma,\pm}=0$ has at most one solution which is a saddle point when $\partial_{\Sigma}\phi_{N,a,\gamma,\pm}=0$. Since the third order derivative stays close to $2$, degenerate stationary phase (or Van der Corput Lemma) in $\Sigma$ provides a factor $\lambda_{\gamma}^{-1/3}$.
\end{proof}
\begin{lemma}
Let $N=0$, $|t|>\gamma$, then for $|\Sigma|\lesssim 1$ we have $\partial^2_{\Sigma}\phi_{0,a,\gamma,\pm}(t,x,y,\Sigma,\rho)\sim t/\sqrt{\gamma}$.
\end{lemma}

\begin{proof}
Let $N=0$, $|t|>\gamma$, then $A_{\pm}$ solve \eqref{AccritScpm} where the term in the left hand side vanishes. The main term containing $A$ comes from $S_{\pm}=S_0\pm \sqrt{A-O(a/\gamma)}$ and $A_{\pm}$ satisfy
\begin{equation}\label{derivApm}
S_0\pm\sqrt{A_{\pm}-O(a/\gamma)}=\Sigma(1+O(\sqrt{\gamma})) - q^{2/3}(\vartheta_c)\frac{(t+O(|y|))}{2\sqrt{\gamma}}(1+O(\gamma)).
\end{equation}
Taking the derivative of \eqref{derivApm} with respect to $\Sigma$, using that $S_0$ depends on $\Sigma$ only through $\vartheta_c$ (which satisfies \eqref{derivomegac}), that $A_{\pm}$ stays close to $1$ and that $t/\sqrt{\gamma}$ is bounded yields
\begin{align*}
\partial_{\Sigma} A_{\pm}&=\pm 2\sqrt{A_{\pm}-O(a/\gamma)}+O(\sqrt{\gamma})+O(\gamma^{3/2}/|y|)\times \frac{(t+O(|y|))}{2\sqrt{\gamma}}\\
&=2\Big(\Sigma(1+O(\sqrt{\gamma})) - q^{2/3}(\vartheta_c)\frac{(t+O(|y|))}{2\sqrt{\gamma}}(1+O(\gamma))\Big)+O(\sqrt{\gamma}),
\end{align*}
and therefore, using that the saddle point $\Sigma$ is bounded since $\Sigma^2\leq A_{\pm}$, we find
\[
\partial^2_{\Sigma}\phi_{0,a,\gamma,\pm}|_{\partial_{\Sigma}\phi_{0,a,\gamma,\pm}=0}=2\Sigma(1+O(\sqrt{\gamma}))-\partial_{\Sigma}A_{\pm}
=q^{2/3}(\vartheta_c)\frac{(t+O(|y|))}{\sqrt{\gamma}}(1+O(\gamma))+O(\sqrt{\gamma})\,.
\]
When $|t|>\gamma$ we obtain the result. When $|t|\lesssim \gamma $, the wave has no time to reach the boundary.
\end{proof}
\begin{lemma}\label{lemN=1transverse}
For $N=\pm 1$: each phase function $\phi_{\pm 1,a,\gamma,\pm}(t,x,y,\Sigma,\rho)$ has at most one degenerate critical point  $\Sigma_c$ of order exactly two;  for $t\neq 0$, the equation $\partial^2_{\Sigma}\phi_{\pm 1,a,\gamma,\mp}=0$ has an unique solution $\Sigma\sim \mp1/\gamma$, while for $|\Sigma|\lesssim 1$, we have $\partial^2_{\Sigma}\phi_{\pm 1,a,\gamma, \mp}\sim (t/\sqrt{\gamma})(1+O(\gamma))$.
\end{lemma}
\begin{proof}
For $\phi_{\pm 1,a,\gamma,\pm}$, the proof of Lemma \ref{LemN>2} applies, for $\phi_{\pm 1,a,\gamma,\mp}$ the proof for  $N=0$ does.  
\end{proof}
\begin{lemma}\label{lemT>lambdasquare}
For $|N|\geq \lambda_{\gamma}^2$, after stationary phase in $\rho$, the critical value $\phi_{N,a,\gamma,\pm}(t,x,y,\Sigma,\rho_c)$ has at most one degenerate critical point of order exactly two in $\Sigma$.
\end{lemma}
\begin{proof}
The proof is essentially the same as the one of Lemma \ref{LemN>2} since the contribution from the derivatives of the critical point with respect to $\rho$ do not affect significantly the third order derivative of the phase with respect to $\Sigma$ .
\end{proof}
\subsubsection{Estimates for the sum over $N$ and end of the proof of Proposition \ref{proptransv}}
We now proceed with estimating the sum over $N$, i.e. proving \eqref{eq:37bis}.
Again, for $N\not\in \Nrond_{1}(t,x,y)$, Proposition \ref{propcardoutN1} provides an $O_{C^{\infty}}(h^{\infty})$ contribution.  With $\sharp N=|\Nrond_{1}(t,x,y)|$, if $\sharp N\leq 2C$, the sum reduces to a finite sum. For those $N\in \Nrond_{1}(t,x,y)$, we have $|N|\leq \lambda^{2}$, then, collecting \eqref{eq:46} (stationary phase in $\vartheta$), factors from stationary phase in $S,A$ and from degenerate stationary phase in $\Sigma$,
  \begin{equation}
    \label{eq:87}
    |V_{N,\gamma}(t,x,y)|\lesssim  \frac{1}{h^{d}} \left(\frac h t\right)^{\frac{d-2}2} \frac{\gamma^{2}} h \frac{1}{\lambda_{\gamma}^{1/2}} \frac{1}{ |N|^{1/2} \lambda^{1/2}_{\gamma}} \frac 1 {\lambda^{1/3}_{\gamma}}\sim \frac {1}{h^{d}} \left(\frac h t\right)^{\frac{d-2}2} \frac{h^{1/3}}{\sqrt N}\,.
  \end{equation}
If $|N|\geq \lambda_{\gamma}^{2}$, we collect the same bounds but with an additional stationary phase in $\rho$,
  \begin{equation}
    \label{eq:89}
   |V_{N,\gamma}(t,x,y)| \lesssim  \frac{1}{h^{d}} \left(\frac h t\right)^{\frac{d-2}2} \frac{\gamma^{2}} h \frac{1}{\lambda_{\gamma}^{1/2}} \frac{\lambda^{1/2}_{\gamma}}{\sqrt N}\frac{1}{ |N|^{1/2} \lambda^{1/2}_{\gamma}} \frac 1 {\lambda^{1/3}_{\gamma}}\sim \frac {1}{h^{d}} \left(\frac h t\right)^{\frac{d-2}2} \frac{h^{1/3}\lambda^{1/2}_{\gamma}}{N}\,.
  \end{equation}
As  $\sharp N \leq C+C \frac{T} {\lambda^{2}}$ and $T\sim N$, we complete the proof of Proposition \ref{proptransv} with
\begin{equation}
  \label{eq:90}
\left|  \sum_{N\in \Nrond_{1}(t,x,y)} V_{N,\gamma} \right| \lesssim  \frac {1}{h^{d}} \left(\frac h t\right)^{\frac{d-2}2}   \frac{h^{1/3}\lambda^{1/2}_{\gamma}}{N} \sharp N\lesssim \frac {1}{h^{d}} \left(\frac h t\right)^{\frac{d-2}2}   \frac{h^{1/3}}{\lambda^{3/2}_{\gamma}}\,.
\end{equation}
\begin{rmq}
In this transverse regime,   we used $\gamma\geq h^{2/3-\eps}$ rather than $a\geq h^{2/3-\eps}$. As such, all estimates hold with any $a>0$. This will be of importance in the next section.
\end{rmq}
\section{Dispersion for small $a\lesssim h^{2/3-\ceps}$}\label{secdispapetit}
We now obtain dispersion for $\mathcal{P}_{h,a}$ for small $a\in (0,h^{2/3-\ceps})$ with a small $0<\ceps<1/12$ using Propositions \ref{propdataapetitpetit1} and \ref{propdataapetitpetit2}. Write
$
\Prond_{h,a}(t,x,y):=\Prond^{1}_{h,a}+\Prond^{2}_{h,a}\,,
$
where $\Prond^{j}_{h,a}$ is given by \eqref{eq:Prond2cut} with $g_{h,a}$ replaced by $g_{h,a,j}$, $j\in\{1,2\}$. Dispersion for $\Prond^{2}_{h,a}$ easily follows using {exactly} the arguments from the transverse case (for $a>h^{\frac{2}{3}-\ceps}$): $\omega$ is large and stationary phase arguments apply. We are left with $\Prond^{1}_{h,a}$.
\begin{prop}
For $h<t\leq  T_0\leq 1$ 
\begin{equation}\label{P1disp}
\|\Prond^{1}_{h,a}(t,.)\|_{L^{\infty}_{x>0,y}}\lesssim h^{-d}\Big(\frac ht\Big)^{\frac{d-2}{2}}\Big(\frac ht\Big)^{1/3}.
\end{equation}
\end{prop}
\begin{proof}
We need the following lemma, relying on Corollary \ref{lemgomegak}: 
\begin{lemma}
There exists symbols $\sigma(\eta,\omega_k)$ 
and $r_j(a,\eta,\omega_k)$, of order $0$, $\sigma, r_0$ elliptic such that 
 \begin{multline}\label{K<} 
K_{\omega_k}(g_{h,a,1})(t,x,y)=\int e^{it\tau_{q}(\omega_k,\eta)}G(x,y,\eta,\omega_k)q^{\frac 16}(\eta)\varkappa(h\eta)\varkappa(h\tau_{q}(\omega_k,\eta))r(\eta,\omega_k)\sigma(\eta,\omega_k)\\
 \cutoffchi^{\flat}\Big(\frac{\omega_k}{\omega_{K_{\ceps}}}\Big) q^{\frac 16}(\eta)\Big(r_0Ai(-\zeta(a,0,\eta, \omega_k))+iq^{-\frac 16}(\eta)r_1Ai'(-\zeta(a,0,\eta,\omega_k))\Big)d\eta\,,
\end{multline}
where $\zeta$ is the phase introduced in Theorem \ref{thmMelrose}.
\end{lemma}
\begin{proof}
Replacing $\hat{g}_{h,a,1}$ given by Corollary \ref{lemgomegak} in formula \eqref{eq:Kequiv} yields
\[
K_{\omega_k}(g_{h,a,1})=\int e^{it\tau_{q}(\omega_k,\eta)}G(x,y,\eta,\omega_k)\chi^{\#}(\omega_k)q^{\frac 16}(\eta)\varkappa(h\eta)\varkappa(h\tau_{q}(\omega_k,\eta))r(\eta,\omega_k)\frac{\sqrt{L'(\omega_k)}}{\sqrt{2\pi}}I_a(\eta,\omega_k)d\eta\,,
\]
with $I_a(\eta,\omega_k)$ defined in \eqref{Io}. To compute $I_a$ use \eqref{defEmathcalM} and $\overline{e( x, y,\eta,\omega_k)}=e^{-i|\eta|B_0}\overline{\tilde e( x, y,\eta,\omega_k)}$
\begin{equation}\label{I<}
 I_a  = \int e^{-i(\tilde y\cdot (\eta-\tilde \eta)+|\eta| B_{0}(\tilde y,\frac{\eta}{|\eta|}))}\varkappa(h\tilde\eta)\cutoffchi^{\flat}\Big(\frac{\omega_k}{\omega_{K_{\ceps}}}\Big) e_k(a,\tilde \eta)\left[\int_{\R_{+}} \overline{\tilde e( \tilde x, \tilde y,\eta,\omega_k)}e_k(\tilde x,\tilde \eta) d\tilde x\right] 
d\tilde \eta d\tilde y\,.
\end{equation}
As in the proof of Lemma \ref{lemestimderivekk}, the bracket term in \eqref{I<} behaves like a symbol. For small $a\lesssim h^{2/3-\ceps}$ and $k\leq K_{\ceps}$, $e_k(a,\tilde\eta)$ can also be included in the symbol, and
 with $\eta=\theta/h$ and $\tilde \eta=\tilde\theta/h$, stationary phase applies in $\tilde y,\tilde \theta$ for the phase $-(\tilde y\cdot\theta+|\theta|B_{0}(\tilde y,\vartheta))+\tilde y\cdot \tilde \theta$, with large parameter $\frac 1h$ and symbol 
\begin{equation}\label{symbK<}
\frac{1}{h^{d-1}}\cutoffchi^{\flat}\Big(\frac{\omega_k}{\omega_{K_{\ceps}}}\Big)\varkappa(\tilde\theta)\chi_0(h^{2/3}\omega_kq^{2/3}(\tilde\theta))e_k(a,\tilde\theta/h)
\Big(\int_0^{\infty}  \overline{\tilde e( \tilde x, \tilde y,\theta/h,\omega_k)}e_k(\tilde x,\tilde \theta/h) d\tilde x\Big)\,.
\end{equation}
Stationary points are such that $\tilde y=0$, $\tilde \theta=\theta+|\theta|\partial_{\tilde y} B_{0}(\tilde y,\vartheta)=\theta$ (as $\partial_y B_0(0,\vartheta)=0$). All derivatives with respect to $\tilde y$ land on $ \overline{\tilde e( \tilde x, \tilde y,\theta/h,\omega_k)}$ and we can use Lemma \ref{lemestimderivekk}.
Derivatives with respect to $\tilde \theta$ land either on cut-offs, $e_k(\tilde x,\tilde \theta/h)$ or $e_k(a,\tilde \theta/h)$ : using \eqref{estimderivekk} for $e_k$,
\[
\Big|\partial^{\beta_1}_{\tilde y}\partial^{\beta_2}_{\tilde \eta}\Big(\int_0^{\infty}  \overline{\tilde e( \tilde x, \tilde y,\eta,\omega_k)}e_k(\tilde x,\tilde \eta) d\tilde x\Big)\Big|\leq \|\partial^{\beta_1}_{\tilde y}\tilde e(.,\omega_k) \|_{L^{2}_{x\geq 0}} \|\partial^{\beta_2}_{\tilde \eta} e_k \|_{L^{2}_{x\geq 0}}\lesssim (\omega_k|\eta|^{1/3})^{|\beta_1|+|\beta_2|}\,.
\]
As $|\eta|\sim 1/h$ and $k\lesssim K_{\ceps}< h^{-1/4+\ceps}$ we find $\omega_k|\eta|^{1/3} \leq h^{-1/2+2\ceps/3}$. As one derivative on $e_k(a,\tilde\theta/h)$ yields at most $\frac ah\sqrt{h^{2/3}\omega_k}$ and we have $\frac ah\sqrt{h^{2/3}\omega_k}\leq h^{-1/12-2\ceps/3}\ll h^{-1/2}$ for $\ceps<5/8$, stationary phase applies in $\tilde y,\tilde\theta$. After stationary phase, $e_k(a,\tilde\theta/h)$ transforms into a linear combination of $e_k$ and its derivative : indeed, repeated derivatives in $\tilde\theta$ lead to $Ai(aq^{1/3}(\tilde\theta/h)-\omega_k)$ and $Ai'(aq^{1/3}(\tilde \theta/h)-\omega_k)$, and we denote $r_0, r_1$ the corresponding asymptotic expansions; $r_0, r_1$ are functions of $(a,\theta/h,\omega_k)$ and $r_0$ is elliptic with main contribution equal to $1$. Taking the difference between phase functions of $e_k(a,\tilde\theta/h)$ and $e(a,0,\theta/h,\omega_k)$ yields $a\tau_{q}(h^{2/3}\omega_k,\theta)A_{\Gamma}(a,0,s q^{1/3}(\theta)/\tau_{q},\theta/\tau_{q})$, which behaves (at worst) like $a(s^2-h^{2/3}\omega_kq^{2/3}(\theta)+O(s^3,s(h^{2/3}\omega_k))$ (where $s^2\lesssim h^{2/3}\omega_k$);
as $ah^{2/3}\omega_k\leq h^{7/6-\ceps/3}$, $e^{\frac ih a\tau_{q}(h^{2/3}\omega_k,\theta)A_{\Gamma}(a,0,s q^{1/3}(\theta)/\tau_{q},\theta/\tau_{q})}$ does not oscillate and can be brought into the symbol.
The $L^2_{x>0}$ product in \eqref{I<} becomes a new symbol $\sigma(\theta/h,\omega_k)$, with main contribution $\sigma_0$
\begin{multline*}
\sigma_{0}=\int_0^{\infty}  \overline{\tilde e( \tilde x, 0,\theta/h,\omega_k)}e_k(\tilde x, \theta/h) d\tilde x=h^{-2/3}\int_0^{\infty}e^{\frac ih (-\tilde s^3/3-\tilde s(\tilde xq^{1/3}(\theta)-h^{2/3}\omega_k)+s^3/3+s(\tilde x q^{1/3}(\theta)-h^{2/3}\omega_k))}\\
\times \frac{2\pi}{L'(\omega_k)}  \frac{q^{1/3}(\theta)}{h^{2/3}}   \tilde p_h(\tilde x,\theta, h^{2/3}\omega_k,\tilde s, h/t) e^{-\frac ih \tilde x \tau_{q}(h^{2/3}\omega_k,\theta)A_{\Gamma}(\tilde x,0,\sigma q^{1/3}(\theta)/\tau_{q},\theta/\tau_{q})} d\tilde s ds d \tilde x,
\end{multline*}
for a symbol $\tilde p_h$ obtained from $p_h$ in \eqref{eq:qh}. For $\tilde x\gtrsim h^{2/3}\omega_k$, repeated integrations by parts in $s$ yield an $O(h^{\infty})$ contribution. For $\tilde x\lesssim h^{2/3}\omega_k$, apply stationary phase in $s,\tilde x$: as $\tilde xq^{1/3}(\theta)=h^{2/3}\omega_k-s^2$ and $s q^{1/3}(\theta)=\tilde s q^{1/3}(\theta)+\tau_{q} (A_{\Gamma}+\tilde x\partial_{\tilde x}A_{\Gamma})$, we are left with an integral in $\tilde s$ with phase $e^{\frac ih (h^{2/3}\omega_k-\tilde s^2)q^{1/3}(-\theta)\tau_{q}(h^{2/3}\omega_k,\theta)A_{\Gamma}((h^{2/3}\omega_k-\tilde s^2)q^{-1/3}(\theta),0,\tilde s q^{1/3}(\theta)/\tau_{q},\theta/\tau_{q})}$. One derivative with respect to $\theta$ on this phase yields at most $\frac 1h (h^{2/3}\omega_k)^2\lesssim  h^{4\ceps/3}$ for $k\leq K_{\ceps}\leq h^{-1/4+\epsilon}$ and therefore $\sigma_0$ is of order $0$. Rewriting $K_{\omega_{k}}(g_{h,a,1})$ after stationary phase, relabelling $r_0$ and $r_1$, we obtain \eqref{K<}.
\end{proof}
Now we evaluate the $L^{\infty}$ norm of $\Prond^{1}_{h,a}(t,.)$. If $d\geq 3$, we perform stationary phase in \eqref{K<} w.r.t $\eta/|\eta|\in \mathbb{S}^{d-2}$: $\eta=\theta/h$, $\theta=|\theta|\vartheta$, $\theta\in \text{supp}(\varkappa)$, the phase of each $K_{\omega_k}(g_{h,a,1})(t,x,y)$ is
\[
t\tau_{q}(h^{2/3}\omega_k,\theta)+|\theta|(y\cdot\vartheta+B_0(y,\vartheta))+O(h^{2/3}\omega_k)=|\theta|(t+y\cdot\vartheta+B_0(y,\vartheta))+O(h^{2/3}\omega_k),
\]
where $O(h^{2/3}\omega_k)$ contains contributions from $B_{\Gamma}-B_0$, $xA_{\Gamma}$ and $aA_{\Gamma}$. Critical points $\vartheta_{\pm}$ are given by \eqref{eqNcontYmodgen} (with $\gamma=a$, $A=h^{2/3}\omega_k/(a|\theta|^{2/3})$, $S=s/(\sqrt{a}|\theta|^{1/3})$, $\Sigma=\sigma/(\sqrt{a}|\theta|^{1/3}$)). As $\nabla^2_{\vartheta}(y\cdot\vartheta+B_0(y,\vartheta))\sim |y|$ and $|y|\sim |t|$ (Lemma \ref{derriere}), stationary phase in $\vartheta$ yields $({h}/{|\theta| |y|})^{d-2} \leq C({h}/{|t|})^{d-2}$ for $\vartheta$ near $\vartheta_{\pm}$. For $\vartheta$ outside a small neighborhood of $\vartheta_{\pm}$, we get $O(\frac{h}{t})^{\infty}$. Stationary phase yields new symbols as asymptotic expansions with small parameter $\frac h t$; the main contribution  $K_{\omega_k}(g_{1,a,h})(t,x,y)$ remains similar to \eqref{K<}, with a front factor $(h/t)^{(d-2)/2}$, $\eta=|\theta| \vartheta_{\pm}/h$ and integration over $\eta$ replaced by integration over $|\theta|$.
\begin{lemma}( \cite[Lemma 3.5]{Annals})
There exists $C_0$ such that for $L\geq 1$,
\begin{equation}\label{estairy2}
\sup_{b\in \R}\Big (\sum_{1\leq k\leq L}\omega_k^{-1/2}Ai^2(b-\omega_{k})\Big)  \leq C_{0}L^{1/3}\,, \quad 
\sup_{b\in\R_+}\Big (\sum_{1\leq k \leq L}\omega_k^{-1/2}h^{2/3}Ai'^2(b-\omega_{k})\Big)  \leq C_{0}h^{2/3}L\,.
\end{equation}
\end{lemma}
Applying Cauchy-Schwarz, dispersion for small $t$ reduces to estimates like \eqref{estairy2}, as 
\begin{equation}\label{estimmodes}
\|  \sum_{k=1}^{L}\frac{2\pi}{L'(\omega_k)}K_{\omega_k}(g_{h,a,1})(t,x,y)\|_{L^{\infty}}\lesssim h^{-d}\Big(\frac ht\Big)^{\frac{d-2}{2}}h^{1/3}\Big(L^{1/3}+h^{1/3}L^{2/3}+h^{2/3}L\Big),
\end{equation}
where $h^{-(d-1)}(h/t)^{\frac{d-2}{2}}$ comes from stationary phase in $\vartheta$, while $h^{-2/3}=h^{-1}h^{1/3}$ arises from $q^{1/3}(\theta/h)$ and $L$-related terms come from \eqref{estairy2}. 
Recall $h^{-2\ceps}\leq K_{\ceps}\leq h^{-1/4+\ceps}$, hence $h^{\ceps}\gg h^{2\ceps}\geq \frac{1}{K_{\ceps}}$. Let $L= h^{-\ceps}$ : for $t\leq h^{\ceps}$ we bound $L\leq \frac 1t$ in \eqref{estimmodes} and get \eqref{P1disp} for the sum up to $h^{-\ceps}$.

When $h^{\ceps}\leq t$, we apply stationary phase in each oscillatory integral in the sum over $k$ for $k\leq  h^{-\ceps}(\ll K_{\ceps})$. Using Corollary \ref{corUpsi} in the Appendix, we reformulate the phase $\psi$ of Theorem \ref{thmMelrose} as
\begin{equation}\label{phasestat}
\psi(x,y,\theta,\omega_k)=y\cdot \theta+\tau_{q}(h^{2/3}\omega_k,\theta)B_{\Gamma}(y,\theta/\tau_{q})+\Upsilon(x,y,\theta,h^{2/3}\omega_k).
\end{equation}
\begin{lemma}\label{lemphasestatksmall}
The stationary phase theorem with respect to $|\theta|$ applies in \eqref{K<} with large parameter $t\omega_kh^{-1/3}$, 
with phase function $\phi_{h,k}(t,x,y,\theta)=t\tau_{q}(h^{2/3}\omega_k,\theta)+\psi(x,y,\theta,\omega_k)$ and with symbol
\begin{multline}\label{symbstat}
\sigma_{h,k}=\Big(p_0Ai(-\zeta(x,y,\theta/h,\omega_k))+ip_1Ai'(-\zeta(x,y,\theta/h,\omega_k))\Big)q^{\frac 13}(\theta)\varkappa(\theta)\varkappa(\tau_{q}(h^{\frac 23}\omega_k,\theta)) \chi^{\flat}\Big(\frac{\omega_k}{\omega_{K_{\ceps}}}\Big)\\
\times  \sigma(\theta/h,\omega_k)r(\theta/h,\omega_k)
\Big(r_0Ai(-\zeta(a,0,\theta/h,\omega_k))+iq^{-1/6}(\theta/h)r_1Ai(-\zeta(a,0,\theta/h,\omega_k))\Big)\,,
\end{multline}
where $p_0, p_1$ were defined in  \eqref{eq:defG} and $\theta=|\theta|\vartheta$ with $\vartheta\in\{ \vartheta_{\pm}\}$.
\end{lemma}
\begin{proof}
Let $r=|\theta|$ and recall $\vartheta\in\{\vartheta_{\pm}\}$ (we already performed stationary phase in $\vartheta$). We have
 \begin{multline}
\phi_{h,k}=t\tau_{q}(h^{2/3}\omega_k,\theta)+\psi(x,y,\theta,\omega_k)= r\Big(t+y\cdot \vartheta+B_0(y,\vartheta)\Big)\\
+\frac{r^{1/3}}{2}h^{2/3}\omega_k q^{2/3}(\vartheta)\Big(t+B_0(y,\vartheta)-B_2(y,\vartheta)+O(|y|(h^{2/3}\omega_k))\Big)+(rt)O((h^{2/3}\omega_k)^2)\\
-x\big(O(x)+O(h^{2/3}\omega_k)\big)-k_1\zeta\Big(\zeta+O(xh^{2/3}\omega_k)+O((h^{2/3}\omega_k)^2)\Big),
 \end{multline}
 where, in the second line, $O(|y|(h^{2/3}\omega_k))$ comes from $B_{\Gamma}-B_0-B_2$, and $(rt)O((h^{2/3}\omega_k)^2)$ comes from $t\tau_{q}(h^{2/3}\omega_k,r\vartheta)$ and where the last line represents the contribution from $\Upsilon$ as given in \eqref{eq:Ups}.

Here the critical points $\vartheta_{\pm}$ are solutions to Equation \eqref{eqNcontYmodgen}, but with $\gamma A$ replaced with $h^{2/3} \omega_{k}/r^{2/3}$ in the whole equation. Then, $y\cdot \vartheta_{\pm}+B_0(y,\vartheta_{\pm})=\pm |y+\nabla B_0(y,\vartheta_{\pm})|+O(a^2)$, $a^2\ll h$. Moreover, $\partial_r\vartheta_{\pm}=O(h^{2/3}\omega_k)$ (deriving \eqref{eqNcontYmodgen}; originally in \eqref{eqNcontYmodgen} we made a suitable change of variables and got rid of $r=|\theta|$, which is no longer possible as we now work with fixed $\omega_k$).
 A critical point with respect to $r$ for $\phi_{h,k}$ is such that
\begin{multline}\label{eqcritr}
0 =t \pm |y+\nabla B_0(y,\vartheta_{\pm})|+O(a^2)+r O(|y|^2)O(h^{2/3}\omega_k)\\
{}+\frac 23 r^{-2/3} h^{2/3}\omega_k q^{2/3}(\vartheta_{\pm})\Big(t+B_0(y,\vartheta_{\pm})-B_2(y,\vartheta_{\pm})+O(|y|(h^{2/3}\omega_k))\Big)+tO((h^{2/3}\omega_k)^2)\\
{}
-\partial_r k_1\zeta^2(x,y,\theta,h^{2/3}\omega_k)+O(x^2,x(h^{2/3}\omega_k))\,,
\end{multline}
where the last line is  $\partial_{r}\Upsilon$, whose main contribution is $-\partial_r k_1\zeta^2(x,y,\theta,\omega_k)$ : as $h^{2/3}\zeta(x,y,\theta/h,\omega_k)=\zeta(x,y,\theta,h^{2/3}\omega_k)=h^{2/3}\omega_k-xq^{1/3}(\theta)e_0(x,y,\vartheta,h^{2/3}\omega_k/|\theta|)$, $\partial_{r}\zeta$ yields a factor $x$, which puts everything else in $O(x^2,x(h^{2/3}\omega_k))$. Moreover, set $t(>h)>0$ as usual, only $\vartheta_-$ is to be kept (at $\vartheta_+$ the phase will be non-stationary).
As $B_0(y,\vartheta_{\pm})-B_2(y,\vartheta_{\pm})=O(|y|^2)=O(t^2)$ and $x\lesssim h^{2/3-\ceps}$, the main contribution of $\partial^{2}_{r}\phi_{h,k}$ comes from first terms on the second and last lines of \eqref{eqcritr}
\begin{multline}\label{eq:estimsecondderivr}
-\frac 49  r^{-5/3} h^{2/3}\omega_kq^{2/3}(\vartheta_{\pm})t\Big(1+O(|y|^2/t)+O(|y|/t)\times (h^{2/3}\omega_k))+O((h^{2/3}\omega_k)^{2})\Big)\\
-\partial^2k_1 \zeta^2(x,y,\theta,h^{2/3}\omega_k)+O(x^2,x(h^{2/3}\omega_k)).
\end{multline}
Observe that $h^{2/3}\omega_k t\gg \zeta^2(x,y,\theta,h^{2/3}\omega_k)$ for all $t>h^{\ceps}$: indeed, $h^{\ceps}\gg h^{2/3}\omega_k$, which holds true for all $\omega_k\leq \omega_{K_{\epsilon}}\leq h^{-1/6+2\ceps/3}$. At the stationary point satisfying \eqref{eqcritr}, $|y|$ behaves like $t$, and therefore $\partial_{r}^{2}\phi_{h,k}$ behaves like $t h^{2/3}\omega_k$  there.
We are left to prove that stationary phase indeed applies with large parameter $(th^{2/3}\omega_k)/h\sim t\omega_kh^{-1/3}$, for $h^{\ceps}\lesssim  t$ and symbol $\sigma_{h,k}$ from \eqref{symbstat}.
Computing  $\partial_{r}^{2}\sigma_{h,k}$ yields at most $\omega_k^{11/4}$ which occurs when $\zeta(x,y,\theta/h,\omega_k)$ is large, $\zeta(a,0,\theta/h,\omega_k)$ is bounded and when both derivatives fall on of the first Airy: hence, it suffices to check $\frac{h^{1/3}}{t\omega_k}\omega_k^{11/4}\ll1 $ for $t\geq h^{\ceps}$, 
 uniformly in $k\lesssim h^{-\ceps}$ : this does hold as long as $\ceps<2/13$ and we already set $\ceps<1/12$. \end{proof}
We then sum up to $k\leq h^{-\ceps}$, for $t\geq h^{\ceps}$ and $\ceps<1/12$, with $L'(\omega_{k})\sim 2 \omega_{k}^{\frac 1 2}$:
\begin{equation}
\|\Prond^{1}_{h,a}(t,x,y)\|_{L^{\infty}}\lesssim \frac{h^{1/3}}{h^{d}}\Big(\frac{h}{t}\Big)^{\frac{d-2}2}\sum_{ k\lesssim h^{-\ceps}}\frac{1}{L'(\omega_k)}\Big(\frac{h^{1/3}}{t\omega_k}\Big)^{\frac 1 2}
\sim \frac 1 {h^{d}}\Big(\frac{h}{t}\Big)^{\frac{d-2}2}\Big(\frac ht\Big)^{\frac 1 2}h^{-\frac{\ceps}{3}}\lesssim \frac 1 {h^{d}}\Big(\frac{h}{t}\Big)^{\frac{d-2}2}\Big(\frac ht\Big)^{\frac 1 3}\,.
\end{equation}
It remains to deal with $h^{-\ceps}\leq k\leq K_{\ceps}$ and $t\in (h,T_0]$. Taking $L=K_{\ceps}$ in \eqref{estimmodes}, we obtain \eqref{P1disp} exactly as before for all $t\leq K_{\ceps}^{-1}$ (up to $k\leq  K_{\ceps}$). Hence we are left with $t\geq K_{\ceps}^{-1}$ and $h^{-\ceps}\leq k\lesssim K_{\ceps}$. It suffices to consider $0\leq x,a\lesssim h^{\frac 23-\ceps}$ (as for $x\geq h^{\frac 23-\ceps}$, the arguments for the transverse case $a>h^{\frac 23-\ceps}$ apply). 
For small values of both $x,a$, $|\zeta(x,y,\eta,\omega_k)|\geq \omega_k/2$ and $|\zeta(a,0,\eta,\omega_k)|\geq \omega_k/2$ ($k\geq h^{-\ceps}$ is large so $\omega_k\geq h^{-2\ceps/3}\gg h^{-\ceps}\gtrsim a/h^{2/3}$). We can write $Ai(-\zeta)=\sum_{\pm}A_{\pm}(\zeta)$ which give rise to factors $\zeta(x,y,\eta,\omega_k)^{-1/4}\zeta(a,0,\eta,\omega_k)^{-1/4}$. If we decide to ignore the integral with respect to $r\in \text{supp}\varkappa$,  we can immediately bound the sum of integrals with
\begin{equation}
\|\Prond^{1}_{h,a}(t,x,y)\|_{L^{\infty}}\lesssim \frac{h^{1/3}}{h^{d}}\Big(\frac{h}{t}\Big)^{\frac{d-2}2}\sum_{h^{-\ceps}\leq k\leq K_{\ceps}}\frac{1}{\omega_k^{1/2}}\frac{1} {\omega_k^{1/2}}
\lesssim h^{-d}\Big(\frac{h}{t}\Big)^{\frac{d-2}2}(hK_{\ceps})^{1/3}\,,
\end{equation}
and as $K_{\ceps}\leq h^{-1/4+\ceps}$,  $(hK_{\ceps})^{1/3}\leq h^{1/4+\ceps/3}$ which is not optimal (a similar crude bound was obtained in \cite{Annals} in that regime.) To obtain sharper bounds when $t\geq K_{\ceps}^{-1}$, $h^{-\ceps}\leq k\leq K_{\ceps}$, decompose ${Ai}(-z)=A_{+}(z)+A_{-}(z)$, with $A_{\pm}(z)=\Psi(e^{\mp i\pi/3} z)e^{\mp\frac 23 i z^{3/2} }$ and 
$4\pi^{3/2}|\Psi(z)|\sim |z|^{-1/4}$ for large $z$. Then there are four different phases in $K_{\omega_k}(g_{h,a,1})$, where $\pm_{1}$ and $\pm_{2}$ mean idependent signs,
\[
\frac 1h \Big(t\tau_{q}(h^{2/3}\omega_k,\theta)+\psi(x,y,\theta,\omega_k)\pm_{1} \frac 23 \zeta(x,y,\theta,h^{2/3}\omega_k)^{3/2}\pm_{2} \frac 23 \zeta(a,0,\theta,h^{2/3}\omega_k)^{3/2}\Big)
\]
and a symbol which behaves like ${h^{-(d-1)-2/3}}q^{1/3}(\theta)\varkappa(\theta)|\zeta(x,y,\theta/h,\omega_k)|^{-1}|\zeta(a,0,\theta/h,\omega_k)|^{-1/4}$, where $h^{-2/3}\zeta(x,y,\theta,h^{2/3}\omega_k)=\zeta(x,y,\theta/h,\omega_k)=\omega_k-xq^{1/3}(\theta)/h^{2/3}e_0(x,y,\theta/h,\omega_k)\geq \omega_k/2$ and also $\zeta(a,0,\theta/h,\omega_k)=\omega_k-aq^{1/3}(\theta/h)e_0(a,0,\theta/h,\omega_k)\geq \omega_k/2$. We  prove that stationary phase applies with the same "large parameter" $th^{2/3}\omega_k/h$ but a new symbol. 
As in \eqref{eqcritr}, we find 
\begin{multline*}
t - |y+\nabla B_0(y,\vartheta_{-})|+r O(|y|^2)O(h^{\frac 23}\omega_k)
+\frac 23 r^{-\frac 23} h^{\frac 23}\omega_k q^{\frac 23}(\vartheta_{-})\Big(t+B_0(y,\vartheta_{-})-B_2(y,\vartheta_{-})-x\mu(y,\vartheta_{-})\Big)\\{}+O((h^{\frac 23}\omega_k)^2)\pm_{1} x\partial_r(q^{\frac 13}(r\vartheta_{-})e_0)\sqrt{ \zeta(x,y,r\vartheta_{-},h^{\frac 23}\omega_k)}\pm_{2} a\partial_r(q^{\frac 13}(r\vartheta_{-})e_0)\sqrt{ \zeta(a,0,r\vartheta_{-},h^{\frac 23}\omega_k)}=0.
\end{multline*}
The second derivative is
\begin{multline*}
-\frac 49  r^{-5/3} h^{2/3}\omega_kq^{2/3}(\vartheta_{-})(t+O(|y|^2))+O((h^{2/3}\omega_k)^{2})) \mp_{1} xO(\sqrt{h^{2/3}\omega_k})(1+O((h^{2/3}\omega_k)^2)\\
 {}\mp_{2} aO(\sqrt{h^{2/3}\omega_k})(1+O((h^{2/3}\omega_k)^2)),
\end{multline*}
and for $t\geq K_{\ceps}^{-1}$, $h^{-\ceps}\leq k\leq K_{\ceps}$, the main contribution comes from the first term : indeed, $t\sqrt{h^{2/3}\omega_k}\geq K_{\ceps}^{-1}h^{1/3(1-\ceps)}\geq h^{7/12-4\ceps/3}\gg h^{2/3-\ceps}$, hence $t(h^{2/3}\omega_k)\gg \max\{x,a\}O(\sqrt{h^{2/3}\omega_k})$ for all $x,a\lesssim h^{2/3-\ceps}$. We then check that $th^{2/3}\omega_k/h$ is a large parameter for $t> K_{\ceps}^{-1}$ (recall that in Lemma \ref{lemphasestatksmall} we had $t\geq h^{\ceps}\gg K_{\ceps}^{-1}$): indeed, $t\omega_k\geq K_{\ceps}^{-1}h^{-2\ceps/3}\geq h^{1/4-5\ceps/3}\gg h^{1/3}$. We find
\begin{equation}
\|\Prond^{1}_{h,a}(t,x,y)\|_{L^{\infty}}\leq \frac{h^{1/3}}{h^d}\Big(\frac{h}{t}\Big)^{\frac{d-2}2}\sum_{1\leq k\lesssim K_{\ceps}}\frac{1}{L'(\omega_k) \omega_k^{1/4+1/4}}\Big(\frac{h^{1/3}}{t\omega_k}\Big)^{1/2}
\sim h^{-d}\Big(\frac{h}{t}\Big)^{\frac{d-2}2}\Big(\frac ht\Big)^{1/2}\log{h^{-1}}\,,
\end{equation}
where we use the symbol of both factors $A_{\pm}$ to get $|\zeta|^{-1/4}<2/\omega_{k}^{1/4}$. We proved \eqref{P1disp} and therefore Theorem \ref{dispintermediaire}.
\end{proof}
\section{Strichartz estimates}
Standard duality and interpolation arguments would lead to Theorem \ref{thStri} with $\upgamma(d)=1/4$ if one only uses dispersion formula \eqref{dispco}. We aim at improving on this straightforward application of our dispersion estimate, combining three ingredients, where the first two were already proven earlier: in our estimates from Propositions \ref{propN=1}, \ref{proptransv}, \ref{proptang} and \ref{P1disp}, we either have better decay or additional (small) factors of $a$ or $\gamma$; we also have the usual dispersion on time intervals with lengh less than $\sqrt \gamma$, in the sum of reflections regime (corresponding to dispersion on the $V_{0}$ term in the expansion over $N$); and finally, we claim that the additional cut-off that we introduced to localize $h^{2/3}\omega\sim \gamma$ is not only a useful technical device but also essentially commutes with the flow itself. Assuming these facts, one can then prove Strichartz estimates at fixed $\gamma$: on intervals of length one, with the $\upgamma(d)=1/4$ loss but with a constant that depends on a positive power of $\gamma$; on intervals of size $\sqrt\gamma$, one has the usual Strichartz estimates, with $\upgamma(d)=0$, and by iteration, one has a corresponding Strichartz estimate on time length one but with a constant that depends on a negative power of $\gamma$; by interpolation, one recovers a Strichartz estimate, at fixed $\gamma$, that improves upon $\upgamma(d)=1/4$, and then summation over $\gamma$ yields the theorem, with strict inequality on $\upgamma(d)$ because of that last summation over $\gamma$. Such a strategy was implemented in \cite{ILP4} and is facilitated there because the $\gamma$ localization commutes with the wave flow. Moreover, a refined analysis of the dispersion estimates around the exceptional times that force $\upgamma(d)=1/4$ allows for a better Strichartz estimate to interpolate with, thus yielding a better result on the 2D model. We expect such result to generalize to the 2D general convex case, while higher dimensions will require additional arguments related to the space-time localization of swallowtails (which can no longer be a set of discrete, exceptional, times). As such, we believe that Theorem \ref{thStri} is of interest as it illustrates, in a relatively simple way, that the rate of dispersion does tell the whole story as far as Strichartz estimates are concerned.

We now turn to the details: we relabel $\mathcal{P}_{h}(t,x,y,a,b)$ our parametrix but with source $(a,b)\in \R^{d}$, and similarly $\mathcal{P}_{h,\gamma}(t,x,y,a,b)$ where an additional cut-off in $\alpha$ was inserted. Inhomogeneous Strichartz estimates follow from estimating the inhomogeneous operator $P_{h}$ (and/or $P_{h,\gamma}$ with obvious notations)
$$
P_{h}f(t,x,y)=\int_{s} \mathcal{P}_{h}(t-s,x,y,a,b) f(s,a,b)\,dadbds\,.
$$
We remark that, with such notations, the homogeneous operator approximating the half-wave flow with data $f_{0}(a,b)$ at time $t=0$ is
$$
(W_{h}f_{0})(t,x,y)=P_{h}(\delta_{s=0}f_{0}(a,b))(t,x,y)\,,
$$
and we define similarly $W_{h,\gamma}$. Now, suppose we restrict ourselves to ${P}_{h}^{\flat}=\sum_{\gamma<h^{1/3}} {P}_{h,\gamma}$: for $\mathcal{P}_{h,\gamma}$ with $\gamma<h^{1/3}$, we have $a^{1/4}h^{1/4}\lesssim h^{1/3}$ and, collecting all bounds from Propositions \ref{propN=1}, \ref{proptransv}, \ref{proptang} and \ref{P1disp} in our regime, we do have
\begin{equation}
  \label{eq:8}
|\mathcal{P}_{h}^{\flat}(t,x,y,a,b)|\lesssim \frac 1 {h^{d}}  \Bigl(\frac h t\Bigr)^{(d-2)/2} \min \left( 1, \Bigl(\frac h t\Bigr)^{1/3}\right)\,.
\end{equation}
Commuting the wave flow and the $\alpha\sim \gamma$ localization in the parametrix is an issue, to be adressed in Proposition \ref{commutgamma} below; we already observe that $\mathcal{P}_{h,\gamma}$ is an approximate solution to the wave equation, and as such, satisfies energy estimates.  We then obtain inhomogeneous Strichartz estimates but with $\upgamma(d)=1/6$ rather than $1/4$, after interpolation between \eqref{eq:8} and energy estimates, followed by the argument from \cite{KT98} for the endpoint (only required for $d\geq 4$) and a dependence on $h$ dictated by scaling, which we encode with $\lesssim_{h}$:
\begin{equation}
  \label{eq:25}
      \|P_{h,\gamma} f\|_{L^{2}_{t} L^{2\frac{3d-4}{3d-7}}_{x}} \lesssim_{h} \| f\|_{L^{2}_{t} L^{2\frac{3d-4}{3d+3}}_{x}}\,,\quad\text{or},\,\text{for}\,d=3,\, \,\, \|P_{h,\gamma} f\|_{L^{\frac 12 5}_{t}L^{\infty}_{x}}\lesssim_{h} \|f\|_{L^{\frac {12} 7}_{t}L^{1}_{x}}\,.
\end{equation}
For now, we focus on a given $\gamma>h^{1/3}$.
First, we remark that ${P}_{h,\gamma}$ satisfies inhomogeneous Strichartz estimates with $\upgamma(d)=1/4$, but with a constant related to $\gamma$ (from the $a^{1/4}$, $\gamma^{1/4}$ factors in \eqref{eq:37}, \eqref{eq:37bis}, \eqref{eq:37ter})
\begin{equation}
  \label{eq:9}
  \|P_{h,\gamma} f\|_{L^{2}_{t} L^{2\frac{2d-3}{2d-7}}_{x}} \lesssim_{h} \gamma^{\frac 1 {2d-3}} \| f\|_{L^{2}_{t} L^{2\frac{2d-3}{2d+4}}_{x}}\,,\quad\text{or},\,\text{for}\,d=3,\, \,\, \|P_{h,\gamma} f\|_{L^{\frac 8 3}_{t}L^{\infty}_{x}}\lesssim_{h} \gamma ^{\frac 1 4} \|f\|_{L^{\frac 8 5}_{t}L^{1}_{x}}\,.
\end{equation}
This follows from performing the usual duality and interpolation argument directly on $P_{h,\gamma}$ rather than $P_{h}$, using Proposition \ref{commutgamma} to handle commutation of the localization with the flow.
At the same time, $W_{h,\gamma}$ satisfies the usual short time dispersion (either from \cite{blsmso08} or from \eqref{eq:36} and close inspection of the $N=1$ cases, especially Lemma \ref{lemN=1transverse}), hence homogeneous Strichartz estimates on a time interval of size $\gamma^{1/2}$ (this may be seen as a direct consequence of short time Strichartz estimates proven in a more general context in \cite{blsmso08}.) With the help of Proposition \ref{commutgamma}, one may sum such $L^{2}_{t}$ estimates over $\gamma^{-1/2}$ intervals of length $\gamma^{-1/2}$, obtaining 
$$
\|W_{h,\gamma} f_{0}\|_{L^{2}_{t} L^{2\frac{d-1}{d-3}}_{x}} \lesssim_{h} \gamma^{-\frac 1 {4}} \| f_{0}\|_{2}
$$
and then revert to inhomogenous estimates,
\begin{equation}
  \label{eq:10}
    \|P_{h,\gamma} f\|_{L^{2}_{t} L^{2\frac{d-1}{d-3}}_{x}} \lesssim_{h} \gamma^{-\frac 1 {2}} \| f\|_{L^{2}_{t} L^{2\frac{d-1}{d+1}}_{x}}\,.
\end{equation}
(abusing the enpoint for $d=3$, for which exponents are to be shifted to avoid the forbidden endpoint). One then interpolates between \eqref{eq:9} and \eqref{eq:10} to retain $\gamma^{\varepsilon(q)}$, with $\varepsilon(q)>0$ dictated by interpolation (to sum over $\gamma$ later),
\begin{equation}
  \label{eq:11}
  \|P_{h,\gamma} f\|_{L^{2}_{t} L^{q}_{x}} \lesssim_{h}  \gamma^{\varepsilon(q)}\| f\|_{L^{2}_{t} L^{q'}_{x}}\,,\quad\text{or},\,\text{for}\,d=3,\, \,\, \|P_{h,\gamma} f\|_{L^{\frac {12} {5-12\varepsilon}}_{t}L^{\frac 1 \varepsilon}_{x}}\lesssim_{h} \gamma ^{\varepsilon} \|f\|_{L^{\frac {12} {7+12\varepsilon}}_{t}L^{\frac 1 {1-\varepsilon}}_{x}}\,.
\end{equation}
with $1/2-1/q>2d/((d-1)(2d-1))$ for $d\geq 4$: this corresponds to $\upgamma(d)>1/4-1/4d$ for $d\geq 3$. Taking advantage of $\gamma^{\varepsilon}$, we sum over dyadic $\gamma$'s to get estimates for $P^{\sharp}_{h}=\sum_{\gamma>h^{1/3}} P_{h,\gamma}$, and the resulting estimate is always worse than \eqref{eq:9} for $P^{\flat}_{h}$, that is to say, $q>2(3d-4)/(3d-7)$. As such, $P_{h}=P^{\sharp}_{h}+P^{\flat}_{h}$ satisfies the same set of estimates as $P^{\sharp}_{h}$, that is, Strichartz estimates with $\upgamma(d)>1/4-1/4d=1/6+(1/4)(1/3-1/d)$. In particular, for $d=3$, we obtain (except for the endpoint) the set of Strichartz estimates that would result from  dispersion with a loss of only $1/6$, which is the loss resulting from cusp-like singularities in the wave front. This completes the proof of Theorem \ref{thStri}, up to proving Proposition \ref{commutgamma}.

It remains to handle the issue of localization with respect to $\alpha$ in our parametrix. For this, we check that we do have a suitable form of the group property for the operator $W_{h,\gamma}(t)$: assume that $\tilde W_{h,\gamma}$ is the same operator (with kernel denoted by $\mathcal{\tilde P}_{h,\gamma}$) but with a cut-off in $\alpha$, denoted $\tilde\chi_1$, which is the identity on the support of the cut-off $\chi_1$ in $W_{h,\gamma}$.
\begin{prop}\label{commutgamma}
  We have $\tilde W_{h,\gamma}(t_{1})\circ W_{h,\gamma}(t_{2})=W_{h,\gamma}(t_{1}+t_{2})$, modulo $O(h^{\infty})$. In terms of kernels,
  \begin{equation}
    \label{eq:7}
    \mathcal{P}_{h,\gamma}(t_1+t_2,x,y;{a,b})=\int_{\tilde x,\tilde y}\mathcal{ \tilde P}_{h,\gamma}(t_1,x,y;{\tilde x,\tilde y}) \mathcal{P}_{h,\gamma}(t_2,\tilde x,\tilde y;a,b) \,d\tilde x d\tilde y+O(h^{\infty})\,.
  \end{equation}
  Similarly, for $\gamma_{1}$ and $\gamma_{2}$ such that the corresponding cut-offs have disjoint support, $ W_{h,\gamma_{1}}(t_1)\circ W_{h,\gamma_{2}}(t_2)=O(h^{\infty})$ and
  \begin{equation}
    \label{eq:7bis}
    \int_{\tilde x,\tilde y}\mathcal{P}_{h,\gamma_1}(t_1,x,y;{\tilde x,\tilde y}) \mathcal{P}_{h,\gamma_2}(t_2,\tilde x,\tilde y;a,b) \,d\tilde x d\tilde y=O(h^{\infty})\,.
  \end{equation}
\end{prop}
\begin{proof}
Recall that $\mathcal{P}_{h,\gamma}$ is given by \eqref{defProndgamma} with  $K_{\omega,\gamma}$ defined in \eqref{KKKgam}. Then we rewrite
\[
\int_{\tilde x,\tilde y}\mathcal{ \tilde P}_{h,\gamma}(t_1,x,y;{\tilde x,\tilde y}) \mathcal{P}_{h,\gamma_1}(t_2,\tilde x,\tilde y;a,b) \,d\tilde x d\tilde y=\sum_{N_1, N_2}\int e^{-iN_1 L(\omega)}\cutoffchi^{\flat}(h^{2/3}\omega/\ceps_{0}) J_{N_2,\gamma_1, \gamma_2}(t_1, t_2,x,y,\omega)d\omega,
\]
where we set
 \begin{align*}
 J_{N_2,\gamma_1, \gamma_2}(t_1,t_2,x,y,\omega) = &\int_{\tilde x,\tilde y} \tilde K_{\omega,\gamma_1}(g_{h,(\tilde x, \tilde y)})(t_1,x,y)F_{N_2, t_2,\gamma_2}(\tilde x,\tilde y)d\tilde x d\tilde y\,, \\
 F_{N_2, t_2, \gamma_2}(\tilde x,\tilde y)= &e^{-i N_2L(\tilde\omega)}\cutoffchi^{\flat}(h^{2/3}\tilde\omega/\ceps_{0})K_{\tilde\omega,\gamma_2}(g_{h,(a,b)})( t_2, \tilde x,\tilde y)\,,
 \end{align*}
 and $\tilde K_{\omega,\gamma_1}(g_{h,(\tilde x, \tilde y)})$ has the same form as \eqref{KKKgam} but with a cut-off $\tilde\chi_1(\frac{\omega}{\gamma_1|\eta|^{2/3}})$ instead of $\chi_1(\frac{\omega}{\gamma_1|\eta|^{2/3}})$, supported near $1$ and equal to $1$ on the support of $\chi_1$.
 
From \eqref{KKKgam}, $F_{N_2, t_2,\gamma_2}(\tilde x,\tilde y)$ expands as
\[
\int G(\tilde x, \tilde y, \tilde \eta,\tilde \omega)\chi^{\#}(\tilde\omega)q^{1/6}(\tilde \eta)\varkappa (h\tilde \eta)\varkappa(h\tau_q(\tilde\omega,\tilde\eta))
 \chi_1(\frac{\tilde\omega}{\gamma_2 |\tilde\eta|^{2/3}})e^{-i N_2L(\tilde\omega)+i t_2\tau_q(\tilde\omega,\tilde\eta)}\hat{g}_{h,(a,b)}(\tilde\eta, \frac{\tilde{\omega}}{h^{1/3}}) d\tilde \omega d\tilde\eta, 
\]
hence $F_{ N_2,  t_2, \gamma_2}(\tilde x,\tilde y)=J(f_{N_2,t_2,\gamma_2})$, where we have set
\begin{equation}\label{ftildentgam}
\hat{f}_{N_2,t_2, \gamma_2}(\tilde \eta, \frac{\tilde \omega}{h^{1/3}})=\chi_1(\frac{\tilde\omega}{\gamma_2 |\tilde\eta|^{2/3}})e^{-i N_2L(\tilde\omega)+i t_2\tau_q(\tilde\omega,\tilde\eta)}\hat{g}_{h,(a,b)}(\tilde\eta, \frac{\tilde{\omega}}{h^{1/3}}).
\end{equation}
We turn to $\tilde K_{\omega,\gamma_1}(g_{h,(\tilde x, \tilde y)})(t_1,x,y)$ which expands as 
\begin{equation}\label{tildeKagam}
\int G(x,y,\eta,\omega)\chi^{\#}(\omega)q^{1/6}( \eta)\varkappa (h \eta)\varkappa(h\tau_q(\omega,\eta))
\tilde\chi_1(\frac{\omega}{\gamma_1 |\eta|^{2/3}})e^{it_1\tau_q(\omega,\eta)}\hat{g}_{h,(\tilde x,\tilde y)}(\eta, \frac{\omega}{h^{1/3}}) d\eta.
\end{equation}
\begin{lemma}
We have
\[
\int_{\tilde x, \tilde y} g_{h, (\tilde x,\tilde y)}(y',\rho) F_{N_2,t_2, \gamma_2}(\tilde x,\tilde y) d\tilde x d\tilde y= J^{-1}( F_{N_2, t_2,\gamma_2})(y',\rho).
\]
\end{lemma}

\begin{proof}
Recall from \eqref{ghaF} that, in the variables $\theta=h\eta$, $\alpha=h^{2/3}\omega$, we may write 
\begin{multline}
 g_{h, (\tilde x,\tilde y)}(y',\rho) =\frac{1}{h^{2d+1}}\int e^{\frac ih (y\cdot '\theta'-\Phi(\overline{x},\overline{y},\theta',\alpha,s)+\rho\alpha+(\overline{x}-\tilde x)\sigma +(\overline{y}-\tilde y)\cdot\theta)} \\
 \times q_h(\overline{x}, \overline{y}, \theta',\alpha, s)q^{-1/6}(\theta')\tilde \chi_0(\sigma)\varkappa(\theta) d\overline{x} d\overline{y} d\theta' d\sigma ds d \alpha d \theta,
\end{multline}
where $\tilde\chi_0$ is supported near $0$ and equals $1$ on the support of $\chi_0$ (of Lemma \ref{lemgab}). We thus obtain
\begin{multline}
\int_{\tilde x, \tilde y} g_{h, (\tilde x,\tilde y)}(y',\rho) F_{N_2,t_2, \gamma_2}(\tilde x,\tilde y) d\tilde x d\tilde y=\frac{1}{h^{2d+1}}\int e^{\frac ih (y'\cdot\theta'-\Phi(\overline{x},\overline{y},\theta',\alpha,s)+\rho\alpha+\overline{x}\sigma +\overline{y}\cdot\theta)} q_h(\overline{x}, \overline{y}, \theta',\alpha, s)q^{-1/6}(\theta')\\
 \times e^{-\frac ih(\tilde x \sigma +\tilde y.\theta)}\tilde\chi_0(\sigma)\varkappa(\theta)\hat{F}_{N_2,t_2, \gamma_2}(\sigma/h,\theta/h).
\end{multline}
As $\hat{ F}_{N_2, t_2, \gamma_2}=\tilde \chi_0(\sigma)\hat{F}_{N_2,t_2, \gamma_2}(\sigma/h,\theta/h)$ (which contains $\chi_0(\sigma)$), the last integral equals $J^{-1}( F_{N_2,t_2, \gamma_2})$.
\end{proof}
Using the lemma and $F_{N_2, t_2, \gamma_2}(\tilde x,\tilde y)=J(f_{N_2,t_2, \gamma_2})$ yields
\begin{equation}\label{fplusinf}
\int_{\tilde x, \tilde y} g_{h, (\tilde x,\tilde y)}(y',\rho) F_{N_2,t_2, \gamma_2}(\tilde x,\tilde y) d\tilde x d\tilde y= J^{-1}\circ J(f_{N_2,t_2,  \gamma_2})(y',\rho)= f_{N_2,t_2, \gamma_2}(y',\rho)+O(h^{\infty}),
\end{equation}
hence taking the Fourier transform of $g_{h, (\tilde x,\tilde y)}(y',\rho)$ (as needed in \eqref{tildeKagam}) gives $\hat{f}_{N_2,t_2,\gamma_2}(\theta/h,\alpha/h)$ modulo $O(h^{\infty})$ terms.
Hence, $J_{N_2,\gamma_1, \gamma_2}(t_1,t_2,x,y,\omega)$ becomes
\[
J_{N_2,\gamma_1, \gamma_2}(t_1,t_2, x,y,\omega)=\int G(x,y,\eta,\omega)\chi^{\#}(\omega)q^{1/6}( \eta)\varkappa (h \eta)\varkappa(h\tau_q(\omega,\eta))
\tilde \chi_1(\frac{\omega}{\gamma_1 |\eta|^{2/3}})e^{it_1 \tau_q(\omega,\eta)}\hat{f}_{N_2,t_2,\gamma_2}(\eta,\omega/h^{1/3}),
\]
with $\hat{f}_{N_2,t_2,\gamma_2}$ given in \eqref{ftildentgam}. We find, modulo $O(h^{\infty})$ terms from \eqref{fplusinf} (which stay $O(h^{\infty})$ as the sums over $N_1,N_2$ below are finite),
\begin{multline}\label{finalcompgam}
\int_{\tilde x,\tilde y}\mathcal{ \tilde P}_{h,\gamma_1}(t_1,x,y;{\tilde x,\tilde y}) \mathcal{P}_{h,\gamma_2}(t_2,\tilde x,\tilde y;a,b) \,d\tilde x d\tilde y=\sum_{N,\tilde N}\int e^{it_1\tau_q(\omega,\eta)-i N_1 L(\omega)}\cutoffchi^{\flat}(h^{2/3}\omega/\ceps_{0})G(x,y,\eta,\omega)\\
\times \chi^{\#}(\omega)q^{1/6}( \eta)\varkappa (h \eta)
 \varkappa(h\tau_q(\omega,\eta))
\tilde\chi_1(\frac{\omega}{\gamma_1 |\eta|^{2/3}})\hat{f}_{N_2,t_2,\gamma_2}(\eta,\omega/h^{1/3})d\omega\\
=\sum_{N_1,N_2}\int e^{i(t_1+t_2)\tau_q(\omega,\eta)-i(N_1+N_2)L(\omega)}G(x,y,\eta,\omega) \chi^{\#}(\omega)q^{1/6}( \eta)\varkappa (h \eta)
 \varkappa(h\tau_q(\omega,\eta))\\
\tilde\chi_1(\frac{\omega}{\gamma_1 |\eta|^{2/3}})\chi_1(\frac{\omega}{\gamma_2 |\eta|^{2/3}})\hat{g}_{h,(a,b)}(\eta, \frac{\omega}{h^{1/3}}).
\end{multline}
Recall that $\tilde\chi_1=1$ on the support of $\chi_1$. Therefore, if $\gamma_1,  \gamma_2\in 1/2^{\mathbb{N}}$, with $\gamma_1\neq\gamma_2$, then $\tilde\chi_1(\frac{\omega}{\gamma_1 |\eta|^{2/3}})\chi_1(\frac{\omega}{\gamma_2 |\eta|^{2/3}})=0$ and in this case we obtain \eqref{eq:7bis}.

It remains to deal with $\gamma_1=\gamma_2=:\gamma$ : in this case, as $\tilde\chi_1\chi_1=\chi_1$, the last sum in \eqref{finalcompgam} equals 
\begin{equation}\label{sumnntilde}
\sum_{N_1,N_2} <e^{-i(N_1+N_2)L(\omega)}, \cutoffchi^{\flat}(h^{2/3}\omega/\ceps_{0})K_{\omega,\gamma}(g_{h,(a,b),\gamma})(t_1+t_2,x,y)>_{\omega}+O(h^{\infty}),
\end{equation}
and we need to show that the double sum (over $N_1,N_2$) yields indeed $\mathcal{P}_{h,\gamma}(t_1+t_2,x,y ;a,b)$. This is the goal of the next lemma.
\begin{lemma}
The sum in \eqref{sumnntilde} can be written as 
\begin{equation}\label{calPfinal}
\sum_{N} <e^{-i N L(\omega)}, \cutoffchi^{\flat}(h^{2/3}\omega/\ceps_{0})K_{\omega,\gamma}(g_{h,(a,b),\gamma})(t_1+t_2,x,y)>_{\omega}.
\end{equation}
\end{lemma}

\begin{proof}
We use arguments similar to those in Lemma \ref{sommeNfinie}, i.e. we check when the phase function 
of $\int e^{-i(N_1+N_2)L(\omega)}K_{\omega,\gamma}( g_{h,(a,b),\gamma})(t_1+t_2,x,y)d\omega$ may be stationary with respect to $\alpha=h^{2/3}\omega$. 

Recall first that, independently of the size of $a$ (or $\tilde x$), the phase function of $\hat{g}_{h,(a,b)}$ (or of $\hat{g}_{h,(\tilde x,\tilde y)}$) equals that of $G(a,b,\eta,\omega)$ (or of $G(\tilde x,\tilde y,\tilde\eta,\tilde\omega)$) : this follows from Propositions \ref{lemp} and and \ref{proptildeJ} in case $a>h^{2/3-\epsilon}$ and from Lemma \ref{proptildeJ} when $a\lesssim h^{2/3-\epsilon}$ and $\omega>h^{-\epsilon}$ (for some $\epsilon>0$). (This property also holds when $a\lesssim h^{2/3-\epsilon}$ and $\omega$ is bounded). These constructions of $g_{h,(a,b)}(y',\rho)$ (resp. $g_{h,(\tilde x,\tilde y)}(y',\rho)$) in both cases are very similar and these functions are compactly supported for $\rho$ near $0$.
As in Lemma \ref{sommeNfinie}, the phase function of $<e^{-i N L(\omega)}, \cutoffchi^{\flat}(h^{2/3}\omega/\ceps_{0})K_{\omega,\gamma}(g_{h,(a,b),\gamma})(t_1+t_2,x,y)>_{\omega}$  (with $\eta=\theta/h$, $\omega=\alpha/h^{2/3}$) is given by \eqref{phaKom} and its critical points with respect to $\alpha,\sigma$ satisfy the equations \eqref{eq:critalphaNN} and \eqref{critsigma}. Therefore, taking $\alpha=\gamma A$, $\sigma=\sqrt{\gamma} \Sigma$, $T_1=t_1/(2\sqrt{\gamma})$, $X=x/\gamma$ we must have
\[
T_1-\Sigma=\rho+2N_1\sqrt{A},\text{ where } \Sigma^2+Xq^{1/3}(\theta)(1+\Lp+O(\gamma))=A \text{ and } \rho \text{ is small}.
\]
If there are two non-trivial contributions for the same $t_1>0$, corresponding to $N_1$ and $N_1-1$, then there exists $A_1\sim 1$ and $\Sigma_1$ satisfying $\Sigma^2_1+Xq^{1/3}(\theta)(1+\Lp+O(\gamma))=A_1$ such that
\[
T_1-\Sigma_1=\rho+2(N_1-1)\sqrt{A_1}=\rho+2N_1\sqrt{A_1}-2\sqrt{A_1}.
\]
The only solutions satisfying the last equations verify $\Sigma/\sqrt{A}\sim -1$ and $\Sigma_1/\sqrt{A_1}\sim 1$. On the other hand, there are no solutions if, instead of $N_1-1$ and $N_1$ we consider $N_1-2$, $N_1$. Therefore, each time that we obtain non-trivial contributions from two consecutive values of $N_1$ in the sum \eqref{sumnntilde}, the corresponding critical points in $\Sigma$ satisfy $\Sigma^2/A\sim 1$. In the same way, if, for fixed $t_2$, $N_2$ and $N_2+1$ provide nontrivial contributions in $f_{N_2,t_2,\gamma}(y',\rho)$ and $f_{N_2+1,t_2,\gamma}(y',\rho)$, then the (normalized) Airy variable $\tilde \Sigma$ in $G(a,b,\tilde\eta, \tilde \omega)$ has to satisfy $\tilde\Sigma^2/\tilde A\sim 1$. 

We introduce $\chi$, supported in $[-3/4,3/4]$ and equal to $1$ on $[-1/2,1/2]$ and $\chi_{\pm}(z):=(1-\chi)(z) $ if $\pm z>0$, and equal to $0$ otherwise: then $\chi+\sum_{\pm}\chi_{\pm}=1$ everywhere. If $\sigma$, $s$ denote the Airy variables of $G(x,y,\eta,\omega)$ and $G(a,b,\eta,\omega)$, we split its symbol in two parts, $p_h=p_h\chi(\sigma^2/\alpha)+p_h(1-\chi(\sigma^2/\alpha))$ and write $K_{\omega,\gamma}(f)=K^0_{\omega,\gamma}(f)+K^{\pm}_{\omega,\gamma}(f)$. We then do the same with $ g_{h,(a,b),\gamma}$, that we write as a sum $g_{h,(a,b),\gamma}= g^0_{h,(a,b),\gamma}+ g^{\pm}_{h,(a,b),\gamma}$, where in the integral form of $\hat{g}_{h,(a,b),\gamma}$ we added the cut-offs $\chi(s^2/\alpha)$ and $(1-\chi)(s^2/\alpha)$.

Using the same arguments as above, we notice that, with $N_t=[t/4\sqrt{\gamma}]$, the only non-trivial contribution in last line of \eqref{finalcompgam} comes from pairs $(N_t,N_{\tilde t}), (N_t\mp 1, N_{\tilde t}),(N_t, N_{\tilde t}\pm 1), (N_t\mp 1, N_{\tilde t}\pm 1)$ which corresponds to the following products of cutoffs (with respect to $(\sigma/\alpha), (s/\alpha)$) : $(\chi,\chi)$, $(\chi_{\pm},\chi)$, $(\chi,\chi_{\pm})$ and $(\chi_{\pm},\chi_{\pm})$. Summing up all these contributions allows to obtain \eqref{calPfinal} (as we recover $(\chi+\sum_{\pm}\chi_{\pm})(\cdot)(\chi+\sum_{\pm}\chi_{\pm})(\cdot)$); the sum of the contributions coming from $|N-N_t|\geq 2$ or $|\tilde N-\tilde N_{\tilde t}|\geq 2$ equals $O(h^{\infty})$ by repeated integrations by parts and using that these sums stay finite. This achieves the proof of Proposition \ref{commutgamma}.
\end{proof}
From  the last Lemma we conclude the proof of Proposition \ref{commutgamma}.
\end{proof}
\section{Appendix}
We provide details on the proof of Lemma \ref{lemestimderivekk} and on obtaining $\canonchi_M$ (from Proposition \ref{lemgamma}).
\subsection{Proof of Lemma \ref{lemestimderivekk}}
We start with twisted modes $\tilde e$ (see Definition \ref{blabla}):
\begin{lemma}\label{lemequivL2normstildee}
Let $k$ such that $\omega_kh^{2/3}\leq \ceps_0$ with $\ceps_0< 1/100$. There exists  constants $l_0>0$ and $0<c_0<C_0$ independent of $k,h,a$ such that, for all $y$ such that $|y|\leq l_0$,
\begin{equation}\label{equivtildee2}
c_0\leq \|\tilde e(\cdot,y,\eta,\omega_k)\|_{L^2}\leq C_0\,.
\end{equation}
\end{lemma}
\begin{proof}
By definition, $\tilde{e}(x,y,\eta,\omega_k)=\frac{q(\eta)^{\frac 16}}{\sqrt{L'(\omega_k)}}e^{-i(y\cdot \eta+|\eta|B_0(y,\eta/|\eta|))}G(x,y,\eta,\omega_k)$ and, using \eqref{eq:propL2}, $L'(\omega_k)=2\pi \|{Ai}(\cdot -\omega_k)\|^2_{L^2}\sim \omega_k^{1/2}$, for all $k\geq 1$. From \eqref{eq:defG}, with $e_0$, $p_{0}$ and $p_{1}$ from Theorem \ref{thmMelrose},
\begin{equation}\label{developtildee}
\tilde{e}(x,y,\eta,\omega_k)=\frac{q(\eta)^{1/6}}{ \sqrt{L'(\omega_k)}}e^{i(\psi(x,y,\eta,\omega_k)-y\cdot \eta-|\eta|B_0(y,\eta/|\eta|))}\Big(p_0{Ai}((-\zeta)+ip_1q^{-1/6}(\eta){Ai}'(-\zeta)\Big),
\end{equation}
where $\zeta(x,y,\eta,\omega_k)=-\omega_k+x|\eta|^{2/3}e_0(x,y,\eta/|\eta|,\omega_k/|\eta|^{2/3})$. 
Using Cauchy-Schwarz,
\begin{multline}\label{estimtildeeL2}
\int_0^{\infty} |\tilde{e}(x,y,\eta,\omega_k)|^2dx 
\leq \frac{q^{1/3}(\eta)}{L'(\omega_k)}\Big[\int_0^{\infty}\Big| p_0{Ai}(-\zeta)\Big|^2dx+2\Big(\int_0^{\infty}\Big| p_0{Ai}(-\zeta)\Big|^2dx\Big)^{1/2}\\ {}\times \Big(\int_0^{\infty}\Big| p_1q^{-1/6}(\eta){Ai}'(-\zeta)\Big|^2dx\Big)^{1/2}
+\int_0^{\infty}\Big| p_1q^{-1/6}(\eta){Ai}'(-\zeta)\Big|^2dx\Big]\,.
\end{multline}
For values $x|\eta|^{2/3}e_0> \omega_k$ both ${Ai}(-\zeta)$ and ${Ai}'(-\zeta)$ are exponentially decreasing.
Setting $X:=xe_0(x,y,\eta/|\eta|,\omega_k/|\eta|^{2/3})$ yields $x=x(X,y,\eta/|\eta|,\omega_k/|\eta|^{2/3})$ and
\[
\frac{q^{1/3}(\eta)}{L'(\omega_k)}\int_0^{\infty}\Big| p_0{Ai}(-\zeta)\Big|^2dx=\frac{q^{1/3}(\eta)}{L'(\omega_k)}\int_0^{\infty}\Big| \tilde p_0{Ai}(X|\eta|^{2/3}-\omega_k)\Big|^2\frac{dx}{dX}dX,
\]
where $\tilde p_0(X,y,\eta,\omega_k):=p_0(x(X,y,\eta/|\eta|,\omega_k/|\eta|^{2/3})$. Here $\omega_k/|\eta|^{2/3}\ll 1$ from $\eta\sim 1/h$ and $h^{2/3}\omega_k\ll 1$. Moreover $\frac{dx}{dX}=\frac{1}{e_0(0,\cdot)}+O(X)$, $e(0,\cdot)$ is close to $1$ and $X\sim x\leq \omega_k|\eta|^{2/3}\ll 1$ (as for large $X$, ${Ai}(\cdot)$ is exponentially decreasing.) With $C:=\sup |\tilde p_0|^2|\frac{dx}{dX}|<\infty$ and rescaling variables $\tilde X=X|\eta|^{2/3}$, 
\begin{equation}\label{estimAisquare}
\frac{q^{1/3}(\eta)}{L'(\omega_k)}\int_0^{\infty}\Big| p_0{Ai}(-\zeta)\Big|^2dx\leq C \frac{q^{1/3}(\eta)/|\eta|^{2/3}}{L'(\omega_k)}\int_0^{\omega_k}{Ai}(\tilde X-\omega_k)^2dX=Cq^{1/3}(\eta/|\eta|).
\end{equation}
For the integral with ${Ai}'$ we proceed similarly: $p_1$ is bounded, $\|{Ai}'(\tilde X-\omega_k)\|^2_{L^2_{\tilde X}}\lesssim \omega_k^{3/2}$ and $\tilde X\lesssim \omega_k$ as for large values we have exponential decay. We compute, with $X=xe_0(x,\cdot)$, then $\tilde X=X|\eta|^{2/3}$,
\begin{multline}\label{estimxAiprimsquare}
\frac{q^{1/3}(\eta)}{L'(\omega_k)}\int_0^{\infty}\Big| p_1q^{-1/6}(\eta){Ai}'(-\zeta)\Big|^2dx\leq \frac{C'}{L'(\omega_k)}\int_0^{\infty}|X\mathrm{Ai}'(X|\eta|^{2/3}-\omega_k)|^2dX\\
\leq\frac{C'|\eta|^{-2/3}}{L'(\omega_k)}\int_0^{\omega_k}\tilde X^2|\eta|^{-4/3}|{Ai}'(\tilde X-\omega_k)|^2d\tilde X
\leq C'|\eta|^{-2}\omega_k^{3+1/2}/L'(\omega_k)\sim C'(\omega_kh^{2/3})^3\ll C',
\end{multline}
where we used $|{Ai}'(z)^2|\leq (1+|z|)^{1/2}$ and that $p_1|_{x=0}=0$. Collecting all bounds yields the upper bound in \eqref{equivtildee2}, which  will be enough  for proving Lemma \ref{lemestimderivekk} below. We now prove the lower bound, which was used for the proof of Proposition \ref{propmatrix}.
 From \eqref{estimtildeeL2}, using \eqref{estimAisquare}, \eqref{estimxAiprimsquare} and $\omega_k h^{2/3}\leq \ceps_0$, we have 
\begin{multline}\label{estimtildeeL2<}
\int_0^{\infty} |\tilde{e}(x,y,\eta,\omega_k)|^2dx\geq \frac{q^{1/3}(\eta)}{L'(\omega_k)}\Big[\int_0^{\infty}\Big| p_0{Ai}(-\zeta)\Big|^2dx-2\Big(\int_0^{\infty}\Big| p_0{Ai}(-\zeta)\Big|^2dx\Big)^{1/2}\\ 
\quad\quad\quad\quad\quad\quad\quad\quad\quad{}\times\Big(\int_0^{\infty}\Big| xp_1q^{-1/6}(\eta){Ai}'(-\zeta)\Big|^2dx\Big)^{1/2}
-\int_0^{\infty}\Big| xp_1q^{-1/6}(\eta){Ai}'(-\zeta)\Big|^2dx\Big]\\
\geq   \frac{q^{1/3}(\eta)}{L'(\omega_k)}\int_0^{\infty}\Big| p_0{Ai}(-\zeta)\Big|^2dx-O(\omega_kh^{2/3})^{3/2}\,.
\end{multline}
As $\omega_k h^{2/3}\leq \ceps_0$, we are left to prove that $\frac{q^{1/3}(\eta)}{L'(\omega_k)}\int_0^{\infty}\Big| p_0{Ai}(-\zeta)\Big|^2dx$ can be bounded from below by a constant independently of $k,h,a$. From ellipticity of  $p_0\sim 1$, there exists $\cveps_1>0$ such that $p_0(x,y,\eta,\omega_k) \geq 1/2$ for all $(x,y)$ such that $|(x,y)|\leq \cveps_1$. On the other hand, for values $x|\eta|^{2/3}e_0(x,y,\eta/|\eta|,\omega_k/|\eta|^{2/3})>\omega_k$ with $|\eta|\sim 1/h$, ${Ai}(-\zeta)$ is exponentially decreasing: thus,  the bulk of the $L^2$ norm of $p_0{Ai}(-\zeta)$ is located for $x\lesssim \omega_kh^{2/3}\leq \ceps_0\ll 1$ and $\int_{4\ceps_0}^{\infty}|p_0{Ai}(-\zeta)|^2dx=O(h^{\infty})$. Taking $\ceps_0$ smaller if necessary such that $\ceps_0<\cveps_1/4$, we have, for all $|y|\leq \cveps_1$
\begin{equation}\label{estimAiminortildee}
\frac{q^{1/3}(\eta)}{L'(\omega_k)}\int_0^{\infty}\Big| p_0{Ai}(-\zeta)\Big|^2dx\geq \frac{q^{1/3}(\eta)}{4L'(\omega_k)}\int_0^{4\ceps_0}|{Ai}(-\zeta)|^2dx+O(h^{\infty})\,.
\end{equation}
Ellipticity of $e_0$ near $(x,y)=(0,0)$ provides $c>0$ and $\cveps_2>0$ such that $e_0(x,y,\cdot)\geq c$ for all $|(x,y)|\leq \cveps_2$. Taken $\cveps_0$ smaller if necessary (so that $\cveps_0<\cveps_2/4$), we can assume that $e_0(x,y,\cdot)\geq c$ for all $x\leq 2\cveps_0$ and all $|y|\leq \cveps_2$. Let $X=xe_0(x,y,\eta/|\eta|,\omega_k/|\eta|^{2/3})$ for $x\leq \cveps_0$ and $|y|\leq l_0:= \min\{\cveps_1,\cveps_2\}$, then $|\frac{dx}{dX}|\geq \frac 1c$ for all $0\leq x\leq 4\cveps_0$ and $|y|\leq l_0$ and 
\[
 \frac{q^{1/3}(\eta)}{4L'(\omega_k)}\int_0^{4\ceps_0}|{Ai}(-\zeta)|^2dx\geq  \frac{q^{1/3}(\eta/|\eta|)}{4cL'(\omega_k)}\int_0^{4\ceps_0|\eta|^{2/3}}{Ai}^2(\tilde X-\omega_k)d\tilde X.
\]
As $\omega_k\leq 4\cveps_0|\eta|^{2/3}$ for $h|\eta|\in [1/2,2]$ and $\omega_kh^{2/3}\leq \cveps_0$, we find
\begin{equation}
\int_0^{4\ceps_0|\eta|^{2/3}}{Ai}^2(\tilde X-\omega_k)d\tilde X=\int_0^{\infty}{Ai}^2(\tilde X-\omega_k)d\tilde X-\int_{4\ceps_0|\eta|^{2/3}}^{\infty}{Ai}^2(\tilde X-\omega_k)d\tilde X=L'(\omega_k)+O(h^{\infty}) 
\end{equation}
and therefore the left hand side term in \eqref{estimAiminortildee} is bounded from below by $\inf_{\Theta\in\mathbb{S}^{d-2}}q^{1/3}(\Theta)/(4c)$ for all $y$ with $|y|\leq l_0$. As $q$ is  positive definite, this completes the proof of Lemma \ref{lemequivL2normstildee}.
\end{proof}
\begin{proof}(of Lemma \ref{lemestimderivekk}) 
Using $B_{2j}(y,\cdot)=O(y)$ for all $j\geq 2$ and Corollary \ref{corUpsi}, the phase $\psi(x,y,\eta,\omega_k)-y\cdot \eta-|\eta|B_0(y,\eta/|\eta|)$ of $\tilde e(x,y,\eta,\omega_k)$ reads (see \eqref{developtildee})
\begin{equation}\label{phasetildee}
(\tau_{q}(\omega_k,\eta)-|\eta|)(B_0+B_2)+\tau_{q}(\omega_k,\eta)(O(x)\mathcal{H}_{j\geq 2}+O(\zeta/|\eta|^{2/3})\mathcal{H}_{j\geq 2}+O(y)\mathcal{H}_{j\geq 3}).
\end{equation}
Using Lemma \ref{lemequivL2normstildee}, $\|\tilde{e}(.,\omega_k)\|_{L^2(x\geq 0)}\lesssim 1$. From $\tau_{q}(\omega,\eta)-|\eta|=\frac{\omega q^{2/3}(\eta)}{|\eta|+\tau_{q}(\omega,\eta)}\sim \omega |\eta|^{1/3}$ and $x\lesssim \omega_k/|\eta|^{2/3}$, taking derivatives with respect to $y$ or $\eta$ of the phase of \eqref{developtildee} provides, at each step, $\omega_k |\eta|^{1/3}\sim \omega_k/h^{1/3}$. On the other hand, taking the derivatives (with respect to $y$ or $\eta$) in the last factor of the right hand side of \eqref{developtildee} and using that ${Ai}''(-\zeta)=\zeta {Ai}(-\zeta)$ provides
\begin{equation}
\partial^{\beta_1}_{y}\partial^{\beta_2}_{\eta}\Big(p_0{Ai}(-\zeta)+ip_1q^{-1/6}(\eta){Ai}'(-\zeta)\Big)=(\omega |\eta|^{1/3})^{\beta_1+\beta_2}\Big(p^{(\beta_1,\beta_2)}_0{Ai}(-\zeta)+i p^{(\beta_1,\beta_2)}_1 q^{-1/6}(\eta){Ai}'(-\zeta)\Big),
\end{equation}
where $p^{(\beta_1,\beta_2)}_0$ and $p^{(\beta_1,\beta_2)}_1$ are asymptotic expansions with main contributions homogeneous of degree $0$ and small parameter $(\omega_k|\eta|^{1/3})^{-1}$. Then, \eqref{estimderivekk} follows from bounds like in \eqref{estimAisquare} and \eqref{estimxAiprimsquare}.
\end{proof}
\subsection{The generating function $\varphi_{\Gamma}$ of $\canonchi_{M}$}
\label{Secgamma}
We aim at proving \eqref{eq:GAB}.
Set $\Theta=\varrho\vartheta$ with $\varrho =|\Theta|$ near $1$ and $\vartheta=\Theta/|\Theta|$. Functions $A_{\Gamma}, B_{\Gamma}$ are to be defined near the glancing set $\mathcal{GL}=\{x=0,\Xi=0,\varrho-1=0\}$ and for $(y,\vartheta)$ near $\{0\}\times \mathbb{S}^{d-1}$. We work with formal Taylor expansions $F$ near $\mathcal{GL}$ such that $F=\sum_{a,b,c} f_{a,b,c}(y,\vartheta){X_{M}}^a(\varrho-1)^b\xi^c$.
We attribute a degree to each factor $x, \varrho-1, \Xi$: a monomial of the form $x^a(\varrho-1)^b\Xi^c$ is homogeneous of degree $k$ if and only if $c+2(a+b)=k$. 
For such a  formal serie $F(x,y,\Xi,\varrho,\vartheta)$, defined near $\mathcal{GL}$, we write $F= \sum_{k\geq 0} F_k$, where $F_k$ is homogeneous of degree $k$; we also write $F\in \mathcal{H}_{\geq j}$ if and only if $F=\sum_{k\geq j} F_k$. Therefore, $F_0=f_0(y,\vartheta)$, $F_1=\Xi f_1(y,\vartheta)$, $
F_2={X_{M}} f_2^0(y,\vartheta)+(\varrho-1)f_2^1(y,\vartheta)+\Xi^2f_2^2(y,\vartheta)$, and so on. Replacing ${X_{M}},\xi,\eta$ by their formulas \eqref{genchi} (as functions of $(x,y,\Xi,\Theta)$) and using that from \eqref{eq:melrose} 
\begin{equation}\label{eqiffchi}
\xi^2+R(x,y,\eta)=1\quad \text{ if and only if }\quad \Xi^2+|\Theta|^2+{X_{M}}q(\Theta)=1,
\end{equation}
(where we notice that there is no $Y_{M}$ in the second equation), we get
\begin{equation}\label{formalseriesBA}
B_{\Gamma}= \sum_{j\geq 0} (\varrho-1)^j B_{2j}(y,\vartheta)\,,\quad
A_{\Gamma}= \sum_{k\geq 1}A_k\,.
\end{equation}
Using the third equation from \eqref{genchi} and $\xi|_{\mathcal{GL}}=0$, we have $A_{0}=0$. We also have $\xi(x,y,\Xi,\Theta)\in \mathcal{H}_{\geq 1}$ and ${X_{M}}(x,y,\Xi,\Theta)\in \mathcal{H}_{\geq 2}$. Moreover, from the proof of Melrose's theorem \cite{mel76}, if formal series of the form \eqref{formalseriesBA} satisfy \eqref{genchi} and \eqref{eqiffchi}, then there exist $C^{\infty}$ functions $A_{\Gamma}, B_{\Gamma}$ with the same Taylor development near $\mathcal{GL}$ satisfying \eqref{genchi} and \eqref{eqiffchi}.

We first consider in \eqref{eqiffchi} homogeneous terms of order $\leq 5$ in the expansion of $\Gamma$:  using ${X_{M}}(x,y,\Xi,\Theta)\in \mathcal{H}_{\geq 2}$, we are to write the explicit form of $A_{\Gamma}$ up to $\mathcal{H}_{j\leq 3}$ and $B_{\Gamma}$ up to $\mathcal{H}_{j\leq 5}$.  Let
$A_1=\Xi\Lp(Y,\vartheta)$, $A_2=\alpha(y,\vartheta)x+\beta(y,\vartheta)(\varrho-1)+\mu(y,\vartheta)\Xi^2$, $A_3=\alpha_1(y,\vartheta)x\Xi+\beta_1(y,\vartheta)(\varrho-1)\Xi+\mu_1(y,\vartheta)\Xi^3$ and
$A_{\Gamma}=A_{1}+A_{2}+A_{3} 
+\mathcal{H}_{j\geq 4}$.
\begin{rmq}
Understanding this process will allow us to proceed with the expansion at any order. Observe that $A_{2j}$ and $A_{2j+1}$ always have the same number of terms: the only way to obtain homogeneous terms of order $2j+1$ is to add a factor $\Xi$ to homogeneous terms of order $2j$. Moreover, each $A_{2j}$ and $A_{2j+1}$ will have $2j+1$ terms. This is of importance to understand why all $A_{k}$ may be obtained from the system of equations that will follow below.
\end{rmq}
Using these expansions for $A_{\Gamma}$ and $B_{\Gamma}$, and omitting variables for the functions $\Lp$, $\mu$, $\alpha$, $\beta$, $\mu$, $\alpha_{1}$, $\beta_{1}$ and $\mu_{1}$, \eqref{genchi} yields
\begin{align}\label{X}
X_{M}   = & x(1+\Lp+2\mu \Xi+\alpha_1x+\beta_1(\varrho-1)+3\mu_1\Xi^2+\mathcal{H}_{j\geq 3})\\
\label{xi} \xi  
 = & (1+\Lp)\Xi+(2\alpha x+\beta(\varrho-1)+\mu \Xi^2)+[2\alpha_1x\Xi+\beta_1(\varrho-1)\Xi+\mu_1\Xi^3]+\mathcal{H}_{j\geq 4}\\
\label{eta}
\eta  
  = &
      \begin{multlined}[t]
\vartheta+\nabla_yB_0+(\varrho-1)(\vartheta+\nabla_yB_2)+x\Xi\nabla_y\Lp+x^2\nabla_y\alpha+x(\varrho-1)\nabla_y\beta\\
        {}+x\Xi^2\nabla_y\mu+(\varrho-1)^2 \nabla_y B_4 +\mathcal{H}_{j\geq 5}\,.      \end{multlined}
\end{align}
Using $x\in \mathcal{H}_{j\geq 2}$ we rewrite $x=x_2+x_3+x_4+\mathcal{H}_{j\geq 5}$ where $x_j\in \mathcal{H}_{j}$; in the same way, $\eta=\eta_0+\eta_1+\eta_2+\eta_3+\eta_4+\mathcal{H}_{j\geq 5}$, with $\eta_j\in\mathcal{H}_{j}$. From \eqref{eta} we obtain 
\begin{equation}\label{eta012}
\eta_0=\vartheta+\nabla_yB_0, \quad \eta_1=0,\quad \eta_2= (\varrho-1)(\vartheta+\nabla_yB_2).
\end{equation}
Then, from $x\in \mathcal{H}_{j\geq 2}$, $\eta_3$ is homogeneous of order $3$ and $\eta_3=x_2\Xi\nabla_y\Lp$, while
\[
\eta_4=x_3\Xi\nabla_y\Lp+x_2^2\nabla_y\alpha+x_2(\varrho-1)\nabla_y\beta+x_2\Xi^2\nabla_y\mu+(\varrho-1)^2\nabla_yB_4\,.
\]
Similarly, $\xi=\xi_1+\xi_2+\xi_3+\mathcal{H}_{j\geq 4}$, $\xi_j\in\mathcal{H}_j$, depending on $x_2,x_3,x_4$ (notice that $\xi_0=0$). For $\xi$, an expansion up to $\mathcal{H}_{j\geq 3}$ is sufficient as  we work with $\xi^2$. From \eqref{xi}, 
\[
\xi_1=(1+\Lp)\Xi\,,\quad \xi_2=2\alpha x_2+\beta(\varrho-1)+\mu\Xi^2\,,\quad \xi_3=2\alpha x_3+2\alpha_1x_2\Xi+\beta_1(\varrho-1)\Xi+\mu_1\Xi^3\,.
\]
\begin{lemma}\label{l5.2}
Let $\mathcal{L}:=\{(x,y,\vartheta,\varrho,\Xi),{X_{M}}q(\Theta)=1-\Xi^2-|\Theta|^2\}$, where $X_{M}=X_{M}(x,y,\vartheta,\varrho,\Xi)$ is given by \eqref{X}, $\Theta=\varrho\vartheta$. 
If $(x,y,\vartheta,\varrho,\Xi)\in\mathcal{L}$ then $x=x_2+x_3+x_4+\mathcal{H}_{j\geq 5}$, with
\begin{gather}\label{x2x3}
x_2=-\frac{(\Xi^2+2(\varrho-1))}{q(\vartheta)(1+\Lp)}\,,\quad x_3=\frac{2\mu\Xi(\Xi^2+2(\varrho-1))}{q(\vartheta)(1+\Lp)^2}\,,\\
\label{x4} \,\,x_4=\frac{1}{q(\vartheta)(1+\Lp)}\Big[\Xi^4\Big(\frac{3\mu_1}{(1+\Lp)}-\frac{4\mu^2}{(1+\Lp)^2}-\frac{\alpha_1}{q(\vartheta)(1+\Lp)^2}\Big)\\
{}+(\varrho-1)^2\Big(\frac{2}{(1+\Lp)}(\frac{-2\alpha_1}{q(\vartheta)(1+\Lp)} +3+\beta_{1})\Big)\\
\quad\quad\quad\quad\quad\quad\quad\quad{}+\Xi^2(\varrho-1)\Big(2+2\Big(\frac{3\mu_1}{(1+\Lp)}-\frac{4\mu^2}{(1+\Lp)^2}-\frac{\alpha_1}{q(\vartheta)(1+\Lp)^2}\Big)\\{}+\frac{1}{(1+\Lp)}(\beta_1-\frac{2\alpha_1}{q(\vartheta)(1+\Lp)})\Big)\Big].
\end{gather}
\end{lemma}
\begin{proof}
Using $X_{M}q(\Theta)=1-\Xi^2-|\Theta|^2$, $q(\Theta)=\varrho^2q(\vartheta)$ and \eqref{X} yields
\begin{equation}\label{eqxX}
x(1+\Lp+2\mu\Xi+\alpha_1x+\beta_1(\varrho-1)+3\mu_1\Xi^2+\mathcal{H}_{j\geq 3})\varrho^2q(\vartheta)=1-\varrho^2-\Xi^2.
\end{equation}
which immediately provides $x$ in terms of $(y,\vartheta)$, $\Xi$ and $\varrho-1$ up to $\mathcal{H}_{j\geq 3}$ as follows
\[
x|_{\mathcal{L}}=\frac{(1-\varrho^2-\Xi^2)}{\varrho^2q(\vartheta)(1+\Lp)}\Big(1-\frac{2\Xi\mu}{1+\Lp}+\mathcal{H}_{j\geq 2}\Big)\,;
\]
hence, using $1-\varrho^2=-2(\varrho-1)-(\varrho-1)^2$, we get \eqref{x2x3}.
In order to obtain homogeneous terms of order $4$ we write $x|_{\mathcal{L}}=x_2+x_3+x_4+\mathcal{H}_{j\geq 5}$ 
and replace $x_2,x_3,x_4$ in \eqref{eqxX}; using moreover $\frac{1}{\varrho^2}= 1-2(\varrho-1)+3(\varrho-1)^2+O((\varrho-1)^3)$,
we find
\begin{multline}
(x_2+x_3+x_4+\mathcal{H}_{j\geq 5})(1+\Lp+2\mu \Xi+\alpha_1 x_2+\beta_1(\varrho-1)+3\mu_1\Xi^2+\mathcal{H}_{j\geq 3})\\
=-\frac{1}{q(\vartheta)}(\Xi^2+2(\varrho-1)+(\varrho-1)^2)\times (1-2(\varrho-1)+3(\varrho-1)^2+\mathcal{H}_{j\geq 6})\\
=-\frac{1}{q(\vartheta)}\Big((\Xi^2+2(\varrho-1))-(\varrho-1)(2\Xi^2+3(\varrho-1))+\mathcal{H}_{j\geq 6}\Big).
\end{multline}
Identifying homogeneous terms of order $2,3,4$, we obtain (again) $x_2,x_3$ as well as $x_4$:
\begin{equation}
x_4(1+\Lp)+2\mu\Xi x_3+(\alpha_1 x_2+\beta_1(\varrho-1)+3\mu_1\Xi^2)x_2=\frac{1}{q(\vartheta)}(\varrho-1)(2\Xi^2+3(\varrho-1)),
\end{equation}
which yields \eqref{x4} by substitution. This completes the proof of Lemma \ref{l5.2}.
\end{proof}
We now replace \eqref{x2x3} and \eqref{x4} in \eqref{xi} and \eqref{eta} and then $x,\xi,\eta$ in $\xi^2+R(x,y,\eta)=1$ to obtain a system of equations with unknown $B_0,B_2,B_4,\Lp, \alpha,\beta,\mu,\alpha_1,\beta_1,\mu_1$ as follows. First,
\begin{multline}
\xi^2=(\xi_1+\xi_2+\xi_3+\mathcal{H}_{j\geq 4})^2
=(1+\Lp)^2\Xi^2+2(1+\Lp)\Xi(2\alpha x_2+\beta(\varrho-1)+\mu\Xi^2)\\
+2(1+\Lp)\Xi(2\alpha x_3+2\alpha_1x_2\Xi+\beta_1(\varrho-1)\Xi+\mu_1\Xi^3)
+(2\alpha x_2+\beta(\varrho-1)+\mu\Xi^2)^2+\mathcal{H}_{j\geq 5}\,.
\end{multline}
Write $\xi^2=(\xi^2)_2+(\xi^2)_3+(\xi^2)_4+\mathcal{H}_{j\geq 5}$, where $(\xi^2)_j\in\mathcal{H}_j$; replacing $x_2$ and $x_3$ by \eqref{x2x3}, we find 
\begin{align}\label{xi22}
(\xi^2)_2= & (1+\Lp)^2\Xi^2 \in \mathcal{H}_2\,,\\
\label{xi23}
(\xi^2)_3= & 2(1+\Lp)\Xi\Big(-2\alpha \frac{(\Xi^2+2(\varrho-1))}{q(\vartheta)(1+\Lp)}+\beta(\varrho-1)+\mu\Xi^2\Big)\in \mathcal{H}_3\,,\\
\label{xi24}
  (\xi^2)_4 = & \Big(-2\alpha \frac{(\Xi^2+2(\varrho-1))}{q(\vartheta)(1+\Lp)}+\beta(\varrho-1)+\mu\Xi^2\Big)^2+2(1+\Lp)\Xi\\
&{}\times  \Big(2\alpha\frac{2\mu\Xi(\Xi^2+2(\varrho-1))}{q(\vartheta)(1+\Lp)^2}
                 -2\alpha_1\Xi\frac{(\Xi^2+2(\varrho-1))}{q(\vartheta)(1+\Lp)} +\beta_1(\varrho-1)\Xi+\mu_1\Xi^3\Big)\in \mathcal{H}_{4}\\
 = &     \begin{multlined}[t]
 (\Xi^2+2(\varrho-1))^2\Big(\mu-\frac{2\alpha}{q(\vartheta)(1+\Lp)}\Big)^2+2(1+\Lp)\mu_1\Xi^4+2(1+\Lp)\beta_1\Xi^2(\varrho-1)\\
  {} +\frac{4\Xi^2}{q(\vartheta)}\Big(\frac{2\alpha\mu}{(1+\Lp)}-\alpha_1\Big)(\Xi^2+2(\varrho-1))\,.     \end{multlined}
\end{align}
We do the same for $\eta=\eta_0+\eta_2+\eta_3+\eta_4+\mathcal{H}_{j\geq 5}$, for which it remains to replace $x_2$ and $x_3$ obtained in \eqref{x2x3} in the expression of $\eta_3$ and $\eta_4$ that we have already obtained from \eqref{eta}. We get
\begin{equation}
\label{eta34}  \eta_3=-\Xi\nabla_y\Lp\frac{(\Xi^2+2(\varrho-1))}{q(\vartheta)(1+\Lp)},\,
\eta_4=\begin{multlined}[t](\varrho-1)^2\nabla_yB_4-\frac{(\Xi^2+2(\varrho-1))}{q(\vartheta)(1+\Lp)}\Big(\Xi^2\nabla_y\mu+(\varrho-1)\nabla_y\beta\Big)\\
{}+\nabla_y\Lp\frac{2\mu\Xi^2(\Xi^2+2(\varrho-1))}{q(\vartheta)(1+\Lp)^2}
+\nabla_y\alpha \frac{(\Xi^2+2(\varrho-1))^2}{q^2(\vartheta)(1+\Lp)^2}\,.
\end{multlined}
\end{equation}
We expand $R(x,y,\eta)R_0(y,\eta)+xR_1(y,\eta)+\frac 12 x^2R_2(y,\eta)+O(x^3)$ ($\eta_0,\eta_2$ were obtained in \eqref{eta012}),
\begin{multline}\label{R}
R(x,y,\eta)  =R_0(y,\eta_0)+(\eta_2+\eta_3+\eta_4)\nabla_{\eta}R_0(y,\eta_0)+\frac 12 \eta_2^2\nabla^2_{\eta,\eta}R_0(y,\eta_0)+\mathcal{H}_{j\geq 5}\\
{}+(x_2+x_3+x_4)(R_1(y,\eta_0)+\eta_2\nabla_{\eta}R_1(y,\eta_0))+\frac 12 x_2^2R_2(y,\eta_0)+\mathcal{H}_{j\geq 5}\,.
\end{multline}
We set $R(x,y,\eta)=(R)_0+(R)_1+(R)_2+(R)_3+(R)_4+\mathcal{H}_{j\geq 5}$, with $(R)_j$ the homogeneous term of order $j$ in $R(x,y,\eta)$; from \eqref{R} and \eqref{eta012} we get $(R)_0=R_0(y,\vartheta+\nabla_yB_0)$, $(R)_1=0$, and
\begin{align}\label{R2}
(R)_2= & \eta_2\nabla_{\eta}R_0(y,\eta_0)+x_2R_1(y,\eta_0)\,,\quad \eta_0=\vartheta+\nabla_yB_0(y,\vartheta)\,,\\
\label{R3}
(R)_3= & \eta_3\nabla_{\eta}R_0(y,\eta_0)+x_3R_1(y,\eta_0)\,,\\
\label{R4}
(R)_4= & \eta_4\nabla_{\eta}R_0(y,\eta_0)+\frac 12 \eta_2^2\nabla^2_{\eta,\eta}R_0(y,\eta_0)+x_2\eta_2\nabla_{\eta}R_1(y,\eta_0)+\frac 12x_2^2R_2(y,\eta_0)\,.
\end{align}
Recall from \eqref{eq:R01} 
that we had set $R_0(y,\eta)=R(0,y,\eta)=|\eta|^2+O(y)$ and $R_1(y,\eta)=\frac{\partial R}{\partial x}(0,y,\eta)$, $q(\eta):=R_1(0,\eta)$. From \eqref{eqiffchi} it follows that for $(x,y,\vartheta,\varrho,\Xi)\in\mathcal{L}$ we must have 
\begin{equation}\label{eqDelta}
\xi^2+R(x,y,\eta)=1\,.
\end{equation}
On $\mathcal{L}$ we have obtained $x|_{\mathcal{L}}$ as a sum of homogeneous terms of the form \eqref{x2x3}, \eqref{x4} which in turn has allowed to do the same for $\xi^2$ and $R(x,y,\eta)$; it remains to get homogeneous terms of order $j$ for $j\in\{0,2,3,4\}$ in \eqref{eqDelta} to obtain a system of equations whose unknown are the coefficients of $A_{\Gamma}$ and $B_{\Gamma}$ (notice there are no terms of order $j=1$). First, we have
$R_0(y,\vartheta+\nabla_yB_0)=1$,
(homogeneous terms of order $0$). 
The next lemma easily follows from solving transport equations:
\begin{lemma}\label{lemB0}
For $\eta\in \mathbb{R}^{d-1}\setminus 0$, there exists an unique function $\phi(y,\eta)$, homogeneous of degree $1$ in $\eta$, solving the eikonal equation $R_0(y,\nabla_y\phi)=|\eta|^2$, with $\phi_{|y\cdot\eta=0}=0$.
\end{lemma}
We then obtain $B_{0}$: from $R_0(y,\eta)=|\eta|^2+O(y)$, we have $\phi(y,\eta)=y\cdot \eta(1+O(y))$, and
\begin{equation}\label{eqB0}
B_0(y,\vartheta)=\phi(y,\vartheta)-y\cdot\vartheta\,.
\end{equation}
As a consequence we have 
$B_0(0,\vartheta)=0$ and $\nabla_y B_0(0,\vartheta)=0$. Back to \eqref{eqDelta}, consider homogeneous terms of order $2$ such that $(\xi^2)_2+(R)_2=0$; using \eqref{xi22} and \eqref{R2}, this translates into 
\begin{equation}\label{eqorder2}
(1+\Lp)^2\Xi^2+(\varrho-1)(\vartheta+\nabla_yB_2)\cdot \nabla_{\eta}R_0(y,\eta_0)-\frac{(\Xi^2+2(\varrho-1))}{q(\vartheta)(1+\Lp)}R_1(y,\eta_0)=0\,.
\end{equation}
We first match coefficients of $\Xi^2$ in \eqref{eqorder2}: 
$(1+\Lp)^2=\frac{R_1(y,\eta_0)}{q(\vartheta)(1+\Lp)}$
which yields $1+\Lp=\Big(\frac{R_1(y,\vartheta+\nabla_yB_0)}{q(\vartheta)}\Big)^{1/3}$. As $R_1(0,\vartheta)=q(\vartheta)$ and $\nabla_yB_0(0,\vartheta)=0$, we obtain $\Lp(0,\vartheta)=0$. We now match coefficients of $\varrho-1$, 
\begin{equation}\label{eqB2tr}
(\vartheta+\nabla_yB_2)\cdot \nabla_{\eta}R_0(y,\eta_0)=\frac{2R_1(y,\eta_0)}{q(\vartheta)(1+\Lp)}\,,
\end{equation}
which is a linear transport equation for $B_2(y,\vartheta)$ and we can take $B_{2|y.\vartheta=0}=0$ : at $y=0$, the transport field is $2\vartheta\cdot \nabla_y$, as $\nabla_{\eta}R_0(y,\eta_0)=2\eta_0+O(y)$, $\eta_0=\vartheta+\nabla_yB_0(y,\vartheta)$ and $\nabla_yB_0(0,\vartheta)=0$.
The first three equations involving $B_0,B_2$ and $\Lp$ can be solved explicitly using only homogeneous contributions up to order $2$. We consider now homogeneous terms of order $3$ in \eqref{eqDelta}, i.e. $(\xi^2)_3+(R)_3=0$. This yields, using \eqref{xi23} together with \eqref{R3}, \eqref{eta34} and \eqref{x2x3}
 \begin{multline}\label{eqorder3}
 2(1+\Lp)\Xi\Big(-2\alpha \frac{(\Xi^2+2(\varrho-1))}{q(\vartheta)(1+\Lp)}+\beta(\varrho-1)+\mu\Xi^2\Big)\\
 -\Xi\frac{(\Xi^2+2(\varrho-1))}{q(\vartheta)(1+\Lp)}\nabla_y\Lp\cdot \nabla_{\eta}R_0(y,\eta_0)+\frac{2\mu\Xi(\Xi^2+2(\varrho-1))}{q(\vartheta)(1+\Lp)^2}R_1(y,\eta_0)=0\,,
 \end{multline}
in which there are only $\Xi^3$ and $\Xi(\varrho-1)$ terms. Exactly like we did for \eqref{eqorder2}, we match coefficients for these terms separately. Unknown functions are $\alpha$, $\beta$ and $\mu$ (we already chose $\Lp$ and $B_0,B_2$.) Using that $q(\vartheta)(1+\Lp)^3=R_1(y,\eta_0)$, we get
\begin{equation}\label{formalsystem}
\left\{ \begin{array}{l}
4\mu (1+\Lp)=\frac{4\alpha}{q(\vartheta)}+\frac{\nabla_y\Lp}{q(\vartheta)(1+\Lp)}\cdot\nabla_{\eta}R_0(y,\eta_0), \text{  terms in } \Xi^3;
\\ 
(2\mu+\beta)(1+\Lp)=\frac{4\alpha}{q(\vartheta)}+\frac{\nabla_y\Lp}{q(\vartheta)(1+\Lp)}\cdot\nabla_{\eta}R_0(y,\eta_0)\,,
\text{ terms in } \Xi(\varrho-1).
 \end{array} \right.
\end{equation}
The last system of two equations and three unknown functions implies $\beta=2\mu$, and provides a relation between $\alpha$ and $\mu$ (given by the first equation in \eqref{formalsystem}). Moreover, at $y=0$, we have $q(\vartheta)=R_1(0,\vartheta)$ and $3\nabla_y\Lp(0,\vartheta)=\frac{\nabla_yR_1(0,\vartheta)}{q(\vartheta)}$.
We summarize what we obtained so far:
\begin{prop}\label{propimpformgamma}
The phase function $\Gamma(x,\cdot)=B_{\Gamma}+x A_{\Gamma}$ is such that, near the glancing set $\mathcal{GL}$,
\begin{equation}\label{AGam}
\left\{ \begin{array}{l}
A_{\Gamma}(x,y,\Xi,\Theta)= \Xi\Lp(y,\vartheta)+\alpha(y,\vartheta)x+\mu(y,\vartheta)(\Xi^2+|\Theta|^2-1)+\mathcal{H}_{j\geq3}\,,\\
B_{\Gamma}(y,\Theta)=B_0(y,\vartheta)+(\varrho-1)B_2(y,\vartheta)+\mathcal{H}_{j\geq3}\,,
 \end{array} \right.
\end{equation}
where $\vartheta=\Theta/|\Theta|$, $\varrho=|\Theta|$ and the functions $B_0$ $B_{2}$ were defined in \eqref{eqB0}, \eqref{eqB2tr}. Morover, 
\[
\beta(y,\vartheta)=2\mu\,, \quad \text{ and } \quad \mu(y,\vartheta)=\frac{\alpha}{q(\vartheta)(1+\Lp)}-\frac{\nabla_{\eta}R_0(y,\vartheta+\nabla_yB_0)}{4q(\vartheta)}\cdot \nabla_y\Big(\frac{1}{1+\Lp}\Big)\,.
\]
We also have $\Lp(0,\vartheta)=0$, $B_0(0,\vartheta)=0$, $\nabla_yB_0(0,\vartheta)=0$, $B_2(0,\vartheta)=0$ and $\nabla_yB_2(0,\vartheta)=0$, $B_0$ (resp. $B_2$) is homogeneous of order $1$ (resp. of order $0$) in the second variable.
\end{prop}
\begin{rmq}
The restriction of $\canonchi_M$ to $\mathcal{GL}_M$ is given by $x=\xi=0$ and $Y_{M}=y+\nabla_{\Theta}B_{\Gamma}(y,\Theta)|_{\varrho=1}$, $\eta=\Theta+\nabla_yB_0(y,\Theta)(y,\Theta)|_{\varrho=1}$.
 It preserves the canonical foliation: $\canonchi_M\Big(\{Y_{M}=Y_0+2s\vartheta_0, \Theta=\vartheta_0,\vartheta_0^2=1\}\Big)$ is an integral curve of $H_{R_0}$ on $R_0=1$.
\end{rmq}
To complete the proof of Proposition \ref{lemgamma}, we are to identify homogeneous terms of order 
at least $ 4$ for $A_{\Gamma}$ and $B_{\Gamma}$ in \eqref{eqiffchi}. One gets a cascade of linear equations (similar to those obtained by identifying homogeneous terms of order $0,1,2,3$
in order to chose $\alpha,\beta,\mu$) which may be solved by induction.
We only do it for homogeneous terms of order $4$ in \eqref{eqDelta} as an example: 
let $(\xi^2)_4+(R)_4=0$ and match coefficients for $\Xi^4$ and $\Xi^2(\varrho-1)$ using \eqref{xi24}, \eqref{R4}, \eqref{eta34}, \eqref{eta012}, \eqref{x2x3},
\begin{align}
\frac{4\alpha_1}{q(\vartheta)}-2(1+\Lp)\mu_1= &  \,9\mu^2+\frac{R_2(y,\eta_0)}{2q^2(\vartheta)(1+\Lp)^2}+\frac{\nabla_{\eta}R_0(y,\eta_0)}{q(\vartheta)(1+\Lp)}\cdot \nabla_y\Big(\frac{\nabla_{\eta}R_0(y,\eta_0)}{4q(\vartheta)}\cdot \nabla_y(\frac{1}{(1+\Lp)})\Big)\\
\frac{4\alpha_1}{q(\vartheta)}-(1+\Lp)\beta_1= &\, 10\mu^2+\frac{R_2(y,\eta_0)}{q^2(\vartheta)(1+\Lp)^2}-\frac{\nabla_{\eta}R_1(y,\eta_0)}{2q(\vartheta)(1+\Lp)}\cdot (\vartheta+\nabla_yB_2)\\
& \quad\quad\quad\quad\quad\quad\quad\quad\quad\quad\quad\quad  {}+\frac{2\nabla_{\eta}R_0(y,\eta_0)}{q(\vartheta)(1+\Lp)}\cdot \nabla_y\Big(\frac{\nabla_{\eta}R_0(y,\eta_0)}{4q(\vartheta)}\cdot \nabla_y(\frac{1}{(1+\Lp)})\Big)\,,
\end{align}
where both RHSs only contain $\mu^2$ and known function such as $\Lp$, $B_0$, $B_2$, $R_0$, $R_1$, $R_2$. 
\begin{rmq}
If we require $\beta_1=2\mu_1$ (which implies $A_3=\Xi(x\alpha_1+(\Xi^2+2(\varrho-1))\mu_1)$), then by difference between the last two equations $\mu^{2}$ is uniquely determined (and therefore $\alpha$ using Proposition \ref{propimpformgamma}). Similarly, we could ask for
\[
A_{2j}=x\alpha_{j}+(\Xi^2+2(\varrho-1))\mu_{j},\quad\quad A_{2j+1}=x\Xi\alpha_{j-1}+(\Xi^2+2(\varrho-1))\Xi\mu_{j-1}, \quad j\geq 2\,,
\]
but this would determine $B_{2j}$ in an unique way. Indeed, for $j=1$ requiring $\beta_1=2\mu_1$ provides a unique $\mu$ and then a unique $\alpha$. Moreover, identifying the coefficients of $(\varrho-1)^2$ in homogeneous terms of degree $4$ in \eqref{eqDelta} does not involve $\alpha_1,\beta_1,\mu_1$, but only $\nabla_yB_4$ (with $\nabla_{\eta}R_0(y,\eta_0\neq 0$) and $\alpha,\mu,\Lp$ (this is the first occurrence of $B_4$.) Indeed, $\eta_4$ does not contain $\alpha_1,\beta_1,\mu_1$, which appear only in $(\xi^2)_4$ (with $\Xi^4$ or $\Xi^2(\varrho-1)$). Therefore for given $\alpha,\mu$, this equation (obtained by identifying coefficients of $(\varrho-1)^2$) determines $\nabla_y B_4$ (and therefore $\nabla_yB_4(0,\vartheta)\neq 0$ unlike for $B_0,B_2$.)
\end{rmq}
\begin{rmq}
That the formal expansion is not uniquely defined reflects that the
group of canonical transformations which preserves the model $\{X_{M}=0, \Xi^2+\vert\Theta\vert^2+X_{M}q(\Theta)=1\}$ is not trivial.
\end{rmq}
\subsection{Equivalence of phase functions for $G(x,y,\eta,\omega)$}
Both phases $\psi+s^3/3-s\zeta$ (from \eqref{eq:defG} in Theorem \ref{thmMelrose}) and $y\cdot \eta+\sigma^3/3+\sigma(xq^{1/3}(\eta)-\omega)+\tau_{q}\Gamma(x,y,\sigma q^{1/3}(\eta)/\tau_{q},\eta/\tau_{q})$ (from \eqref{eq:Gosc}) define the same Lagrangian. We now explain how they are related. From a classical result (see \cite{ChFrUr}) on the normal form of integrals whose phases have degenerate critical points of order $2$, we have:
\begin{lemma}\label{lemphaseG}
Let $\phi(x,y,\theta,\alpha,\sigma)=\sigma^3/3+\sigma(xq^{1/3}(\theta)-\alpha)+x\tau_{q} A_{\Gamma}(x,y,\sigma q^{1/3}(\theta)/\tau_{q},\theta/\tau_{q})$. There exists a unique map $\sigma\rightarrow s$ and $\Upsilon(x,y,\theta,\alpha)\in C^{\infty}$ such that $\phi(x,y,\theta,\alpha,\sigma)=s^3/3-s\zeta(x,y,\theta,\alpha)+\Upsilon(x,y,\theta,\alpha)$
 and $\frac{ds}{d\sigma}\notin\{0,\infty\}$. Let $w:=(x,y,\theta,\alpha)$ and denote $\sigma_0(w)$ the unique solution to $\partial^2_{\sigma,\sigma}\phi(w,\sigma)=0$; then the two saddle points of $\phi$, that we denote $\sigma_{\pm}(w)$, correspond to the critical points $s_{\pm}(w):=\pm\sqrt{\zeta(w)}$ and such that $\sigma_{\pm}(w)=\sigma_0(w)\pm \sqrt{\zeta(w)}k(\pm\sqrt{\zeta(w)},w)$, with $k(u,w)= 1+\sum_{j\geq 1}k_j(w)u^j$, where $k_j$ are smooth functions of $w$.
Moreover, 
\begin{equation}\label{formephaseUpsilon}
\frac 34 \zeta^{3/2}(w)=\phi(w,\sigma_-(w))-\phi(w,\sigma_+(w)),\quad \Upsilon(w):=\frac 12\Big(\phi(w,\sigma_+(w))+\phi(w,\sigma_-(w))\Big)\,.
\end{equation}
\end{lemma}
\begin{cor}\label{corUpsi}
We may write $\psi(x,y,\theta,\alpha)=y\cdot \theta+\tau_{q}(\alpha,\theta)B_{\Gamma}(y,\theta/\tau_{q})+\Upsilon(x,y,\theta,\alpha)$, with
\begin{equation}
  \label{eq:Ups}
\Upsilon(x,y,\theta,\alpha)
= -x\big(x \mu(y,\theta/|\theta|)q(\theta)/\tau_{q}(\alpha,\theta)(1+\Lp(y,\theta/|\theta|)+\mathcal{H}_{j\geq 4}\big)-k_1\zeta(\zeta+x\mathcal{H}_{j\geq 2}+\zeta\mathcal{H}_{j\geq 2})\,.
\end{equation}
\end{cor}
\begin{proof}
Compute $\sigma_0$: as $\partial^2_{\sigma,\sigma}\phi(w,\sigma)=2\sigma+2x\big(\mu(y,\theta/|\theta|)q^{2/3}(\theta)/\tau_{q}(\alpha,\theta)+\mathcal{H}_{j\geq 1}\big)=0$, where $\mathcal{H}_{j=1}$ contains only multiples of $\sigma$, we get 
$\sigma_0(x,y,\theta,\alpha)= -x\big(\mu(y,\theta/|\theta|)q^{2/3}(\theta)/\tau_{q}(\alpha,\theta)+\mathcal{H}_{j\geq 2}\big)$,
where all $\mathcal{H}_{j\geq 2}$ in the RHS come with weights $x$ and $\alpha$. We develop $\phi(w,\sigma)$ near $\sigma=0$ and replace $\sigma$ by $\sigma_{\pm}(w)$: as $\frac 12(\sigma_++\sigma_-)=\sigma_0+k_1(w)\zeta(w)(1+O(\zeta(w)))$, using  \eqref{formephaseUpsilon} yields \eqref{eq:Ups}. Moreover, $\psi$ contains $B_{\Gamma}$ as it does not depend on $\sigma$. \end{proof}

\section{Index of notations}
Below is a commented list of the main notations, with reference to their very first occurence. 
\subsection{General notations (used consistently throughout the paper)}
\begin{itemize}
\item $(\Omega,g)$ = $d$ dimensional manifold, $d\geq 2$, $\Delta_g$ its Laplace Beltrami operator, section \ref{intro}.
  \item $\upgamma(d)$ encodes the loss in Strichartz estimates w.r.t. the case without boundary, Theorem \ref{thStri}.
\item $(x,y)$, boundary normal coordinates; $t$ the time variable; locally, $\Omega=\{(x,y) : x>0, y\in\mathbb{R}^{d-1}\} $, section \ref{parconstruction}.
\item $(\xi,\eta,\tau)$, dual variables: $(x,y,t,\xi,\eta,\tau)\in T^*(\Omega\times\mathbb{R}_t)$.
For $(x,y)$ near $(0,y)$, the metric is $\xi^2+R(x,y,\eta)$. In a neighborhood of $(0,0)\in \partial\Omega$, 
\[
  R_0(y,\partial_y):=R(0,y,\eta)
  \,, \quad R_1(y,\eta):=\partial_x R(0,y,\partial)
  \,,
\]
section \ref{parconstruction} and \eqref{eq:R01}.
\item $\Delta_M=\partial^2_x+\sum_j\partial^2_{y_j}+x\sum_{j,k=1}^{d-1}R^{j,k}_1(0)\partial_{y_j}\partial_{y_k}$: model Laplace operator, \eqref{eq:LapM};  Multipliers 
 \[
q(\eta)=\sum_{j,k=1}^{d-1}R^{j,k}_1(0)\eta_j\eta_k,\quad \tau_q(\omega,\eta):=\sqrt{|\eta|^2+\omega q(\eta)^{2/3}}\,.
 \]
\item $\{e_k(x,\eta)\}_{k\geq 0}$: in the spectral decomposition of $-\Delta_M$ (section \ref{sss231}), an explicit orthonormal base of eigenfunctions associated to eigenvalues $\lambda_k(\eta)$, where
\[
  \lambda_k(\eta)=|\eta|^2+\omega_k q(\eta)^{2/3}=\tau_q^2(\omega_k,\eta).
\]
\item $\{-\omega_k\}_{k\geq 0}$: zeros of the Airy function in decreasing order. Everywhere in the paper $\omega >1$ and serves as a substitute to the $\xi$ variable: if $Q_{y}$ is the differential operator with symbol $q$, $\alpha=h^{2/3}\omega$ quantizes the operator $x-Q_{y}^{-1}\partial_{x}^{2}$.
\item $s$, $\sigma$: integration variables in Airy type oscillatory integrals, \eqref{eq:Gosc}, \eqref{eq:15}.
\item  $(a,b)$ coordinates of the source point, mostly set with $b=0$, Theorem \ref{disper}.
\item  $h\in (0,1)$ (Theorem \ref{disper}), $\gamma\in (0,1)$ with $1/\gamma \in 2^{\mathbb{N}}$ (Section \ref{sectcardN1}): small parameters.
    \item $\lambda=a^{3/2}/h$, $\lambda_{\gamma}=\gamma^{3/2}/h$: large parameters, Section \ref{sectcardN1}.
    \item $(X,Y,T)$, rescaled coordinates (using some combination of $a$, $h$ or $\lambda$, $\lambda_{\gamma}$ as rescaling parameters), Section \ref{sectcardN1}.
    \item $\Sigma$, $S$, $A$: rescaled variables in Airy-type oscillatory integrals, Section \ref{sectcardN1}.
      \item  $\omega$: \eqref{eq:LapM}, parameter and integration variable, successively rescaled to $\alpha$ (\eqref{eq:Phi}) and then $A$ (Section \ref{sectcardN1}). Stationary phases in oscillatory integrals are performed with respect to $\alpha$, $\sigma$, $s$ or their rescaled versions $A$, $\Sigma$, $S$, less frequently $\eta$, with a large parameter being $1/h$, $\lambda$ or $\lambda_{\gamma}$.
    \item $\theta$: (Section \ref{parconstruction}) rescaled $\eta$, near $\mathbb{S}^{d-1}$, and $\rho=|\theta|$, $\vartheta=\theta/\rho$.
\end{itemize}
\subsection{Localisations in phase space} 
\begin{itemize}
\item We localize $\tau_q(\omega,\eta)\sim 1/h$ and $|\eta|\sim 1/h$.
For small $x$, this corresponds to large frequencies $-\Delta\sim -\Delta_M \sim 1/h^2$ and "tangent" directions: the number of reflections on the boundary may be quite large.
\item A further localization is to values $\omega/|\eta|^{2/3}\sim \gamma$. Informally, it relates to  the angle of incidence at the boundary for a ray starting tangentially from $(\gamma,0)$.  %
\item Cut-offs : $\varkappa\geq 0$ is a cut-off function in $ C_{0}^{\infty}(\mathbb{R}^{m})$ with $m=1$ or with $m=d-1$, localizing around  a small neighbourhood of $1$ (for $d=1$), or near $\mathbb{S}^{m-1}$  for $m=d-1$; $\varkappa_{1}$ is a 1-d $\varkappa$. We also have $\chi^{\flat}\in C^{\infty}(\mathbb{R})$ such that $\chi^{\flat}=1$ on $(-\infty, 1]$ and $\chi^{\flat}=0$ on $[2,\infty)$ and $\chi^{\sharp}=1-\chi^{\flat}$. Also, $\chi_0\in C^{\infty}_0(\mathbb{R})$ is supported is a small, fixed neighborhood of $0$.
\end{itemize}
\subsection{Operators, kernels and quasimodes}
\begin{itemize}
\item $G(x,y,\eta,\omega)$: a quasimode, \eqref{eq:eqG} or \eqref{eq:Gosc}; satisfies \eqref{eq:eqG}  $-\Delta G=\tau_q^2 G+O_{C^{\infty}}(\tau_q^{-\infty})$.
\item $K_{\omega}(f)(t,x,y)$: operator related to wave flow \eqref{eq:Kequiv}, acting on smooth $f$.
\item $J(f)(x,y)$: Fourier integral operator, \eqref{eq:J}.
\item $\mathcal{P}_{h,a}(t,x,y)$, \eqref{eq:Prond} and \eqref{eq:Prond2}, our parametrix for the wave equation
\item $V_{N}$: a wave in the expansion over $N$ of $\mathcal{P}_{h,a}$, \eqref{eq:newVN}.
\item $V_{N,\gamma}$: further localized with $\chi_{1}(\omega/(\gamma |\eta|^{2/3}))$, \eqref{eq:newVNgam}.
  \item $\mathcal{P}_{h,a,\gamma}$: the corresponding sum over $N$, \eqref{eq:newProndgam}. 
  \item $\mathcal{E}_M(\cdot,\omega_k)$:  galery modes for the model Laplacian $\Delta_{M}$, \eqref{defEmathcalM}.
  \item $e_{k}(x,\eta)$: eigenfunctions of $\mathcal{F}_{y}(\Delta_{M})$, \eqref{eig_k}.
    \item  $e(x,y,\eta,\omega)$: quasimodes for $\Delta_{g}$, \eqref{defeomkgen}.
\item $g_{h,a}$, $g_{h,a,j}$, $j\in\{1,2\}$: functions to serve as arguments to $J$ and $K_{\omega}$ to construct a suitable smoothed out Dirac data, \eqref{eq:gosc}, Propositions \ref{propdataapetitpetit1} and \ref{propdataapetitpetit2}.
\item $F_{\omega_k}(g)(x,y)$: operator acting on functions $g\in L^2(\mathbb{R}^{d-1})$, average (with density $\hat g(\eta)$) of quasimodes $e(x,y,\eta,\omega_{k})$, \eqref{defFk}.
\item $\Lo(f)(y)$: operator actiong on $f\in L^2(\mathbb{R}^{d-1})$ which allows to "get rid" of the term $B_0$ in the phase of $e(x,y,\eta,\omega)$,  \eqref{def:L}.
\item $\tilde F_{\omega_k}(f)(x,y)=F_{\omega_k}\circ \Lo(f)(x,y)$: its main property is that it can be inverted.
\end{itemize}
\subsection{Phase functions and canonical transformation} 
\begin{itemize}
\item $\zeta(x,y,\eta,\omega)$, $\psi(x,y,\eta,\omega)$ : the phase functions of $G(x,y,\eta,\omega)$ from Theorem \ref{thmMelrose}.
\item $\Sigma_0$: \eqref{eq:18}, a neighborhood of a glancing point in the model case.
  \item $ \canonchi_M$, \eqref{eq:melrose},the canonical transformation defined in a conic neighborhood of $\Sigma_0$ mapping the model case (variables $(X_{M},Y_{M}, \Xi, \Theta)$) to the general case (variables $(x,y,\xi,\eta)$).
  \item $\varphi_{\Gamma}(x,y,\Xi,\Theta)=x\Xi+y\Theta+\Gamma(x,y,\Xi,\Theta)$: Proposition \ref{lemgamma}, the generating function for $\canonchi_M$, with $\Gamma(x,y,\Xi,\Theta)=B_{\Gamma}(y,\Theta)+xA_{\Gamma}(x,y,\Xi,\Theta)$ from \eqref{eq:GAB}.
    \item $A_{\Gamma}$, $B_{\Gamma}$: phase functions that are formal series \eqref{formalseriesBA} from Section \ref{Secgamma}, defined near $\mathcal{GL}=\{x=0,\Xi=0,\varrho-1=0\}$ and for $(y,\vartheta)$ near $\{0\}\times \mathbb{S}^{d-1}$, where $\Theta=\varrho\vartheta$. Their explicit form is given in \eqref{AGam}.
\item $\mathcal{H}_{\geq j}=\{ F \text{ such that } F=\sum_{k\geq j} F_k, \text{ with } F_k \text{ homogeneous of degree } k\geq 1\}$, where a monomial of the form $x^a(\varrho-1)^b\Xi^c$ is homogeneous of degree $k$ if $c+2(a+b)=k$.

\item $\Lp(y,\vartheta)$ (which defines $A_1=\Xi \Lp$), $\alpha(y,\vartheta)$, $\beta(y,\vartheta)$, $\mu(y,\vartheta)$ (which define $A_2=\alpha(y,\vartheta)x+\beta(y,\vartheta)(\varrho-1)+\mu(y,\vartheta)\Xi^2$) so that $A_{\Gamma} =A_1+A_2+\mathcal{H}_{\geq 3}$: other functions related to $\Gamma $ from Section \ref{Secgamma}.
  \item $B_{\Gamma}(y,\Theta)=B_0(y,\vartheta)+(\varrho-1)B_2(y,\vartheta)+\mathcal{H}_{j\geq3}$, whose properties are stated in Proposition \ref{propimpformgamma}.
\item $\mathcal{L}=\{(x,y,\vartheta,\varrho,\Xi), p_{M,2}(X_M,Y_M,\Xi,\Theta)=0\}$ defined in Lemma \ref{l5.2}.
\item $\WF_{{h}}$, the semiclassical wavefront set (see \cite{ZworBook}).
\item $\Phi_{N,a,\gamma}$: the phase function of $V_{N,\gamma}$ defined in \eqref{PhiNagamma}.
\item $\Phi^M_{N,a,\gamma}$: the phase function in the model case $(\Omega,\Delta_M)$. (In general, a notation with an additional $M$ indicates that we consider the model situation.)
\item $\tilde\Phi_{N,a,\gamma}$, a rescaled $\Phi_{N,a,\gamma}$ (see \eqref{deftildephiNagamma}) with $\sigma=\sqrt{\gamma}|\theta|^{1/3}\Sigma$, $\alpha=\gamma |\theta|^{2/3} A$, $s=\sqrt{\gamma}|\theta|^{1/3}S$.
  \item $\Psi_{N,a,\gamma}$: relabeled $\tilde\Phi_{N,a,\gamma}$ after rescaling $x=\gamma X$, $t=\sqrt{\gamma}T$, $y=\sqrt{\gamma}Y$.
    
  \item $\mathcal{N}(t,x,y)$: the set of $N$ with significant contributions of the phase $\Phi_{N,a,\gamma}(t,x,y,\cdot)$, defined in \eqref{Ncal}.
  \item $\mathcal{C}_{\gamma}$: a cylinder defined by \eqref{eq:113}.
    \item $\mathcal{N}^1_{d}$: enlargement of $\mathcal{N}$ defined by $ \mathcal{N}^1_{d}(t,x,y)=\cup_{\mathcal{C}_{\gamma}(t,x,y)}\mathcal{N}(t',x',y')$ in \eqref{Ncal1}.
 
 \item  $\tilde\Gamma_{\gamma}$, $\breve \Gamma_{\gamma}$ (see \eqref{deftildeGammagam}), $\mathcal{E}_{\pm}$, $\tilde{\mathcal{E}}$ (see Lemma \ref{lemomegac}), $\tilde{\mathcal{E}}$ (see Lemma \ref{lemdetails}): remainder terms in phase functions (they do not influence the behaviour of the corresponding phase functions.)
    \end{itemize}


\def\cprime{$'$} \def\cprime{$'$}

\end{document}